\documentclass{amsart}

\newif\ifpreprint

\newcommand{\preprintonly}[1]{%
  \ifpreprint
    #1
  \else
    
  \fi
}

\usepackage{pgf,tikz,pgfplots}
\pgfplotsset{compat=1.15}
\usepackage{mathrsfs}
\usetikzlibrary{arrows}

\usepackage[T1]{fontenc}
\usepackage{microtype}
\usepackage{lmodern}
\usepackage[colorlinks=true,urlcolor=blue, citecolor=red,linkcolor=blue,linktocpage,pdfpagelabels, bookmarksnumbered,bookmarksopen]{hyperref}
\usepackage[hyperpageref]{backref}
\usepackage{amsthm} 
\usepackage{latexsym,amsmath,amssymb}

\usepackage{accents}
\usepackage{esint}

\usepackage{soul}
\usepackage{mathtools} 
\usepackage{xparse} 
\usepackage[capitalize]{cleveref}

\usepackage[shortlabels]{enumitem}

\title{Scale-invariant tangent-point energies for knots}
\author[S. Blatt]{Simon Blatt}
\address[Simon Blatt]{Paris Lodron Universit\"at Salzburg, Hellbrunner Strasse 34, 5020 Salzburg, Austria}
\email{simon.blatt@sbg.ac.at}

\author[$\Phi$. Reiter]{Philipp Reiter}
\address[Philipp Reiter]
{Chemnitz University of Technology, Faculty of Mathematics, 09107 Chemnitz, Germany} \email{reiter@math.tu-chemnitz.de}

\author[A. Schikorra]{Armin Schikorra}
\address[Armin Schikorra]{Department of Mathematics,
University of Pittsburgh,
301 Thackeray Hall,
Pittsburgh, PA 15260, USA}
\email{armin@pitt.edu}

\author[N. Vorderobermeier]{Nicole Vorderobermeier}
\address[Nicole Vorderobermeier]{Paris Lodron Universit\"at Salzburg, Hellbrunner Strasse 34, 5020 Salzburg, Austria}
\email{nicole.vorderobermeier@sbg.ac.at}

\newcommand{\n}{\mathcal{N}}

\newcommand{\N}{{\mathbb N}}

\renewcommand{\S}{{\mathbb S}}

\newtheorem{theorem}{Theorem}
\newtheorem{lemma}[theorem]{Lemma}
\newtheorem{corollary}[theorem]{Corollary}
\newtheorem{proposition}[theorem]{Proposition}

\theoremstyle{definition}
\newtheorem{definition}[theorem]{Definition}
\newtheorem{example}[theorem]{Example}

\theoremstyle{remark}
\newtheorem{remark}[theorem]{Remark}

\newcommand\diam{{\rm diam\,}}
\newcommand\dist{{\rm dist\,}}

\newcommand\lip{{\rm Lip\,}}

\newcommand\supp{{\rm supp\,}}

\newcommand{\R}{\mathbb{R}}

\newcommand{\Z}{\mathbb{Z}}

\newcommand{\brac}[1]{\left (#1 \right )}
\newcommand{\abs}[1]{\left\lvert #1 \right \rvert}

\newcommand{\barint}{
\rule[.036in]{.12in}{.009in}\kern-.16in \displaystyle\int }

\newcommand{\barcal}{\text{$ \rule[.036in]{.11in}{.007in}\kern-.128in\int $}}

\def\mvint_#1{\mathchoice
          {\mathop{\vrule width 6pt height 3 pt depth -2.5pt
                  \kern -8pt \intop}\nolimits_{\kern -3pt #1}}%
          {\mathop{\vrule width 5pt height 3 pt depth -2.6pt
                  \kern -6pt \intop}\nolimits_{#1}}%
          {\mathop{\vrule width 5pt height 3 pt depth -2.6pt
                  \kern -6pt \intop}\nolimits_{#1}}%
          {\mathop{\vrule width 5pt height 3 pt depth -2.6pt
                  \kern -6pt \intop}\nolimits_{#1}}}

\numberwithin{theorem}{section} \numberwithin{equation}{section}

\newcommand{\lap}{\Delta }
\newcommand{\aleq}{\lesssim}
\newcommand{\ageq}{\succsim}
\newcommand{\aeq}{\approx}

\newcommand{\laps}[1]{|D|^{#1}}

\newcommand{\lapla}[1]{(-\lap)^{#1}}
\newcommand{\lapin}[1]{\mathcal{I}_{#1}}

\newcommand{\M}{\mathcal{M}}
\newcommand{\p}{\mu}
\newcommand{\E}{\mathcal{E}}
\newcommand{\RZ}{\R / \Z}
\usepackage{scalerel}[2014/03/10]
\usepackage[usestackEOL]{stackengine}
\def\avint{\,\ThisStyle{\ensurestackMath{%
			\stackinset{c}{.2\LMpt}{c}{.5\LMpt}{\SavedStyle-}{\SavedStyle\phantom{\int}}}%
		\setbox0=\hbox{$\SavedStyle\int\,$}\kern-\wd0}\int}

\let\latexchi\chi
\makeatletter
\renewcommand\chi{\@ifnextchar_\sub@chi\latexchi}
\newcommand{\sub@chi}[2]{
  \@ifnextchar^{\subsup@chi{#2}}{\latexchi^{}_{#2}}%
}
\newcommand{\subsup@chi}[3]{
  \latexchi_{#1}^{#3}%
}
\makeatother
\newcommand{\eps}{\varepsilon}

\newcommand{\tp}{{\rm TP}}

\begin{document}

\begin{abstract}
We investigate  minimizers and critical points for scale-invariant tangent-point energies ${\rm TP}^{p,q}$ of closed curves. We show that a) minimizing sequences in ambient isotopy classes converge to locally critical embeddings in all but finitely many points and b) show regularity of locally critical embeddings.

Technically, the convergence theory a) is based on a gap-estimate of a fractional Sobolev spaces in comparison to the tangent-point energy. The regularity theory b) is based on constructing a new energy $\mathcal{E}^{p,q}$ and proving that the derivative $\gamma'$ of a parametrization of a ${\rm TP}^{p,q}$-critical curve $\gamma$ induces a critical map with respect to $\mathcal{E}^{p,q}$ acting on torus-to-sphere maps.
\end{abstract}

\preprinttrue

\maketitle
\tableofcontents
\section{Introduction and main results}
When modeling and simulating topological effects in the sciences like, e.g. protein knotting, one has to make a choice how to incorporate the avoidance of
interpenetration of matter, i.e. self-intersections. Either one explicitly models partial differential equations that incorporate effects such as self-repulsion through penalization. Or one constructs a comprehensive variational energy that includes self-repulsive behavior, and hopes that minimizing the energy (or following the steepest descent) delivers a realistic description. Several such self-repulsive energy functionals have been proposed and studied extensively over the last forty years\footnote{for an overview we refer the interested reader to \cite{strvdM,blatt-reiter-survey}} -- and all have one thing in common: modeling topological resistance, i.e.\  self-repulsion, they are necessarily nonlocal functionals. Consequently, questions of most central interest such as existence and regularity for minimizing configurations are very challenging.
This holds especially true for the geometrically most interesting case of \emph{scale-invariant} knot energies, i.e. energies for which a curve $\gamma$ and any scaled version $\lambda \gamma$ have the same value for any $\lambda > 0$.

\subsection*{O'Hara and M\"obius knot energies}
The first \emph{knot energies} have been introduced by Fukuhara \cite{Fukuhara1988} and O'Hara \cite{OH91,OH92,OH94}, and are known as \emph{O'Hara energies}. Let $\gamma: \R/\Z \to \R^3$ be the parametrization of a closed regular Lipschitz curve, i.e., $\gamma$ is both an immersion and an embedding. For any $\alpha p \geq 4$ and $p \geq 2$, the O'Hara energy  $\mathcal{O}^{\alpha,p}(\cdot)$ is given by
\[
 \mathcal{O}^{\alpha,p}(\gamma) := \int_{\R/\Z}\int_{\R/\Z}\brac{\frac{1}{|\gamma(x)-\gamma(y)|^\alpha}-\frac{1}{d_\gamma(\gamma(x),\gamma(y))^\alpha}}^{\frac{p}{2}}\, |\gamma'(x)|\, |\gamma'(y)|\, dx\, dy
\]
where $d_\gamma$ is the intrinsic distance on the submanifold $\gamma(\R/\Z) \subset \R^3$.

These energies are scale-invariant functionals if $\alpha p = 4$. In this scaling invariant regime, until recently, only the case of the so-called \emph{M\"obius energy} $\mathcal{O}^{2,2}$  was understood at all. This was due to the celebrated work by Freedman, He, and Wang~\cite{FHW94}. They discussed existence of minimizers within prime knot classes and established $C^{1,1}$-regularity of local minimizers. One can then bootstrap to smoothness \cite{he} and even analyticity~\cite{BV19}.\footnote{Higher regularity via bootstrapping is possible because $p=2$ implies the functional is a Hilbert-space functional; in particular, such arguments are independent of the presence of scale-invariance, see~\cite{reiter,blatt-reiter1-3,V19}.}

The techniques employed in \cite{FHW94} by Freedman et al.
crucially rely on
the M\"obius invariance of $\mathcal{O}^{2,2}$, and largely fall apart for $\mathcal{O}^{4/p,p}$ when $p \neq 2$ since M\"obius invariance does not hold anymore, cf.~\cite{BRS19}. Indeed, 
there was no progress on neither existence nor regularity of scale-invariant knot energies besides the M\"obius energy for a long time, until in the two recent works \cite{BRS16,BRS19}, three of the authors established the regularity theory for all scale-invariant O'Hara energies $\mathcal{O}^{\alpha,p}$ (for critical points and minimizers) via a new approach. Namely, they showed that critical knots $\gamma$ induce via their derivative $\gamma'$ a sort of fractional harmonic map between $\R/\Z{}\cong\S^1$ and $\S^2$. Then, extending the tools developed for fractional harmonic maps \cite{DLR11a,S15}, they obtained a regularity theory via arguments based on compensation effects and harmonic analysis.

\subsection*{Tangent-point energies for curves}
In this work we are interested in scale-invariant \emph{tangent-point} energies. As in the case of O'Hara energies, the scale-invariant situation is the most interesting and challenging one, and up to now it was completely out of reach.
Due to the lack of Möbius invariance,\footnote{
As for O'Hara energies~\cite{BRS19} one can check this assertion by numerically computing the energy of a stadion curve before and after applying a M\"obius transform.}
the geometric techniques of Freedman, He, and Wang~\cite{FHW94} cannot be applied. Let us stress that M\"obius invariance of an energy does not have any impact for applications and it might be considered a curiosity mostly of geometric-topological interest. Actually,	from the point of view of applications, one can argue that the tangent-point energies might be preferable to O'Hara energies because they are numerically simpler to compute~\cite{BRR18,bartels2018stability,BR20,RS20}, and they have a natural generalization to embedded surfaces~\cite{SvdM13} 
which seems to be more convenient than higher dimensional analogues of O'Hara-type energies, cf.~\cite{OH20,KvdM20}.

The ``classical'' tangent-point energy
has been studied first by Buck and Orloff~\cite{buck-orloff}.
It amounts to the {double} integral over the reciprocal of
\[
 R_t(x,y):=\frac{|\gamma(x)-\gamma(y)|^2}{2\left |\gamma'(x) \wedge (\gamma(x)-\gamma(y)) \right |}
\]
which is the smallest radius of {a sphere} 
passing through $\gamma(x)$ and $\gamma(y)$ while being tangential at $\gamma(x)$, see \Cref{fig:tangentpoint}. Later on, Gonzalez and Maddocks~\cite{GM99}
obtained a family of energies by taking the integrand to suitable powers.
Decoupling these powers as proposed in~\cite{BR15},
we arrive at the two-parameter family
\[
 \begin{split}
 \tp^{p,q}(\gamma) :=& \int_{\R /\Z} \int_{\R /\Z} \frac{\left |\frac{\gamma'(x)}{|\gamma'(x)|} \wedge (\gamma(x)-\gamma(y)) \right |^q}{|\gamma(x)-\gamma(y)|^p}\, |\gamma'(x)| |\gamma'(y)|\, dx\, dy\\
 =&\int_{\R /\Z} \int_{\R /\Z} \frac{\left |\gamma'(x) \wedge (\gamma(x)-\gamma(y)) \right |^q}{|\gamma(x)-\gamma(y)|^p}\, |\gamma'(x)|^{1-q} |\gamma'(y)|\, dx\, dy
 \end{split}
\]
for any embedded $\gamma \in C^2(\R/\Z, \R^3)$.
As a standing assumption, we will always restrict the parameters to satisfy $q+2\leq p<2q+1$. If $p \ge 2q+1$, the energy is infinite even for some smooth diffeomorphisms. If $p < q+2$, then $\tp^{p,q}$ is not self-repulsive.
The subfamily studied by Gonzalez and Maddocks can be recovered
by letting $p=2q$; the Buck--Orloff functional
corresponds to $\tp^{{4,2}}$.
While the O'Hara energies $\mathcal{O}^{\alpha,p}$ for $\alpha \to 0$ and $p = \frac{4}{\alpha}$ converge to Gromov's \emph{distortion} functional, cf. \cite{gromov,OH92,BRS19}, the tangent point energies $\tp^{2q,q}$ converge to Federer's \emph{reach} as $q \to \infty$, cf. \cite{Federer, GM99}. 


	\begin{figure}
		\centering
		\includegraphics[scale=0.4]{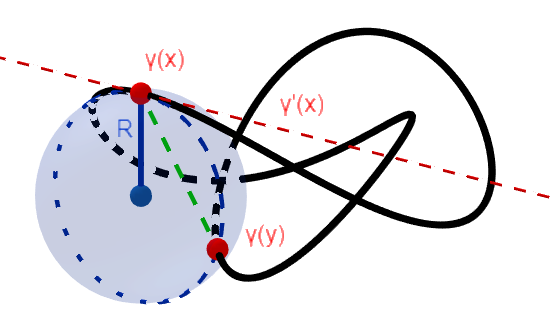}
		\caption{\label{fig:tangentpoint} The tangent-point radius $R_t$ is the radius of the smallest sphere tangential to $\gamma$ at $\gamma(x)$ and traversing $\gamma(y)$. It
		tends to zero when $\gamma(x) \to \gamma(y)$ while $\gamma(x)$ and $\gamma(y)$ belong to two different strands of $\gamma$. Its reciprocal converges to the local curvature as $y\to x$.}
	\end{figure}

Strzelecki and von der Mosel \cite{SvdM12} obtained the first {and so far only} fundamental result concerning the scale-invariant case $p=q+2$. They showed in particular that the images of curves with finite $\tp^{q+2,q}$-energy form a topological one-manifold. However, this could be a nonsmooth object, e.g., a non-differentiable curve, see \Cref{ex:nonsmooth} -- or even worse: 
a doubly-traversed line which has zero energy, see \Cref{ex:weirdfiniteenergy}! So there is an issue with even defining the notion of \emph{minimizing} embedded curves of the tangent-point energies. While the energy of the doubly-traversed line is zero and thus the global minimizer, it is certainly not a smooth manifold and {therefore should not count as an acceptable minimizer}. Let us remark that none of this was an issue for O'Hara energies which would be infinite on any periodic parametrization of a straight segment -- \emph{the tangent point energies are more extrinsic than the O'Hara energies}.

\subsection*{Main results}
With the example of the doubly-traversed line in mind, in order to discuss minimizers in the class of knots (i.e.\ closed embedded curves), we  restrict our interest to {those} curves which appear as limits of diffeomorphisms.

Let us introduce the localized energy for $A \subset \R/\Z$ by
\[
 \tp^{p,q}(\gamma;A) := |\gamma'|^{2-q}\int_{A} \int_{A} \frac{\left |\gamma'(x)\wedge (\gamma(x)-\gamma(y)) \right |^q}{|\gamma(x)-\gamma(y)|^p}\, dx\, dy
\]
where we assume $\gamma$ to be parametrized by arclength, {$|\gamma'| \equiv const$.}

Following the spirit of an analogue strategy for Willmore surfaces~\cite[Definition~I.1]{R08}, we
introduce the following terminology. 
\begin{definition}[Homeomorphisms with locally small tangent-point energy]\label{def:homeotp}
A Lipschitz map $\gamma: \R /\Z \to \R^3$
is called a \emph{homeomorphism with locally $\eps$-small tangent-point energy at $x\in\R/\Z$} if there exists an open interval $B(x,r) \subset \R/ \Z$ and a sequence of $C^1$-homeomorphisms $\gamma_k: \R/\Z\to \R^3$, $|\gamma_k'| \equiv c >0$, such that 
\begin{enumerate}
 \item $\gamma_k$ converges uniformly to $\gamma$ on $\R / \Z$,
 \item $\sup_k \tp^{p,q}(\gamma_k;\R / \Z) < \infty$, and
 \item $\sup_k \tp^{p,q}(\gamma_k;B(x,r)) < \eps$ for some $r>0$.
 \end{enumerate}
\end{definition}

Our first main result states that \emph{sequences of curves with uniformly bounded tangent-point energy converge to homeomorphisms with locally $\eps$-small tangent-point energy outside of at most finitely many points}.
More precisely, we will prove the following assertion.

\begin{theorem}\label{th:weaklimitareweakimm}
Let $p= q+2$, $q>1$, $\Lambda > 0$, and $\eps > 0$. Then there exists an integer $K = K(q,\eps,\Lambda)$ such that
any sequence $\left(\gamma_k\right)_{k\in\N} \subset C^1(\R/\Z, \R^3)$ of closed embedded curves
with
\[
 \sup_{k \in \N} \tp^{q+2,q}(\gamma_k) < \Lambda
\]
converges -- after possibly translating, rescaling, and reparametrizing $\gamma_k$ and passing to a subsequence -- to a Lipschitz map $\gamma: \R / \Z \to  \R^3$
with the following properties.
\begin{itemize}
 \item \emph{embeddedness:} $\gamma$ is a bi-Lipschitz homeomorphism.
 \item \emph{arclength parametrization:} $|\gamma'(x)| = 1$ for a.e. $x \in \R / \Z$.
 \item \emph{lower semicontinuity:} $\tp^{q+2,q}(\gamma) \leq \liminf_{k \to \infty} \tp^{q+2,q}(\gamma_{k})$.
 \item \emph{subcritical Sobolev-space:} $\gamma \in W^{1+s,q}(\R/\Z,\R^3)$ for any $0<s < \frac{1}{q}$.
 \end{itemize}
 \emph{Moreover -- %
 and this is crucial -- we 
 locally control the critical Sobolev norm outside of a singular set $\Sigma$ containing at most $K$ points:} for any $x_0 \in (\R / \Z) \backslash \Sigma$ there exists some $r_{x_0} > 0$ such that 
 \begin{itemize} 
 \item $\sup_{k} \tp^{q+2,q}(\gamma_k,B(x_0,r_{x_0})) < \eps$,
 \item $\gamma_k$ weakly converges to $\gamma$ in the Sobolev space $W^{1+\frac{1}{q},q}(B(x_0,r_{x_0}),\R^3)$,
 \item and \quad
 \(\displaystyle
 \sup_{k} [\gamma_k]_{W^{1+\frac{1}{q},q}(B(x_0,r_{x_0}))}^q \aleq \sup_{k} \tp^{q+2,q}(\gamma_k,B(x_0,r_{x_0})) < \eps 
 \).
 \end{itemize}

\end{theorem}
\begin{remark}[Pull-tight]\label{re:exsingularnonempty}
In analogy to harmonic maps and 
O'Hara energies, we expect
examples for $p=q+2$ where the singular set $\Sigma$ in \Cref{th:weaklimitareweakimm} is nonempty.
The idea is as follows.
Take a (closed) smooth curve containing a piece of a straight
line and replace the latter by a small \emph{nontrivially} knotted arc.
Shrinking this arc to zero (``pull-tight'') produces a sequence of curves
with uniformly bounded energy, cf.~\cite[Thm.~3.1]{OH92}.
In the limit curve we observe a change of topology
along with a loss of energy, ruling out strong convergence
with respect to the Sobolev norm. 

{For the M\"obius energy $\mathcal{O}^{2,2}$, one can use the M\"obius invariance to rewrite a minimizing sequence into one that avoids ``concentration of topology in a small set'' see \cite{FHW94} for this notion, and \cite{N18} for a survey. Even more is true in this special case: the M\"obius energy can be decomposed into several M\"obius invariant energies that control different features, cf. \cite{Nagasawa1,Nagasawa2,Nagasawa3,Nagasawa4,Nagasawa5,Nagasawa6}}.
\end{remark}

\begin{remark}
In light of two-dimensional analogues \cite{H57,MS95}, one might conjecture that
the limit curve $\gamma$ from Theorem~\ref{th:weaklimitareweakimm}
does not necessarily belong to a Sobolev space on $\R/\Z$,
but this is certainly not clear to us.
\end{remark}

\Cref{th:weaklimitareweakimm} 
is much simpler to prove if $p>q+2$, cf.~\cite{blattTP,BR15}.  
In our limiting range $p=q+2$,
\Cref{th:weaklimitareweakimm} can be understood as a one-dimensional counterpart to the fundamental theorem of M\"uller and Sverak~\cite{MS95}, who showed that surfaces with small second fundamental form  w.r.t.\ the $L^2$-norm can be conformally parametrized. We also refer to earlier works by Toro \cite{T94,T95} 
as well as~\cite{KL12,LiLuoTang13,KS12,R15,R16,S18}. Indeed,
\Cref{th:weaklimitareweakimm} is strongly inspired by  the ``weak closure theorem'' for the Willmore energy, see \cite[Theorem~3.55]{R16}.

As a particular consequence of \Cref{th:weaklimitareweakimm}, homeomorphisms with locally small tangent-point energy as described in \Cref{def:homeotp} appear as limits of smooth minimizing sequences (minimizing, e.g., with respect to isotopy classes, see~\Cref{s:topology}). Since the convergence of minimizing sequence is only weak, in general the limits of minimizing sequences may not be minimizers -- indeed they may not belong to the same isotopy class due to bubbling effects (also called pull-tights). This is why we introduce, once more in analogy to Willmore surfaces \cite[Definition I.2]{R08}, the notion of local critical embeddings.

\begin{definition}[Locally critical embedding]\label{def:wcp}
A homeomorphism with locally small tangent-point energy at $x\in\R/\Z$ as in \Cref{def:homeotp} is a
\emph{locally critical embedding} in  $B(x,r)$ if
\[ \delta\tp^{p,q}(\gamma,\varphi) = 0
\qquad\text{for all $\varphi \in C_c^\infty(B(x,r),\R^3)$}. \]
\end{definition}
The notion of locally critical embedding as in \Cref{def:wcp} can be justified by the following theorem which 
states that
\emph{any minimizing sequence of curves \emph{(w.r.t.\ isotopy classes, cf.\ \Cref{s:topology})} converges away from finitely many points to a local critical embedding}. 

\begin{theorem}\label{th:minarecritical} Let $p =q+2$.
Let $[\gamma_0]$ be an ambient isotopy class and let $\left(\gamma_k\right)_{k\in\N} \subset [\gamma_0]$ be a minimizing sequence for 
\[
 \Lambda := \inf_{\gamma \in [\gamma_0]} \tp^{p,q}(\gamma)
\]
in the sense that $\gamma_k \in C^1(\R / \Z , \R^3)$ are homeomorphisms and $\gamma_k(\R / \Z)$ belongs to the ambient isotopy class $[\gamma_0]$ for all $k\in\N$.

Then, {up to reparametrization, translation, rescaling} and 
{passing to} a subsequence, $\gamma_k$ uniformly converges to a limit map $\gamma: \R/\Z \to \R^3$ which is a locally critical embedding in the sense of \Cref{def:wcp}
except for a finite exception set 
$\Sigma \subset \R/\Z$
whose cardinality is bounded in terms of $\Lambda$.
\end{theorem}

Our last main result concerns regularity: \emph{the limit of minimizing sequences is regular outside a finite singular set $\Sigma$}. Indeed, we have regularity theory for critical points as in \Cref{def:wcp}.
\begin{theorem}\label{th:mainreg}
	Let $q \geq 2$. Let $\gamma:\R/\Z\rightarrow \R^3$ be a homeomorphism with finite global and small local tangent-point energy around $B(x_0,r)$ as in \Cref{def:homeotp} that is a locally critical embedding in $B(x_0,r)$ of $\tp^{q+2,q}$ as in \Cref{def:wcp}. Then $\gamma \in C^{1,\alpha}(B(x_0,\tfrac r 2),\R^3)$ for some uniform constant $\alpha=\alpha(q) > 0$.
\end{theorem}

From the previous two results we draw the following conclusion.

\begin{corollary}\label{co:regisotopy}
Let $[\gamma_0]$ be an ambient isotopy class and let $\gamma_k \subset [\gamma_0]$ be a minimizing sequence for 
\[
 \Lambda := \inf_{\gamma \in [\gamma_0]} \tp^{p,q}(\gamma)
\]
in the sense that $\left(\gamma_k\right)_{k\in\N} \subset C^\infty(\R / \Z , \R^3)$ is a sequence of homeomorphisms and $\gamma_k(\R / \Z)$ belongs to the knot class $[\gamma_0]$ for all $k\in\N$.

Then, up to reparametrization, translation, and rescaling and taking a subsequence, $\gamma_k$ converges to a limit map $\gamma: \R/\Z \to \R^3$ which is a weak critical point in the sense of \Cref{def:wcp} outside of finitely many points (whose number is bounded in terms of $\Lambda$). In particular, in view of \Cref{th:mainreg}, the limit is $C^{1,\alpha}$ outside of finitely many points.
\end{corollary}

\begin{remark}In \Cref{th:mainreg} we restrict to the scale-invariant case $p=q+2$. For the non scaling-invariant case $q \geq 2$ and $p > q+2$ a $C^{1,\alpha}$-regularity is a consequence of previous work by \cite{SvdM12}, with $\alpha$ depending on $p-q-2>0$. Let us remark that a slight adaptation of our arguments, similar in spirit to adaptations carried out in \cite{MS20} implies $C^{1,\alpha}$-regularity with $\alpha$ independent of $p-q-2 \geq 0$.

Also, in \Cref{th:mainreg} we restrict our attention to $q \geq 2$ (which includes the ``classical'' tangent point energy $T^{4,2}$), but we expect that it is only a minor technical difficulty to obtain the same result for the case $q > 1$. 

Lastly, we consider the target space $\R^3$ throughout this paper to keep the notation simple, but again we expect our results to carry over to curves of arbitrary co-dimension without more than minor technical difficulties.
\end{remark}

\subsection*{Outline and comments on the proofs}
In \Cref{s:sobolev} we introduce the Sobolev spaces 
that are essential for this article.
In \Cref{s:topology} we review the notion of ambient isotopy and adapt this concept to $W^{1+s,\frac{1}{s}}$-curves. While this is, to the best of our knowledge, a new result, the main ideas are 
related to the well-established theory of homotopy groups of Sobolev maps, e.g. in \cite{Schoen-Uhlenbeck-1983,Bethuel-1991}.

In \Cref{s:weaklimitimmersion} we prove our first main theorem, \Cref{th:weaklimitareweakimm},
{which states that} sequences of diffeomorphisms with uniformly bounded tangent-point energy converge outside of a finite singular set. The argument is based on a gap-estimate, vaguely reminiscent of and substantially inspired by arguments due to M\"uller--Sverak \cite{MS95} and H\'elein \cite{Helein-2002} who showed that limits of conformally parametrized two-dimensional maps with a sufficiently small $L^2$-bound on the {second} fundamental form are either point maps or bi-Lipschitz. A further crucial ingredient is an adaptation of the ``straightness'' analysis developed by Strzelecki and von der Mosel \cite{SvdM12} (which in their case leads to the fact that finite energy curves are topological one-manifolds).

In \Cref{s:proofth:minarecritical} we prove our second main theorem, \Cref{th:minarecritical}, {which asserts that} minimizing sequences converge to critical points. This is based on \Cref{th:weaklimitareweakimm} combined with a fractional Luckhaus-type lemma, \Cref{la:frluckhaus}, and the theory of isotopy classes for Sobolev maps from \Cref{s:topology}.

In \Cref{s:regularity} we prove the regularity theory, \Cref{th:mainreg}. We follow the spirit of~\cite{BRS19}, building a bridge to harmonic map theory. Namely, we introduce an energy $\mathcal{E}^{q}$ such that the arclength parametrization $\gamma$ of a critical knot $\tp^{q+2,q}$ induces via its derivative $\gamma'$ a critical map of $\mathcal{E}^{q}$ in the class of maps from $\R/\Z$ to the sphere $\S^{2}$. The energy $\mathcal{E}^{q}$ is structurally similar to the $W^{\frac{1}{q},q}$-seminorm whose critical points are 
called \emph{$W^{\frac{1}{q},q}$-harmonic maps}. For $q=2$ techniques for regularity theory of $W^{\frac{1}{2},2}$ harmonic maps between manifolds were introduced in the pioneering work by Da Lio and Rivi\`ere~\cite{DLR11a,DLR11b}; this was extended to $W^{\frac{1}{q},q}$ harmonic maps into spheres in~\cite{S15}. Here, we extend the techniques of~\cite{S15} to obtain the regularity for derivatives $\gamma'$ of the arclength-parametrization of critical knots $\gamma$.

\subsection*{Notation} 
When $A \leq C B$ for some constant $C$, we write $A \aleq B$
or $B\ageq A$.
We use the notation $A \aeq B$
if both $A \aleq B$ and $B \aleq A$.
Throughout this work, 
constants will depend on ``unimportant'' factors like $p$ and $q$ and may change from line to line.

Balls (i.e.\  intervals) in $\R$ will be denoted by $B(x,\rho)$. We will allow ourselves an abuse of notation to denote \emph{geodesic balls} in $\R/\Z$ by the same notation. All our arguments are local in nature, so that we only need to work with balls which correspond to Euclidean balls.

\subsection*{Acknowledgement}
Grant funding is acknowledged as follows.
\begin{itemize}
 \item Austrian Science Fund (FWF), Grant P29487 (SB, NV)
 \item German Research Foundation (DFG), grant RE~3930/1--1 ($\Phi$R)
 \item National Science Foundation (NSF) Career DMS 2044898 (AS)
 \item Simons foundation, grant no 579261 (AS)
\end{itemize}
Part of this work was conducted while the authors were at the IMA Minnesota and visits to each other's institutions. A long-term visit by NV at the University of Pittsburgh funded by the Austrian Marshall Plan Scholarship is also gratefully acknowledged.

\section{Preliminaries on Sobolev maps}\label{s:sobolev}
In this section we recall some basic notation and properties of Sobolev maps.

For $s \in (0,1)$, $p \in (1,\infty)$, $\Omega \subset \R$ open, the Sobolev space $W^{s,p}(\Omega)$ is defined as all maps $f \in L^{p}(\Omega)$ such that 
\[
 [f]_{W^{s,p}(\Omega)} := \brac{\int_{\Omega}\int_{\Omega} \frac{|f(x)-f(y)|^p}{|x-y|^{1+sp}} \, dx\, dy }^{\frac{1}{p}} < \infty.
\]
For $s \in (1,2)$ the Sobolev space $W^{s,p}(\Omega)$ is defined to be the space of $f \in L^{p}(\Omega)$, $f'\in L^{p}(\Omega)$, and 
\[
 [f']_{W^{s-1,p}(\Omega)} < \infty.
\]
One important observation, cf. \cite[Lemma 2.1]{B12}, is that small $W^{1+s,\frac{1}{s}}$-Sobolev norm implies a bi-Lipschitz estimate if $|\gamma'| > 0$. Namely we have,
\begin{lemma}\label{la:bilip}
Let $s \in (0,1)$. For any $\lambda_1 > \lambda_2 > 0$ there exists $\eps = \eps(\lambda_1,\lambda_2,s) > 0$ such that the following holds.
For any $-\infty<a<b<\infty$ and for any $\gamma \in \lip([a,b],\R^3)$ such that \[\inf_{[a,b]} |\gamma'| \geq \lambda_1\] and
\[
 [\gamma']_{W^{s,\frac{1}{s}}((a,b))} < \eps,
\]
we have
\[
 |\gamma(x)-\gamma(y)|\geq \lambda_2 |x-y|.
\]
\end{lemma}
\preprintonly{
\begin{proof}
First we prove the following inequality
\begin{equation}\label{eq:bilip:goal}
 \frac{\abs{\gamma(y)-\gamma(x)}^2}{|x-y|^2} \geq (\lambda_1)^2 - \frac{1}{2|x-y|^2} \int_{[x,y]} \int_{[x,y]} |\gamma'(z_1)-\gamma'(z_2)|^2 \, dz_1 \, dz_2.
\end{equation}
The estimate \eqref{eq:bilip:goal} is a consequence of the fundamental theorem of calculus, which says that for any $x\neq y$ we have
\[
 \gamma(y)-\gamma(x) = \int_{x}^y \gamma'(z)\, dz.
\]
Then
\[
\begin{split}
 |\gamma(y)-\gamma(x)|^2 =& \int_{x}^y \int_{x}^y  \langle \gamma'(z_1), \gamma'(z_2)\rangle \, dz_1\, dz_2\\
 =& \frac{1}{2} \int_{x}^y \int_{x}^y  |\gamma'(z_1)|^2+|\gamma'(z_2)|^2\, dz_1\, dz_2 -\frac{1}{2} \int_{x}^y \int_{x}^y  |\gamma'(z_1)-\gamma'(z_2)|^2\, dz_1\, dz_2\\
 \geq&|x-y|^2\, (\lambda_1)^2-\frac{1}{2} \int_{[x,y]} \int_{[x,y]} |\gamma'(z_1)-\gamma'(z_2)|^2 \, dz_1\,  dz_2.
 \end{split}
\]
This establishes \eqref{eq:bilip:goal}.

The claim of \Cref{la:bilip} follows from \eqref{eq:bilip:goal} once we show that for any $s \in (0,1)$ there exists a constant $C = C(s)$ such that 
\begin{equation}\label{eq:bilip:goal2}
 \frac{1}{2|x-y|^2} \int_{[x,y]} \int_{[x,y]} |\gamma'(z_1)-\gamma'(z_2)|^2 \, dz_1 \, dz_2 \leq C(s)\, [\gamma']_{W^{s,\frac{1}{s}}((a,b))}^2.
\end{equation}
Indeed, once \eqref{eq:bilip:goal2} is established, we choose $\eps > 0$ such that
\[
 (\lambda_1)^2 - C(s) \eps^2 > (\lambda_2)^2.
\]
Then -- under the assumptions of \Cref{la:bilip} -- we conclude that
\[
 \frac{\abs{\gamma(y)-\gamma(x)}^2}{|x-y|^2} \geq (\lambda_2)^2,
\]
which is what we wanted to show.

It remains to establish \eqref{eq:bilip:goal2}, and for this we consider three cases.\\
For \underline{$s =\frac{1}{2}$}, \eqref{eq:bilip:goal2} is a consequence of the following observation
\[
\begin{split}
 &\frac{1}{2|x-y|^2} \int_{[x,y]} \int_{[x,y]} |\gamma'(z_1)-\gamma'(z_2)|^2 \, dz_1 \, dz_2 \\
 \leq& \frac{1}{2} \int_{[x,y]} \int_{[x,y]} \frac{|\gamma'(z_1)-\gamma'(z_2)|^2}{|z_1-z_2|^2} \, dz_1 \, dz_2\\
 =& \frac{1}{2}[\gamma']_{W^{\frac{1}{2},2}([x,y])}^2 \leq \frac{1}{2}[\gamma']_{W^{\frac{1}{2},2}([a,b])}^2.
 \end{split}
\]
For the \underline{case $s \in (\frac{1}{2},1)$}, we additionally observe that by \Cref{la:sob1} there exists a constant $C=C(s)$ such that 
\[
[\gamma']_{W^{\frac 12,2}((a,b))}^2 \leq C(s) \, [\gamma']_{W^{s,\frac 1s }((a,b))}^2.
\]
This establishes \eqref{eq:bilip:goal2} for all $s \in [\frac{1}{2},1)$. \\
If \underline{$s \in (0,\frac{1}{2})$}, then $\frac{1}{2s} > 1$. Hence, by Jensen's inequality,
\[
\begin{split}
&|x-y|^{-2} \int_{[x,y]} \int_{[x,y]} |\gamma'(z_1)-\gamma'(z_2)|^2 \, dz_1 \, dz_2 \\
=&\brac{|x-y|^{-2} \int_{[x,y]} \int_{[x,y]} |\gamma'(z_1)-\gamma'(z_2)|^2 \, dz_1 \, dz_2}^{\frac{1}{2s}\, 2s} \\
\leq&\brac{|x-y|^{-2} \int_{[x,y]} \int_{[x,y]} |\gamma'(z_1)-\gamma'(z_2)|^{\frac{1}{s}} \, dz_1 \, dz_2}^{2s} \\
\leq& \brac{\int_{[x,y]} \int_{[x,y]} \frac{|\gamma'(z_1)-\gamma'(z_2)|^{\frac{1}{s}}}{|z_1-z_2|^2} \, dz_1 \, dz_2}^{2s} .
\end{split}
\]
This establishes \eqref{eq:bilip:goal2} for $s \in (0,\frac{1}{2})$, and we can conclude.
\end{proof}}

Let us also remark the following consequence of \Cref{la:bilip}, which states that closed curves have minimal $W^{1+s,\frac{1}{s}}$-energy.
\begin{corollary}
Let $s\in(0,1)$, $-\infty < a < b < \infty$. For any $\lambda > 0$ there exists $\eps =\eps(\lambda,a,b,s) > 0$ so that the following holds.\\
Whenever $\gamma \in \lip((a,b),\R^3)\cap C^0([a,b])$ with $\gamma(a) = \gamma(b)$ and 
$\inf |\gamma'| \geq \lambda$, 
then $[\gamma']_{W^{s,\frac{1}{s}}([a,b])} \geq \eps$.
\end{corollary}
\begin{proof}
If $[\gamma']_{W^{s,\frac{1}{s}}([a,b])} < \eps$ for small enough $\eps$ we know from \Cref{la:bilip} that $\gamma$ is bilipschitz, and thus
\[
 |\gamma(a)-\gamma(b)| = \lim_{x \to a^+} \lim_{y \to b^-} |\gamma(x)-\gamma(y)| \geq c |b-a|.
\]
\end{proof}

\section{Ambient isotopy for Sobolev curves}\label{s:topology}
A homotopy theory for Sobolev maps has been introduced and established a long time ago. The spirit is that for maps in VMO (see e.g. the phenomenal work \cite{BN95,BN96}) homotopy classes exist, and by Sobolev embedding, homotopy groups for $W^{s,\frac{n}{s}}(\Sigma^n,\mathcal{N})$ coincide with the classical homotopy groups for continuous maps. In particular, this leads to a beautiful theory of density \cite{Schoen-Uhlenbeck-1983,Bethuel-1991}. There are many extensions, e.g. to more general Sobolev spaces \cite{Hajlasz-1994,Riviere-2000,Mironescu-2004,Bousquet-Ponce-VanSchaftingen-2013,Brezis-Mironescu-2015,Bousquet-Ponce-VanSchaftingen-2015}.

We begin here to introduce the fundamental results on isotopy classes (for curves) in fractional Sobolev spaces following the spirit of homotopy classes. To the best of our knowledge, the results in this section are new, in particular our main result, \Cref{th:localambisotchange}, which says that small Sobolev-variations of smooth curves do not change their isotopy class. However, there is some overlap with \cite{BGRvdM} where an isotopy theory for closed sets with controlled bi-Lipschitz constant is developed.

We begin by defining ambient isotopy (by which we mean $C^1$-ambient isotopy).
\begin{definition}\label{def:smoothisotopy}
Two sets $X,Y \subset \R^3$ are called ambient isotopic, if there exists an ambient isotopy, namely $I \in C^1([0,1] \times \R^3)$, such that 
\begin{itemize}
\item $I(t,\cdot): \R^3 \to \R^3$ is a surjective diffeomorphism for all $t \in [0,1]$,
 \item $I(0,p) = p$ for all $p \in \R^3$, 
 \item $I(1,\cdot) : X \to Y$ is a surjective homeomorphism.
 \end{itemize}
\end{definition}

When working with parametrized curves, the following result is very useful: Smooth enough isotopy coincides with ambient isotopy, see \cite[Chapter 8, Theorem 1.6, p.181]{H94}.
\begin{theorem}\label{th:hambientisotop}
Let $\gamma_0,\gamma_1 \in C^1(\R / \Z,\R^3)$ be two diffeomorphisms, and assume there exists a $C^1$-isotopy between them, that is $\Gamma \in C^1([0,1] \times \R/\Z, \R^3)$ such that $\Gamma(0,\cdot) = \gamma_0(\cdot)$ and $\Gamma(1,\cdot) = \gamma_0(\cdot)$ and 
\[
 \Gamma(t,\cdot): \R/\Z \to \R^3 \quad \text{is a diffeomorphism for all $t \in [0,1]$.}
\]
Then the images $\gamma_0(\R/\Z)$ and $\gamma_1(\R/\Z)$ are ambient isotopic.
\end{theorem}

We begin by defining ambient isotopy classes for regular $W^{1+s,\frac{1}{s}}$-homeomorphisms.
Observe that in view of the formal analogy to homotopy classes, having the techniques by Brezis and Nirenberg \cite{BN95,BN96}, one might hope for an ``s=0'' theory (i.e.\  $\gamma' \in VMO$), but we will not pursue that question here. Also, we will make no attempt to consider the higher-dimensional version, but rather focus on the situation at hand, which are curves.

\begin{definition}[Regular Sobolev homeomorphism]
A homeomorphism $\gamma \in W^{1+s,\frac{1}{s}}(\R/\Z,\R^3)$ is called \emph{regular} if
\[
0< \inf |\gamma'| \leq \sup |\gamma'|<\infty
\]
where $\inf$ and $\sup$ are the essential infimimum and supremum, respectively. 
\end{definition}

The isotopy class is derived from smooth approximating maps, whose existence is the content of the following lemma.

\begin{lemma}\label{la:diffeoapprox}
Let $\gamma \in W^{1+s,\frac{1}{s}}(\R/\Z,\R^3)$ be a  regular homeomorphism 
. Then there exists a sequence of smooth diffeomorphisms $\gamma_k: \R/\Z \to \R^3$ with 
\[
\frac{1}{2} \inf |\gamma'| \leq \inf |\gamma_k'| \leq \sup |\gamma_k'|  \leq \sup |\gamma'| \quad \text{for all $k \in \N$}
\]
such that 
\[
 \|\gamma_k-\gamma\|_{L^\infty(\R/\Z)} + [\gamma_k-\gamma]_{W^{1+s,\frac{1}{s}}(\R/\Z)} \xrightarrow{k \to \infty} 0.
\]
\end{lemma}
\begin{proof}
Set 
\[
 \lambda := \inf |\gamma'|.
\]
Fix some $\eps_0 > 0$ to be specified later.

By absolute continuity of the integral there exists $\delta_0 =\delta_0(\gamma) \in (0,1)$ such that
\[
\sup_{B(10\delta_0) \subset \R/\Z} [\gamma']_{W^{s,\frac{1}{s}}(B(10\delta_0))} < \eps_0.
\]
Since $\gamma$ is a continuous and injective map, the following infimum is attained and larger than $0$ 
\[
\eps_1 := \inf_{|x-y| \geq \frac{1}{2}\delta_0} |\gamma(x)-\gamma(y)| > 0.
\]
Let $\eta \in C_c^\infty(B(0,1),[0,1])$, $\int \eta = 1$, be the usual mollifier kernel, and $\eta_\delta := \delta^{-1} \eta(\cdot/\delta)$. Set 
\[
 \gamma_\delta(x) := \eta_\delta \ast \gamma(x) =  \int_{-1}^{1} \eta(z)\, \gamma(x+\delta z) \, dz.
\]
By periodicity of $\gamma$, $\gamma_\delta$ is $1$-periodic, and thus is well-defined $\gamma_\delta: \R/\Z \to \R^3$. 
Moreover, $[\gamma_\delta-\gamma]_{W^{1+s,\frac{1}{s}}(\R/\Z)} \xrightarrow{\delta \to 0} 0$, and by Sobolev embedding $\|\gamma_\delta-\gamma\|_{L^\infty(\R/\Z)} \xrightarrow{\delta \to 0}  0$.

Let $\delta_1 \in (0,\delta_0)$ such that 
\begin{equation}\label{eq:la:4654785}
\|\gamma_\delta-\gamma\|_{L^\infty(\R/\Z)} < \frac{1}{10} \eps_1 \quad  \text{for all $\delta \in [0,\delta_1]$.}
\end{equation}
We need to show (for the right choice of $\eps_0$) that for $\delta \in (0,\delta_1)$ $\gamma_\delta$ is a diffeomorphism as requested in the claim.

It is easy to see from the definition of $\gamma_\delta$ that
\[
 |\gamma_\delta'(x)| \leq \sup |\gamma'| \quad \forall \delta > 0.
\]
First we observe that for almost any $x \in \R/\Z$ and almost any $z \in \R/\Z$ we have 
\[
 |\gamma_\delta'(x)| \geq |\gamma'(z)| - |\gamma_\delta'(x)-\gamma'(z)| \geq \lambda -  |\gamma_\delta'(x)-\gamma'(z)|.
\]
Since this holds for almost every $z \in \R / \Z$, we can integrate this inequality, $\mvint_{B(x,\delta)}\, dz$, and find
\[
 |\gamma_\delta'(x)| \geq \lambda -  \mvint_{B(x,\delta)}|\gamma_\delta'(x)-\gamma'(z)|\, dz.
\]
Now we have $\gamma_\delta' = \eta_\delta \ast (\gamma')$, and thus
\[
\begin{split}
 &\mvint_{B(x,\delta)}|\gamma_\delta'(x)-\gamma'(z)| \, dz\\
 \leq& \ \mvint_{B(x,\delta)} \mvint_{B(x,\delta)}|\gamma'(z_1)-\gamma'(z)|\, dz_1 \, dz\\
 \aleq&  \ [\gamma']_{W^{s,\frac{1}{s}}(B(x,\delta))} < \eps_0.
 \end{split}
\]
That is, we have shown 
\[
 |\gamma_\delta'(x)| \geq \lambda -  C\eps_0 \quad \text{for almost any $x \in \R / \Z$}.
\]
So if we choose $\eps_0 < \frac{\lambda}{2C}$, we therefore have 
\[
 \inf |\gamma_\delta'| \geq \frac{\lambda}{2} \quad \forall \delta \in (0,\delta_0).
\]
Now choosing $\eps_0$ possibly even smaller (depending on $\lambda$), we obtain from \Cref{la:bilip} that 
\begin{equation}\label{eq:la:difap:bilip1}
 \frac{\lambda}{4} |x-y| \leq |\gamma_\delta(x)-\gamma_\delta(y)| \quad \forall \delta \in [0,\delta_0], \ |x-y| < \delta_0 .
\end{equation}
On the other hand, 
\[
\begin{split}
 |\gamma_\delta(x)-\gamma_\delta(y)|  \geq |\gamma(x)-\gamma(y)|-2\|\gamma-\gamma_\delta\|_{L^\infty}.\\
 \end{split}
\]
In view of \eqref{eq:la:4654785}, we thus have
\begin{equation}\label{eq:la:difap:bilip2}
\begin{split}
 |\gamma_\delta(x)-\gamma_\delta(y)|  \geq 
 (1-\frac{2}{10})\eps_1 \quad \forall |x-y| > \frac{1}{2}\delta_0, \quad \delta \in [0,\delta_1].
 \end{split}
\end{equation}
Combining \eqref{eq:la:difap:bilip1} and \eqref{eq:la:difap:bilip2}, we obtain
\[
 |\gamma_\delta(x)-\gamma_\delta(y)| \geq \frac{1}{100} \min\{\lambda, \delta_0\eps_1\} |x-y| \quad \forall x,y \in\R/\Z\quad \forall \delta \in [0,\delta_1].
\]
Consequently, $\gamma_\delta$ is a one-to-one map with $\inf |\gamma_\delta'| > 0$.

Hence, $\gamma_\delta$ is a smooth homeomorphism with nonvanishing derivative (i.e.\  its an immersion), thus $\gamma_\delta$ is a diffeomorphism. The proof is concluded by choosing $\gamma_k := \gamma_{\frac{\delta_1}{k}}$.
\end{proof}

Now that we have approximating smooth diffeomorphisms, we argue that they all eventually are of the same ambient isotopy type.

\begin{proposition}\label{pr:closebycurves}
Let $\gamma \in W^{1+s,\frac{1}{s}}(\R / \Z,\R^3)$ be a regular homeomorphism. Then there exists $\eps = \eps(\gamma,s) > 0$ such that for any $\tilde{\gamma}_1,\tilde{\gamma}_2 \in C^1(\R/\Z,\R^3)$ with 
\begin{equation}\label{eq:cbc:2}\|\tilde{\gamma}_i - \gamma\|_{L^\infty} +[\tilde{\gamma}_i' - \gamma']_{W^{s,\frac{1}{s}}} < \eps\end{equation} and 
\[
\frac{1}{2} \inf |\gamma'| \leq\inf |\tilde{\gamma}_i'| \leq
\sup |\tilde{\gamma}_i'| \leq 2\sup |\gamma'| \quad \text{$i=1,2$,}
\]
we have that $\tilde{\gamma}_1$ is ambient isotopic to $\tilde{\gamma}_2$.
\end{proposition}
\begin{proof}
The strategy of the proof is very similar to the proof of \Cref{la:diffeoapprox}.

Set 
\[
 \lambda := \frac{1}{2}\inf |\gamma'|.
\]
Fix some $\eps_0 > 0$ to be specified later.

By absolute continuity there exists $\delta_0 =\delta_0(\gamma) \in (0,1)$ such that 
\[
\sup_{B(10\delta_0) \subset \R/\Z} [\gamma']_{W^{s,\frac{1}{s}}(B(10\delta_0) )} < \frac{1}{2}\eps_0.
\]
Since $\gamma$ is continuous and injective, the following infimum is attained and larger than $0$ 
\[
\eps_1 := \frac{1}{10}\inf_{|x-y| \geq \frac{1}{2}\delta_0} |\gamma(x)-\gamma(y)| > 0.
\]

Assume $\eps < \frac{1}{100}\min\{\eps_0,\eps_1\}$ in \eqref{eq:cbc:2}, then we have
\[
\sup_{B(10\delta_0) \subset \R/\Z} [\tilde{\gamma}_i']_{W^{s,\frac{1}{s}}(B(10\delta_0))} < \eps_0 \quad i=1,2
\]
and
\[
\inf_{|x-y| \geq \frac{1}{2}\delta_0} |\tilde{\gamma}_i(x)-\tilde{\gamma}_i(y)| > \eps_1 >0.
\]

Let $\eta \in C_c^\infty(B(0,1),[0,1])$, $\int \eta = 1$, be the usual mollifier kernel, and $\eta_\delta := \delta^{-1} \eta(\cdot/\delta)$. Set 
\[
 \gamma_\delta := \eta_\delta \ast \gamma
\]
and
\[
 \tilde{\gamma}_{i,\delta} := \eta_\delta \ast \tilde{\gamma}_i.
\]
Let $\delta_1$ be such that 
\[
\|\gamma_\delta - \gamma\|_{L^\infty} < \frac{1}{100}\eps_1 \quad \forall \delta \in [0,\delta_1]. 
\]
By the choice of $\eps \ll \eps_1$ we then have
\[
\|\tilde{\gamma}_{i,\delta} - \tilde{\gamma}_{i}\|_{L^\infty} < \frac{1}{10}\eps_1 \quad \forall \delta \in [0,\delta_1],\, i=1,2.
\]
As in the proof of \Cref{la:diffeoapprox}, we obtain that for each $\delta \in [0,\delta_0]$,
$\gamma_\delta$, $\tilde{\gamma}_{i,\delta}$ are diffeomorphisms satisfying
\[
|\tilde{\gamma}_{i,\delta}(x)-\tilde{\gamma}_{i,\delta}(y)|, |\gamma_\delta(x)-\gamma_\delta(y)| \geq \tilde{\lambda} |x-y| \quad \forall x,y \in\R/\Z \quad \forall \delta \in [0,\delta_1]
\]
for $\tilde{\lambda} := \frac{1}{100} \min\{\lambda, \delta_0\eps_1\}$.

Since $\tilde{\gamma}_{i,\delta} \xrightarrow{\delta \to 0} \tilde{\gamma}_{i}$, and $\tilde{\gamma}_i$ are smooth diffeomorphisms, we get from \Cref{th:hambientisotop} that $\tilde{\gamma}_{i,\delta_0}$ is ambient isotopic to $\tilde{\gamma}_{i}$ for $i=1,2$.

Now we show that $\tilde{\gamma}_{i,\delta_0}$ is ambient isotopic to $\tilde{\gamma}_{\delta_0}$ for $i=1,2$. Indeed set 
\[
 \Gamma_i(\cdot,t) := t \tilde{\gamma}_{i,\delta_0} + (1-t) \gamma_{\delta_0}.
\]
This is clearly a smooth homotopy, we only need to show that for each fixed $t \in [0,1]$ it is a diffeomorphism. But observe that
\[
\begin{split}
 &|\Gamma_i(x,t)-\Gamma_i(y,t)|\\
 \geq& |\gamma_{\delta_0}(x)-\gamma_{\delta_0}(y)| - |x-y|\, \|\gamma_{\delta_0}'-\tilde{\gamma}_{i,\delta_0}'\|_{L^\infty} \\
 \geq&\brac{\tilde{\lambda} -\|\gamma_{\delta_0}'-\tilde{\gamma}_{i,\delta_0}'\|_{L^\infty}} |x-y|.
 \end{split}
\]
Now
\[
 \|\gamma_{\delta_0}'-\tilde{\gamma}_{i,\delta_0}'\|_{L^\infty} \leq \frac{1}{\delta_0} \|\eta'\|_{L^1}\, \|\gamma-\tilde{\gamma}_i\|_{L^\infty} \leq \frac{\eps}{\delta_0} \|\eta'\|_{L^1}.
\]
That is, 
\[
 |\Gamma_i(x,t)-\Gamma_i(y,t)|  \geq\brac{\tilde{\lambda} -\|\eta'\|_{L^1}\, \frac{\eps}{\delta_0}} |x-y|\quad \forall x,y {\in\R/\Z}.
\]
So if we choose $\eps$ even smaller, namely if we ensure that $\eps < \frac{\tilde{\lambda}}{100} \delta_0 \|\eta'\|_{L^1}$, then we have found that $\Gamma(t,\cdot)$ is globally bilipschitz, and thus a diffeomorphism for each $t \in [0,1]$. This and \Cref{th:hambientisotop} imply that $\gamma_{\delta_0}$ and $\tilde{\gamma}_{i,\delta_0}$ are ambient isotopic, for each $i=1,2$. 

In particular, $\tilde{\gamma}_{1,\delta_0}$ is ambient isotopic to $\tilde{\gamma}_{2,\delta_0}$. Since we have already shown that $\tilde{\gamma}_{i,\delta_0}$ is ambient isotopic to $\tilde{\gamma}_{i}$ for each $i=1,2$, we conclude that $\tilde{\gamma}_{1}$ is ambient isotopic to $\tilde{\gamma}_{2}$.
\end{proof}

Since by \Cref{la:diffeoapprox} any $W^{1+s,\frac{1}{s}}$-regular Sobolev homeomorphism has an approximation by a regular diffeomorphisms, and by \Cref{pr:closebycurves} these approximating diffeomorphisms eventually are all of the same ambient isotopy type, the following definition is justified.

\begin{definition}\label{d:aisob}
\begin{itemize}
\item Let $[\gamma_0]$ be an ambient isotopy class. We say that a regular $W^{1+s,\frac{1}{s}}$-Sobolev homeomorphism $\gamma_1$ belongs to $[\gamma_0]$ if there exist approximating diffeomorphisms $\tilde{\gamma}_{1,k}$ such that $\tilde{\gamma}_{1,k} \in [\gamma_0]$ for eventually all $k \in \N$.
\item Equivalently, let $\gamma_1, \gamma_2 \in W^{1+s,\frac{1}{s}}(\R/\Z,\R^3)$ be two regular homeomorphisms.

We say that $\gamma_1$ and $\gamma_2$ are (Sobolev-)ambient isotopic,
\[
 \gamma_{1} \sim \gamma_{2},
\]
if the following properties are met: There exist approximating diffeomorphisms $\tilde{\gamma}_{1,k}$ and $\tilde{\gamma}_{2,k}$ converging in $L^\infty \cap W^{1+s,\frac{1}{s}}(\R/\Z,\R^3)$ to $\gamma_1$ and $\gamma_2$, respectively, and each satisfy
\[
\frac{1}{2} \inf |\gamma_i'| \leq\inf |\tilde{\gamma}_{i,k}'| \leq
\sup |\tilde{\gamma}_{i,k}'| \leq 2\sup |\gamma_i'| \quad \text{$i=1,2$}.
\]
Moreover, $\gamma_{1,k}$ and $\gamma_{2,\ell}$ are ambient isotopic for all but finitely many $k$ and $\ell$.
\end{itemize}
\end{definition}

Our main result in this section is that two curves which differ only locally and in a set where they have small critical Sobolev norm, have the same ambient isotopy type.
\begin{theorem}\label{th:localambisotchange}
 There exists a uniform $\eps > 0$ such that the following holds.

Let $\gamma_1,\gamma_2 \in C^1(\R/\Z,\R^3)$ be diffeomorphisms with 
\begin{itemize}
 \item $\frac{3}{4}\leq |\gamma_i'| \leq \frac{5}{4}$ in $\R / \Z$ for $i=1,2$
\end{itemize}
and assume that there is a ball $B(\rho) \subset \R / \Z$ such that the following holds:
\begin{itemize}
 \item $\frac{3}{4} |x-y|\leq |\gamma_i(x)-\gamma_i(y)|\leq \frac{5}{4} |x-y|$ for all $x,y \in B(10\rho)$,
  \item $[\gamma_i']_{W^{s,\frac{1}{s}}(B(10\rho))} < \eps$, $i=1,2$,
 \item we have $\dist\brac{\gamma_i(\R / \Z \backslash B(9\rho)),\gamma_i(B(8\rho)} \geq \frac{1}{1000} \rho$ for $i=1,2$,
 \item $\gamma_1(x) = \gamma_2(x)$ for all $x \in \R/\Z \backslash B(3\rho)$,
 \item $\|\gamma_1 -\gamma_2\|_{L^\infty(\R/\Z)} < \frac{1}{100}\frac{1}{1000} \rho$.
\end{itemize}
Then $\gamma_1(\R/\Z)$ and $\gamma_2(\R/\Z)$ are ambient isotopic.
\end{theorem}
\begin{proof}
Let $\eta \in C_c^\infty(B(0,2),[0,1])$, $\int_{\R} \eta = 1$. Denote by $\eta_\delta := \delta^{-1} \eta(\cdot/\delta)$.

Let $\theta \in C_c^\infty(B(5\rho),[0,1])$ such that $\theta \equiv 1$ in $B(4\rho)$. We can construct $\theta$ such that $\|\theta'\|_{L^\infty} \aleq \frac{1}{\rho}$.

Set 
\[
 \gamma_{i,\delta}(x) := \int_{B(0,1)} \eta(z)\, \gamma_i(x+\delta \theta(x)z)\, dz.
\]

In the following we need to choose first some $\sigma > 0$, and then obtain some $\eps > 0$ depending on $\sigma$.

\underline{$\gamma_i$ is ambient isotopic to $\gamma_{i,\sigma \rho}$ for some uniform constant $\sigma \in (0,\frac{1}{2})$}

$\gamma_{i,\delta}$ would be the  usual convolution if $\theta \equiv 1$.

First observe that $\gamma_{i,\delta}(x) = \gamma_i(x)$ for $x \in \R / \Z \backslash B(5\rho)$. Moreover, we have 
\[
 \gamma_{i,\delta}(x)-\gamma_{i}(x) = \int_{B(0,1)} \eta(z)\, \brac{\gamma_i(x+\delta \theta(x)z)-\gamma_i(x)}\, dz \aleq \|\gamma_i'\|_{L^\infty} \delta \|\theta\|_{L^\infty} \xrightarrow{\delta \to 0} 0.
\]
Also we have 
\[
 \gamma_{i,\delta}'(x) = \int_{B(0,1)} \eta(z)\, \gamma_i'(x+\delta \theta(x)z) \brac{1+\delta \theta'(x)z}\, dz.
\]
Thus (recall $|\theta'|\aleq \frac{1}{\rho}$)
\[
 |\gamma_{i,\delta}'(x)| \leq \|\gamma_i'\|_{L^\infty}\, \brac{1+C\frac{\delta}{\rho}},
\]
that is for a certain (essentially uniform) $\sigma \in (0,1)$ we can ensure that
\[
 \|\gamma_{i,\delta}'\|_{L^\infty} \leq \frac{11}{8} \quad \forall \delta < \sigma \rho.
\]
It is the other direction which is more tricky. We have for any $x,z \in B(5\rho)$
\[
\begin{split}
 &\left ||\gamma_{i,\delta}'(x)|-1 \right |\\
 \leq &\left ||\gamma_{i,\delta}'(x)|-|\gamma_{i}'(z)| \right | + \frac{1}{4}\\
 \leq &\left |\gamma_{i,\delta}'(x)-\gamma_{i}'(z)\right | + \frac{1}{4}.
\end{split}
 \]
In particular, for any $z \in B(0,1)$,
\[
\begin{split}
 &\left ||\gamma_{i,\delta}'(x)|-1 \right |\\
 \leq &\left |\gamma_{i,\delta}'(x)-\gamma_{i}'(x+\delta \theta(x)z)\right | + \frac{1}{4}.
\end{split}
 \]
Integrating in $z$, as the left-hand side is a constant, we obtain
\[
\begin{split}
 &\left ||\gamma_{i,\delta}'(x)|-1 \right |\\
  =&\mvint_{B(0,1)} \left |\int_{B(0,1)} \eta(z_2) \gamma_{i}'(x+\delta \theta(x) z_2)(1+\delta \theta'(x) z_2)dz_2-\gamma_{i}'(x+\delta \theta(x)z) \right | dz + \frac{1}{4}\\
  =&\mvint_{B(0,1)} \left |\int_{B(0,1)} \eta(z_2) \brac{\gamma_{i}'(x+\delta \theta(x) z_2)(1+\delta \theta'(x) z_2)-\gamma_{i}'(x+\delta \theta(x)z)} dz_2 \right | dz + \frac{1}{4}\\
\leq&\mvint_{B(0,1)} \int_{B(0,1)} \left |\brac{\gamma_{i}'(x+\delta \theta(x) z_2)(1+\delta \theta'(x) z_2)-\gamma_{i}'(x+\delta \theta(x)z)}\right | \, dz_2  \, dz + \frac{1}{4}\\
\leq&\mvint_{B(0,1)} \int_{B(0,1)} \left |\gamma_{i}'(x+\delta \theta(x) z_2)-\gamma_{i}'(x+\delta \theta(x)z) \right | \, dz\, dz_2 + \frac{1}{4}+C\, \frac{\delta}{\rho}.\\
\end{split}
 \]
Now observe that $\theta(x)$ is a fixed, nonnegative number. If $\theta(x) = 0$, then the double integral is zero. If $\theta(x) > 0$, by substitution,
\[
\begin{split}
&\mvint_{B(0,1)} \int_{B(0,1)} \left |\gamma_{i}'(x+\delta \theta(x) z_2)-\gamma_{i}'(x+\delta \theta(x)z) \right | \, dz \, dz_2\\
\aleq &\frac{1}{(\delta \theta(x))^2} \int_{B(x,\delta \theta(x))} \int_{B(x,\delta \theta(x))} \left |\gamma_{i}'(\tilde{x})-\gamma_{i}'(\tilde{y}) \right | \, d\tilde{x} \,  d\tilde{y}\\
\aleq &[\gamma_i']_{W^{\frac{1}{q},q}(B(x,\delta \theta(x))}\\
\leq &[\gamma_i']_{W^{\frac{1}{q},q}(B(10\rho))} < \eps.
\end{split}
\]
That is, we have shown that for each $x \in B(5\rho)$ we have
\[
 \left ||\gamma_{i,\delta}'(x)|-1 \right | \leq \frac{1}{4} + C (\frac{\delta}{\rho} +\eps).
\]
The constant $C$ is uniform, so if $\delta < \sigma \rho$ and $\eps \ll 1$ (uniform constant), we get that 
\begin{equation}\label{eq:asdkgammaideltap2}
 \frac{5}{8} \leq |\gamma_{i,\delta}'(x)| \leq \frac{11}{8} \quad \forall x \in \R /\Z.
\end{equation}
(Observe this estimate is trivial for all $x$ where $\gamma_{i,\delta} = \gamma_{i}$).

Next we estimate the Sobolev norm, namely we have
\[
 [\gamma_{i,\delta}']_{W^{\frac{1}{q},q}(B(9\rho))} \aleq [\gamma_{i}']_{W^{\frac{1}{q},q}(B(10\rho))} + \frac{\delta}{\rho}.
\]
Indeed, we have 
\[
\begin{split}
 &|\gamma_{i,\delta}'(x)-\gamma_{i,\delta}'(y)| \\
 \leq &\int_{B(0,1)}\,\left  |\gamma_{i}'(x+\delta \theta(x) z)(1+\delta \theta'(x) z)-\gamma_{i}'(y+\delta \theta(y) z)(1+\delta \theta'(y) z) \right |\, dz\\
 \leq &(1+C \frac{\delta}{\rho})\int_{B(0,1)}\,\left  |\gamma_{i}'(x+\delta \theta(x) z)-\gamma_{i}'(y+\delta \theta(y) z) \right |\, dz\\
 &+\|\gamma_i'\|_{L^\infty}\int_{B(0,1)}\,\left  |\delta \theta'(x) - \delta \theta'(y) \right |\, dz.\\
 \end{split}
\]
First we observe 
\[
\int_{B(0,1)}\,\left  |\delta \theta'(x) - \delta \theta'(y) \right |\, dz \aleq \frac{\delta}{\rho} |x-y|
\]
and consequently, for any $q > 1$
\[
 \begin{split}
\int_{B(9\rho)}\int_{B(9\rho)} 
\frac{\brac{\int_{B(0,1)}\,\left  |\delta \theta'(x) - \delta \theta'(y) \right |\, dz}^q}{|x-y|^{2}}\, dx\, dy \aleq \brac{\frac{\delta}{\rho}}^q \rho^q = \delta^q.
 \end{split}
\]
We thus arrive at
\[
\begin{split}
&[\gamma_{i,\delta}']_{W^{\frac{1}{q},q}(B(9\rho))}^q \\
\leq& C\delta^q + (1+C \frac{\delta}{\rho})^q
\int_{B(0,1)} \int_{B(9\rho)}\int_{B(9\rho)} 
\frac{\left  |
\gamma_{i}'(x+\delta \theta(x) z)-\gamma_{i}'(y+\delta \theta(y) z)
\right | ^q}{|x-y|^{2}}\, dx\, dy \, dz.
\end{split}
\]
Observe that $|1+\delta \theta'(x)z| \geq 1-C\frac{\delta}{\rho}\geq 1-C\sigma$ (for $\delta \leq \sigma \rho$). Moreover, 
\[ 
 \left |x+\delta \theta(x) z -\brac{y+\delta \theta(y) z}\right | \leq |x-y| + \frac{\delta}{\rho} |x-y| \leq 2 |x-y|.
\]
So we can use the change of variables formula, to obtain
\[
\begin{split}
 &\int_{B(9\rho)}\int_{B(9\rho)} 
\frac{\left  |
\gamma_{i}'(x+\delta \theta(x) z)-\gamma_{i}'(y+\delta \theta(y) z)
\right |^q}{|x-y|^{2}}\, dx\, dy\\
\aleq &\frac{1}{(1-C\sigma)^2} \int_{B(9\rho)}\int_{B(9\rho)} 
\frac{\left  |
\gamma_{i}'(x+\delta \theta(x) z)-\gamma_{i}'(y+\delta \theta(y) z)
\right |^q}{|
x+\delta \theta(x) z - (y+\delta \theta(y) z)
|^{2}}\, |1+\delta \theta'(x) z| dx\, |1+\delta \theta'(y) z| dy\\
\aleq &\frac{1}{(1-C\sigma)^2} \int_{B(10\rho)}\int_{B(10\rho)} \frac{|\gamma_i'(x)-\gamma_i'(y)|^q}{|x-y|^{2}}\, dx\, dy.
\end{split}
\]
In conclusion, for $\sigma$ (and $\eps$) small enough we have shown
\begin{equation}\label{eq:asdkgammaideltap}
 [\gamma_{i,\delta}']_{W^{\frac{1}{q},q}(B(9\rho))} \aleq \eps_0 \quad \forall \delta \leq \sigma \rho.
\end{equation}
Here $\eps_0$ can be the small constant of \Cref{la:bilip}, and get that $\gamma'_{i,\delta}$ is uniformly bilipschitz in $B(9\rho)$. Outside it does not even change for all $\delta \leq \sigma \rho$.

Consequently,
\[
 \gamma_{i,\delta}: [0,\sigma\rho] \times \R/\Z \to \R^3
\]
is an isotopy with uniformly bounded bilipschitz estimate, and by \Cref{th:hambientisotop}, $\gamma_{i}$ and $\gamma_{i,\sigma \rho}$ are ambient isotopic.

\underline{$\gamma_{1,\sigma\rho}$ and $\gamma_{2,\sigma\rho}$ are ambient isotopic}

Let
\[
 \Gamma: [0,1] \times \R/\Z \to \R^3
\]
given by the convex combination
\[
 \Gamma(t,x) := t \gamma_{1,\sigma \rho}(x) + (1-t) \gamma_{2,\sigma \rho}(x).
\]
From the estimates above, we have  
\[
 \|\partial_x \Gamma(t,\cdot)\|_{L^\infty} \leq \frac{11}{8} \quad \forall t \in [0,1].
\]
We need to get a uniform bilipschitz estimate for $\Gamma$. 

First, since $\gamma_1 = \gamma_2$ in $\R/\Z \backslash B(3\rho)$ and $\sigma \in (0,\frac{1}{2})$,
\[
 \gamma_{1,\sigma \rho}(x) = \gamma_{2,\sigma\rho}(x) \quad \forall \text{$x \in \R / \Z \backslash B(\tfrac{7}{2}\rho)$}. 
 \]
 Thus,
 \[
  \Gamma(t,x) = \gamma_{1,\sigma \rho}(x) \quad \forall \text{$x \in \R / \Z \backslash B(\tfrac{7}{2}\rho)$}.
 \]
In view of \eqref{eq:asdkgammaideltap} and \eqref{eq:asdkgammaideltap2} combined with \Cref{la:bilip}, we know that $\gamma_{1,\sigma \rho}$ is  uniformly bilipschitz in $B(9\rho)$, namely
\begin{equation}\label{eq:bilip:344646}
\inf_{t \in [0,1]} \inf_{x,y \in B(9\rho)\backslash B(7\rho/2)} \frac{\Gamma(t,x)-\Gamma(t,y)}{|x-y|} > 0.
\end{equation}
Secondly, since $\theta \equiv 0$ in $\R/\Z \backslash B(5\rho)$, 
\[
 \gamma_{1,\sigma \rho}(x) = \gamma_{2,\sigma\rho}(x) = \gamma_1(x)=\gamma_2(x) \quad \forall x \in \R / \Z \backslash B(5\rho).
\]
Thus,
 \[
  \Gamma(t,x) = \gamma_{1}(x)\quad  \forall x \in \R / \Z \backslash B(5\rho).
 \]
Since $\gamma_1$ is a regular diffeomorphism, this implies that 
\begin{equation}\label{eq:bilip:344647}
\inf_{t \in [0,1]} \inf_{x,y \in \R / \Z \backslash B(5\rho)} \frac{\Gamma(t,x)-\Gamma(t,y)}{|x-y|} > 0.
\end{equation}
Combining \eqref{eq:bilip:344646} and \eqref{eq:bilip:344647} we have 
\begin{equation}\label{eq:bilip:344648}
\inf_{t \in [0,1]} \inf_{x,y \in \R / \Z \backslash B(\frac{3}{2}\rho)} \frac{\Gamma(t,x)-\Gamma(t,y)}{|x-y|} > 0.
\end{equation}

It remains to show the bilipschitz estimate for $|\Gamma(t,x)-\Gamma(t,y)|$ only for $x,y \in B(4\rho)$. For such $x,y$,
\[
\begin{split}
 |\Gamma(t,x)-\Gamma(t,y)| \geq& |\gamma_{1,\sigma \rho}(x)-\gamma_{1,\sigma \rho}(y)|\\
&- |\Gamma(t,x)-\gamma_{1,\sigma \rho}(x)-(\Gamma(t,y)-\gamma_{1,\sigma \rho}(y))|\\
\geq& \frac{1}{2}\, |x-y|\\
&- |(t-1)\gamma_{1,\sigma \rho}(x)+(1-t)\gamma_{2,\sigma \rho}(x) -((t-1)\gamma_{1,\sigma \rho}(y)+(1-t)\gamma_{2,\sigma \rho}(y))|\\
=& \frac{1}{2}\, |x-y|\\
&- (1-t) |\gamma_{1,\sigma \rho}(y)-\gamma_{1,\sigma \rho}(x)-\brac{\gamma_{2,\sigma \rho}(y)-\gamma_{2,\sigma \rho}(x) }|\\
\geq& \frac{1}{2}\, |x-y|\\
&- \int_{[x,y]}|\gamma'_{1,\sigma \rho}(z)-\gamma_{2,\sigma \rho}'(z)|\, dz \\
\geq& \brac{\frac{1}{2}-\|\gamma'_{1,\sigma \rho}-\gamma_{2,\sigma \rho}'\|_{L^\infty(B(4\rho))} }\, |x-y|.\\
\end{split}
 \]
Since $x,y \in B(4\rho)$, $\theta(x) = \theta(y) = 1$. Thus,
\[
\begin{split}
 \|\gamma'_{1,\sigma \rho}-\gamma_{2,\sigma \rho}'\|_{L^\infty(B(4\rho))}
 =&\|\eta_{\sigma \rho}\ast \gamma'_{1}-\eta_{\sigma \rho}\ast \gamma_{2}'\|_{L^\infty(B(4\rho))}\\
 \aleq& \frac{1}{\sigma \rho}\|\gamma'_1 - \gamma_2'\|_{L^1(B(6\rho))}\\
 \aleq& \frac{\rho}{\sigma \rho} [\gamma'_1 - \gamma_2']_{W^{\frac{1}{q},q}(B(6\rho))}\\
 \leq& \frac{\eps}{\sigma}.
\end{split}
\]
In the last step we used Poincar\'e inequality (and the fact that $\gamma'_1 = \gamma_2'$ close to $\partial B(6\rho)$).

So if we choose $\eps$ such that $\eps \ll \sigma$ we obtain 
\[
 |\Gamma(t,x)-\Gamma(t,y)| \geq \frac{1}{4}\, |x-y| \quad \forall x,y \in B(4\rho).
 \]
Combining this with \eqref{eq:bilip:344648}, we have shown that $\Gamma$ is an isotopy. In view of \Cref{th:hambientisotop}, $\gamma_{1,\sigma}(\R/\Z)$ and $\gamma_{2,\sigma}(\R/\Z)$ are ambient isotopic. Since in turn $\gamma_{i,\sigma}$ and $\gamma_i$  are ambient isotopic for $i=1,2$, we have proven that $\gamma_1$ is ambient isotopic to $\gamma_2$.
\end{proof}

\section{Homeomorphisms appear as limits: Proof of Theorem~\ref{th:weaklimitareweakimm}}\label{s:weaklimitimmersion}
It is easy to construct Lipschitz parametrization of curves $\gamma: \R /\Z \to \R^3$ with vanishing tangent-point energy $\tp^{p,q}(\gamma) = 0$, $p \geq q+2$, but with no reasonable regularity, namely $\gamma \not \in C^1(\R/\Z,\R^3)$ and $\gamma \not \in W^{1+\frac{p-q-1}{q},q}(\R/\Z,\R^3)$.

\begin{example}\label{ex:weirdfiniteenergy}
For any Lipschitz map $\tilde{\gamma}: \R/\Z \to [0,1/2]$ with $|\tilde{\gamma}'| \equiv 1$, if we set $\gamma(x) := (\tilde{\gamma},0,0) \in \R^3$ then
\[
 |\gamma'(x) \wedge (\gamma(x)-\gamma(y))| = 0.
\]
In particular, if for any $x \in \R/\Z$ there are only finitely many $y \in \R/\Z$ such that $\tilde{\gamma}(x) = \tilde{\gamma}(y)$, we have that 
\[
 \frac{|\gamma'(x) \wedge (\gamma(x)-\gamma(y))|^q}{|\gamma(x)-\gamma(y)|^{p}} = 0 \quad \text{$\mathcal{L}^2$-a.e. $(x,y) \in (\R/\Z)^2$},
\]
and thus $\tp^{p,q}(\gamma) = 0$.

E.g., take $\tilde{\gamma}$ to be
\[
 \tilde{\gamma}(t) := \begin{cases}
                       t \quad &t < \frac{1}{2}\\
                       \frac{1}{2}-t \quad &t \in [\frac{1}{2},1]
                      \end{cases}.
\]
Then $\gamma'$ has a jump discontinuity at $t= \frac{1}{2}$ and $t=0$. Thus $\gamma' \not \in C(\R/\Z, \R^3)$ and $\gamma' \not \in W^{\frac{p-1}{q},q}(\R/\Z,\R^3)$ whenever $\gamma \not \in W^{1+\frac{p-q-1}{q},q}(\R/\Z,\R^3)$ for any $p \geq q+2$ and $q \in (1,\infty)$.

It is easy to extend this example into  a map $\gamma$ with countably many points of non-differentiability but still $\tp^{p,q}(\gamma) = 0$.
\end{example}
See also example of $k$-covered circle \cite[after Theorem 1.1]{SvdM12}.

\Cref{ex:weirdfiniteenergy} shows that there is no hope to classify a reasonable energy space of Lipschitz maps $\gamma: \R/\Z \to \R^3$ with finite tangent-point energy. Rather we investigate the space of diffeomorphisms with finite tangent-point energy, which turns out to be more manageable -- this is the content of the following \Cref{th:convoutsidesingular} which is the main theorem of this section. In particular, \Cref{th:convoutsidesingular} implies \Cref{th:weaklimitareweakimm}.

\begin{theorem}\label{th:convoutsidesingular}
    For any $\Lambda > 0$ and $\eps > 0$ there exists an $L = L(\eps,\Lambda) \in \N$ such that the following holds.

	Let $\gamma_k \in C^1(\R/\Z,\R^3)$, $|\gamma_k'| \equiv 1$, be homeomorphisms with 
	\[
	\sup_{k}\| \gamma_k\|_{L^\infty} + \sup_{k} \tp^{p,q} (\gamma_k) \leq \Lambda.
	\]
	Then there exists a subsequence $(\gamma_{k_i})_{i \in \N}$ and $\gamma \in \lip(\R/\Z,\R^3)$ such that the following holds for some finite set $\Sigma \subset \R/\Z$ with $\#\Sigma \leq L$.  
	
	\begin{enumerate}
	\item $\gamma_{k_i}$ converges uniformly to $\gamma$ and $\gamma \in \lip(\R/\Z,\R^3)$.
    \item For any $x_0 \in \R / \Z \backslash \Sigma$ there exists a radius $\rho(x_0) > 0$ such that $\gamma_{k_i}$ weakly converges to $\gamma$ in $W^{1+\frac{p-q-1}{q},q}(B(x_0,\rho))$.
	\item $|\gamma'| = 1$ a.e.
	\item $\gamma$ is uniformly bilipschitz in $B(x_0,\rho)$ with the estimate
	\[
	 (1-\eps) |x-y| \leq |\gamma(x)-\gamma(y)| \leq |x-y| \quad \forall x,y \in B(x_0,\rho).
	\]
	
	\item We have lower semicontinuity, namely
	\[
	 \tp^{p,q} (\gamma) \leq \liminf_{k \to \infty} \tp^{p,q} (\gamma_k).
	\]
    
    \item $\gamma$ is a bi-Lipschitz homeomorphism.
    \item $\gamma \in W^{1+s,q}(\R/\Z,\R^3)$ for any $0<s < \frac{1}{q}$.
	\end{enumerate}
\end{theorem}
We will prove a more detailed version of \Cref{th:convoutsidesingular} in \Cref{pr:convoutsidesingular}.

In order to prove \Cref{th:convoutsidesingular}, we proceed in several steps.
\begin{itemize}
 \item First we prove in \Cref{s:uniformsmallness} that for the approximating sequence $\gamma_k$ the local tangent-point energy is uniformly small away from a finite set $\Sigma$  (we will refer to it as the ``singular set'') of points of energy concentration. 
 \item In \Cref{s:sobolevconv} we obtain the Sobolev estimate for smooth curves whenever the tangent-point energy is locally small, see \Cref{th:gammasobvstp2}, and as a consequence a bi-Lipschitz estimate. This estimate is obtained by a gap-estimate. In particular this method characterizes the energy space for the tangent-point energies in the scale-invariant case.
 \item In \Cref{s:svdm} we adapt an argument due to Strzelecki and von der Mosel \cite{SvdM12} to obtain a uniform estimate on global injectivity of the approximating sequence $\gamma_k$ away from the singular points, see~\Cref{th:smallenergyinjectivity}.
  \item In \Cref{s:convoutsidesingular} we then obtain in \Cref{pr:convoutsidesingular} the convergence outside the singular set which implies \Cref{th:convoutsidesingular}.
\end{itemize}

\subsection{Locally uniform smallness}\label{s:uniformsmallness}
In the first step we ensure that away from a discrete set we have locally uniformly small energy in the approximating sequence.

\begin{proposition}\label{pr:uniformsmallness}
For any $\eps > 0$ and $\Lambda > 0$ there exists $L = L(\eps,\Lambda)$ such that the following holds.

For any sequence $\gamma_k \in \lip(\R /\Z ,\R^3)$, $|\gamma_k'|\equiv 1$, such that 
\[
 \sup_{k} \tp^{p,q}(\gamma_k) \leq \Lambda,
\]
there exists a subsequence $\gamma_{k_i}$ and set $\Sigma \subset \R /\Z$ consisting of at most $L$ points such that for any $x_0 \in \R /\Z \backslash \Sigma$ there exists a radius $\rho =\rho_{x_0} > 0$ and an index $K \in \N$ such that 
\[
 \sup_{i \geq K} \int_{B(x_0,\rho_{x_0})} \int_{\R/\Z} \frac{|\gamma_{k_i}'(x) \wedge (\gamma_{k_i}(x)-\gamma_{k_i}(y))|^q}{|\gamma_{k_i}(x)-\gamma_{k_i}(y)|^{p}}\, dy\, dx < \eps.
\]
\end{proposition}
\Cref{pr:uniformsmallness} follows for any integral energy from a relatively standard covering argument, see e.g. \cite[Proposition 4.3 and Theorem 4.4.]{SU81}.
\preprintonly{\begin{proof}
We give the details for the convenience of the reader. 

Pick $\delta << \frac{\eps}{2\Lambda}$ and let $m\in\N$. Then cover $\RZ$ by at most $2 \lceil(\delta 2^{-m})^{-1}\rceil$ intervals $B(x_i,\delta 2^{-m})$ such that every point $x\in\RZ$ is covered at most two times. Then we have
	\begin{align*}
		& \sum_i \int_{B(x_i,\delta 2^{-m})} \int_{\RZ} \frac{|\gamma_{k}'(x) \wedge (\gamma_{k}(x)-\gamma_{k}(y))|^q}{|\gamma_{k}(x)-\gamma_{k}(y)|^{q+2}}\, dy\, dx  \\
		& \leq 2 \int_{\RZ} \int_{\RZ} \frac{|\gamma_{k}'(x) \wedge (\gamma_{k}(x)-\gamma_{k}(y))|^q}{|\gamma_{k}(x)-\gamma_{k}(y)|^{q+2}}\, dy\, dx < 2\Lambda = \tfrac{2\Lambda}{\eps} \eps.
	\end{align*}
	Hence for every $\gamma_k$ there exist at most $L:=L(\eps,\Lambda):= \lfloor \tfrac{2\Lambda}{\eps} \rfloor$ intervals $B(x_i,\delta 2^{-m})$ such that
	\[
		\int_{B(x_i,\delta 2^{-m})} \int_{\RZ} \frac{|\gamma_{k}'(x) \wedge (\gamma_{k}(x)-\gamma_{k}(y))|^q}{|\gamma_{k}(x)-\gamma_{k}(y)|^{q+2}}\, dy\, dx \geq \eps.
	\]
	Now assume that we have already shown for $i\in \{1,\ldots, n\}$ that 
	\[
		\sup_k \int_{B(x_i,\delta 2^{-m})} \int_{\RZ} \frac{|\gamma_{k}'(x) \wedge (\gamma_{k}(x)-\gamma_{k}(y))|^q}{|\gamma_{k}(x)-\gamma_{k}(y)|^{q+2}}\, dy\, dx < \eps.
	\]
	If there exist more than $L$ intervals $B(x_i,\delta 2^{-m})$, $i>n$, with
	\[
	\sup_k \int_{B(x_i,\delta 2^{-m})} \int_{\RZ} \frac{|\gamma_{k}'(x) \wedge (\gamma_{k}(x)-\gamma_{k}(y))|^q}{|\gamma_{k}(x)-\gamma_{k}(y)|^{q+2}}\, dy\, dx \geq \eps,
	\]
	there must exist at least one $B(x_{n+1},\delta 2^{-m})$ among them and a subsequence of $\gamma_k$ 
such that 
	\[
	\sup_k \int_{B(x_{n+1},\delta 2^{-m})} \int_{\RZ} \frac{|\gamma_{k}'(x) \wedge (\gamma_{k}(x)-\gamma_{k}(y))|^q}{|\gamma_{k}(x)-\gamma_{k}(y)|^{q+2}}\, dy\, dx < \eps.
	\]
	By repeating this step, we find a subsequence of $\gamma_k$ for which 
	\[
		\sup_k \int_{B(x_i,\delta 2^{-m})} \int_{\RZ} \frac{|\gamma_{k}'(x) \wedge (\gamma_{k}(x)-\gamma_{k}(y))|^q}{|\gamma_{k}(x)-\gamma_{k}(y)|^{q+2}}\, dy\, dx < \eps
	\]
	holds for all given intervals apart from $L$ many $B(x_{i,m},\delta 2^{-m})$.

	Applying this method iteratively for $m \rightarrow \infty$, we can construct a series of subsequences such that for each subsequence $\gamma_{k,m}$ we have
	\[
	\sup_{k} \int_{B(x_i,\delta 2^{-m})} \int_{\RZ} \frac{|\gamma_{k,m}'(x) \wedge (\gamma_{k,m}(x)-\gamma_{k,m}(y))|^q}{|\gamma_{k,m}(x)-\gamma_{k,m}(y)|^{q+2}}\, dy\, dx < \eps,
	\]
	where $B(x_{i},\delta 2^{-m}) \subset \RZ \backslash \bigcup_{j\leq m-1} B(x_{i,j},\delta 2^{-j})$. 
	
	Now we choose a diagonal subsequence $\gamma_{k_i}$ (one element per $\gamma_{k,m}$).
	Since 
	\[
	\bigcup_m (\RZ \backslash \bigcup_{i\leq L} B(x_{i,m}, \delta 2^{-m})) = \RZ \backslash \bigcap_m \bigcup_{i\leq L} B(x_{i,m}, \delta 2^{-m}) = \RZ \backslash \{x_1,\ldots,x_L\}
	\]
	for at most  $x_1,\ldots,x_L\in\RZ$, 
	there indeed exists for any $x_0\in\RZ\backslash \{x_1,\ldots,x_L\}$ a $\rho_{x_0}>0$ such that $B(x_0,\rho_{x_0}) \subset B(z_i,\delta 2^{-K})$ for a $K\in\N$ and therefore
	\[
	\sup_{i\geq K} \int_{B(x_0,\rho_{x_0})} \int_{\RZ} \frac{|\gamma_{k_i}'(x) \wedge (\gamma_{k_i}(x)-\gamma_{k_i}(y))|^q}{|\gamma_{k_i}(x)-\gamma_{k_i}(y)|^{q+2}}\, dy\, dx < \eps.
	\]
\end{proof}}

\subsection{Small local energy implies local Sobolev space estimates}\label{s:sobolevconv}
The main novel ingredient underlying our argument for \Cref{th:convoutsidesingular} is a gap estimate for Sobolev spaces with respect to the tangent-point energy. 

As discussed in \Cref{ex:weirdfiniteenergy}, it is impossible to control the Sobolev norm of $\gamma$ in terms of the tangent-point energy of $\gamma$, $\tp^{p,q}(\gamma)$, without assuming a priori bilipschitz estimates (as it was done in \cite{BR15} see also \cite{B19}). This is however not a viable method for the scale-invariant case $p=q+2$ because the bi-Lipschitz constant is not uniformly controlled as a sequence $\gamma_k$ converges to $\gamma$. We turn this argument around and first a priori assume that the Sobolev norm is finite, and then conclude that this is an estimate which is uniform for sequences $\gamma_k$ converging to $\gamma$.

The first step is the following gap estimate\footnote{\Cref{la:gammasobvstp1} is called a gap estimate, because it implies the following: For $\eps := \brac{\frac{1}{2C(p,q)}}^{\frac{1}{q}}$ we have either $[\gamma']_{W^{\frac{1}{q},q}(B)}^q \leq 2C(p,q)\, \tp^{q+2,q}(\gamma,B)$ or $[\gamma']_{W^{\frac{1}{q},q}(B)} \geq \eps$.
}.
\begin{lemma}\label{la:gammasobvstp1}
Let $p \in [q+2,2q+1)$ and $q>1$. Let $\gamma \in \lip(\R/\Z,\R^3)$, $|\gamma'| \equiv 1$. Then for any ball $B \subset \R / \Z$ of diameter less than $\frac{1}{2}$,
\begin{equation}\label{eq:gammasobvstp1:goal}
 [\gamma']_{W^{\frac{p-q-1}{q},q}(B)}^q \leq C(p,q)\, \tp^{p,q}(\gamma,B) + C(p,q)\, [\gamma']_{W^{\frac{p-q-1}{q},q}(B)}^{2q},
\end{equation}
whenever the right-hand side is finite.
\end{lemma}
\begin{proof}[Proof of \Cref{la:gammasobvstp1}]
The assumption that $B$ has diameter less than $\frac{1}{2}$ implies that $B$ is a geodesic ball and thus convex with respect to the $\R/\Z$-metric. To simplify matters even more, we assume w.l.o.g. that the ball $B$ is centered at $0$ so that $|x-y|$ is actually the Euclidean distance.

Recall the Lagrange identity for $v,w \in \R^3$, $|v| = 1$, 
\[
 |v \wedge w|^2 = |w|^2 - |v \cdot w|^2.
\]
Moreover, observe that $|\gamma'| \equiv 1$ implies
\[
 |\gamma(x)-\gamma(y)| \leq |x-y|.
\]
Then,
\[
\begin{split}
 &\frac{\left |\gamma'(x)\wedge (\gamma(x)-\gamma(y)) \right |^q}{|\gamma(x)-\gamma(y)|^p}\\
 \geq&\frac{\left |\gamma'(x)\wedge (\gamma(x)-\gamma(y)) \right |^q}{|x-y|^p} \\
=&\frac{\brac{\left |\gamma'(x)\wedge \brac{\gamma(y)-\gamma(x)-\gamma'(x)(y-x)} \right |^2}^{\frac{q}{2}}}{|x-y|^p}\\
=&\frac{\brac{\left |\brac{\gamma(y)-\gamma(x)-\gamma'(x)(y-x)} \right |^2-\left |\gamma'(x) \cdot \brac{\gamma(y)-\gamma(x)-\gamma'(x)(y-x)} \right |^2}^{\frac{q}{2}}}{|x-y|^p}.\\ 
 \end{split}
\]
We have
\begin{equation}\label{eq:simplecomp}
\begin{split}
  \gamma'(x) \cdot \brac{\gamma'(z)-\gamma'(x)}
  =&\frac{1}{2} \brac{2\gamma'(x) \cdot \gamma'(z)-1-1}\\
  =&-\frac{1}{2} \abs{\gamma'(x)-\gamma'(z)}^2,\\
  \end{split}
\end{equation}
so
\[
\begin{split}
 &\left |\gamma'(x) \cdot \brac{\gamma(y)-\gamma(x)-\gamma'(x)(y-x)} \right |\\
 =&|y-x| \left | \mvint_{(x,y)} \gamma'(x) \cdot (\gamma'(z)-\gamma'(x)) \,dz\right |\\
 =& \frac{1}{2} |y-x| \mvint_{(x,y)} |\gamma'(x) -\gamma'(z)|^2\, dz.
 \end{split}
\]
Consequently,
\[
\begin{split}
 &\frac{\left |\gamma'(x)\wedge (\gamma(x)-\gamma(y)) \right |^q}{|\gamma(x)-\gamma(y)|^p}\\
\geq&\frac{\brac{\left |\brac{\gamma(y)-\gamma(x)-\gamma'(x)(y-x)} \right |^2-\frac{1}{4}|y-x|^2\left |\mvint_{(x,y)} |\gamma'(x)-\gamma'(z)|^2\, dz\right |^2}^{\frac{q}{2}}}{|x-y|^p}.\\ 
 \end{split}
\]
Observe that from our computations we know that in particular,
\[
 \left |\brac{\gamma(y)-\gamma(x)-\gamma'(x)(y-x)} \right |^2 \geq \frac{1}{4}|y-x|^2\left |\mvint_{(x,y)} |\gamma'(x)-\gamma'(z)|^2\, dz\right |^2.
\]
Also observe that for any $r > 0$ there exists $c_r \in (0,1)$, $C_r > 0$ such that for any $a \geq b \geq 0$ 
\[
 (a-b)^r \geq c_r\, a^r - C_r\, b^r.
\]
From this, and by Jensen's inequality
\[
\begin{split}
 &\frac{\left |\gamma'(x)\wedge (\gamma(x)-\gamma(y)) \right |^q}{|\gamma(x)-\gamma(y)|^p}\\
\geq&c\frac{\left |\brac{\gamma(y)-\gamma(x)-\gamma'(x)(y-x)} \right |^q}{|x-y|^p}-C\frac{|y-x|^q\left |\mvint_{(x,y)} |\gamma'(x)-\gamma'(z)|^2\, dz\right |^q}{|x-y|^p}\\ 
=&c\frac{\left | \frac{\gamma(y)-\gamma(x)-\gamma'(x)(y-x)}{|x-y|} \right |^q}{|x-y|^{p-q}}-C\frac{\mvint_{(x,y)} |\gamma'(x)-\gamma'(z)|^{2q}\, dz}{|x-y|^{p-q}}\\ 
 \end{split}
\]
Integrating $x$ and $y$ in $B$ we obtain
\begin{equation}\label{eq:23231}
\begin{split}
 &\int_{B}\int_{B} \frac{\abs{\frac{\gamma(y)-\gamma(x)-\gamma'(x)(y-x)}{|x-y|}}^q}{|x-y|^{p-q}}\, dx\, dy \\
 \leq& C_q \brac{\tp^{p,q}(\gamma,B) + \int_{B}\int_{B} \frac{\mvint_{(x,y)} |\gamma'(x)-\gamma'(z)|^{2q}\, dz}{|x-y|^{p-q}}\, dx\, dy}.
 \end{split}
\end{equation}
Recall that $p\in[q+2,2q+1)$, so $\frac{p-q-1}{q} \in (0,1)$. From \Cref{la:id1}
\[
 [\gamma']_{W^{\frac{p-q-1}{q},q}(B)}^q \aeq \int_{B}\int_{B} \frac{\abs{\frac{\gamma(y)-\gamma(x)-\gamma'(x)(y-x)}{|x-y|}}^q}{|x-y|^{p-q}}\, dx\, dy.
\]
From \Cref{la:id2}
\[
 [\gamma']_{W^{\frac{p-q-1}{2q},2q}(B)}^{2q} \aeq \int_{B}\int_{B} \frac{\mvint_{(x,y)} |\gamma'(x)-\gamma'(z)|^{2q}\, dz}{|x-y|^{p-q}}\, dx\, dy.
\]
Also, from the Sobolev inequality, \Cref{la:sob1}, we have
\[
 [\gamma']_{W^{\frac{p-q-1}{2q},2q}(B)} \aleq (\diam B)^{\frac{p-q-2}{2q}} [\gamma']_{W^{\frac{p-q-1}{q},q}(B)} \leq [\gamma']_{W^{\frac{p-q-1}{q},q}(B)}.
\]
Thus, \eqref{eq:23231} implies
\[
 [\gamma']_{W^{\frac{p-q-1}{q},q}(B)}^q \aleq \tp^{p,q}(\gamma,B) + [\gamma']_{W^{\frac{p-q-1}{q},q}(B)}^{2q},
\]
which is \eqref{eq:gammasobvstp1:goal}. The proof is concluded.
\end{proof}

The gap-estimate leads to the following control of the Sobolev norm. We stress that we need to assume \emph{a priori} that $\gamma$ already belongs to the Sobolev space in question, which rules out the irregular curves in \Cref{ex:weirdfiniteenergy}.

\begin{theorem}\label{th:gammasobvstp2}
Let $q_0,p_0 > 1$, $q_1 < \infty$, $p_1 < \infty$ such that $p_1 - 2q_0 < 1$. Let $\eps > 0$, then there exists $\delta = \delta(q_0,p_0,q_1,p_1,\eps) > 0$ and a constant $C=C(q_0,p_0,q_1,p_1) > 0$ such that the following holds for any $p \in [p_0,p_1]$ and $q \in [q_0,q_1]$.

Let $\gamma \in \lip(\R/\Z,\R^3)$, $|\gamma'| \equiv 1$, and assume that for some ball $B \subset \R / \Z$, $\diam(B) < \frac{1}{2}$, we have 
\[
 \tp^{p,q}(\gamma,B) < \delta
\]
and
\begin{equation}\label{eq:eitherc1orsob}
 \text{either }\gamma \in C^1(B) \quad \text{ or }\quad [\gamma']_{W^{\frac{p-q-1}{q},q}(B)}^q < \infty.
\end{equation}
Then
\begin{equation}\label{eq:gsob:goal1}
 [\gamma']_{W^{\frac{p-q-1}{q},q}(B)}^q \leq C(q_0,p_0,q_1,p_1)\, \tp^{p,q}(\gamma,B)
\end{equation}
and we have the bi-Lipschitz estimate
\begin{equation}\label{eq:gsob:goal2}
 (1-\eps)|x-y| \leq |\gamma(x)-\gamma(y)| \leq |x-y| \quad \forall x,y \in B.
\end{equation}
\end{theorem}
 
\begin{proof}
In view of \cite[Theorem 1.1, Remark 1.6]{BR15}, $\gamma \in C^1(B)$ and $\tp^{p,q}(\gamma,B) < \infty$ implies
\[
 [\gamma']_{W^{\frac{p-q-1}{q},q}(B)}^q < \infty,
\]
so that in \eqref{eq:eitherc1orsob} we can assume the Sobolev-space estimate to hold.

Assume that $B = B(x_0,R)$ for some $x_0 \in \R/\Z$ and $R \in (0,\frac{1}{4}]$.

Set 
\[
 \lambda := \tp^{p,q}(\gamma,B) < \delta.
\]

For $r \in (0,R]$ set $B(r) := B(x_0,r)$.

By \Cref{la:gammasobvstp1} we have for any $r \in (0,R]$,
\[
 [\gamma']_{W^{\frac{p-q-1}{q},q}(B(r))}^q \leq C_1 \tp^{p,q}(\gamma,B(r)) + C_2 [\gamma']_{W^{\frac{p-q-1}{q},q}(B(r))}^{2q}.
\]
Set 
\[
 f(r) := [\gamma']_{W^{\frac{p-q-1}{q},q}(B(r))}^q.
\]
Then we have
\[
 f(r) \leq C_1 \lambda + C_2 (f(r))^2 \quad \forall r \in (0,R).
\]
Setting $p(t) := C_2 t^2 - t + C_1 \lambda$, we have 
\begin{equation}\label{eq:gap:pfr}
 p(f(r)) \geq 0 \quad \forall r \in (0,R).
\end{equation}
The roots of the polynomial $p$ are
\[
 t_{\lambda;1} := \frac{1}{2C_2} - \sqrt{\frac{1}{(2C_2)^2} - \frac{C_1}{C_2} \lambda}, \quad t_{\lambda;2} := \frac{1}{2C_2} + \sqrt{\frac{1}{2C_2^2} - \frac{C_1}{C_2} \lambda}.
\]
Let $\delta$ be small enough so that $\frac{1}{(2C_2)^2} - {\frac {C_1} {C_2 }}\delta > \frac{1}{(4C_2)^2}$. Then, whenever $\lambda \in (0,\delta)$ we have $t_{\lambda;1} < t_{\lambda;2}$ and moreover
\begin{equation}\label{eq:gap:tlambda}
 t_{\lambda;1} \leq C(C_1,C_2)\, \lambda  \quad \forall \lambda \in (0,\delta).
\end{equation}
The polynomial $p$ is negative only on the interval $(t_{\lambda,1},t_{\lambda,2})$. From \eqref{eq:gap:pfr} we deduce that for each $r \in (0,R)$ either $f(r) \leq t_{\lambda,1}$ or $f(r) \geq t_{\lambda,2}$. Since $f(0) = 0$ and $f$ is continuous, we conclude that necessarily \[f(r) \leq t_{\lambda,1} \quad\text{for each $r < R$}.\]
That is, in view of \eqref{eq:gap:tlambda},
\[
[\gamma']_{W^{\frac{p-q-1}{q},q}(B)}^q \leq C(C_1,C_2)\, \lambda.
\]
Recalling the definition of $\lambda$, we conclude \eqref{eq:gsob:goal1}.

Choosing $\eps > 0$ possibly even smaller, we obtain also \eqref{eq:gsob:goal2} as a consequence of \eqref{eq:gsob:goal1} and \Cref{la:bilip}.
\end{proof}

Let us also remark, for the sake of completeness, that the argument in the proof of \Cref{la:gammasobvstp1} also gives a real classification of the energy space, if one assumes a priori bi-Lipschitz estimates (cf. \cite[Proposition 2.4]{BR15}).
\begin{lemma}\label{la:energyspace}
 Let $p \in [q+2,2q+1)$ and $q>1$. Let $\gamma \in \lip(\R/\Z,\R^3)$, $|\gamma'| \equiv 1$ and $\gamma: \R/\Z \to \R$ be bi-Lipschitz, i.e.\ 
 \[
  (1-\lambda)|x-y| \leq |\gamma(x)-\gamma(y)| \leq |x-y|
 \]
Then for any ball $B \subset \R / \Z$ of diameter less than $\frac{1}{2}$,
\[
 \tp^{p,q}(\gamma,B) \leq C(p,q,\lambda)[\gamma']_{W^{\frac{p-q-1}{q},q}(B)}^q + C(p,q,\lambda)\, [\gamma']_{W^{\frac{p-q-1}{q},q}(B)}^{2q}.
\]
\end{lemma}
With the help of \Cref{la:energyspace} we obtain
\begin{example}\label{ex:nonsmooth}
There exists a homeomorphism $\gamma: \R /\Z \to \R^3$ which is bilipschitz, whose derivative is not everywhere continuous, but has finite tangent-point energy $\tp^{q+2,q}(\gamma)$ for any $q > 1$. Moreover, there exists a sequence of $C^\infty$-diffeomorphisms $\gamma_k$ converging uniformly to $\gamma$ with uniformly bounded tangent-point energy, i.e.\  $$\sup_{k \in \N} \tp^{q+2,q}(\gamma_k) < \infty. $$

Indeed, denote by $N = (0,0,1)$ the north pole of $\S^2$.

Let $u\in W^{\frac{1}{q},q}([-\frac{1}{4},\frac{1}{4}],\S^2) \backslash C^0([-\frac{1}{4},\frac{1}{4}],\R^3)$ such that 
\[
 \langle u,N \rangle \geq \frac{1}{4}
\]
and $u$ is constant for $|x| \leq -\frac{1}{8}$ and $|x| \geq \frac{1}{8}$.

For example for any $\eta \in C_c^\infty((-\frac{1}{8},\frac{1}{8}),[0,1])$ with $\eta \equiv 1$ in $[-\frac{1}{16},\frac{1}{16}]$ we could set
\[
 u(x) = \brac{\frac{1}{\sqrt{2}}\sin(\eta(x) \log \log 1/|x|),\frac{1}{\sqrt{2}}\cos(\eta(x)\log \log 1/|x|),\frac{1}{\sqrt{2}}}.
\]
Now let for $x \in [-\frac{1}{4},\frac{1}{4}]$,
\[
 \gamma(x) = \int_{-\frac{1}{4}}^x u(z)\, dz.
\]
Then $\gamma$ is bilipschitz in $[-\frac{1}{4},\frac{1}{4}]$ because 
\[
 |\gamma(x)-\gamma(y)|\geq \langle \gamma(x)-\gamma(y),N\rangle = \int_{[x,y]} \langle u, N \rangle \geq \frac{1}{4} |x-y|.
\]
Observe that $\gamma'$ is constant around $x \approx -\frac{1}{4}$ and $x \approx \frac{1}{4}$, so $\gamma$ can be smoothly extended into a closed curve on $[-\frac{1}{2},\frac{1}{2}]$ which is a smooth 1-D manifold outside of $[-\frac{1}{4},\frac{1}{4}]$. By \Cref{la:energyspace} the curve $\gamma$ has finite tangent-point energy $\tp^{q+2,q}$ but $\gamma$ is not $C^1$ since $\gamma'$ is discontinuous. 

On the other hand, in view of \Cref{s:topology} any regular homeomorphism $\gamma \in W^{1+\frac{1}{q},q}$ can be approximated by smooth homeomorphisms with uniformly controlled bi-Lipschitz constant, so that in view of \Cref{la:energyspace} the tangent-point energy $\tp^{q+2,q}$ is uniformly bounded.
\end{example}

\subsection{The Strzelecki--von der Mosel argument: Locally small energy implies global injectivity}\label{s:svdm}
In this section we provide a reformulation of a powerful argument due to Strzelecki and von der Mosel, \cite{SvdM12}, see also \cite[Appendix]{BR15}. They used it to show that the image of a curve with finite tangent-point energy ($p\geq q+2$) is a topological $1$-manifold embedded into $\R^3$. Recall that this manifold could be the twice covered straight line, \Cref{ex:weirdfiniteenergy}.

We rework their argument to provide us with uniform injectivity for intervals with small energy, \Cref{th:smallenergyinjectivity}.

The following is essentially a reformulation (with a slight refinement) of \cite[Lemma 2.1]{SvdM12}.

\begin{lemma}[Strzelecki--von der Mosel]\label{la:vdms}
    Let $p \geq q+2$. For any $\eps > 0$ there exists $\delta > 0$ such that the following holds.

	Let $\gamma \in \lip(\R/\Z,\R^3)$, $|\gamma'| \equiv 1$, and assume that for some $x_0\in \R / \Z$ and $\rho >0$ we have 
	\begin{equation}\label{eq:svdm:1}
	\int_{B(x_0,\rho)} \int_{\R / \Z} \frac{\left |\gamma'(x) \wedge (\gamma(y)-\gamma(x)) \right |^{q}
}{|\gamma(x)-\gamma(y)|^{p}} \, dy \,  dx < \delta.
	\end{equation}
	Moreover, assume that there is $y_0 \in \R / \Z$ with $d := |\gamma(y_0)-\gamma(x_0)| \leq \rho$.
	
	Then
	\[
		\gamma(\R / \Z) \cap B_{2d}(\gamma(x_0)) \subset B_{\eps d} (L(\gamma(x_0),\gamma(y_0))),
	\]
	where $L(\gamma(x_0),\gamma(y_0))$ is the straight line containing $\gamma(x_0)$ and $\gamma(y_0)$ defined by
	\[
	 L(\gamma(x_0),\gamma(y_0)) = \left \{(1-t) \gamma(x_0) + t \gamma(y_0), t \in \R \right \}.
	\]
\end{lemma}
\begin{proof}
	For $r>0$ and $p,v \in \R^3$ we define
	\begin{align*}
		A(r,p) & := \{x \in \R / \Z: |\gamma(x)-p| < r\}, \\
		X(r,v) & := \left \{x \in \R / \Z: \abs{\gamma'(x) \wedge v} \geq r \right \}.
	\end{align*}
	Fix $x_0, y_0 \in \R / \Z$ and $d:= |\gamma(x_0) - \gamma(y_0)|$ as in the assumption. Set $v:=\frac{\gamma(x_0)-\gamma(y_0)}{|\gamma(x_0)-\gamma(y_0)|}$.
	
	\underline{Step 1} is devoted to show the following: There exists $\delta_0,\sigma_0\in (0,1)$ depending only on $p$ and $q$ such that if $\delta<\delta_0$ in \eqref{eq:svdm:1} and $\sigma < \sigma_0$, we have
	\begin{equation}\label{eq:stzvdm:step1}
		|B(x_0,\rho) \cap A(\sigma^2 d,\gamma(x_0)) \cap X(\sigma,v)^c| > \frac 32 \sigma^2 d.
	\end{equation}
    In order to establish \eqref{eq:stzvdm:step1}, we first observe that for any $x\in A(\sigma^2 d, \gamma(x_0))$ and $y \in A(\sigma^2 d, \gamma(y_0))$ we have 
	\[
		|\gamma(x)-\gamma(y)| \in [(1-2\sigma^2) d,(1+2\sigma^2) d].
	\]
	Moreover, for $x \in A(\sigma^2 d, \gamma(x_0)) \cap X(\sigma,v)$ and $y \in A(\sigma^2 d, \gamma(y_0))$, 
	\[
	\begin{split}
	 &\abs{\gamma'(x) \wedge \frac{\gamma(x)-\gamma(y)}{|\gamma(x)-\gamma(y)|} }\\
	 \geq& \frac{1}{(1+2\sigma^2)d} \abs{\gamma'(x) \wedge \brac{\gamma(x)-\gamma(y)}}\\
	 \geq& \frac{1}{(1+2\sigma^2)d} \brac{\abs{\gamma'(x) \wedge \brac{\gamma(x_0)-\gamma(y_0)}}-2\sigma^2d}\\
	 =&\frac{1}{(1+2\sigma^2)} \brac{\abs{\gamma'(x) \wedge \frac{\gamma(x_0)-\gamma(y_0)}{|\gamma(x_0)-\gamma(y_0)|}}-2\sigma^2}\\
	 \geq&\frac{1}{(1+2\sigma^2)} \brac{\sigma-2\sigma^2}\\
	 =&\sigma \frac{1-2\sigma}{1+2\sigma^2}.
	 \end{split}
	\]
	Hence, whenever $x \in A(\sigma^2 d, \gamma(x_0)) \cap X(\sigma,v)$ and $y \in A(\sigma^2 d, \gamma(y_0))$, 
	\[
	\frac{\abs{\gamma'(x) \wedge \brac{\gamma(x)-\gamma(y)}}^q}{\abs{\gamma(x)-\gamma(y)}^{p}} \geq  \frac{\sigma^q}{d^{p-q}} \brac{\frac{1-2\sigma}{1+2\sigma^2}}^q \frac{1}{(1+2\sigma^2)^{p-q}}.
	\]
	From \eqref{eq:svdm:1} we find
	\[
\delta > \frac{\sigma^q}{d^{p-q}} \brac{\frac{1-2\sigma}{1+2\sigma^2}}^q \frac{1}{(1+2\sigma^2)^{p-q}}\, |B(x_0,\rho) \cap A(\sigma^2 d,\gamma(x_0)) \cap X(\sigma,v)|\, |A(\sigma^2 d,\gamma(y_0))|.
	\]
Observe that since $|\gamma'|\equiv 1$, we have $|A(\sigma^2 d,\gamma(y_0))|\geq 2\sigma^2 d$.
Then we have
\begin{equation}\label{eq:svdm:3535}
|B(x_0,\rho) \cap A(\sigma^2 d,\gamma(x_0)) \cap X(\sigma,v)|  \leq \frac{d^{p-q-1}}{\sigma^{q+2}}  \brac{\frac{1+2\sigma^2}{1-2\sigma}}^{q} (1+2\sigma^2)^{p-q}\frac{1}{2} \delta.
\end{equation}
Since $d < \rho$, $\sigma \in (0,1)$, and $|\gamma'| \equiv 1$, we have
\begin{equation}\label{eq:svdm:3536}
 |B(x_0,\rho) \cap A(\sigma^2 d,\gamma(x_0))| \geq 2\sigma^2 d.
\end{equation}
Taking $\sigma_0$ and $\delta_0$ to be small enough, combining \eqref{eq:svdm:3535} and \eqref{eq:svdm:3536}, we obtain for any $\sigma \in (0,\sigma_0)$ and any $\delta \in (0,\delta_0)$,
\[
\begin{split}
	&|B(x_0,\rho) \cap A(\sigma^2 d,\gamma(x_0)) \cap X(\sigma,v)^c| \\
	\geq& 2\sigma^2 d - |B(x_0,\rho) \cap A(\sigma^2 d,\gamma(x_0)) \cap X(\sigma,v)| \\
	\geq& \frac{3}{2} \sigma^2 d.
	\end{split}
	\]
This establishes \eqref{eq:stzvdm:step1}.

\underline{Step 2:}
    We are going to show below that if $\sigma$ is small enough and
	\begin{equation}\label{eq:sdvdm1:asbycontr}
		\gamma(\R / \Z) \cap B_{2d}(\gamma(x_0)) \not \subset B_{20\sqrt{\sigma} d} (L(\gamma(x_0),\gamma(y_0))),
	\end{equation}
	then necessarily for a uniform constant $C=C(p,q)$
	\begin{equation}\label{eq:sdvdm1:asbycontr:goal}
    1 < 2C\, \delta   \sigma^{p-q-{4}-\frac{q}{2}}.
	\end{equation}
	Once we have obtained that we can argue by contradiction: Choose $\sigma$ small enough so that in particular $20\sqrt{\sigma} < \eps$. Pick $\delta$ depending on $\sigma$ such that \eqref{eq:sdvdm1:asbycontr:goal} is false. Then \eqref{eq:sdvdm1:asbycontr} would lead to a contradiction and thus is false, meaning 
	\[
		\gamma(\R / \Z) \cap B_{2d}(\gamma(x_0)) \subset B_{20\sqrt{\sigma} d} (L(\gamma(x_0),\gamma(y_0))) \subset B_{\eps d} (L(\gamma(x_0),\gamma(y_0))), 
	\]
	which is what is claimed in the statement.
	
	So what we have to show is that \eqref{eq:sdvdm1:asbycontr} implies \eqref{eq:sdvdm1:asbycontr:goal}, which we will do now.
		
	If \eqref{eq:sdvdm1:asbycontr} was true, there necessarily must be a point $z_0 \in \R/\Z$ such that 
	\begin{equation}\label{eq:sdvdm1:gammaz0}
	\gamma(z_0) \in B_{2d}(\gamma(x_0)) \backslash B_{20\sqrt{\sigma} d} (L(\gamma(x_0),\gamma(y_0))).
	\end{equation}
	Denote the angle between $\gamma(z_0)-\gamma(x_0)$ and $\gamma(y_0)-\gamma(x_0)$ by $\alpha$, cf. \Cref{fig:svdm1}. 

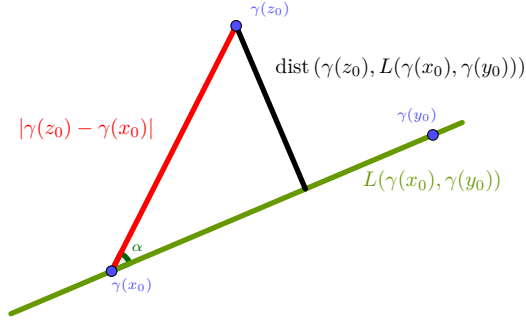
\begin{figure}	

	\definecolor{qqwuqq}{rgb}{0,0.39215686274509803,0}
\definecolor{ffqqqq}{rgb}{1,0,0}
\definecolor{wwzzqq}{rgb}{0.4,0.6,0}
\definecolor{ududff}{rgb}{0.30196078431372547,0.30196078431372547,1}

\vspace{-1cm}
\begin{tikzpicture}[line cap=round,line join=round,>=triangle 45,x=.75cm,y=.75cm, scale=0.8, transform shape]
\clip(-7,-6.60843746495283) rectangle (11.139334568632782,6.672764200933142);
\draw [shift={(-2.77,-2.72)},line width=2pt,color=qqwuqq,fill=qqwuqq,fill opacity=0.10000000149011612] (0,0) -- (22.98458350389757:0.44369270598728267) arc (22.98458350389757:63.098907048459004:0.44369270598728267) -- cycle;
\draw [line width=2pt,color=wwzzqq,domain=-5:5] plot(\x,{(-11.001--3.02*\x)/7.12});
\draw [line width=2pt,color=ffqqqq] (-2.77,-2.72)-- (-0.01,2.72);
\draw [line width=2pt] (-0.01,2.72)-- (1.5247497943652746,-0.8983505085697853);
\draw [color=ffqqqq](-5,0.845599995633506) node[anchor=north west] {$|\gamma(z_0)-\gamma(x_0)|$};
\draw (0.7,2.2) node[anchor=north west] {$\dist(\gamma(z_0),L(\gamma(x_0),\gamma(y_0)))$};
\draw [color=wwzzqq](2.635224370543196,-0.3227907967996698) node[anchor=north west] {$L(\gamma(x_0),\gamma(y_0))$};
\begin{scriptsize}
\draw [fill=ududff] (-2.77,-2.72) circle (2.5pt);
\draw[color=ududff] (-2.319344179648127,-3.036711181755211) node {$\gamma(x_0)$};
\draw [fill=ududff] (4.35,0.3) circle (2.5pt);
\draw[color=ududff] (4,0.7346768191366855) node {$\gamma(y_0)$};
\draw [fill=ududff] (-0.01,2.72) circle (2.5pt);
\draw[color=ududff] (0.7421354916641235,3.08624816086928) node {$\gamma(z_0)$};
\draw[color=qqwuqq] (-2.2,-2.2) node {$\alpha$};
\end{scriptsize}
\end{tikzpicture}

\vspace{-1cm}
\caption{Definition of $\alpha$}\label{fig:svdm1}

	\end{figure}
    Observe
	\[
	\begin{split}
	 \left |\frac{\gamma(z_0)-\gamma(x_0)}{|\gamma(z_0)-\gamma(x_0)|} \wedge \frac{\gamma(y_0)-\gamma(x_0)}{|\gamma(y_0)-\gamma(x_0)|}\right | =& |\sin(\alpha)| = \frac{\dist(\gamma(z_0),L(\gamma(x_0),\gamma(y_0)))}{|\gamma(z_0)-\gamma(x_0)|}\\
	 \overset{\eqref{eq:sdvdm1:gammaz0}}{\geq}&\frac{20 \sqrt{\sigma} d}{2d}=10\sqrt{\sigma}.
	 \end{split}
	\]
Denote by $\beta$ the angle between $\gamma'(x)$ and $\gamma(y_0)-\gamma(x_0)$ and $\theta$ the angle between $\gamma'(x)$ and $\gamma(z_0)-\gamma(x_0)$. Then we have
\[
 |\alpha| \leq |\beta|+|\theta|.
\]
Thus,
\[
\begin{split}
 &\abs{\gamma'(x) \wedge \frac{\gamma(z_0)-\gamma(x_0)}{|\gamma(z_0)-\gamma(x_0)|}}\\
 =&|\sin(\theta)|\\
 \geq&|\sin(\alpha)|-|\sin(\beta)|\\
 \geq&\left |\frac{\gamma(z_0)-\gamma(x_0)}{|\gamma(z_0)-\gamma(x_0)|} \wedge \frac{\gamma(y_0)-\gamma(x_0)}{|\gamma(y_0)-\gamma(x_0)|}\right | - \abs{\gamma'(x) \wedge \frac{\gamma(y_0)-\gamma(x_0)}{|\gamma(y_0)-\gamma(x_0)|}}\\
 \geq&10\sqrt{\sigma} - \abs{\gamma'(x) \wedge \frac{\gamma(y_0)-\gamma(x_0)}{|\gamma(y_0)-\gamma(x_0)|}}.
 \end{split}
\]
Consequently, for any $x \in X(\sigma,v)^c$,
\[
\begin{split}
 &\abs{\gamma'(x) \wedge \frac{\gamma(z_0)-\gamma(x_0)}{|\gamma(z_0)-\gamma(x_0)|}}\\
 \geq&10\sqrt{\sigma} - \sigma.
 \end{split}
\]
For all $\sigma$ small enough this implies that for any $x \in X(\sigma,v)^c$ necessarily
\begin{equation}\label{eq:svdm1:gammaprimez}
 \abs{\gamma'(x) \wedge \frac{\gamma(z_0)-\gamma(x_0)}{|\gamma(z_0)-\gamma(x_0)|}}
 \geq9\sqrt{\sigma}.
\end{equation}
Observe next that for any $x \in B(x_0,\rho) \cap A(\sigma^2 d,\gamma(x_0))$ and $z\in A(\sigma^2 d,\gamma(z_0))$ (it is very important that $z \in \R/\Z$ and does not need to lie in $B(x_0,\rho)$) we have 
	\[
	 |\gamma(z)-\gamma(x)| \geq |\gamma(z_0)-\gamma(x_0)|-2\sigma^2 d\geq \sqrt{\sigma}(20-2\sigma^{\frac{3}{2}}) d.
	\]
    and
    \[
	 |\gamma(z)-\gamma(x)| \leq |\gamma(z_0)-\gamma(x_0)|+2\sigma^2 d\leq (2+2\sigma^{2}) d.
	\]
	That is, for all $\sigma$ small enough
	\begin{equation}\label{eq:svdm1:gammzdif}
	 |\gamma(z)-\gamma(x)| \in (19\sqrt{\sigma} d,(2+2\sigma^2)d).
	\end{equation}
	
	We combine \eqref{eq:svdm1:gammzdif} and \eqref{eq:svdm1:gammaprimez}, and obtain for any 
	$x \in B(x_0,\rho) \cap A(\sigma^2 d,\gamma(x_0)) \cap X(\sigma,v)^c$ and $z \in A(\sigma^2 d,\gamma(z_0))$,
	\[
\begin{split}
 &\abs{\gamma'(x) \wedge \frac{\gamma(x)-\gamma(z)}{|\gamma(x)-\gamma(z)|}}\\
 \geq& \frac{1}{|\gamma(x)-\gamma(z)|} \abs{\gamma'(x) \wedge \brac{\gamma(x)-\gamma(z)}}\\
 \geq&\frac{1}{|\gamma(x)-\gamma(z)|} \brac{\abs{\gamma'(x) \wedge \brac{\gamma(x_0)-\gamma(z_0)}} - 2\sigma^2 d}\\
 \geq&\frac{1}{|\gamma(x)-\gamma(z)|} \brac{|\gamma(x_0)-\gamma(z_0)|\abs{\gamma'(x) \wedge \frac{\gamma(x_0)-\gamma(z_0)}{|\gamma(x_0)-\gamma(z_0)|}} -  2\sigma^2 d}\\
 \geq&\frac{d}{|\gamma(x)-\gamma(z)|} \brac{180 \sigma    -  2\sigma^2 }\\
 \geq&\frac{\sigma}{2+2\sigma^2} \brac{180     -  2\sigma }.\\
 \end{split}
\]
Again, for all small enough $\sigma$ this implies for any 
	$x \in B(x_0,\rho) \cap A(\sigma^2 d,\gamma(x_0)) \cap X(\sigma,v)^c$ and $z \in A(\sigma^2 d,\gamma(z_0))$,
	\begin{equation}\label{eq:svdm1:gammapxzvar}
	\abs{\gamma'(x) \wedge \frac{\gamma(x)-\gamma(z)}{|\gamma(x)-\gamma(z)|}}\\
 \geq3\sqrt{\sigma}.
 \end{equation}
Integrating this inequality in $x \in B(x_0,\rho) \cap A(\sigma^2 d,\gamma(x_0)) \cap X(\sigma,v)^c$ and $z \in A(\sigma^2 d,\gamma(z_0))$, we obtain from \eqref{eq:svdm:1}, \eqref{eq:svdm1:gammzdif}, and \eqref{eq:svdm1:gammapxzvar},
\[
 \delta > \frac{(3\sqrt{\sigma})^q}{(\sigma d)^{p-q}} |B(x_0,\rho) \cap A(\sigma^2 d,\gamma(x_0)) \cap X(\sigma,v)^c|\ |A(\sigma^2 d,\gamma(z_0))|.
\]
With \eqref{eq:stzvdm:step1} we find
\[
\delta > \frac{(3\sqrt{\sigma})^q}{(\sigma d)^{p-q}} \frac{3}{2} \sigma^2 d\ |A(\sigma^2 d,\gamma(z_0))|.
\]
On the other hand, since $|\gamma'| \equiv 1$, we have $|A(\sigma^2 d,\gamma(z_0))| \geq 2\sigma^2 d$. That is,
\[
 2\sigma^2 d \leq |A(\sigma^2 d,\gamma(z_0))|  < C\, \delta d^{p-q-1} \sigma^{p-q-2-\frac{q}{2}}.
\]
That is,
\[
d^{-(p-2-q)} < 2C\, \delta   \sigma^{p-q-4-\frac{q}{2}}.
\]
If $p \geq q+2$ we have (observe that $d \leq 1$ since $\diam(\gamma(\R/\Z)) \leq 1$),
\[
1 < 2C\, \delta   \sigma^{p-q-4-\frac{q}{2}}.
\]
That is, under the assumption \eqref{eq:sdvdm1:asbycontr}, we have shown \eqref{eq:sdvdm1:asbycontr:goal} which, as explained above, implies the claim of \Cref{la:vdms}.
\end{proof}

\begin{theorem}\label{th:smallenergyinjectivity}
Let $p \geq q+2$. There exists $\delta > 0$ such that the following holds.

Let $\gamma \in \lip(\R/\Z,\R^3)$ be a homeomorphism, $|\gamma'| \equiv 1$, and assume that for some $x_0\in \R / \Z$ and $\rho >0$ we have that
\begin{equation}\label{eq:svdm:24}
\text{either $\gamma \in C^1(B(x_0,\rho))$ or }[\gamma']_{W^{\frac{p-q-1}{q},q}(B(x_0,\rho))}<\infty.
\end{equation}
Also assume
\begin{equation}\label{eq:svdm:23}
\int_{B(x_0,\rho)} \int_{\R / \Z} \frac{\left |\gamma'(x) \wedge (\gamma(y)-\gamma(x)) \right |^{q}
}{|\gamma(x)-\gamma(y)|^{p}} \, dy\, dx < \delta.
\end{equation}

If for any $z_0 \in \R/\Z$ we have 
\[
 |\gamma(x_0)-\gamma(z_0)| < \frac{1}{10}\rho,
\]
then there exists $\bar{x} \in B(x_0,\rho)$ such that $\gamma(\bar{x}) = \gamma(z_0)$. In particular, we have $z_0\in B(x_0,\rho)$.
\end{theorem}
\begin{proof}
Fix $\sigma, \eps > 0$ to be specified later.
Take $\delta$ small enough so that \Cref{la:bilip} and \Cref{th:gammasobvstp2} are applicable (in view of in view of \eqref{eq:svdm:24}), so that we have
\begin{equation}\label{eq:svdm:bilip2}
 (1-\sigma)|x-y|\leq |\gamma(x)-\gamma(y)| \leq |x-y|\quad \forall x,y \in B(x_0,\rho).
\end{equation}
Moreover, we can assume that $\delta$ is small enough so that \Cref{la:vdms} is applicable.

Starting from $x_0$, we are going to construct a sequence $(x_k)_{k=0}^\infty \subset B(x_0,\rho)$ and a sequence $(d_k)_{k=0}^\infty \subset (0,\infty)$ such that for all $k \geq 0$,
\begin{itemize}
 \item $|x_k-x_{k+1}| \leq \frac{1}{1-\sigma}d_k$,
 \item $|\gamma(x_{k})-\gamma(z_0)| = d_{k}$,
 \item $d_{k+1} \leq \frac{1}{100} d_k$,
 \item $B(x_k,10d_k) \subset B(x_0,\rho)$.
\end{itemize}
Once we have constructed this sequence, we see that $x_k$ is convergent to some $\bar{x} := \lim_{k \to \infty} x_k \in B(x_0,\rho)$ and $\gamma(\bar{x})=\lim_{k \to \infty} \gamma(x_k) = \gamma(z_0)$. Since $\gamma$ is injective, this implies $\bar{x} = z_0$ and thus $z_0 \in B(x_0,\rho)$.

Let $x_k$ be already given, we need to construct $x_{k+1}$.

Set 
\[
 d_k := |\gamma(x_k)-\gamma(z_0)|
\]
and 
\[
 \eta_k(t) := \gamma(x_k+td_k).
\]
Observe that for $1 \leq |t| \leq (1-\sigma)^{-1}$, by \eqref{eq:svdm:bilip2},
\[
 d_k \leq |\eta_k(t)-\gamma(x_k)| \leq (1-\sigma)^{-1}\, d_k.
\]
Since $\eta_k(0)-\gamma(x_k) = 0$, we obtain from the intermediate value theorem that there must be $t_- \in [-(1-\sigma)^{-1},-1]$ and $t_+ \in [1,(1-\sigma)^{-1}]$ such that 
\[
 |\eta_k(t_{\pm}) -\gamma(x_k)| = d_k.
\]
W.l.o.g. (otherwise we interchange the role of $t_-$ and $t_+$ below) we may assume that 
\begin{equation}\label{eq:svdm2:goodguybadguy}
 |\eta_k(t_-)-\gamma(z_0)| \geq \frac{1}{2}d_k.
\end{equation}
Indeed, if the inequality was false for both $t_+$ and $t_-$, we would have
\[
 |\eta_k(t_-)-\eta_k(t_+)|\leq d_k,
\]
which violates the bi-Lipschitz assumption \eqref{eq:svdm:bilip2}.

Set \[
y_k := x_k+t_+ d_k.
\]
Denote the line through $\gamma(x_k)$ and $\gamma(y_k)$ by $L_{x_k,y_k}$, more precisely let
\[
 L_{x_k,y_k}(t) := (1-t)\gamma(x_k)+t \gamma(y_k), \quad t \in \R.
\]
Since $B(x_k,d_k) \subset B(x_0,\rho)$, we can apply \Cref{la:vdms} and have
\begin{equation}\label{eq:vdmsappl:24}
 |L_{x_k,y_k}(t_1)-\gamma(z_0)| = \inf_{t \in \R} |L_{x_k,y_k}(t)-\gamma(z_0)| \leq \eps d_k.
\end{equation}
for some $t_1 \in [-1,1]$. By Pythagoras,
\[
 (d_k)^2(1-\eps^2) = |L_{x_k,y_k}(t_1)-\gamma(x_k)|^2 = t_1^2 (d_k)^2.
\]
That is, if $\eps$ is chosen small enough, $|t_1| \geq 1-2\eps$.

We now argue that $t_1 \geq 1-2\eps$. Indeed, if we had $t_1 \leq -1+2\eps$, then from the bilipschitz estimate 
\[
 |\eta_k(t_1)-\eta_k(t_+)| \geq (1-\sigma)(2-2\eps)d_k.
\]
On the other hand,
\[
 |L_{x_k,y_k}(t_1)-\eta_k(t_+)|= |1-t_1|\, | \gamma(x_k) -\gamma_k(y_k)| \geq (2-2\eps)d_k. 
\]
Moreover,
\[
  |\eta_k(t_1)-\gamma(x_k)|\leq d_k
\]
and
\[
 |L_{x_k,y_k}(t_1)-\gamma(x_k)| = |t_1| d_k \leq d_k.
\]
So we have that 
\[
 L_{x_k,y_k}(t_1) , \eta_k(t_1) \in B(\gamma(x_k),d_k) \backslash B(\eta_k(t_+),(1-\sigma)(2-2\eps)d_k).
\]
But we also have that 
\[
 |\eta_k(t_+)-\gamma(x_k)| = d_k.
\]
From elementary geometry this implies that for small enough $\eps$ and $\sigma$ we have 
\[
 |L_{x_k,y_k}(t_1) - \eta_k(t_1)| < \frac{1}{4} d_k.
\]
But then from the projection assumption \eqref{eq:vdmsappl:24}
\[
\begin{split}
 |\eta_k(t_-)-\gamma(z_0)| \leq& |\eta_k(t_-) - \eta_k(t_1)| + |\eta_k(t_1) - L_{x_k,y_k}(t_1)| +|L_{x_k,y_k}(t_1)-\gamma(z_0)|\\
 \leq& |t_--t_1| d_k + \frac{1}{4}d_k + \eps d_k\\
 <& \frac{1}{2} d_k,
 \end{split}
\]
for small enough $\eps$ and $\sigma$. This contradicts \eqref{eq:svdm2:goodguybadguy}. That is, we have established that $t_1 \in [1-2\eps,1]$.

Now we compare
\[
 |\gamma(y_k) - L_{x_k,y_k}(t_1)| \leq (1-t_1) |\gamma(x_k)-\gamma(y_k)| \leq 2\eps d_k.
\]
Consequently, by \eqref{eq:vdmsappl:24}, for $\eps$ small enough
\[
 |\gamma(y_k)-\gamma(z_0)| \leq 3\eps d_k < \frac{1}{100} d_k.
\]
So if we set $x_{k+1} := y_k= x_k+t_+ d_k \in B(x_0,\rho)$, we have from \eqref{eq:svdm:bilip2}.
\[
 |x_{k+1}-x_k| \leq (1-\sigma)^{-1} d_k.
\]
Also, by the definition, $d_{k+1} := |\gamma(x_{k+1})-\gamma(z_0)|< \frac{1}{100} d_k$. Lastly, 
\begin{equation}
\begin{split}
B(x_{k+1},10d_{k+1}) \subset& B(x_{k+1},\frac{1}{10} d_{k}) \\
\subset& B(x_k, (\frac{1}{10} +(1-\sigma)^{-1} )d_k) \subset B(x_k,10d_k) \subset B(x_0,\rho).
\end{split}
\end{equation}
We thus have constructed $x_{k+1}$ with the required properties, and can conclude.
\end{proof}

\subsection{Convergence}\label{s:convoutsidesingular}

\begin{proposition}\label{pr:convoutsidesingular}
    For any $\Lambda > 0$ and $\eps > 0$ there exists an $L = L(\eps,\Lambda) \in \N$ such that the following holds.

	Let $\gamma_k : \R/\Z \to \R^3$, $|\gamma_k'| \equiv 1$, be $C^1$-homeomorphisms with 
	\[
	\sup_{k}\| \gamma_k\|_{L^\infty} + \sup_{k} \tp^{q+2,q} (\gamma_k) \leq \Lambda
	\]
	Then there exists a subsequence $(\gamma_{k_i})_{i \in \N}$ and $\gamma \in \lip(\R/\Z,\R^3)$ such that the following holds for some finite set $\Sigma \subset \R/\Z$ with $\#\Sigma \leq L$.  
	
	\begin{enumerate}
	\item \label{i:cos:1} $\gamma_{k_i}$ converges uniformly to $\gamma$ and $\gamma \in \lip(\R/\Z,\R^3)$.
    \item \label{i:cos:2} For any $x_0 \in \R / \Z \backslash \Sigma$ there exists a radius $\rho(x_0) > 0$ such that $\gamma_{k_i}$ weakly converges to $\gamma$ in $W^{1+\frac{p-q-1}{q},q}(B(x_0,\rho))$.
	\item  \label{i:cos:3}$|\gamma'| = 1$ a.e.
	\item  \label{i:cos:4} $\gamma_{k_i}$ and $\gamma$ are uniformly bilipschitz in $B(x_0,\rho)$ with the estimates
	\begin{equation}\label{eq:conv:bilipk}
	 (1-\eps) |x-y| \leq |\gamma_{k_i}(x)-\gamma_{k_i}(y)| \leq |x-y| \quad \forall x,y \in B(x_0,\rho) \, \forall i.
	\end{equation}
    and
    \begin{equation}\label{eq:conv:bilipg}
	 (1-\eps) |x-y| \leq |\gamma(x)-\gamma(y)| \leq |x-y| \quad \forall x,y \in B(x_0,\rho).
	\end{equation}
	
	\item  \label{i:cos:5} For any point $x_0 \in \R /\Z \backslash \Sigma$ and any $y_0 \in \R/\Z$ with $|\gamma_{k_i}(x_0)-\gamma_{k_i}(y_0)| \leq \frac{1}{100} \rho(x_0)$ or $|\gamma(x_0)-\gamma(y_0)| \leq \frac{1}{100} \rho(x_0)$ we have $|x_0-y_0| \leq \rho(x_0)$.
	
	\item  \label{i:cos:6} In particular, whenever $\gamma(x) =\gamma(y)$, then either $x=y$ or $\{x,y\} \subset \Sigma$.
	
	\item \label{i:cos:7} We have lower semi-continuity, namely
	\begin{equation}\label{eq:cv:lsc}
	 \tp^{p,q} (\gamma) \leq \liminf_{k \to \infty} \tp^{p,q} (\gamma_k).
	\end{equation}
    
    \item \label{i:cos:8}  $\gamma: \R/\Z \to \R^3$ is a homeomorphism.
    
    \item \label{i:cos:9} $\gamma$ is globally bi-Lipschitz.
    
    \item \label{i:cos:10} $\gamma \in W^{1+s,q}(\R/\Z,\R^3)$ for any $0<s < \frac{1}{q}$.
	\end{enumerate}
\end{proposition}

\begin{proof}
\underline{\cref{i:cos:1}:}
By Arzela-Ascoli, up to taking a subsequence, we may assume that \eqref{i:cos:1} holds.

\underline{\cref{i:cos:2}:}
Fix $\delta > 0$ to be specified later. By \Cref{pr:uniformsmallness} up to taking a further subsequence we find a discrete singular set $\Sigma$ with $\# \Sigma \leq L$ such that for any $x_0 \in \R/\Z \backslash \Sigma$ there exists $\rho_{x_0} > 0$ such that
\begin{equation}\label{eq:smalllimsup:23}
 \limsup_{k \to \infty} \int_{B(x_0,10\rho_{x_0})} \int_{\R/\Z} \frac{|\gamma_{k}'(x) \wedge (\gamma_{k}(x)-\gamma_{k}(y))|^q}{|\gamma_{k}(x)-\gamma_{k}(y)|^{p}}\, dx\, dy < \delta.
\end{equation}
From \Cref{th:gammasobvstp2} we find that for each $x_0 \in \R/\Z \backslash \Sigma$,
\begin{equation}\label{eq:cv:gammaksmall}
 \sup_{k} [\gamma_{k}']_{W^{\frac{p-q-1}{q},q}(B(x_0,\rho_{x_0}))}^q \aleq C \delta.
\end{equation}
By reflexivity of $W^{\frac{p-q-1}{q},p}(B(x_0,\rho_{x_0}))$ and Banach-Alaoglu, combined with Rellich's theorem, we find that $\gamma_k'$ weakly converges to $\gamma'$ in $W^{\frac{p-q-1}{q},p}(B(x_0,\rho_{x_0}))$ and the convergence is pointwise a.e. in $B(x_0,\rho_{x_0})$ and strong in $L^1(B(x_0,\rho_{x_0}))$.
Observe that by uniqueness of the weak limit we do not need to pass to a further subsequence here.
This establishes \eqref{i:cos:2}.

\underline{\cref{i:cos:3}:}
In particular, $\gamma'_k$  a.e. converges to $\gamma$ in $\R/\Z \backslash \Sigma$, and since $\Sigma$ is a $\mathcal{L}^1$-zero-set. we have that $\gamma'_k$ a.e. converges to $\gamma$. This implies 
\[
|\gamma'(x)| = \lim_{k \to \infty} |\gamma_k'(x)|=1 \quad \text{a.e. $x \in \R/\Z$}.
\]
We have established \eqref{i:cos:3}.

\underline{\cref{i:cos:4}:}
Having chosen $\delta$ in \eqref{eq:smalllimsup:23} small enough, we get a small $W^{\frac{p-q-1}{q},p}(B(x_0,\rho_{x_0}))$-norm from \eqref{eq:cv:gammaksmall}, and from \Cref{la:bilip} we obtain \eqref{eq:conv:bilipk}. From the uniform convergence $\gamma_k \to \gamma$ we obtain \eqref{eq:conv:bilipg}. This establishes \eqref{i:cos:4}.

\underline{\cref{i:cos:5}:}
The statement in \eqref{i:cos:5} for $\gamma_k$ is a consequence of \Cref{th:smallenergyinjectivity}. By uniform convergence this takes over to $\gamma$.

\underline{\cref{i:cos:6}:}
Assume $x,y \in \R /\Z$ with $\gamma(x) = \gamma(y)$. If $x \not \in \Sigma$, we obtain from \eqref{i:cos:5} that $y \in B(x,\rho(x))$. But by \eqref{eq:conv:bilipg} this implies $x = y$. Similarly, if $y \not \in \Sigma$ we obtain $y=x$. So if $\gamma(x) = \gamma(y)$ then either $x=y$ or $x$ and $y$ both belong to $\Sigma$.

\underline{\cref{i:cos:7}:}
In order to prove \eqref{eq:cv:lsc} observe that by \eqref{i:cos:6} we have 
\[
|\gamma(x)-\gamma(y)| = 0 \Rightarrow \left \{(x,y) \subset \R/\Z \times \R / \Z: x=y \right \} \cup \Sigma \times \Sigma
\]
In particular,
\[
 \left \{(x,y) \in \R/\Z \times \R/\Z: \quad |\gamma(x)-\gamma(y)| = 0 \right \}\quad \text{is a $\mathcal{L}^2$-zero set}.
\]
Consequently, we have from the pointwise convergence of $\gamma_k$ to $\gamma$ and the $\mathcal{L}^1$-almost everywhere convergence of $\gamma_k'$ to $\gamma'$ that 

\[
\frac{|\gamma'(x) \wedge (\gamma(x)-\gamma(y))|^q}{|\gamma(x)-\gamma(y)|^p} = \lim_{k \to \infty} \frac{|\gamma_k'(x) \wedge (\gamma_k(x)-\gamma_k(y))|^q}{|\gamma_k(x)-\gamma_k(y)|^p} \quad \text{$\mathcal{L}^2$-a.e. in $\R/\Z \times \R/\Z$}.
\]
From Fatou's lemma we thus have
\[
 \tp^{p,q}(\gamma) \leq \liminf_{k \to \infty} \tp^{p,q}(\gamma_k),
\]
and \eqref{i:cos:7} is established.

\underline{\cref{i:cos:8}:}
By now we know that $\gamma$ has finite tangent-point energy. By \cite[Theorem~1.1.]{SvdM12} this implies that $\gamma(\R/\Z)$ is a topological $1$-manifold. On the other hand $\gamma: \R/\Z \to \R^3$ has only at most finitely many self-intersection points, namely $\gamma(\Sigma)$. This means that there can be no intersection points at all.

Indeed, assume there is an intersection $x_1, x_2  \in \Sigma$, $\gamma(x_1)=\gamma(x_2)$. Assume by contradiction that $x_1 \neq x_2$. Then there exists $\eps > 0$ such that $[x_i-\eps,x_i+\eps] \cap \Sigma = \{x_i\}$, $i=1,2$, since $\Sigma$ is a discrete set. Now $\gamma: [x_i-\eps,x_i+\eps] \to \R^3$ is a one-to-one map and it is (even Lipschitz-)continuous.

Now let denote by $C$ a cross, 
\[C := \{z=(z_1,z_2) \in [-1,1]^2, z_1 = 0 \text{ or } z_2 =0\} \]
and define $f: C \to \R^3$
\[
 f(z) := \begin{cases}
         \gamma(x_1+\eps z_1) \quad z = (z_1,0),\\
         \gamma(x_2+\eps z_2) \quad z = (0,z_2).\\
         \end{cases}
\]
Then $f$ is injective and continuous. Since $C$ is a compact set, we conclude that $f: C \to f(C)\subset \gamma(\R/\Z)$ is a homeomorphism. Since $\gamma(\R/\Z)$ is a one-dimensional topological manifold, so around any point $p_0 \in \gamma(\R/\Z)$ there exists a homeomorphism $h: \gamma(\R/\Z) \cap B_{\tilde{\delta}}(p_0) \to \R$ for some small $\tilde{\delta} > 0$. Taking $p_0 := \gamma(x_1)$ we see that $h \circ f: C\cap B_\delta(0) \to \R$ is a homeomorphism for a smaller $\delta>0$. But the cross $C$ is not homeomorphic to any subset in $\R$, and we have a contradiction to our assumption that $x_1 \neq x_2$. In conclusion, whenever $\gamma(x) = \gamma(y)$, we have $x=y$, that is $\gamma$ is injective. Since $\gamma: \R/\Z \to \R^3$ is continuous and $\R/\Z$ compact we conclude that $\gamma: \R/\Z \to \R^3$ is a homeomorphism.
Thus \eqref{i:cos:8} is established.

\underline{\Cref{i:cos:9}:} Assume 
\[
 \inf_{x\neq y \in \R/\Z} \frac{|\gamma(x)-\gamma(y)|}{|x-y|} = 0.
\]
Then there exists a (w.l.o.g. convergent) sequences $\R/\Z \ni x_k \xrightarrow{k \to \infty} \bar{x}$ and $\R/\Z \ni y_k \xrightarrow{k \to \infty} \bar{y}$ with
\begin{equation}\label{eq:cos:9:infzero}
 \lim_{k \to \infty} \frac{|\gamma(x_k)-\gamma(y_k)|}{|x_k-y_k|} = 0.
\end{equation}
We make several observations:
\begin{itemize}
\item $\bar{x} = \bar{y}$. Indeed, assume that $\bar{x} \neq \bar{y}$, then continuity of $\gamma$ combined with \eqref{eq:cos:9:infzero} implies
\[
 0 = \gamma(\bar{x})-\gamma(\bar{y}).
\]
Since $\gamma$ is injective, this implies $\bar{x} = \bar{y}$, contradiction.
\item $\bar{x} \in \Sigma$. Indeed, assume that $\bar{x} \not \in \Sigma$. Then for all $k$ sufficiently large, $x_k,y_k \in B(\bar{x},\rho_{\bar x})$, thus by \eqref{eq:conv:bilipg},
\[
 \frac{|\gamma(x_k)-\gamma(y_k)|}{|x_k-y_k|} \geq 1-\eps \quad \forall k \gg 1.
\]
This contradicts \eqref{eq:cos:9:infzero}.
\item For all but finitely many $k \in \N$ (up to interchanging the role of $x_k$ and $y_k$) we have $x_k < \bar{x} < y_k$ for all $k \in \N$.

Indeed let $K > 0$ such that 
\begin{equation}\label{eq:co:9:inflimitK}
 \frac{|\gamma(x_{k})-\gamma(y_k)|}{|x_k-y_k|} < \frac{1}{4} \quad \forall k \geq K.
\end{equation}
Also, combining \Cref{la:bilip} and \Cref{th:gammasobvstp2}, let $\delta = \delta(\eps) > 0$ such that for any ball $B \subset \R/\Z$ of diameter $\leq \frac{1}{2}$ 
\begin{equation}\label{eq:co:9:tpsmallbilip}
\begin{split}
 &\tp^{p,q}(\gamma,B) < \delta \text{ and } [\gamma']_{W^{\frac{p-q-1}{q},q}(B))} < \infty\\ \text{ implies }& |\gamma(x)-\gamma(y)| \geq \frac{1}{2} |x-y| \quad \forall x,y \in B.
 \end{split}
\end{equation}
By absolute continuity of the integral, and since $\tp^{p,q}(\gamma) < \infty$, there exists a $\bar{\rho} > 0$ such that 
\begin{equation}\label{eq:co:9:tpsmall} 
\tp^{p,q}(B) < \delta\text{ for all balls }B \subset \R/\Z \text{ with $|B| < \bar{\rho}$}. 
\end{equation}

Now assume by contradiction that $x_k,y_k \in B(\bar{x},\bar{\rho}/2)$ and $x_k < y_k \leq \bar{x}$ for some $k > K$. Take a sequence $(\tilde{y}_{k;i})_{i \in \N}$ such that $x_k < \tilde{y}_{k;i} < \bar{x}$ with $\tilde{y}_{k;i} \xrightarrow{i \to \infty} y_k$. There exists an open ball $B_{k,i} \subset \overline{B_{k,i}} \subset \R/\Z \backslash \Sigma$ of radius $< \bar{\rho}$ such that $x_{k},\tilde{y}_{k;i} \subset B_{k,i}$. By a covering argument, from \eqref{i:cos:2}, we obtain that $\gamma' \in W^{\frac{p-q-1}{q},q}(B_{k,i})$ (without any estimate for the norm). However, since the ball $B_{k,i}$ is small enough, we have $\tp^{p,q}(\gamma,B_{k,i}) < \delta$ by \eqref{eq:co:9:tpsmall}, so from \eqref{eq:co:9:tpsmallbilip} we obtain
\[
 \frac{|\gamma(x_{k})-\gamma(\tilde{y}_{k;i})|}{|x_k-\tilde{y}_{k;i}|} \geq \frac{1}{2}.
\]
This holds for all $i \in \N$, so letting $i \to \infty$ we get
\[
 \frac{|\gamma(x_{k})-\gamma(y_k)|}{|x_k-y_k|} \geq \frac{1}{2},
\]
a contradiction to \eqref{eq:co:9:inflimitK}.
\item There exists $K \in \N$ so that for any $k \geq K$
\begin{equation}\label{eq:i:cos:9:45235}
 |\gamma(x_k)-\gamma(y_k)| \geq \frac{1}{20} \max\{|x_k-\bar{x}| ,|y_k-\bar{x}| \}.
\end{equation}
Indeed, let $\delta > 0$ be from \Cref{th:smallenergyinjectivity}, and let $R > 0$ be such that \eqref{eq:svdm:23} is satisfied for any $x_0 \in \R/\Z$ and any $\rho \in (0,R)$ -- such an $R>0$ exists by absolute continuity of the integral and since $\tp^{p,q}(\gamma) < \infty$. We can assume by taking $R$ possibly smaller that $B(\bar{x},R) \cap \Sigma = \{\bar{x}\}$.

Let $K \in \N$ such that $|x_k-y_k| < \frac{1}{100} R$ for all $k \geq K$.

Set $\rho := \frac{1}{2}|x_k-\bar{x}| < R$. Observe that $B(x_k,\rho) \cap \Sigma = \emptyset$ and thus $\gamma' \in W^{\frac{p-q-1}{q}}(B(x_k,\rho))$ by a covering argument and \eqref{i:cos:2}. Since $\bar{x} \not \in B(x_k,\rho)$ and $x_k < \bar{x} < y_k$, we have that $y_k \not \in B(x_k,\rho)$. Applying \Cref{th:smallenergyinjectivity} in $B(x_k,\rho)$, we conclude that 
\[
|\gamma(x_k)-\gamma(y_k)| \geq \frac{1}{20}|x_k-\bar{x}| .
\]
By a similar argument we also obtain
\[
|\gamma(x_k)-\gamma(y_k)| \geq \frac{1}{20}|y_k-\bar{x}| .
\]
\end{itemize}
To conclude \eqref{i:cos:9}, observe that by triangular inequality
\[
\max\{|x_k-\bar{x}|,|y_k-\bar{x}|\} \geq \frac{1}{2} |x_k-y_k|.
\]
Combining this with \eqref{eq:i:cos:9:45235}, implies 
\[
 |\gamma(x_k)-\gamma(y_k)| \geq \frac{1}{40} |x_k-y_k|,
\]
which is a contradiction to \eqref{eq:cos:9:infzero}. This establishes \eqref{i:cos:9}.

\underline{\Cref{i:cos:10}:} There are at most finitely many points $\Sigma$ where we do not know already that $\gamma$ is Sobolev. For simplicity of notation assume that $0 \in \Sigma$. Take $r > 0$ small enough such that $B(r) \cap \Sigma = \{0\}$, and that for $\delta$ from \Cref{th:gammasobvstp2},
\[
 \tp^{p,q}(\gamma,B(r)) < \delta . 
\]
For small $\sigma > 0$ let 
\[
 X_\sigma := \brac{[-r,-\sigma] \times [-r,-\sigma]} \cup \brac{[\sigma,r] \times [\sigma,r]} \cup \brac{[-r,-\sigma] \times [\sigma,r]}.
\]
Since $\mathcal{L}^2([-r,r]^2 \backslash X_\sigma) \xrightarrow{\sigma \to 0} 0$, it suffices to show that 
\[
 \limsup_{\sigma \to 0^+} \int\int_{X_\sigma} \frac{|\gamma'(x)-\gamma'(y)|^q}{|x-y|^{1+sq}} \, dx\, dy< \infty.
\]
Since we already know that $\gamma' \in W^{\frac{p-q-1}{q},q}([-r,-\sigma]) \cup W^{\frac{p-q-1}{q},q}([\sigma,r])$, we can use \Cref{th:gammasobvstp2} to obtain that 
\[
\begin{split}
 &\int\int_{X_\sigma} \frac{|\gamma'(x)-\gamma'(y)|^q}{|x-y|^{1+sq}}\, dx \, dy\\
 \aleq &\int\int_{[\sigma,r]^2 \cup [-r,-\sigma]^2} \frac{|\gamma'(x)-\gamma'(y)|^q}{|x-y|^{p-q}} \, dx\, dy + \|\gamma'\|_{L^\infty} \int_{(-r,0)} \int_{(0,r)} \frac{1}{|x-y|^{1+sq}}\, dx\, dy.
 \end{split}
\]
The second integral is finite for $s < \frac{1}{q}$.
\end{proof}

\begin{remark}
It is unclear to us whether \Cref{pr:convoutsidesingular}, property \ref{i:cos:10}, holds for $s=\frac{1}{q}$. If one were able to prove it, then there is a chance to remove the singular set for the regularity theory in \Cref{co:regisotopy} with the removability argument as in \cite{MS20}.
\end{remark}

\section{Weak limits of minimizing sequences are critical: Proof of Theorem~\ref{th:minarecritical}}\label{s:proofth:minarecritical}
We would like to compare our minimizing sequence $\gamma_k$ with the variation $\gamma + t\varphi$, where $\varphi$ is a locally supported test-function. Computing the Euler-Lagrange equations then proves \Cref{co:regisotopy}. For notational convenience, we restrict ourselves to the case $p=q+2$ instead of $p \geq q+2$. The case $p > q+2$ follows the same way, but it can also be obtained by simpler, more direct methods.

\begin{theorem}[Minimizing sequence becomes critical point]\label{th:minseqcrit}
There exists $\eps_0 > 0$ such that the following holds. 

Let {$q>1$,} $\gamma_k \in C^1(\R/\Z,\R^3)$, $\gamma \in \lip(\R / \Z,\R^3)$, $|\gamma_k'| = |\gamma'| = 1$ a.e., with
\[
 \sup_{k} \tp^{q+2,q}(\gamma_k) < \infty.
\]
Assume that $\gamma_k$ is approximately minimizing, in the sense that
\[
 \tp^{q+2,q}(\gamma_k) \leq \tp^{q+2,q}(\tilde{\gamma}) + \frac{1}{k}
\]
holds for any $\tilde{\gamma}$ ambient isotopic to $\gamma_k$.

Assume that $\gamma_k$ uniformly converges to $\gamma$ in $\R/\Z$ and for a geodesic ball $B(100\rho) \subset \R/\Z$, e.g. $\rho < \frac{1}{1000}$,
\begin{equation}\label{eq:tpl:smallen1}
\sup_k \int_{B(100\rho)}\int_{\R/\Z} \mu(\gamma,x,y)\, dy\, dx <\eps_0.
\end{equation}
Then for any $\varphi \in C_c^\infty(B(\rho),\R^3)$ there exists $t_0 > 0$ such that 
\[
 \int\int _{(\R/\Z)^2 \backslash (B(\rho)^c)^2} \mu(\gamma,x,y)\, dx\, dy \leq
 \int\int _{(\R/\Z)^2 \backslash (B(\rho)^c)^2} \mu(\gamma+t\varphi,x,y)\, dx\, dy \quad \forall t \in (-t_0,t_0).
\]
Here,
\[
  \mu(\sigma,x,y) := \frac{|\sigma'(x)\wedge (\sigma(x)-\sigma(y))|^{q}}{|\sigma(x)-\sigma(y)|^{q+2}}\, |\sigma'(x)|^{1-q}\, |\sigma'(y)|.
 \]

\end{theorem}

For the proof of \Cref{th:minseqcrit} we need to obtain first a fractional version of the Luckhaus Lemma.

\subsection{A fractional Luckhaus Lemma in one dimension}
The Luckhaus Lemma, \cite[Lemma~1]{L88}, is an important tool for harmonic maps, usually it is given in the following form, see \cite[Section 2.6, Lemma~1]{S96}.
It essentially provides a way to glue together two maps $u$ and $v$ along the boundary $\partial B(0,1)$ with explicit dependence on the size $\delta$ of the glued region.
\begin{lemma}\label{la:luckhaus}
Let $\n$ be any compact subset of $\R^3$, $n \geq 2$.

Assume $u,v \in W^{1,2}(\S^{n-1})$ and $\delta \in (0,1)$. Then there exists $w \in W^{1,2}(\S^{n-1} \times [0,\delta])$ with $w = u$ in a neighborhood of $\S^{n-1} \times \{0\}$, $w = v$ in a neighborhood of $\S^{n-1} \times \{\delta\}$, 
\[
 \int_{\S^{n-1} \times [0,\delta]} |\nabla w|^n \leq C \delta \int_{\S^{n-1}} \brac{|\nabla u|^2 + |\nabla v|^2 } + C \delta^{-1} \int_{\S^{n-1}} |u-v|^2,
\]
and
\[
 \dist^2(w,\n) \leq C \delta^{1-n} \brac{\int_{\S^{n-1}} |\nabla u|^2 + |\nabla v|^2}^{\frac{1}{2}}\, \brac{\int_{\S^{n-1}} |u-v|^2}^{\frac{1}{2}} + C\delta^{-n} \int_{\S^{n-1}} |u-v|^2.
\]

\end{lemma}

We will need a version of this lemma for fractional Sobolev spaces in one dimension. Working in one-dimension has advantages and disadvantages: The advantage is that the boundary of a ball consists of two points, and the possibility of explicit computations. The main disadvantage is that there may be no reasonable trace spaces for $W^{s,p}([0,1])$ when $sp \leq 1$. In any case, the following might be interesting on its own.
\begin{lemma}\label{la:frluckhaus}
Assume $u,v: \R \to \R^3$ be locally integrable, have a Lebesgue point in $x = \pm 1$, and
\[
 \int_{\R} \frac{|u(x)-u(y)|^{p}}{|x - y|^{1+sp}} dy+ \int_{\R} \frac{|v(x)-v(y)|^{p}}{|x-y|^{1+sp}} dy < \infty \quad x= -1,1.
\]
Then for any $\delta \leq \frac{1}{2}$ there exists $w: (-2,2)\rightarrow \R^3$ with the following properties:
\begin{itemize}
 \item $w(x) = u(x)$ for $|x| > 1$ and $w(x) = v(x/(1-\delta))$ for $|x| < 1-\delta$, namely we can choose
\[
w(x) = \begin{cases} 
u(x) \quad & |x| > 1\\
(1-\eta_\delta(x)) u(-1) + \eta_\delta(x) v(-1) \quad &x \in [-1,-1+\delta]\\
v(x/(1-\delta)) \quad & |x| < 1-\delta\\
(1-\eta_\delta(x)) u(1) + \eta_\delta(x) v(1) \quad &x \in [1-\delta,1]\\
\end{cases}
\]
where $\eta:\R \to [0,1]$ is any smooth map such that $\eta \equiv 0$ for $|x| > 1-\frac{\delta}{4}$, $\eta \equiv 1$ for $|x| \leq 1-\frac{1}{2}\delta$, and $|\eta'| \aleq \frac{1}{\delta}$.
 \item For any $r > 1$ we have the estimate
 \begin{equation}\label{eq:frl:estw}
 \begin{split}
  [w]_{W^{s,p}(-r,r)}^p \leq&  \int_{1<|y|<r} \int_{r>|x| > 1} \frac{|u(x)-u(y)|^p}{|x-y|^{1+sp}}\, dx\, dy + (1-\delta)^{1-sp}[v]_{W^{s,p}(-1,1)}^p\\
  &+2(1-\delta)\int_{|y|<1} \int_{r>|x| > 1} \frac{|u(x)-v(y)|^p}{|x-(1-\delta)y|^{1+sp}}\, dx\, dy\\
  &+C\,\ \delta \brac{\int_{r>|y|>1} \frac{|u(1)-u(y)|^{p}}{|1 - y|^{1+sp}} dy+\int_{|y|>1} \frac{|u(-1)-u(y)|^{p}}{|-1- y|^{1+sp}} dy }\\
  &+C\,\ \delta (1-\delta)^{-sp} \brac{\int_{|y|<1} \frac{|v(1)-v(y)|^{p}}{|1 - y|^{1+sp}} dy+\int_{|y|<1} \frac{|v(-1)-v(y)|^{p}}{|-1- y|^{1+sp}} dy }\\
  &+ C\, \delta^{1-sp}\, \brac{|u(1)-v(1)|^p+|u(-1)-v(-1)|^p}.
  \end{split}
 \end{equation}
\item If we set $K := u(-2,2) \cup v(-2,2)$, then \[\dist(w,K) \leq |u(-1)-v(-1)|+|u(1)-v(1)|.\]
\end{itemize}
\end{lemma}
\begin{proof}
Let $\eta: (-2,2) \to [0,1]$ such that $\eta \equiv 0$ for $|x| > 1-\frac{\delta}{4}$ and $\eta \equiv 1$ for $|x| \leq 1-\frac{1}{2}\delta$.
We can find $\eta$ such that $|\eta'| \aleq \frac{1}{\delta}$.
 
Set 
\[
w(x) := \begin{cases} 
u(x) \quad & |x| > 1\\
(1-\eta(x)) u(-1) + \eta(x) v(-1) \quad &x \in [-1,-1+\delta]\\
v(x/(1-\delta)) \quad & |x| < 1-\delta\\
(1-\eta(x)) u(1) + \eta(x) v(1) \quad &x \in [1-\delta,1].\\
\end{cases}
\]
Then
\[
\begin{split}
 &\int_{(-2,2)} \int_{(-2,2)} \frac{|w(x)-w(y)|^p}{|x-y|^{1+sp}}\, dx\, dy\\
=& \int_{|y|>1} \int_{|x|>1} \frac{|u(x)-u(y)|^p}{|x-y|^{1+sp}}\, dx\, dy\\
&+(1-\delta)^{1-sp} \int_{|y|<1} \int_{|x|<1}  \frac{|v(x)-v(y)|^p}{|x-y|^{1+sp}}\, dx\, dy\\
&+III+2IV+2V+2VI
\end{split}
\]
where
\[
\begin{split}
III:=&\int_{|y| \in (1-\delta,1)} \int_{|x| \in (1-\delta,1)} \frac{|w(x)-w(y)|^p}{|x-y|^{1+sp}}\, dx\, dy,\\
IV&:=\int_{|y|<1-\delta} \int_{|x| > 1} \frac{|w(x)-w(y)|^p}{|x-y|^{1+sp}}\, dx\, dy,\\
V:=&\int_{|y| \in (1-\delta,1)} \int_{|x|<1-\delta} \frac{|w(x)-w(y)|^p}{|x-y|^{1+sp}}\, dx\, dy,\\
VI&:=\int_{|y| \in (1-\delta,1)} \int_{|x|>1} \frac{|w(x)-w(y)|^p}{|x-y|^{1+sp}}\, dx\, dy.\\
\end{split}
\]
We have
\[
 III = |v(\pm 1)-u(\pm 1)|^p\int_{|y| \in (1-\delta,1)} \int_{|x| \in (1-\delta,1)} \frac{|\eta(x)-\eta(y)|^p}{|x-y|^{1+sp}}\, dx\, dy \aleq \delta^{1-sp}\, |v(\pm 1)-u(\pm 1)|^p.
\]
Also 
\[
 \begin{split}
IV=&\int_{|y|<1-\delta} \int_{|x| > 1} \frac{|u(x)-v(y/(1-\delta))|^p}{|x-y|^{1+sp}}\, dx\, dy\\
=&(1-\delta)\int_{|y|<1} \int_{|x| > 1} \frac{|u(x)-v(y)|^p}{|x-(1-\delta)y|^{1+sp}}\, dx\, dy.\\
 \end{split}
\]
The tricky terms (that need to vanish as $\delta \to 0$) are the remaining ones:
\[
\begin{split}
 V =& \int_{y \in (1-\delta,1)} \int_{|x|<1-\delta} \frac{|v(x/(1-\delta))-(1-\eta(y)) u(1) - \eta(y) v(1)|^p}{|x-y|^{1+sp}}\, dx\, dy\\
 &+\int_{y \in (-1,-1+\delta)} \int_{|x|<1-\delta} \frac{|v(x/(1-\delta))-(1-\eta(y)) u(-1) - \eta(y) v(-1)|^p}{|x-y|^{1+sp}}\, dx\, dy\\
=& \int_{y \in (1-\delta,1)} \int_{|x|<1-\delta} \frac{|v(x/(1-\delta))-v(1)+(1-\eta(y)) (v(1)-u(1))|^p}{|x-y|^{1+sp}}\, dx\, dy\\
 &+\int_{y \in (-1,-1+\delta)} \int_{|x|<1-\delta} \frac{|v(x/(1-\delta))-v(-1)+(1-\eta(y)) (v(-1)-u(-1) )|^p}{|x-y|^{1+sp}}\, dx\, dy\\
=& \int_{y \in (1-\delta,1)} \int_{|x|<1-\delta} \frac{|v(x/(1-\delta))-v(1)+(\eta(x)-\eta(y)) (v(1)-u(1))|^p}{|x-y|^{1+sp}}\, dx\, dy\\
 &+\int_{y \in (-1,-1+\delta)} \int_{|x|<1-\delta} \frac{|v(x/(1-\delta))-v(-1)+(\eta(x)-\eta(y)) (v(-1)-u(-1) )|^p}{|x-y|^{1+sp}}\, dx\, dy.\\
 \end{split}
\]
That is,
\[
\begin{split}
 V \aleq& \int_{y \in (1-\delta,1)} \int_{|x|<1-\delta} \frac{|v(x/(1-\delta))-v(1)|^p}{|x-y|^{1+sp}}\, dx\, dy\\
 &+\int_{y \in (-1,-1+\delta)} \int_{|x|<1-\delta} \frac{|v(x/(1-\delta))-v(-1)|^p}{|x-y|^{1+sp}}\, dx\, dy\\
 &+\delta^{1-sp} \brac{|u(1)-v(1)|^p+|u(-1)-v(-1)|^p}.\, 
 \end{split}
\]
Now for $y \in (1-\delta,1)$ and $x < 1-\delta$ we have $|x-y| \geq |x-(1-\delta)|$, thus
\[
 \begin{split}
  &\int_{y \in (1-\delta,1)} \int_{|x|<1-\delta} \frac{|v(x/(1-\delta))-v(1)|^p}{|x-y|^{1+sp}}\, dx\, dy\\
  \leq&\int_{y \in (1-\delta,1)} \int_{|x|<1-\delta} \frac{|v(x/(1-\delta))-v(1)|^p}{|x-(1-\delta)|^{1+sp}}\, dx\, dy\\
  =&\delta \int_{|x|<1-\delta} \frac{|v(x/(1-\delta))-v(1)|^p}{|x-(1-\delta)|^{1+sp}}\, dx\\
  =&\delta (1-\delta)^{-sp} \int_{|x|<1} \frac{|v(x)-v(1)|^p}{|x-1|^{1+sp}}\, dx.
 \end{split}
\]
Arguing similarly for the other part, we obtain 
\[
\begin{split}
 V \leq& C\, \delta (1-\delta)^{-sp} \int_{|x|<1} \frac{|v(x)-v(-1)|^p}{|-1-x|^{1+sp}}\, dx\\
 &+C\, \delta (1-\delta)^{-sp} \int_{|x|<1} \frac{|v(x)-v(1)|^p}{|x-1|^{1+sp}}\, dx\\
 &+\delta^{1-sp} \brac{|u(1)-v(1)|^p+|u(-1)-v(-1)|^p}.\, 
 \end{split}
\]
Now for $VI$,
\[
\begin{split}
 VI =& \int_{|y| \in (1-\delta,1)} \int_{|x|>1} \frac{|u(x)-w(y)|^p}{|x-y|^{1+sp}}\, dx\, dy.\\
 \end{split}
\]
Again we have two parts, of which we only estimate 
\[
\begin{split}
 &\int_{y \in (1-\delta,1)} \int_{|x|>1} \frac{|u(x)-w(y)|^p}{|x-y|^{1+sp}}\, dx\, dy\\
 =&\int_{y \in (1-\delta,1)} \int_{|x|>1} \frac{|u(x)-(1-\eta(y)) u(1) - \eta(y) v(1)|^p}{|x-y|^{1+sp}}\, dx\, dy\\
 =&\int_{y \in (1-\delta,1)} \int_{|x|>1} \frac{|u(x)-u(1)+\eta(y) (u(1) - v(1))|^p}{|x-y|^{1+sp}}\, dx\, dy\\
 \aleq&\int_{y \in (1-\delta,1)} \int_{|x|>1} \frac{|u(x)-u(1)|^p}{|x-y|^{1+sp}}\, dx\, dy +|u(1) - v(1)|^p \delta^{1-sp} \\
 \aleq&\int_{y \in (1-\delta,1)} \int_{|x|>1} \frac{|u(x)-u(1)|^p}{|x-1|^{1+sp}}\, dx\, dy +|u(1) - v(1)|^p \delta^{1-sp}\\
 =&\delta \int_{|x|>1} \frac{|u(x)-u(1)|^p}{|x-1|^{1+sp}}\, dx +|u(1) - v(1)|^p \delta^{1-sp} .
 \end{split}
\]
For the last inequality we used that whenever $\delta \in (0,1)$ if $y \in (1-\delta,1)$ and $x> 1$ or $x<-1$, we have that $|x-y| \geq c |x-1|$ for some uniform constant $c$.
\end{proof}

\subsection{Proof of Theorem~\ref{th:minseqcrit}}
Armed with the Luckhaus Lemma, we can now prove \Cref{th:minseqcrit}. The idea is to work with $\gamma_k'$ and $\gamma' + t\varphi'$, using Luckhaus Lemma to glue $\gamma'+t\varphi'$ into $\gamma_k'$ where the gluing happens in an annulus $B(R)\backslash B(R(1-\delta))$. The important observation is that even under the assumption of only weak convergence the norms on the annulus $B(R)\backslash B(R(1-\delta))$ vanish in the limit.

We face an additional technicality, which is that we want to glue derivatives. Up to a adding a corrector term after integration, this leads to a curve $\sigma_{k,\delta}$ wich coincides with $\gamma_k$ outside of a ball $B(R)$. The map $\sigma_{k,\delta}$ may not coincide with $\gamma+t\varphi$ inside the ball $B(R)$, however, its derivative $\sigma_{k,\delta}'$ is essentially $\gamma'$ inside the ball, which is good enough for our purposes.

\begin{proof}[Proof of Theorem~\ref{th:minseqcrit}]
We may assume that the ball $B(\rho)$ is centered at zero. Since $\rho$ is small, w.l.o.g. we will assume for simplicity that the distance $|x-y|$ corresponds to the Euclidean distance.

For any $\eps > 0$ there is $\eps_0 > 0$ such that \eqref{eq:tpl:smallen1} implies
\begin{equation}\label{eq:tpl:smallen}
 [\gamma']_{W^{\frac{1}{q},q}(B(100\rho))}+\sup_k [\gamma_k']_{W^{\frac{1}{q},q}(B(100\rho))}<\eps .
\end{equation}
Indeed this follows, similar to the arguments in \Cref{s:weaklimitimmersion}, from the fact that $\gamma_k$ is smooth and \eqref{eq:tpl:smallen1} implies by \Cref{th:gammasobvstp2} a uniform bound of $\gamma_k'$ in $W^{\frac{1}{q},q}(B(100\rho))$, so that a subsequence of $\gamma_k'$ converges to $\gamma$ weakly in $W^{\frac{1}{q},q}(B(100\rho))$. 

Let $\varphi \in C_c^\infty(B(\rho))$.

By Fubini's theorem, there exists $R \in (\rho,2\rho)$ such that $\pm R$ are Lebesgue points of $\gamma_k'$ and $\gamma'$, 
\begin{equation}\label{eq:tpl:fub1}
  \int_{\R} \frac{|\gamma'(x)-\gamma'(y)|^{q}}{|x-y|^{2}} dy + \sup_{k} \int_{\R} \frac{|\gamma_k'(x)-\gamma_k'(y)|^{q}}{|x - y|^{2}} dy < \infty \quad x= -R,R,
\end{equation}
and we can also assume that $\gamma_k'(\pm R)$ converges to $\gamma'(\pm R)$ for this $R$, and that $\gamma_k' \xrightarrow{k \to \infty} \gamma'$ weakly in $W^{\frac{1}{q},q}(B(100\rho))$ and strongly in $L^q(B(100\rho))$.

Observe, this also implies 
\begin{equation}\label{eq:tpl:fub2}
 \int_{\R} \frac{|(\gamma+t\varphi)'(x)-(\gamma+t\varphi)'(y)|^{q}}{|x-y|^{2}} dy < \infty \quad \quad x= -R,R, \ \forall t \in \R.
\end{equation}

\underline{Construction of a comparison map $\sigma_{k,\delta}$}

Fix $\delta > 0$. Apply Luckhaus Lemma, \Cref{la:frluckhaus}, to $\gamma' + t\varphi'$ (within $B(R)$) and $\gamma_k'$ (in $B(R)^c$).
Then we obtain 
\[
g_{k,\delta}(x) = \begin{cases} 
\gamma_k'(x) \quad & x <- R\\
(\gamma+t\varphi)'(x/(1-\delta)) \quad & |x| < (1-\delta)R\\
\gamma_k'(x) \quad & x > R.\\
\end{cases}
\]
For $(1-\delta)R\leq |x|\leq R$ we have a interpolation between $\gamma'(\pm R)$ and $\gamma_k'(\pm R)$ as in \Cref{la:frluckhaus}. 

Observe that for $t_0 \ll 1$ we have that $|(\gamma+t\varphi)'|$ is as close to $1$ as we want, so we also get the estimate
\[
 \dist(g_{k,\delta}'(x),\S^{N-1})\leq \eps \quad \text{for $k \gg 1$, $t_0 \ll 1$}.
\]
Pick $\theta \in C^\infty(\R)$ with $\theta \equiv 0$ for $x < -R/2$, $\theta \equiv 1$ for $x \geq R/2$, and with $|\theta'| \aleq \frac{1}{R}$. We set 
\begin{equation}\label{eq:testfunction}
 \sigma_{k,\delta}(x) := \int_{-1}^x g_{k,\delta}(z)\, dz + \gamma_k(-1) + \theta(x) \brac{\int_{-R}^{R} \gamma_k'(z)-g_{k,\delta}(z)\, dz} .
\end{equation}

\underline{Properties of $\sigma_{k,\delta}$}
We need to show that $\sigma_{k,\delta}$ (for all small $t$, small $\delta$ and large $k$) is a comparison function for $\gamma_k$.

First we show 
\begin{equation}\label{eq:tpl:sigmakcomp}
 \sigma_{k,\delta} \equiv \gamma_k \quad \text{on $\R/\Z \backslash [-R,R]$}.
\end{equation}
Indeed, observe for $x < -R/2$ that we have
\[
 \sigma_{k,\delta}(x) = \int_{-1}^x g_{k,\delta}(z)\, dz + \gamma_k(-1),
\]
and thus $\sigma_{k,\delta}'(x) = g_{k,\delta}(x) = \gamma_k'(x)$ whenever $x < -R$. Since moreover $\sigma_{k,\delta}(-1) = \gamma_k(-1)$, we have
\[
 \sigma_{k,\delta} \equiv \gamma_k \quad \text{on $[-1,-R)$}.
\]
Moreover, for $x > R$ we have
\[
\begin{split}
 \sigma_{k,\delta}(x) =& \int_{-1}^x g_{k,\delta}(z)\, dz + \gamma_k(-1) + \brac{\int_{-R}^{R} \gamma_k'(z)-g_{k,\delta}(z)\, dz}\\
 =& \int_{-1}^x \gamma_k'(z)\, dz + \gamma_k(-1).
 \end{split}
\]
Again, this implies that $\sigma_{k,\delta}'(x) = \gamma_k'(x)$ for $x> R$ and since $\int_{-1}^1 \gamma_k' = 0$, we have that $\sigma_{k,\delta}(1) = \gamma_k(-1) = \gamma_k(1)$. 

Therefore, \eqref{eq:tpl:sigmakcomp} is established.

Next, we show that there exists $t_0 > 0$, $\delta_0 > 0$ and $K_0 \in \N$ such that for 
\begin{equation}\label{eq:tpl:sigmaarclength}
 \frac{1}{2} \leq |\sigma_{k,\delta}'|\leq \frac{3}{2} \quad \forall |t| < t_0, \ \delta \in (0,\delta_0), \ k \geq K_0.
\end{equation}
Indeed, there exists $t_0 > 0$ such that $|\gamma'+t\varphi'| \in (1-\frac{1}{1000},1+\frac{1}{1000})$ for all $|t| < t_0$. From \Cref{la:frluckhaus} we then have
\[
 |g_{k,\delta}(x)-1| \leq \frac{1}{1000} + |\gamma_k'(-R)-\gamma'(-R)|+|\gamma_k'(R)-\gamma'(R)|.
\]
Since by assumption $\gamma_k'(\pm R) \xrightarrow{k \to \infty} \gamma'(\pm R)$, we can find $K_0 \in \N$ such that 
\begin{equation}\label{eq:tpl:gkalmost1}
  |g_{k,\delta}(x)-1| \leq \frac{2}{1000} \quad \forall |t| < t_0,\, k \geq K_0.
\end{equation}
On the other hand, the term involving $\theta$ is converging in $C^\infty(\R)$ to zero. Namely, we have (recall that $\varphi(\pm R) = 0$)
\[
\begin{split}
 &\left |\int_{-R}^R \gamma_k'(z)-g_{k,\delta}(z) dz \right |\\
 \leq&\left |\int_{-R}^R \gamma_k'(z)- \int_{-(1-\delta)R}^{(1-\delta)R} (\gamma+t\varphi)'(z/(1-\delta))dz \right | + 2 \delta R \|g_{k,\delta}\|_{L^\infty}\\
 \overset{\eqref{eq:tpl:gkalmost1}}{\leq} &\left |\int_{-R}^R \gamma_k'(z)- \int_{-(1-\delta)R}^{(1-\delta)R} (\gamma+t\varphi)'(z/(1-\delta))dz \right | + 4\delta R\\
 =&\left |\gamma_k(R)-\gamma_k(-R)- \brac{(1-\delta) (\gamma+t\varphi)(R) - (1-\delta) (\gamma+t\varphi)(-R)} \right | + 4\delta R\\
 =&\left |\gamma_k(R)-\gamma_k(-R)- \brac{(1-\delta) (\gamma)(R) - (1-\delta) \gamma(-R)} \right | + 4\delta R\\
 \leq&2\max_{x \in \{-R,R\}}\left |\gamma_k(x)- \gamma(x)\right |+\delta |\gamma(x)| 
 + 4\delta R.\\
 \end{split}
\]
Since $\gamma_k$ uniformly converges to $\gamma$, we obtain that for any $L \in \N$
\begin{equation}\label{eq:tpl:corrector}
 \forall \tilde{\eps} > 0 \exists \delta_0 > 0, K_0 \in \N: \quad \left \|\theta(\cdot) \int_{-R}^R \gamma_k'(z)-g_{k,\delta}(z) dz\right \|_{C^L(\R)} < \tilde{\eps}\quad \forall \delta \in (0,\delta_0), \ k \geq K_0.
\end{equation}
\eqref{eq:tpl:gkalmost1} and \eqref{eq:tpl:corrector} readily imply 
\eqref{eq:tpl:sigmaarclength}.

Next we estimate the Sobolev norm of $\sigma_{k,\delta}$. From Luckhaus Lemma, \eqref{eq:frl:estw}, we obtain the estimate
\[
 \begin{split}
  &[g_{k,\delta}]_{W^{\frac{1}{q},q}(B(100\rho))}^q\\
  \aleq&  [\gamma_k']_{W^{\frac{1}{q},q}(B(100\rho))}^q+[\gamma']_{W^{\frac{1}{q},q}(B(100\rho))}^q+|t| [\varphi']_{W^{\frac{1}{q},q}(B(100\rho))}^q\\
  &+C(R)\,\ \delta \brac{\int_{\R} \frac{|(\gamma+t\varphi)'(R)-(\gamma+t\varphi)'(y)|^{q}}{|R - y|^{2}} dy+\int_{\R} \frac{|(\gamma+t\varphi)'(-R)-(\gamma+t\varphi)'(y)|^{q}}{|-R- y|^{2}} dy }\\
  &+C(R)\,\ \delta \brac{\int_{\R} \frac{|\gamma_k'(R)-\gamma_k'(y)|^{q}}{|R - y|^{2}} dy+\int_{\R} \frac{|\gamma_k'(-R)-\gamma_k'(y)|^{q}}{|-R- y|^{2}} dy }\\
  &+ C(R)\, \brac{|\gamma_k'(R)-\gamma'(R)|^q+|\gamma_k'(-R)-\gamma'(-R)|^q}\\
  &+C(R) \delta^{-2}  \|\gamma_k'-\gamma\|_{L^q(B(100\rho))}^q.
  \end{split}
 \]
Here we have used that 
\[
\begin{split}
&2(1-\delta)\int_{|y|<R} \int_{100\rho>|x| > R} \frac{|\gamma_k'(x)-(\gamma+t\varphi)'(y)|^q}{|x-(1-\delta)y|^{2}}\, dx\, dy\\
= \, &2(1-\delta)\int_{|y|<R} \int_{100\rho>|x| > R} \frac{|(\gamma_k+t\varphi)'(x)-(\gamma+t\varphi)'(y)|^q}{|x-(1-\delta)y|^{2}}\, dx\, dy\\
\aleq\, & (1-\delta)[(\gamma_k+t\varphi)']_{W^{\frac{1}{q},q}(B(100\rho))}^q + (1-\delta) \delta^{-{2}} \|\gamma_k'-\gamma'\|_{L^q(B(R))}^q.
\end{split}
\]
In view of the convergence $\gamma_k'(\pm R) \xrightarrow{k \to \infty} \gamma'(\pm R)$, the $L^q$-convergence of $\gamma_k'$ to $\gamma'$,  \eqref{eq:tpl:fub1}, and \eqref{eq:tpl:fub2} together with \eqref{eq:tpl:smallen}, we obtain that there are $\delta_0$, $t_0$ such that
\[
\forall \delta \in (0,\delta_0)\, \exists K(\delta):\
[g_{k,\delta}]_{W^{\frac{1}{q},q}(B(100\rho))}^q < 4\eps \quad \forall k \geq K(\delta),\ |t| < t_0.
\]
From \eqref{eq:tpl:corrector} we thus conclude that taking possibly $\delta_0$ and $t_0$ smaller and $K(\delta)$ larger,
\[
\forall \delta \in (0,\delta_0)\, \exists K(\delta):\ [\sigma_{k,\delta}]_{W^{\frac{1}{q},q}(B(100\rho))}^q < 4\eps \quad \forall k \geq K(\delta),\ |t| < t_0.
\]
\underline{The comparison}

In view of uniform convergence, \eqref{eq:tpl:smallen1}, and \eqref{eq:tpl:smallen}, we can apply \Cref{th:smallenergyinjectivity} and \Cref{la:bilip} to conclude that the assumptions of \Cref{th:localambisotchange} are satisfied for $\sigma_{k,\delta}$ and $\gamma_k$ for all $k \geq K(\delta)$ for some large number $K$.

\Cref{th:localambisotchange} implies $\sigma_{k,\delta}$ and $\gamma_k$ are ambient isotopic (technically: we can mollify $\sigma_{k,\delta}$ around $B(5\rho)$, then this mollification is ambient isotopic to $\gamma_k$, so we get the estimate and remove the mollification again (observe that $\gamma_k$ is smooth)), and thus
\[
 \tp^{q+2,q}(\gamma_k) \leq \tp^{q+2,q} (\sigma_{k,\delta}) + \frac{1}{k} \quad \forall k \geq K(\delta).
\]
Since $\gamma_k \equiv \sigma_{k,\delta}$ on $B(R)^c = \R/\Z \backslash B(R)$, this implies
\[
\begin{split}
 \int\int_{(\R/\Z)^2 \backslash (B(R)^c)^2} \mu(\gamma_k,x,y)\, dx\, dy \leq \int\int_{(\R/\Z)^2 \backslash (B(R)^c)^2} \mu(\sigma_{k,\delta},x,y)\, dx\, dy + \frac{1}{k}.
 \end{split}
\]
But if $(x,y)\in (\R/\Z)^2 \backslash (B(2\rho)^c)^2$, we have that $\gamma(x)=\gamma(y)$ if and only if $x = y$ (a set of measure zero), that is we have 
\[
 \mu(\gamma,x,y) = \lim_{k \to \infty} \mu(\gamma_k,x,y) \quad \text{a.e. }(x,y) \in (\R/\Z)^2 \backslash (B(2\rho)^c)^2.
 \]
From Fatou's Lemma we then find
\[
\begin{split}
 \int\int_{(\R/\Z)^2 \backslash (B(R)^c)^2} \mu(\gamma,x,y)\, dx\,  dy \leq \liminf_{k \to \infty} \int\int_{(\R/\Z)^2 \backslash (B(R)^c)^2} \mu(\sigma_{k,\delta},x,y)\, dx\, dy. 
 \end{split}
\]
\underline{The convergence of $\sigma_{k,\delta}$ as $k \to \infty$}
We now fix $\delta \in (0,\delta_0)$ and consider the limit as $k \to \infty$.
Set
\[
g_{\delta}(x) := \begin{cases} 
\gamma'(x) \quad & x <- R\\
\gamma'(-R) \quad & x \in (-R,-R(1-\delta))\\
(\gamma+t\varphi)'(x/(1-\delta)) \quad & |x| < (1-\delta)R\\
\gamma'(R) \quad & x \in (R(1-\delta),R)\\
\gamma'(x) \quad & x > R\\
\end{cases}
\]
and
\[
 \sigma_\delta(x) := \int_{-1}^x g_{\delta}(z)\, dz - \gamma(-1) + \theta (x) \brac{\int_{-R}^{R} \gamma'(z)-g_{\delta}(z)\, dz} .
\]
We claim that 
\begin{equation}\label{eq:tpl:kconv}
  \limsup_{k \to \infty} \left |\int\int_{(\R/\Z)^2 \backslash (B(R)^c)^2} \mu(\sigma_{k,\delta},x,y)\, dx \, dy  - \int\int_{(\R/\Z)^2 \backslash (B(R)^c)^2} \mu(\sigma_{\delta},x,y)\, dx \, dy \right | \leq C \delta.
\end{equation}
First observe that 
\begin{equation}\label{eq:tql:arclenghtunifconv}
 |\sigma_{\delta,k}'(x)| \xrightarrow{k \to \infty} |\sigma_\delta'(x)| \quad \forall \delta < \delta_0, \text{uniformly for $x\in\R/\Z$}.
\end{equation}
Indeed, by the support of $\varphi$,
\[
 |x| > R: |\sigma_{\delta,k}'(x)| = |\gamma_k'(x)| = 1 = |\gamma'(x)|=|\sigma_\delta'(x)|.
\]
Also for $|x| \in ((1-\delta)R,R)$,
\[
\begin{split}
 |g_{k,\delta}(x)|=& |\eta(x)\gamma'(\pm R)+(1-\eta(x)) \gamma'_k(\pm R)|\\ 
  \xrightarrow{k \to \infty} &|\eta(x)\gamma'(\pm R)+(1-\eta(x)) \gamma'(\pm R)|\\
  =&|\gamma'(\pm R)| = |g_\delta(x)|.
 \end{split}
\]
and for $|x| < (1-\delta)R$ we have 
\[
 |g_{k,\delta}(x)| = |g_{\delta}(x)|.
\]
By the definition of $\sigma_\delta$ and in view of \eqref{eq:tpl:corrector} we have established \eqref{eq:tql:arclenghtunifconv}.

Next we observe that for $(x,y) \in (\R/\Z)^2 \backslash (B(R)^c)^2$ with $|x-y| > \frac{\delta}{100}$ we have by the bilipschitz estimate or the global distance estimate 
\[
 \inf_{k} |\sigma_{k,\delta}(x)-\sigma_{k,\delta}(y)| \ageq C(\delta).
\]
Thus we can make a brute force estimate
\[
\begin{split}
 &\left |
 \int\int_{(\R/\Z)^2 \backslash (B(R)^c)^2 \cap |x-y| \ageq \delta} \mu(\sigma_{k,\delta},x,y)\, dx\,  dy  - \int\int_{(\R/\Z)^2 \backslash (B(R)^c)^2 \cap |x-y| \ageq \delta} \mu(\sigma_{\delta},x,y)\, dx \, dy \right |\\
 \aleq& C(\delta, \|\gamma_k'\|_{L^\infty},\|\gamma_k\|_{L^\infty})\, \int_{\R/\Z} |\gamma_k'(x)-\gamma'(x)|^q+|\gamma_k(x)-\gamma(x)|^q \, dx. 
 \end{split}
 \]
Since $\gamma_k \xrightarrow{k \to \infty} \gamma$ a.e., by dominated convergence (recall $|\gamma_k'| \leq 1$), we have
\[
 \lim_{k \to \infty} \left |
 \int\int_{(\R/\Z)^2 \backslash (B(R)^c)^2 \cap |x-y| \ageq \delta} \mu(\sigma_{k,\delta},x,y)\, dx \, dy  - \int\int_{(\R/\Z)^2 \backslash (B(R)^c)^2 |x-y| \ageq \delta} \mu(\sigma_{\delta},x,y)\, dx \, dy \right | = 0.
\]
That is, 
\[
\begin{split}
  &\limsup_{k \to \infty} \left |\int\int_{(\R/\Z)^2 \backslash (B(R)^c)^2} \mu(\sigma_{k,\delta},x,y)\, dx\, dy  - \int\int_{(\R/\Z)^2 \backslash (B(R)^c)^2} \mu(\sigma_{\delta},x,y)\, dx\, dy \right | \\
  \aleq &\limsup_{k \to \infty} \left |\int\int_{(\R/\Z)^2 \backslash (B(R)^c)^2, |x-y| \leq \frac{1}{100}\delta} \mu(\sigma_{k,\delta},x,y)- \mu(\sigma_{\delta},x,y)\, dx\, dy \right |.\\
  \end{split}
\]
Also observe that $\mu(\sigma_{k,\delta},x,y) = \mu(\sigma_{\delta},x,y)$ for all $x,y \in B((1-\delta)R)$, $k \in \N$.

So we have to consider the following terms, (observe $\mu$ is not symmetric in $x$,$y$)
\[
\begin{split}
 &\limsup_{k \to \infty} \left |\int\int_{(\R/\Z)^2 \backslash (B(R)^c)^2} \mu(\sigma_{k,\delta},x,y)\, dx\, dy  - \int\int_{(\R/\Z)^2 \backslash (B(R)^c)^2} \mu(\sigma_{\delta},x,y)\, dx\, dy \right | \\
\aleq&\limsup_{k \to \infty}\left |
 \sum_{\ell = 1}^5\int\int_{A_\ell}
\brac{
\frac{\abs{\sigma_{k,\delta}(x)' \wedge (\sigma_{k,\delta}(x)-\sigma_{k,\delta}(y))}^q}{\abs{\sigma_{k,\delta}(x)-\sigma_{k,\delta}(y)}^{q+2}}
- \frac{\abs{\sigma_{\delta}(x)' \wedge (\sigma_{\delta}(x)-\sigma_{\delta}(y))}^q}{\abs{\sigma_{\delta}(x)-\sigma_{\delta}(y)}^{q+2}}
} |\sigma_{\delta}'(x)|^{1-q} |\sigma_{\delta}'(y)|\, dx\, dy
\right |
\end{split}
\]
where

\[
\begin{split}
A_1 :=& (\pm R,\pm R(1+\delta)) \times (\pm R(1-\delta),\pm R)\\
A_2:= &(\pm R(1-\delta),\pm R) \times (\pm R,\pm R(1+\delta))\\
A_3:= &(\pm R(1-2\delta),\pm R(1-\delta)) \times  (\pm R(1-\delta),\pm R)\\
A_4:= &(\pm R(1-\delta),\pm R) \times (\pm R(1-2\delta),\pm R(1-\delta))\\
A_5:= &(\pm R(1-\delta),\pm R) \times (\pm R(1-\delta),\pm R).
\end{split}
\]
Observe that in each of these regimes $\theta \equiv 1$ or $\theta \equiv 0$, that is
\[
\begin{split}
 &\limsup_{k \to \infty} \left |\int\int_{(\R/\Z)^2 \backslash (B(R)^c)^2} \mu(\sigma_{k,\delta},x,y)\, dx\, dy  - \int\int_{(\R/\Z)^2 \backslash (B(R)^c)^2} \mu(\sigma_{\delta},x,y)\, dx\, dy \right | \\
\aleq&\limsup_{k \to \infty}\left |
 \sum_{\ell = 1}^5\int\int_{A_\ell}
\brac{
\frac{\abs{g_{k,\delta}(x) \wedge \int_x^yg_{k,\delta}(z)\, dz}^q}{\abs{\sigma_{k,\delta}(x)-\sigma_{k,\delta}(y)}^{q+2}}
- \frac{\abs{g_{\delta}(x) \wedge \int_x^y g_{\delta}(z)\,dz}^q}{\abs{\sigma_{\delta}(x)-\sigma_{\delta}(y)}^{q+2}}
} |\sigma_{\delta}'(x)|^{1-q} |\sigma_{\delta}'(y)|dx\, dy
\right |\\
=&\limsup_{k \to \infty}\left |
 \sum_{\ell = 1}^5\int\int_{A_\ell}
\brac{
\frac{\abs{g_{k,\delta}(x) \wedge \int_x^yg_{k,\delta}(z)\, dz}^q}{\abs{\int_x^y g_{k,\delta}(z)\, dz}^{q+2}}
- \frac{\abs{g_{\delta}(x) \wedge \int_x^y g_{\delta}(z)\,dz}^q}{\abs{\int_x^y g_{\delta}(z)\,dz}^{q+2}}
} |\sigma_{\delta}'(x)|^{1-q} |\sigma_{\delta}'(y)|dx\, dy
\right |\\
\aleq&\limsup_{k \to \infty}
 \sum_{\ell = 1}^5\int\int_{A_\ell}
\frac{\abs{\abs{g_{k,\delta}(x) \wedge \int_x^yg_{k,\delta}(z)\, dz}^q
- \abs{g_{\delta}(x) \wedge \int_x^y g_{\delta}(z)\,dz}^q} }
{\abs{\int_x^y g_{k,\delta}(z)\,dz}^{q+2}}
 |\sigma_{\delta}'(x)|^{1-q} |\sigma_{\delta}'(y)|dx\, dy\\
&+\limsup_{k \to \infty}
 \sum_{\ell = 1}^5\int\int_{A_\ell} \frac{\abs{\abs{\int_x^y g_{\delta}(z)\,dz}^{q+2}-\abs{\int_x^y g_{k,\delta}(z)\,dz}^{q+2}}}{
 \abs{\int_x^y g_{\delta}(z)\,dz}^{q+2}
 }
\frac{\abs{g_{\delta}(x) \wedge \int_x^y g_{\delta}(z)\,dz}^q}
{\abs{\int_x^y g_{k,\delta}(z)\,dz}^{q+2}}
 |\sigma_{\delta}'(x)|^{1-q} |\sigma_{\delta}'(y)|dx\, dy\\
\end{split}
\]
By the uniform bi-Lipschitz estimate we have 
\[
\begin{split}
    \aleq&\limsup_{k \to \infty}
 \sum_{\ell = 1}^5\int\int_{A_\ell}
\frac{\abs{\abs{g_{k,\delta}(x) \wedge \int_x^yg_{k,\delta}(z)\, dz}^q
- \abs{g_{\delta}(x) \wedge \int_x^y g_{\delta}(z)\,dz}^q} }
{\abs{x-y}^{q+2}}
 dx\, dy\\
&+\limsup_{k \to \infty}
 \sum_{\ell = 1}^5\int\int_{A_\ell} \frac{\abs{\abs{\int_x^y g_{\delta}(z)\,dz}^{q+2}-\abs{\int_x^y g_{k,\delta}(z)\,dz}^{q+2}}}{|x-y|^{q+2}}
\frac{\abs{g_{\delta}(x) \wedge \int_x^y g_{\delta}(z)\,dz}^q}
{\abs{\int_x^y g_{\delta}(z)\,dz}^{q+2}}
 |\sigma_{\delta}'(x)|^{1-q} |\sigma_{\delta}'(y)|dx\, dy\\
\end{split}
\]
The second set of integrals converges to zero by dominated convergence theorem since $g_{k,\delta}$ converges a.e. to $g_\delta$. To establish \eqref{eq:tpl:kconv} we argue as in the proof \Cref{la:frluckhaus} to obtain that the first terms satisfy
\[
\limsup_{k \to \infty}
 \sum_{\ell = 1}^5\int\int_{A_\ell}
\frac{\abs{\abs{g_{k,\delta}(x) \wedge \int_x^yg_{k,\delta}(z)\, dz}^q
- \abs{g_{\delta}(x) \wedge \int_x^y g_{\delta}(z)\,dz}^q} }
{\abs{x-y}^{q+2}}
 dx\, dy \aleq \delta.
\]

\underline{The convergence of $\sigma_{\delta}$ as $\delta \to 0$}

By now we have shown that for any $\delta \in (0,\delta_0)$
\[
\begin{split}
 \int\int_{(\R/\Z)^2 \backslash (B(R)^c)^2} \mu(\gamma,x,y)\, dx \, dy \leq \int\int_{(\R/\Z)^2 \backslash (B(R)^c)^2} \mu(\sigma_{\delta},x,y)\, dx\,  dy  + C\delta.
 \end{split}
\]
Observe that $\sigma_{\delta} \xrightarrow{\delta \to 0} \gamma + t\varphi$, in view of \eqref{eq:tpl:corrector} -- indeed essentially repeating the Luckhaus-Lemma argument from above, we see that as $\delta \to 0$ we get 
\[
 \int\int_{(\R/\Z)^2 \backslash (B(R)^c)^2} \mu(\sigma_{\delta},x,y)\, dx\,  dy  \xrightarrow{\delta \to 0} \int\int_{(\R/\Z)^2 \backslash (B(R)^c)^2} \mu(\gamma+t\varphi,x,y)\, dx\,  dy.
\]

\end{proof}


\section{The Regularity theory for critical points: Proof of Theorem~\ref{th:mainreg}}
\label{s:regularity}

This section is dedicated to show $C^{1,\alpha}$-regularity of locally critical points for scale-invariant tangent-point energies $\tp^{q+2,q}$ with $q\geq 2$. Our main goal is the following  decay estimate. 
\begin{proposition}[Local decay estimate]\label{pr:decayest}
	Let $q\geq 2$ and $\gamma$ be a locally critical embedding in the sense of \Cref{def:wcp} with small tangent-point energy $\tp^{q+2,q}$ around a geodesic ball $B(x_0,r) \subset \R/\Z$, and assume $|\gamma'| \equiv const > 0$ almost everywhere. 
	Let $u:= \frac{\gamma'}{|\gamma'|}$, that is $u:\RZ\rightarrow \S^2$ such that $\int_{\RZ} u = 0$, and let $\tilde{u}: \R \to \R^3$ be a $L^\infty \cap W^{\frac 1q,q}$-extension of $u|_{B(x_0,r)}$ from  $B(x_0,r)$ to $\R$.
	Then there exist $\eps,\tau, \theta\in(0,1)$ and $N_0\in\N$ such that the following holds.
	
	If $N\geq N_0$, $\rho> 0$, and $y\in B(x_0,\tfrac r2)$ such that $B_{2^N\rho}:= B(y,2^N\rho) \subset B(x_0,r)$, and $[\tilde u]_{W^{\frac 1q,q}(B_{2^N\rho})} \leq \eps$, then
	\begin{align*}
		& [u]^q_{W^{\frac{1}{q}, q}(B_\rho)}  \leq \tau [u]^q_{W^{\frac{1}{q}, q}(B_{2^N\rho})}   + \sum_{l=1}^\infty 2^{-\theta (N+l)} [\tilde u]^{q}_{W^{\frac 1q,q }(B_{2^{N+l}\rho})} + \rho.
	\end{align*}
\end{proposition}

\Cref{pr:decayest} implies \Cref{th:mainreg} by the usual Dirichlet growth-type iteration techniques.
\begin{proof}[Proof of \Cref{th:mainreg}]
First note that by definition $u$ and $\tilde u$ coincide locally around $B(x_0,r)$, therefore we have for any ball $B\subset B(x_0,r)$ that
$$
	[u]_{W^{\frac{1}{q}, q}(B)} = [\tilde u]_{W^{\frac{1}{q}, q}(B)}.
$$
By iterating the decay estimate on small balls, cf. \cite[Lemma A.8]{BRS16}, we obtain a $\sigma > 0$ such that 
\[
\sup_{0<\rho<\frac r2, \, y\in B(x,\frac r2)} \rho^{-\sigma} [u]_{W^{\frac 1q,q}(B(y,\rho))}^q \aleq C(u).
\]
By Jensen's inequality, we conclude that 
\[
 \sup_{0<\rho<\frac r2, \, y\in B(x,\frac r2)} \rho^{-\sigma} \mvint_{B(y,\rho)} |u(z)-(u)_{B(y,\rho)}|^q \, dz \aleq C(u),
\]
where $(u)_{B(y,\rho)}$ denotes the mean value of $u$ in $B(y,\rho)$, and hence $u$ belongs to the Campanato space $\mathcal{L}^{q,1+\sigma}(B(x_0,r))$.
The characterization of Campanato spaces with H\"older spaces, cf.~\cite[Theorem~1.2]{Giaquinta1983}, implies that $u \in C^{\frac \sigma q}_{loc}(B(x_0,\frac{r}{2}))$, which concludes the proof of \Cref{th:mainreg}. 
\end{proof}

In order to obtain \Cref{pr:decayest} -- inspired by the investigations of critical O'Hara energies in \cite{BRS19} by comparison to the theory of fractional harmonic maps, cf. \cite{DLR11a,S15} -- we proceed as follows: 

\begin{itemize}
	\item In \Cref{s:newenergy} we relate critical knots of the tangent-point energies $\tp^{p,q}$ with $p\in [ q+2,2q+1)$ and $q>1$ to fractional harmonic maps: We first define a suitable energy $\E^{p,q}$ such that the unit tangent $u:= \frac{\gamma^\prime}{|\gamma^\prime|}$ of locally critical embeddings $\gamma$ of $\tp^{p,q}$ with constant-speed parametrization are locally critical maps of the energy $\E^{p,q}$ in the class of maps $v:\R/\Z \rightarrow \S^2$. We then establish that the new energy $\E^{p,q}$ is locally comparable to a $W^{\frac{p-q-1}{q}, q}$-seminorm, see \Cref{s:weaklimitimmersion}. Consequently, the equations, that the critical maps $u$ satisfy, are indeed structurally similar to the Euler-Lagrange equations of fractional harmonic maps into the sphere $\S^2$ as treated in \cite{S15}.
	\item In \Cref{s:EulerLagrangeEqu} we derive the Euler-Lagrange equations of the new energies $\E^{p,q}$ for $p\in [ q+2,2q+1)$, $q>1$, and study the highest order and remainder terms of the Lagrangian.
	\item In \Cref{subsec:RegularityTheory} we finally treat the actual decay estimate of \Cref{pr:decayest}.
\end{itemize}

Before continuing with the upcoming subsections, we need to introduce some notation for integration on $\R/\Z$, cf. \cite[Remark 2.2]{BRS19}: 
\begin{enumerate}[(1)]
	\item  We identify by $\rho (x,y)$ the distance of two points $x,y\in \R/\Z$ on $\R/\Z$, in particular $\rho(x,y)=|x-y| \,  {\rm{mod}} \tfrac 12$. 
	\item If $x$ and $y$ are not antipodal, which means $|x-y|\neq \tfrac 12$, we denote by $x\triangleright y$ the shortest geodesic from $x$ to $y$. Hence, we define for any $\Z$-periodic $f$
	\[
	\oint_{x\triangleright y} f := \int_x^{\tilde{y}} f(z) \, dz,
	\]
	where $\tilde{y}\in y + \Z$ such that $|x-\tilde{y}|< \tfrac 12$. 
	\item Furthermore, we write 
	\[
	\sigma(x\triangleright y) = {\rm sgn} \oint_{x\triangleright y} 1.
	\]
	That means, if $x\triangleright y$ is positively oriented, we have $\sigma(x\triangleright y) = 1$, and if $x\triangleright y$ is negatively oriented, we get $\sigma(x\triangleright y) =- 1$.
	\item Now given a $\Z$-periodic function $f$, we define
	\[
	\avint_{x\triangleright y} f := \frac{\sigma(x\triangleright y)}{\rho(x,y)} \oint_{x\triangleright y} f.
	\]
\end{enumerate}

We also have to deal with the fact that the critical embeddings of interest are only locally known to be of class $W^{1+\frac 1q,q}(\R/\Z,\R^3)$, which motivates the use of the extension $\tilde{u}$ as described below. This is a mere technical inconvenience, and we recommend the first-time reader to mentally identify $u$ and $\tilde{u}$ in the arguments to come.

\begin{remark}\label{reg:rm:extension}
	Let $p\in[q+2,2q+1)$, $q>1$, and $\gamma: \R/\Z\rightarrow \R^3$ be a homeomorphism with locally small tangent-point energy according to \Cref{def:homeotp} that is a locally critical embedding in $B(x_0,r)$ of $\tp^{p,q}$ as in \Cref{def:wcp}. Then \Cref{th:weaklimitareweakimm} implies that $\gamma$ is globally bi-Lipschitz and of class $W^{1+s,q}(\R/\Z,\R^3)$ for any $0<s<\tfrac 1q$. However, $\gamma$ is not known to globally belong to the class $W^{1+\frac 1q,q}$ or even $W^{1+\frac {p-q-1}q,q}$, we only have the local statement $\gamma \in W^{1+\frac {p-q-1}q,q}(B(x_0,r),\R^3)$ due to \Cref{th:gammasobvstp2}. 

   Although we aim to mostly work with the local $W^{\frac {p-q-1}q,q}$-Gagliardo seminorm of $\gamma'$, we also have to acknowledge global terms thereof on account of the non-locality of the proposed problem. For this reason, when necessary, we may interpret $B(x_0,r)$ as an interval in $[-1,2]$ and extend $\gamma'|_{B(x_0,r)}$ from $B(x_0,r)$ to a function $\tilde u  \in W^{\frac {p-q-1}q,q}(\R,\R^3)$ such that $\tilde u$ is uniformly bounded. Such extension exists since $\|\gamma'\|_{L^\infty}\leq 1$ and by standard construction of extensions, e.g. \cite[Theorem~5.4]{Hitchhiker}. Note that in this setting we have for any ball $B\subset B(x_0,r)$ that
	\[
	[\tilde u]_{W^{\frac{p-q-1}{q}, q}(B)} = [ \gamma']_{W^{\frac{p-q-1}{q}, q}(B)}.
	\]
\end{remark}

\subsection{A new energy $\E^{p,q}$} \label{s:newenergy}

Our first objective in this subsection is to construct a new energy $\E^{p,q}$, which coincides with the tangent-point energies $\tp^{p,q}$ for sufficiently regular curves $\gamma$, but only depends on the first derivative $\gamma^\prime$. We then show that any locally critical embedding of the tangent-point energies $\tp^{p,q}$ parametrized by arclength produces a locally critical $\S^2$-valued map of the new energy  $\E^{p,q}$. 

For this purpose, we recall that the tangent-point energies are for any $\gamma \in C^{0,1}(\R/\Z, \R^3)$ given by
\[
\tp^{p,q}(\gamma) = \int_{\R /\Z} \int_{\R /\Z} \frac{\left|\gamma'(x) \wedge (\gamma(x)-\gamma(y)) \right|^q}{|\gamma(x)-\gamma(y)|^p} |\gamma'(x)|^{1-q} |\gamma'(y)| \, dy\, dx. 
\]
Now we transform the wedge product in the numerator by Lagrange's identity and the fundamental theorem of calculus to
\begin{align*}
& \left|\gamma'(x) \wedge (\gamma(y)-\gamma(x)) \right|^2 \\
& = \left|\gamma'(x) \wedge (\gamma(x)-\gamma(y)- \gamma'(x)(y-x)) \right|^2\\ 
&= |\gamma'(x)|^2|\gamma(y)-\gamma(x)-\gamma'(x)(y-x)|^2 - (\gamma'(x) \cdot (\gamma(y) - \gamma(x)-\gamma'(x)(y-x)))^2 \\
& = |y-x|^2 \left( |\gamma'(x)|^2 |\avint_{x\triangleright y} \gamma'(z) \, dz - \gamma'(x)|^2 - |\, |\gamma'(x)|^2-\avint_{x\triangleright y} \gamma'(x) \cdot \gamma'(z) \, dz|^2 \right)\\
& = |y-x|^2 \bigg( |\gamma'(x)|^2 |\avint_{x\triangleright y} \gamma'(z)-\gamma'(x) \, dz |^2 - \tfrac 14 | \avint_{x\triangleright y} |\gamma'(x) - \gamma'(z)|^2 \, dz \\
& \qquad + |\gamma'(x)|^2 - \avint_{x\triangleright y} |\gamma'(z)|^2 \, dz|^2 \bigg).
\end{align*}
Additionally, observe that 
\begin{align}\label{eq:reg:rewritefund}
\begin{split}
\frac{|\gamma(y)-\gamma(x)|^2}{|y-x|^2} & = \avint_{x\triangleright y} \avint_{x\triangleright y} \gamma'(s)\cdot \gamma'(t) \, ds \, dt \\
& =  \avint_{x\triangleright y} |\gamma'(z)|^2 \, dz- \frac 12 \avint_{x\triangleright y} \avint_{x\triangleright y} |\gamma'(s)-\gamma'(t)|^2 \, ds \, dt .
\end{split}
\end{align}
Therefore, we can rewrite  $\tp^{p,q}(\gamma)$ in terms of the first derivative $\gamma'$ as 
{\small
	\begin{align*}
	& \tp^{p,q}(\gamma) \\
	& = \int_{\R /\Z} \int_{\R /\Z} \frac{( |\gamma'(x)|^2 |\avint_{x\triangleright y} \gamma'(z)-\gamma'(x) \, dz |^2 - \tfrac 14 | \avint_{x\triangleright y} |\gamma'(x) - \gamma'(z)|^2 \, dz +|\gamma'(x)|^2 - \avint_{x\triangleright y} |\gamma'(z)|^2 \, dz|^2)^{\frac q2}}{|y-x|^{p-q}} \\
	& \hspace{5em}\cdot \left(\avint_{x\triangleright y} |\gamma'(z)|^2 \, dz- \frac 12 \avint_{x\triangleright y} \avint_{x\triangleright y} |\gamma'(s)-\gamma'(t)|^2 \, ds \, dt  \right)^{-\frac p2} |\gamma'(x)|^{1-q} |\gamma'(y)|  \, dy\, dx.
	\end{align*} }

This motivates to introduce the following real-valued energy $\E^{p,q}$ for any maps $u:\R / \Z \rightarrow \R^3$ 
\begin{align*}
\E^{p,q}(u) & := \int_{\R /\Z} \int_{\R /\Z} \bigg(|u(x)|^2 |\avint_{x\triangleright y} u(z) -u(x) \, dz |^2 \\
& \hspace{3em}- \tfrac 14 \left(\avint_{x\triangleright y} |u(x) - u(z)|^2 \, dz + |u(x)|^2 - \avint_{x\triangleright y} |u(z)|^2 \, dz \right)^2 \bigg)^{\frac q2} \\
& \cdot \left(\avint_{x\triangleright y} |u(z)|^2 \, dz-\tfrac 12 \avint_{x\triangleright y} \avint_{x\triangleright y} |u(s)-u(t)|^2 \, ds \, dt \right)^{-\frac p2} \\
&  \cdot  |u(x)|^{1-q} |u(y)|  \frac{\, dy\, dx}{\rho(x,y)^{p-q}}.
\end{align*}

For $\eta \in C^\infty(\R,[0,\infty))$, we set moreover
\begin{align*}
\E^{p,q}_\eta(u) & := \int_{\R /\Z} \int_{\R /\Z} \bigg(|u(x)-\eta(x)(u)_{\RZ}|^2 |\avint_{x\triangleright y} u(z) -u(x) \, dz |^2 \\
& - \tfrac 14 \left(\avint_{x\triangleright y} |u(x) - u(z)|^2 \, dz + |u(x)-\eta(x)(u)_{\RZ}|^2 - \avint_{x\triangleright y} |u(z)-\eta(z)(u)_{\RZ}|^2 \, dz \right)^2 \bigg)^{\frac q2} \\
& \cdot \left(\avint_{x\triangleright y} |u(z)-\eta(z)(u)_{\RZ}|^2 \, dz-\tfrac 12 \avint_{x\triangleright y} \avint_{x\triangleright y} |u(s)-u(t)|^2 \, ds \, dt \right)^{-\frac p2} \\
&  \cdot  |u(x)-\eta(x)(u)_{\RZ}|^{1-q} |u(y)-\eta(y)(u)_{\RZ}|  \frac{\, dy\, dx}{\rho(x,y)^{p-q}}.
\end{align*}

Its localized version for any $D\subset \R/\Z\times\R/\Z$ is denoted by
\begin{align*}
\E_{\eta,D}^{p,q}(u)& := \iint_{D} \bigg(|u(x)-\eta(x)(u)_{\RZ}|^2 |\avint_{x\triangleright y} u(z) -u(x) \, dz |^2 \\
& - \tfrac 14 \left(\avint_{x\triangleright y} |u(x) - u(z)|^2 \, dz + |u(x)-\eta(x)(u)_{\RZ}|^2 - \avint_{x\triangleright y} |u(z)-\eta(z)(u)_{\RZ}|^2 \, dz \right)^2 \bigg)^{\frac q2} \\
& \cdot \left(\avint_{x\triangleright y} |u(z)-\eta(z)(u)_{\RZ}|^2 \, dz-\tfrac 12 \avint_{x\triangleright y} \avint_{x\triangleright y} |u(s)-u(t)|^2 \, ds \, dt \right)^{-\frac p2} \\
&  \cdot  |u(x)-\eta(x)(u)_{\RZ}|^{1-q} |u(y)-\eta(y)(u)_{\RZ}|  \frac{\, dy \, dx}{\rho(x,y)^{p-q}}.
\end{align*}
We observe that the energies $\tp^{p,q}$ and $\E^{p,q}$ coincide for our embeddings of interest accordingly, recalling that $(\gamma')_{\R/\Z}=0$.

\begin{lemma}\label{la:reg:tpqisepq}
	Let $p\in [q+2,2q+1)$ and $q>1$. For any embedding $\gamma:\R/\Z\rightarrow \R^3$ with finite tangent-point energies $\tp^{p,q}$ and constant speed parametrization as well as $\eta \in C^\infty(\R,[0,\infty))$, we have
	\[
	\tp^{p,q}(\gamma) = \E^{p,q}_\eta(\gamma^\prime) = \E^{p,q}(\gamma^\prime),
	\]
	and, in particular, for any subset $D\subset \R/\Z\times \R/\Z$,
	\[
	\iint_{D} \frac{\left|\gamma'(x) \wedge (\gamma(x)-\gamma(y)) \right|^q}{|\gamma(x)-\gamma(y)|^p} |\gamma'(x)|^{1-q} |\gamma'(y)| \, dy\, dx = \E^{p,q}_{\eta,D}(\gamma^\prime).
	\]
\end{lemma}

It remains to show that locally critical embeddings of the tangent-point energies $\tp^{p,q}$, cf. \Cref{def:wcp}, indeed induce locally critical maps into the sphere $\S^2$ of $\E^{p,q}$. 

The main result of this section is the analogue of \cite[Theorem 2.1]{BRS19}: (Locally) Critical knots $\gamma: \R/\Z \to \R^3$ induce a (locally) $\mathcal{E}^{p,q}_\eta$-critical maps $u: \R/\Z \to \S^2$ by setting $u := \frac{\gamma'}{|\gamma'|}$. 
\begin{theorem}\label{reg:th:critmap}
	Let $p\in [q+2,2q+1)$, $q>1$,  and $\gamma:\R /\Z \rightarrow \R^3$ be a homeomorphism with locally small tangent-point energy $\tp^{p,q}$ around the open interval $B(x_0,r)$, in the sense of \Cref{def:homeotp}, and assume $|\gamma'| \equiv const$. Denote the unit tangent field of $\gamma$ by  $u:=\frac{\gamma'}{|\gamma'|}:\RZ\rightarrow \S^2$. \\
	If $\gamma$ is a locally critical embedding of $\tp^{p,q}$ in $B(x_0,r )$, in the sense of \Cref{def:wcp}, then there exists some $\eta \in C_c^\infty(B(x_0,r),[0,\infty))$ such that the map $u: \R/\Z \to \S^2$ is a critical map of $\E^{p,q}_\eta$ in $B(x_0,r)$ in the class of maps $v: \R/\Z \rightarrow \S^2$.
	
	Namely, we have for any $\varphi \in C^\infty_c(B(x_0,r),\R^3)$, if we set $u_\eps = \frac{u+\eps\varphi}{|u+\eps \varphi|}$,
	\begin{align*}
	\frac{d}{d\eps} \Big |_{\eps = 0}\E_{\eta}^{p,q} (u_\eps) = \frac{d}{d\eps} \Big |_{\eps = 0}\E^{p,q} (u_\eps-\eta (u_\eps)_{\R/\Z}) = 0.
	\end{align*}
\end{theorem}
\begin{proof}
We argue similar to the proof of \cite[Theorem 2.1]{BRS19}, but additional problems appear since we only have the criticality in a ball, not globally (which is what makes us to introduce $\eta \in C_c^\infty(B(x_0,r))$, while in \cite[Theorem 2.1]{BRS19} we can choose $\eta \equiv 1$).

For simplicity we assume that $x_0 \in (0,1)$ and that $r \ll \min\{x_0,1-x_0\}$ (namely, we can always assume that $x_0=\frac{1}{2}$ by the periodicity of the problem).

First define for $\varphi\in C^\infty_c (B(x_0,r),\R^3)$ the maps
\[
u_\eps := \frac{u+\eps \varphi}{|u+\eps \varphi|}.
\]
Pick $\theta \in C^\infty(\R)$ with $\theta \equiv 0$ for $x < x_0-r/2$ and $\theta \equiv 1$ for $x \geq x_0+r/2$, and with $|\theta'| \aleq \frac{1}{r}$. We can also assume that $\theta' \geq 0$. Below we will chose $\eta := \theta'$.

Similar to \eqref{eq:testfunction},
	\[
 \gamma_\eps(x) := \gamma(0)+\int_{0}^x u_\eps(z) dz +  \theta(x) \brac{\int_{x_0-r}^{x_0+r} \gamma'(z)-u_\eps(z)\, dz} .
\]
We observe that $\gamma_\eps(x) = \gamma(x)$ for $x < x_0-r$ and $x > x_0+r$. Indeed, for $x < x_0-r$ we have 
\[
 \gamma_\eps(x) = \gamma(0) + \int_0^{x} \gamma'(z) dz = \gamma(x),
\]
and for $x > x_0+r$ we have 
\[
\begin{split}
 \gamma_\eps(x) =& \gamma(0)+\int_{0}^{x_0-r} \gamma'(z) dz + \int_{x_0-r}^{x_0+r} u_\eps(z) dz + \int_{x_0+r}^{x} \gamma'(z) dz 
 +  \theta(x) \brac{\int_{x_0-r}^{x_0+r} \gamma'(z)-u_\eps(z)\, dz} \\
 =& \gamma(0)+\int_{0}^{x_0-r} \gamma'(z) dz + \int_{x_0-r}^{x_0+r} u_\eps(z) dz + \int_{x_0+r}^{x} \gamma'(z) dz 
 +  1\brac{\int_{x_0-r}^{x_0+r} \gamma'(z)-u_\eps(z)\, dz} \\
 =& \gamma(0)+\int_{0}^{x_0-r} \gamma'(z) dz + \int_{x_0+r}^{x} \gamma'(z) dz 
 +  \int_{x_0-r}^{x_0+r} \gamma'(z)\, dz \\
 =&\gamma(0)+\int_{0}^{x} \gamma'(z) dz \\
 =&\gamma(x).
 \end{split}
\]
Moreover, we find for almost every $x$ that
\[
 u_\eps(x) = \gamma'(x) + \eps \brac{\varphi(x)-\langle\varphi(x),  \gamma'(x)\rangle \gamma'(x)} + O(\eps^2),
\]
and, using again the support of $\varphi$,
\[
\begin{split}
 \gamma_\eps(x) =& \  \gamma(x)\\
 &+\eps \brac{\int_{0}^x \brac{\varphi(z)-\langle \varphi(z),\gamma'(z)\rangle \gamma'(z)} dz -  \theta(x) \brac{\int_{0}^{1} \brac{\varphi(z)-\langle \varphi(z),\gamma'(z)\rangle \gamma'(z)}\, dz}} \\
 &+O(\eps^2)
 \end{split}
\]
as well as 
\[
\begin{split}
 \gamma_\eps'(x) =& \gamma'(x)\\
 &+\eps \brac{\varphi(z)-\langle \gamma'(x),\varphi(x)\rangle \gamma'(x) -  \theta'(x) \brac{\int_{0}^{1} \brac{\varphi(z)-\langle \varphi(z),\gamma'(z)\rangle \gamma'(z)}\, dz}} \\
 &+O(\eps^2).
 \end{split}
\]
By \Cref{la:reg:tpqisepq}, we have for small $\eps>0$
\[
 \tp^{p,q} (\gamma_\eps)=\E^{p,q} (\gamma_\eps').
\]
Since by assumption $\gamma$ is critical in $B(x_0,r)$ and $\gamma_\eps$ is a permissible variation we obtain
\[
 0 = \frac{d}{d\eps} \Big |_{\eps = 0}\E^{p,q} (\gamma_\eps') = \delta \E^{p,q} (u)[\psi],
\]
where we observe that 
\[
 \psi := \varphi(x)-\langle \gamma'(x),\varphi(x)\rangle \gamma'(x)-  \theta'(x) \brac{\int_{0}^{1} \brac{\varphi(z)-\langle \varphi(z),\gamma'(z)\rangle \gamma'(z)}\, dz} 
\]
has support in $B(x_0,r)$. Setting $\eta := \theta'$, we conclude the proof.
\end{proof}
\begin{remark}
	The function $\eta$ appearing in the previous theorem might remind of a Lagrange multiplier. However, in our setting $\eta$ can be chosen more freely. The presented construction of $\eta$ in the proof above is only one out of many possibilities to define permissible functions $\eta$.
\end{remark}


\subsection{Euler-Lagrange equations of \texorpdfstring{$\E^{p,q}$}{E}} \label{s:EulerLagrangeEqu} 

In this section, we derive the Euler-Lagrange equations of $\E_\eta^{p,q}$ for $p\in [q+2,2q+1)$, $q>1$, and suitable $\eta \in C^\infty(\R,[0,\infty))$. We realize that the new energies $\E_\eta^{q,q}$ have a nonlinear and nonlocal Lagrangian. Furthermore, we obtain a decomposition of the Lagrangian into a term of highest order, denoted by $Q$, and terms of lower order, denoted by $R$.

The leading order operator $Q$ on a subset $D \subset \R/\Z \times \R/\Z$ for $u: \R/\Z \to \R^3$ and $\varphi: \R/\Z \to \R^3$ is given by
\begin{align}\label{def:reg:QB}
\begin{split}
& Q^{(p,q)}_{D}(u,\varphi ) := \\
&  q \iint_{D} |\avint_{x\triangleright y} u(z) -u(x) \, dz |^{q-2}  \left(1-\frac 12 \avint_{x\triangleright y} \avint_{x\triangleright y} |u(s)-u(t)|^2 \, ds \, dt \right)^{-\frac{p}{2}} \\
& \qquad \cdot  \avint_{x\triangleright y}\avint_{x\triangleright y} (u(z_1)-u(x))\cdot (\varphi(z_2)-\varphi(x) ) \, dz_1 \, d z_2  \frac{ \, dy\, dx}{\rho(x,y)^{p-q}}\\
=&\frac{q}{2} \iint_{D} a^{q-2}  \left(1-\frac 12 c \right)^{-\frac{p}{2}}  a'(\varphi) \frac{ \, dy\, dx}{\rho(x,y)^{p-q}}.\\
\end{split}
\end{align}
(For the definition of $a,a',c$, etc., see below).

The remainders, which are, as we shall see, ``of lower order'', are given as follows
\begin{align*}
& R^{1,(p,q)}_{D}(u,\varphi) := \tfrac q2 \iint_{D}  \left( (a - \tfrac 14 b^2)^{\frac{q-2}{2}} - a^{\frac{q-2}{2}} \right) \left(1-\tfrac 12 c \right)^{-\frac{p}{2}} a'(\varphi) \frac{dy\, dx}{\rho(x,y)^{p-q}},  \\
& R^{2,(p,q)}_{D}(u,\varphi) :=  - \tfrac q4 \iint_{D}  (a - \tfrac 14 b^2)^{\frac{q-2}{2}} \left(1-\tfrac 12 c \right)^{-\frac{p}{2}} b \, b'(\varphi)  \frac{dy\, dx}{\rho(x,y)^{p-q}},  \\
& R^{3,(p,q)}_{D}(u,\varphi) :=  \tfrac{p}{4} \iint_{D}  (a - \tfrac 14 b^2)^{\frac{q}{2}} \left(1-\tfrac 12 c \right)^{-\frac{p+2}{2}}  c'(\varphi) \frac{dy\, dx}{\rho(x,y)^{p-q}}, \\
& R^{4,(p,q)}_{\eta,D}(u,\varphi) :=  - \tfrac{p}{2} \iint_{D} (a - \tfrac 14 b^2)^{\frac{q}{2}} \left(1-\tfrac 12 c \right)^{-\frac{p+2}{2}}  d'(\varphi) \frac{dy\, dx}{\rho(x,y)^{p-q}}, \\
& R^{5,(p,q)}_{\eta,D}(u,\varphi) :=  q \iint_{D}  (a - \tfrac 14 b^2)^{\frac{q-2}{2}}\left(1-\tfrac 12 c \right)^{-\frac{p}{2}} (a-\tfrac 12 b) \, e'(\varphi) \frac{dy\, dx}{\rho(x,y)^{p-q}} , \\
& R^{6,(p,q)}_{\eta,D}(u,\varphi) := \tfrac q4 \iint_{D} (a - \tfrac 14 b^2)^{\frac{q-2}{2}} \left(1-\tfrac 12 c \right)^{-\frac{p}{2}} b \, d'(\varphi) \frac{dy\, dx}{\rho(x,y)^{p-q}} , \\
& R^{7,(p,q)}_{\eta,D}(u,\varphi) :=  \iint_{D} (a - \tfrac 14 b^2)^{\frac{q}{2}} \left(1-\tfrac 12 c \right)^{-\frac{p}{2}} ((1-q)e'(\varphi) + f'(\varphi)) \frac{dy\, dx}{\rho(x,y)^{p-q}}.
\end{align*}
Here
\begin{align*}
a & := \left|\avint_{x\triangleright y} u(z) -u(x)\, dz \right|^2, \quad a'(\varphi)  = 2 \avint_{x\triangleright y} \avint_{x\triangleright y}  (u(z_1) - u(x)) \cdot(\varphi(z_2) - \varphi(x)) \, dz_1 \, dz_2,\\
b & :=  \avint_{x\triangleright y} |u(z)-u(x)|^2 \, dz, \quad  b'(\varphi)  =  2 \avint_{x\triangleright y} (u(z)-u(x)) \cdot (\varphi(z)-\varphi(x)) \, dz, \\
c & := \avint_{x\triangleright y} \avint_{x\triangleright y} |u(s)-u(t)|^2 \, ds \, dt, \quad  c'(\varphi)  = 2 \avint_{x\triangleright y} \avint_{x\triangleright y} (u(s)-u(t)) \cdot (\varphi(s)-\varphi(t)) \, ds \, dt,
\end{align*}
and 
\begin{align*}
d'(\varphi) &= -2 \avint_{x\triangleright y} \eta(z)u(z) \cdot (\varphi)_{\RZ}\, dz,\\
e'(\varphi) & = -\eta(x)u(x) \cdot (\varphi)_{\RZ}, \\
f'(\varphi) &= -\eta(y)u(y) \cdot (\varphi)_{\RZ}.
\end{align*}

Note that in case of considering the entire domain $D= \R/\Z \times \R/\Z$, we drop the label $D$ in the definition of $Q^{(p,q)}_{D}(u,\varphi )$ and the remainders $R^{k,(p,q)}_{ D}(u,\varphi)$ and $R^{k,(p,q)}_{\eta, D}(u,\varphi)$.

\begin{lemma}(Euler-Lagrange equations)\label{la:EulerLagrange}
	Let $p\in[q+2,2q+1)$, $q>1$, and $u:\RZ\rightarrow \S^2$ with  $\int_{\RZ} u = 0$ be a locally critical map of $\E_\eta^{p,q}$ around the interval $B(x_0,r)$ in the class of maps $v: \RZ \rightarrow \S^2$ and let $\eta \in C_c^\infty(B(x_0,r),[0,\infty))$. Then for any test function $\varphi\in W^{\frac{p-q-1}{q}, q}_{0}(B(x_0,r),\R^3)$, which is also tangential, i.e. $\varphi \in T_u\S^2$, if we set $u_\eps = u+\eps\varphi$, it holds 
	\[
	\frac{d}{d\eps} \Big |_{\eps = 0}\E_\eta^{p,q} (u_\eps)  = Q^{(p,q)}(u,\varphi) + \sum_{k=1}^3 R^{k,(p,q)}(u,\varphi) + \sum_{k=4}^7 R_{\eta}^{k,(p,q)}(u,\varphi) =0.
	\] 
\end{lemma}
\begin{proof}
	Let us recall the definition of $\E_\eta^{p,q}$ first by
	\begin{align*}
	\E_\eta^{p,q}(u) & = \int_{\R/\Z} \int_{\R /\Z} \bigg(|u(x)-\eta(x)(u)_{\RZ}|^2 |\avint_{x\triangleright y} u(z) -u(x) \, dz |^2 \\
	& - \frac 14 \left(\avint_{x\triangleright y} |u(x) - u(z)|^2 \, dz + |u(x)-\eta(x)(u)_{\RZ}|^2 - \avint_{x\triangleright y} |u(z)-\eta(z)(u)_{\RZ}|^2 \, dz \right)^2 \bigg)^{\frac q2} \\
	& \cdot \left(\avint_{x\triangleright y} |u(z)-\eta(z)(u)_{\RZ}|^2 \, dz-\frac 12 \avint_{x\triangleright y} \avint_{x\triangleright y} |u(s)-u(t)|^2 \, ds \, dt \right)^{-\frac p2} \\
	&  \cdot  |u(x)-\eta(x)(u)_{\RZ}|^{1-q} |u(y)-\eta(y)(u)_{\RZ}|  \frac{\, dy\, dx}{\rho(x,y)^{p-q}}.
	\end{align*}
	We set 
	\[
	F(a,b,c,d,e) := (e^2 a - \tfrac 14 (b + e^2 - d)^2)^{\frac q2} \left(d-\tfrac 12 c \right)^{-\frac{p}{2}},
	\]
	which implies
	\begin{align*}
	\E^{p,q}(u) = \int_{\R/\Z} \int_{\R /\Z} F(a(0),b(0),c(0),d(0),e(0)) \ e(0)^{1-q} \, f(0) \frac{\, dy\, dx}{\rho(x,y)^{p-q}},
	\end{align*}
	where
	\begin{align*}
	a(\eps) & := \left|\avint_{x\triangleright y} u_\eps(z) -u_\eps(x)\, dz \right|^2, \\
	b(\eps) & :=  \avint_{x\triangleright y} |u_\eps(z)-u_\eps(x)|^2 \, dz, \\
	c(\eps) & := \avint_{x\triangleright y} \avint_{x\triangleright y} |u_\eps(s)-u_\eps(t)|^2 \, ds \, dt, \\
	d(\eps) & := \avint_{x\triangleright y} |u_\eps(z)- \eta(z)(u_\eps)_{\RZ}|^2 \, dz ,\\
	e(\eps) & := |u_\eps (x) - \eta(x) (u_\eps)_{\RZ}|  ,\\
	f(\eps) & := |u_\eps (y) - \eta(y) (u_\eps)_{\RZ}|.
	\end{align*}
	First we note that $d(0)=e(0)=f(0)=1$ since $|u|\equiv 1$ and $(u)_{\RZ}=0$.
	Furthermore, observe that
	\begin{align*}
	a'(0) & \equiv a'(0) (u,\varphi) = 2 \avint_{x\triangleright y} \avint_{x\triangleright y}  (u(z_1) - u(x)) \cdot(\varphi(z_2) - \varphi(x)) \, dz_1 \, dz_2, \\
	b'(0) & \equiv b'(0) (u,\varphi) =  2 \avint_{x\triangleright y} (u(z)-u(x)) \cdot (\varphi(z)-\varphi(x)) \, dz,\\
	c'(0) & \equiv  c'(0) (u,\varphi) = 2 \avint_{x\triangleright y} \avint_{x\triangleright y} (u(s)-u(t)) \cdot (\varphi(s)-\varphi(t)) \, ds \, dt,
	\end{align*}
	and due to $u \cdot \varphi \equiv 0$ and $(u)_{\RZ}=0$
	\begin{align*}
	& d'(0) \equiv d'(0) (u,\varphi, \eta) = -2 \avint_{x\triangleright y} \eta (z)u(z) \cdot (\varphi)_{\RZ}\, dz, \\
	&  e'(0)  \equiv e'(0)(u,\varphi, \eta) = -\eta (x)u(x) \cdot (\varphi)_{\RZ}, \\
	&  f'(0) \equiv f'(0) (u,\varphi, \eta) = -\eta (y)u(y) \cdot (\varphi)_{\RZ}.
	\end{align*}
	
	Hence, we obtain  by the product rule 
	\begin{align*}
	& \frac{d}{d\eps}\bigg|_{\eps=0} \int_{\R/\Z} \int_{\R/\Z} F(a(\eps),b(\eps),c(\eps),d(\eps),e(\eps) ) \ d(\eps)^{1-q} \, e(\eps) \frac{dy\, dx}{\rho(x,y)^{p-q}}  \\
	& = \tfrac q2 \int_{\R/\Z} \int_{\R/\Z}   a(0)^{\frac{q-2}{2}} \left(1-\tfrac 12 c(0) \right)^{-\frac{p}{2}} a'(0) \frac{dy\, dx}{\rho(x,y)^{p-q}}  \\
	& \quad + \tfrac q2 \int_{\R/\Z} \int_{\R/\Z}  \left( (a(0) - \tfrac 14 b(0)^2)^{\frac{q-2}{2}} - a(0)^{\frac{q-2}{2}} \right) \left(1-\tfrac 12 c(0) \right)^{-\frac{p}{2}} a'(0) \frac{dy\, dx}{\rho(x,y)^{p-q}}  \\
	& \quad - \tfrac q4 \int_{\R/\Z} \int_{\R/\Z}  (a(0) - \tfrac 14 b(0)^2)^{\frac{q-2}{2}} \left(1-\tfrac 12 c(0) \right)^{-\frac{p}{2}} b(0) \, b'(0)  \frac{dy\, dx}{\rho(x,y)^{p-q}}  \\
	& \quad + \tfrac{p}{4} \int_{\R/\Z} \int_{\R/\Z}  (a(0) - \tfrac 14 b(0)^2)^{\frac{q}{2}} \left(1-\tfrac 12 c(0) \right)^{-\frac{p+2}{2}}  c'(0) \frac{dy\, dx}{\rho(x,y)^{p-q}} \\
	& \quad - \tfrac{p}{2} \int_{\R/\Z} \int_{\R/\Z}  (a(0) - \tfrac 14 b(0)^2)^{\frac{q}{2}} \left(1-\tfrac 12 c(0) \right)^{-\frac{p+2}{2}}  d'(0) \frac{dy\, dx}{\rho(x,y)^{p-q}} \\
	& \quad + q \int_{\R/\Z} \int_{\R/\Z}   (a(0) - \tfrac 14 b(0)^2)^{\frac{q-2}{2}}\left(1-\tfrac 12 c(0) \right)^{-\frac{p}{2}} (a(0)-\tfrac 12 b(0)) \, e'(0) \frac{dy\, dx}{\rho(x,y)^{p-q}}  \\
	& \quad + \tfrac q4 \int_{\R/\Z} \int_{\R/\Z}  (a(0) - \tfrac 14 b(0)^2)^{\frac{q-2}{2}} \left(1-\tfrac 12 c(0) \right)^{-\frac{p}{2}} b(0) \, d'(0) \frac{dy\, dx}{\rho(x,y)^{p-q}}  \\
	& \quad + \int_{\R/\Z} \int_{\R/\Z}  (a(0) - \tfrac 14 b(0)^2)^{\frac{q}{2}} \left(1-\tfrac 12 c(0) \right)^{-\frac{p}{2}} ((1-q)e'(0) + f'(0)) \frac{dy\, dx}{\rho(x,y)^{p-q}}    .
	\end{align*}
\end{proof}


\begin{remark}\label{rm:reg:boundedfactorc}
	For a given homeomorphism $\gamma:\R /\Z \rightarrow \R^3$ with locally small tangent-point energy $\tp^{p,q}$ and its unit tangent field $u:\RZ\rightarrow \S^2$, we introduce the following abbreviation
	\[
	k(x,y) := 1-\tfrac 12 \avint_{x\triangleright y} \avint_{x\triangleright y} |u(s)-u(t)|^2 \, ds \, dt = 1-\tfrac 12 c(0).
	\]
	Note that in all terms of the Euler-Lagrange equation either $k(x,y)^{-\frac p2}$ or $k(x,y)^{-\frac{p+2}{2}}$ appears as a factor. The motivation behind this definition is the observation that $k(x,y)^{-r}$ for any $r>0$ is bounded: On the one hand it is easy to see that 
	\[
	1-\tfrac 12 \avint_{x\triangleright y} \avint_{x\triangleright y} |u(s)-u(t)|^2 \, ds \, dt \leq 1.
	\]
	On the other hand, there exists a constant $c>0$ such that 
	\[
	0 < c \leq  1-\tfrac 12 \avint_{x\triangleright y} \avint_{x\triangleright y} |u(s)-u(t)|^2 \, ds \, dt.
	\]
	The latter can be shown by recalling that  $u$ denotes the unit tangent field of $\gamma$ and by applying the fundamental theorem of calculus as in \eqref{eq:reg:rewritefund} as well as the global bi-Lipschitz continuity of $\gamma$ due to \Cref{th:weaklimitareweakimm}, from which we conclude that
	\[
	\begin{split}
	& 1 - \tfrac 12 \avint_{x\triangleright y} \avint_{x\triangleright y} |\gamma'(s)-\gamma'(t)|^2 \, ds \, dt  = \frac{|\gamma(y)-\gamma(x)|^2}{|y-x|^2} \geq (1-\eps)^2 > 0
	\end{split}
	\]
	for any $\eps>0$ small and $x\neq y$. 
\end{remark}

As next steps we are going to show that $Q$ is indeed the leading order operator and the remainder $R$ are of ``lower order''.
Namely, in \Cref{pr:Qequiv} we essentially show that $Q(u,\varphi)$ controls the Sobolev-norm $[u]_{W^{\frac{p-q-1}{q},q}(B(\rho))}^{q-1}$ for a good choice of $\varphi \in C_c^\infty(B_\rho)$, in particular, whenever $B_\rho$ is a ball compactly contained in $B(x_0,r)$ and $[\varphi]_{W^{\frac{p-q-1}{q},q}(\R)} \leq 1$.
Then, \Cref{pr:reg:estremainderwedge} shows that each of the remainder terms $R$ essentially satisfy the following estimate
\begin{equation}\label{eq:remaindergoal}
	|R^{k,(p,q)} (u,(\varphi u\wedge)_{ij})|\aleq   [u]_{W^{\tilde{q},q}(B_{2\rho})}^{\tilde{q}}  +   \sum_{l=1}^\infty 2^{-\sigma l}[u]_{W^{\frac{p-q-1}{q},q}(B_{2^{l+1}\rho})}^{q-1} +  C(r,u) \rho^\sigma
\end{equation}
for some $\tilde{q} > q-1$ and some $\sigma > 0$. Such terms on the right-hand side can absorbed by an iteration argument, as discussed in the proof of \Cref{th:mainreg}.

Next we show that the leading order term  $Q^{p,q}$ controls the Sobolev norm.

\begin{proposition}\label{pr:Qequiv}
	Let $p\in [ q+2,2q+1)$, $q>1$, and $\gamma:\R /\Z \rightarrow \R^3$ be a homeomorphism with locally small tangent-point energy $\tp^{p,q}$ around the interval $B(x_0,r)$, in the sense of \Cref{def:homeotp}. Furthermore, denote the unit tangent field of $\gamma$ by  $u:\RZ\rightarrow \S^2$ such that $\int_{\RZ} u = 0$, let $y_0\in B(x_0,\tfrac r2)$, and choose $\rho>0$ such that $B_\rho := B(y_0,\rho) \subset B(x_0,r)$.
	Then we have
	\begin{align*}
	[u]_{W^{\frac{p-q-1}{q}, q}(B_\rho)}^q & \aleq \iint_{{B_{\rho}}^2} k(x,y)^{-\frac p2} 
	\frac{ |\avint_{x\triangleright y} u(z) -u(x)\, dz |^{q}}{\rho(x,y)^{p-q}} \, dy\, dx
	 \\
	& \approx Q^{(p,q)}_{B_\rho\times B_\rho}(u,u) 
	\end{align*}
	with constants only depending on $p$ and $q$.
\end{proposition}
\begin{proof}
	We begin with recalling the definition of the main term $Q_{B_\rho\times B_\rho}^{(p,q)}$ in \eqref{def:reg:QB} and test it with $u$, such that the expression simplifies to
	\begin{align*}
	& Q_{B_\rho\times B_\rho}^{(p,q)}(u,u)  \\
	& =  q \iint_{B_\rho^2}  |\avint_{x\triangleright y} u(z) -u(x) \, dz |^{q-2} \avint_{x\triangleright y}\avint_{x\triangleright y} (u(z_1)-u(x))\cdot (u(z_2)-u(x) ) \, dz_1 \, d z_2  \\
	& \hspace{5em} \cdot k\left(x,y\right)^{-\frac{p}{2}} \frac{ \, dy\, dx}{\rho(x,y)^{p-q}} \\
	& = q \iint_{B_\rho^2} k\left(x,y\right)^{-\frac{p}{2}} \frac{ |\avint_{x\triangleright y} u(z) -u(x) \, dz |^{q}}{\rho(x,y)^{p-q}} \, dy\, dx.
	\end{align*}
	Note that the factor $k\left(x,y \right)^{-\frac{p}{2}}$ is strictly positive and bounded by \Cref{rm:reg:boundedfactorc}.
	Furthermore, we have 
	\[
	|u(y)-u(x)|\leq |u(y)- \avint_{y\triangleright x} u(z) \, dz| + |\avint_{x\triangleright y} u(z) \, dz - u(x)|,
	\]
	which implies by the previous arguments 
	\begin{align*}
	& [u]_{W^{\frac{p-q-1}{q}, q}(B_\rho)}^q = \iint_{{B_\rho}^2} \frac{|u(y) - u(x)|^q }{\rho(x,y)^{p-q}} \, dy \, dx \\
	& \aleq \iint_{{B_\rho}^2} \frac{|\avint_{x\triangleright y} u(z) - u(x) \, dz|^q }{\rho(x,y)^{p-q}} \, dy \, dx + \iint_{{B_\rho}^2} \frac{|\avint_{y\triangleright x} u(z) - u(y) \, dz|^q }{\rho(x,y)^{p-q}}  \, dy \, dx \\
	& \aleq \iint_{{B_\rho}^2} k\left(x,y\right)^{-\frac{p}{2}} \frac{|\avint_{x\triangleright y} u(z) - u(x) \, dz|^q }{\rho(x,y)^{p-q}} \, dy \, dx \approx Q_{B_\rho\times B_\rho}^{(p,q)}(u,u) .
	\end{align*}
\end{proof}

It remains to obtain the ``lower order''-property for the the remainder terms $R$.
We recall that for any $v\in \R^3$ the linear map $v\wedge: \R^3 \to \R^3$ is represented by the $\R^{3\times 3}$-matrix 
\[
v\wedge = \begin{pmatrix}
0 & -v_3 & v_2 \\
v_3 & 0 & - v_1 \\
-v_2 & v_1 & 0
\end{pmatrix}.
\]

\begin{proposition}\label{pr:reg:estremainderwedge}
	Let $p\in [q+2,2q+1)$, $q\geq2$, and $\gamma:\R /\Z \rightarrow \R^3$ be a homeomorphism with locally small tangent-point energy $\tp^{p,q}$ around the interval $B(x_0,r)$, in the sense of \Cref{def:homeotp}. We denote the unit tangent field of $\gamma$ by  $u:\RZ\rightarrow \S^2$ such that $\int_{\RZ} u = 0$ and take $\eta \in C_c^\infty(B(x_0,\tfrac r2),[0,\infty))$. Furthermore, let $y_0\in B(x_0,\tfrac r2)$, choose $\rho>0$ such that $B(y_0,4\rho)\subset B(x_0,r)$, and define $B_\rho:=B(y_0,\rho)$. Let $\varphi \in C_c^\infty(B_\rho,\R)$ such that $[\varphi]_{W^{\frac{p-q-1}{q},q}(\R)} \leq 1$. Then the following holds for any $j=1,2,3$: 
	
	For the first remainder $k=1$ in case of $2 < q < 4$, we have
	\begin{align*}
	& |R^{1,(p,q)} (u,(\varphi u\wedge)_{ij})| \aleq 
	[u]_{W^{\frac{p-q-1}{q},q}(B_{2\rho})}^{2q-3}   +   \sum_{l=1}^\infty 2^{-(l+1)\frac{p-q}{q}}[\tilde{u}]_{W^{\frac{p-q-1}{q},q}(B_{2^{l+1}\rho})}^{q-1}  +  \rho \,  r^{-(p-q+1)}.
	\end{align*}
	For $q=2$ we have $R^{1,(p,q)} \equiv 0$. 
	For the remainders $k=1$ in case of $q\geq 4$ and $k=2,3$ for any $q\geq 2$, we have 
	\begin{align*}
	& |R^{k,(p,q)} (u,(\varphi u\wedge)_{ij})|\aleq   [u]_{W^{\frac{p-q-1}{q},q}(B_{2\rho})}^{q+1}  +   \sum_{l=1}^\infty 2^{-(l+1)\frac{p-q}{q}}[\tilde{u}]_{W^{\frac{p-q-1}{q},q}(B_{2^{l+1}\rho})}^{q-1} +  \rho \,  r^{-(p-q+1)},
	\end{align*}
	and for the remaining terms $k=4,5,6,7$ with $q\geq 2$, we have 
	\begin{align*}
	\sum_{k=4}^7 |R^{k,(p,q)}_\eta (u,(\varphi u\wedge)_{ij})| \aleq  \rho \left(\E^{p,q}(u) +  r^{-(p-q)} + r^{\frac{p-q-2}{2}}[u]_{W^{\frac{p-q-1}{q},q}(B(x_0,r))}^q \right),
	\end{align*}
	where $\tilde{u}$ denotes a $W^{\frac{p-q-1}{q},q}$-extension of $u|_{B(x_0,r)}$ from $B(x_0,r)$ to $\R$ as discussed in \Cref{reg:rm:extension}.
	The constants in these inequalities depend on the $p$ and $q$ and may also depend on global properties of $u$ such as $\|u\|_{L^\infty}$, $[u]_{W^{\frac{p-q-1}{q},q}(B(x_0,r))}$ and $[\tilde u]_{W^{\frac{p-q-1}{q},q}(\R)}$.
\end{proposition}
A very similar statement holds for $q \in (1,2)$, only the tail's exponents
change, but one still obtains an estimate as in \eqref{eq:remaindergoal}. We leave the details to the reader.
\begin{proof}[Proof of \Cref{pr:reg:estremainderwedge}]
	We begin with making some general observations on factors appearing in the integrands of the remainder terms. 
	First the remainders contain a factor of the form $k(x,y)^{-r}$, for some $r>0$, which is strictly positive and bounded by \Cref{rm:reg:boundedfactorc}. Next we consider the factors $(a-\tfrac 14 b^2)^{\frac{q-2}{2}}$ and $(a-\tfrac 14 b^2)^{\frac{q}{2}}$ appearing for the cases $k=2,\ldots,7$. Recall that  
	\[
	\begin{split}
	a = | \avint_{x\triangleright y} u(z)-u(x) \, dz|^2 \quad \text{and} \quad b =  \avint_{x\triangleright y} |u(z)-u(x)|^2 \, dz.
	\end{split}
	\]
	As we have by \eqref{eq:simplecomp}
	\begin{align}\label{eq:reg:a<=14b}
	\tfrac 12 \avint_{x\triangleright y} |u(z)-u(x)|^2 \, dz = |\avint_{x\triangleright y} u(x) \cdot ( u(z)-u(x)) \, dz| \leq | \avint_{x\triangleright y} u(z)-u(x) \, dz|,
	\end{align}
	it follows $ 0\leq  a- \tfrac 14 b^2$ and thus the factors can be simplified, when necessary, to
	\[
	\begin{split}
	(a-\tfrac 14 b^2)^{\frac{q-2}{2}} & \leq a^{\frac{q-2}{2}} =  \left| \avint_{x\triangleright y} u(z)-u(x) \, dz\right|^{q-2},  \\
	(a-\tfrac 14 b^2)^{\frac{q}{2}} & \leq a^{\frac{q}{2}} =  \left| \avint_{x\triangleright y} u(z)-u(x) \, dz\right|^q.
	\end{split}
	\]
	For the case $k=1$ we have to study the factor  $((a - \tfrac 14 b^2)^{\frac{q-2}{2}} - a^{\frac{q-2}{2}})$ instead. For $q=2$ this factor equals $0$, whereas for $2<q<4$ it can be estimated by 
	\[
	\begin{split}
	(a - \tfrac 14 b^2)^{\frac{q-2}{2}} - a^{\frac{q-2}{2}} & \leq  2^{2-q} b^{q-2} = 2^{2-q} \left(\avint_{x\triangleright y} |u(z)-u(x)|^2 \, dz\right)^{q-2},
	\end{split}
	\]
	since $x^r - y^r  \leq (x-y)^r$ for any $x\geq y \geq 0$ and $0\leq r<1$, and for $q\geq 4$ by 
	\[
	(a - \tfrac 14 b^2)^{\frac{q-2}{2}} - a^{\frac{q-2}{2}} \aleq b^2  a^{\frac{q-4}{2}} = \left(\avint_{x\triangleright y} |u(z)-u(x)|^2 \, dz \right)^2 \left| \avint_{x\triangleright y} u(z)-u(x) \, dz\right|^{q-4}
	\]
	since we have $|x^r - y^r| \leq c(r) |x-y| |x^{r-1} + y^{r-1}|$ for any $x \geq y \geq 0$ and $r\geq 0$.
	Last but not least we note that the test functions $(\varphi u\wedge)_{ij}$ are tangential, i.e. $(\varphi u\wedge)_{ij} \in T_u\S^2$, for any $j=1,2,3$ due to the fact that $u\wedge u = 0$.
	
	After these first considerations, we proceed with studying the full remainder terms.

	 We begin with the first remaining term $R^{1,(p,q)}_D$ for some general $D\subset \R/\Z\times\R/\Z$ in the case of $2\leq q<4$ and gain with help of the introductory comments on occurring factors for any $j=1,2,3$
	\[
	\begin{split}
	&|R^{1,(p,q)}_D(u,(\varphi u\wedge)_{ij})| \\
	& \aleq \iint_{D} \left|\avint_{x\triangleright y} |u(z)-u(x)|^2 \, dz \right|^{q-2} \\
	& \cdot\left|\left(\avint_{x\triangleright y} u(z_1)-u(x) \, dz_1 \right) \cdot \left(\avint_{x\triangleright y} (\varphi u\wedge)_{ij}(z_2)-(\varphi u\wedge)_{ij}(x) \, dz_2 \right)  \right| \frac{dx\, dy}{\rho(x,y)^{p-q}}.
	\end{split}
	\]
	If we consider the exemplary case $j=1$, the dot product reads as
	\[
	\begin{split}
	&\left(\avint_{x\triangleright y} u(z_1)-u(x) \, dz_1 \right) \cdot \left(\avint_{x\triangleright y} (\varphi u\wedge)_{i1}(z_2)-(\varphi u\wedge)_{i1}(x) \, dz_2 \right) \\ 
	& = \left(\avint_{x\triangleright y} u_2(z_1)-u_2(x) \, dz_1 \right) \left(\avint_{x\triangleright y} \varphi(z_2) \, u_3(z_2)-\varphi(x) \, u_3(x) \, dz_2 \right) \\
	& + \left(\avint_{x\triangleright y} u_3(z_1)-u_3(x) \, dz_1 \right) \left(\avint_{x\triangleright y} \varphi(z_2) \, (-u_2)(z_2)-\varphi (x)\,  (-u_2)(x) \, dz_2 \right).
	\end{split}
	\]
	By adding $0=\varphi(x)\, u_3(z_2) - \varphi(x)\, u_3(z_2) $ to second factor in the first summand, respectively, $0=\varphi(x)\, (-u_2)(z_2) - \varphi(x)\, (-u_2)(z_2) $ to the second factor in the second summand, the dot product turns to
	\[
	\begin{split}
	&\left(\avint_{x\triangleright y} u(z_1)-u(x) \, dz_1 \right) \cdot \left(\avint_{x\triangleright y} (\varphi u\wedge)_{i1}(z_2)-(\varphi u\wedge)_{i1}(x) \, dz_2 \right) \\ 
	& = \left(\avint_{x\triangleright y} u(z_1)-u(x) \, dz_1 \right) \cdot \left(\avint_{x\triangleright y} (\varphi(z_2) - \varphi(x))  \, (u\wedge)_{i1}(z_2) \, dz_2 \right) \\
	& + \varphi(x) \,  \left(\avint_{x\triangleright y} u(z_1)-u(x) \, dz_1 \right) \cdot \left(\avint_{x\triangleright y}   (u\wedge)_{i1}(z_2) - (u\wedge)_{i1}(x)  \, dz_2 \right).
	\end{split}
	\]
	But since $(u\wedge)_{i1}\in T_u\S^2$, the last summand in the previous equation vanishes.
	Now by the same arguments for $j=2,3$, we conclude 
	\begin{equation}\label{reg:est:remainder1global}
	\begin{split}
	&|R^{1,(p,q)}_D(u,(\varphi u\wedge)_{ij})| \\
	& \aleq \|u\|_{L^\infty} \iint_{D} \left(\avint_{x\triangleright y} |u(z)-u(x)|^2 \, dz \right)^{q-2} \left(\avint_{x\triangleright y} |u(z_1)-u(x)| \, dz_1 \right)  \\
	& \quad \cdot  \left( \avint_{x\triangleright y} |\varphi(z_2) - \varphi(x)|  \, dz_2 \right)  \frac{ dy\, dx}{\rho(x,y)^{p-q}} .
	\end{split}
	\end{equation}
	Using $\|u\|_{L^\infty}\leq 1$, simplifies this inequality.
	Now to take advantage of the local behavior of $u$ and $\varphi$, we split the integration domain into
	\[
	\begin{split}
		& |R^{1,(p,q)}_{\R/\Z\times \R/\Z}| \\
		&  \leq  |R^{1,(p,q)}_{B_{2\rho}\times B_{2\rho}}| +
	|R^{1,(p,q)}_{B_{2\rho}\times (\R/\Z\setminus B_{2\rho})}| + |R^{1,(p,q)}_{(\R/\Z\setminus B_{2\rho}) \times B_{2\rho}} |+ |R^{1,(p,q)}_{(\R/\Z\setminus B_{2\rho})\times (\R/\Z\setminus B_{2\rho})}|.
	\end{split}
	\]
	
	The first term can be estimated by \eqref{reg:est:remainder1global}, H\"older's inequality for $1= \tfrac{q-2}{q} + \tfrac 1q + \tfrac 1q$, Jensen's inequality, and identification \Cref{la:id2} such that for any $j=1,2,3$
	\[
	\begin{split}
	& |R^{1,(p,q)}_{B_{2\rho}\times B_{2\rho}}(u,(\varphi u\wedge)_{ij})|  \\
	& \aleq\iint_{B_{2\rho}^2} \left(\avint_{x\triangleright y} |u(z)-u(x)|^2 \, dz \right)^{q-2} \left(\avint_{x\triangleright y} |u(z_1)-u(x)| \, dz_1 \right)  \\
	& \cdot  \left( \avint_{x\triangleright y} |\varphi(z_2) - \varphi(x)|  \, dz_2 \right)  \frac{ dy\, dx}{\rho(x,y)^{p-q}} \\
	& \aleq  \left(\iint_{B_{2\rho}^2} \frac{\avint_{x\triangleright y} |u(z_0) - u(x)|^{2q} \, dz_0}{\rho(x,y)^{p-q}}\, dy \, dx\right)^{\frac{q-2}{q}}   \left(\iint_{B_{2\rho}^2}  \frac{\avint_{x\triangleright y} |u(z_1) - u(x)|^q \, dz_1}{\rho(x,y)^{p-q}} \, dy\, dx\right)^{\frac 1q} \\
	& \cdot  \left(\iint_{B_{2\rho}^2}  \frac{\avint_{x\triangleright y} |\varphi(z_2) - \varphi(x)|^q \, dz_2}{\rho(x,y)^{p-q}} \, dx\, dy\right)^{\frac 1q} \\
	& \aeq   [u]_{W^{\frac{p-q-1}{2q},2q}(B_{2\rho})}^{2q-4} [u]_{W^{\frac{p-q-1}{q},q}(B_{2\rho})}  [\varphi]_{W^{\frac{p-q-1}{q},q}(B_{2\rho})} \\
	& \aleq  [u]_{W^{\frac{p-q-1}{q},q}(B_{2\rho})}^{2q-3} , 
	\end{split}
	\]
	where we applied the Sobolev embedding \Cref{la:sob1} and the assumption on $\varphi$, i.e. $[\varphi]_{W^{\frac{p-q-1}{q},q}(\R)} \leq 1$ in the last inequality.
	For the second term of the splitting, we subdivide the integration domain $B_{2\rho} \times (\R/\Z\setminus B_{2\rho})$ into $B_{2\rho} \times (B(x_0,r)\setminus B_{2\rho})$, for using the local fractional Sobolev regularity of $u$ in $B(x_0,r)$, and the rest $B_{2\rho} \times (\R/\Z \setminus B(x_0,r))$. At this point recall that  $\tilde{u}$ denotes a $W^{\frac{p-q-1}{q},q}$-extension of $u|_{B(x_0,r)}$ from $B(x_0,r)$ to $\R$, which coincides with $u$ on $B(x_0,r)$ by construction, cf.~\Cref{reg:rm:extension}.
	We then get by inequality \eqref{reg:est:remainder1global} and the disjoint support estimate \Cref{la:RminusBest}
	\begin{equation}\label{reg:eq:tailestforR1}
	\begin{split}
	& |R^{1,(p,q)}_{B_{2\rho}\times (B(x_0,r)\setminus B_{2\rho})}(u,(\varphi u\wedge)_{ij})| \\
	& \aleq \int_{B_{2\rho}} \int_{B(x_0,r) \setminus B_{2\rho}} \left(\avint_{x\triangleright y} |\tilde u(z)- \tilde u(x)|^2 \, dz \right)^{q-2} \left(\avint_{x\triangleright y} |\tilde u(z_1)-\tilde u(x)| \, dz_1 \right)  \\
	& \cdot  \left( \avint_{x\triangleright y} |\varphi(z_2) - \varphi(x)|  \, dz_2 \right)  \frac{ dy\, dx}{\rho(x,y)^{p-q}} \\
	& \aleq  \sum_{l=1}^\infty 2^{- (l+1)\frac{p-q}{q}}[\tilde u]_{W^{\frac{p-q-1}{q},q}(B_{2^l\rho})}^{2q-3}  [\varphi]_{W^{\frac{p-q-1}{q},q}(B_\rho)}	\\
	& \aleq [\tilde u]_{W^{\frac{p-q-1}{q},q}(\R)}^{q-2} \sum_{l=1}^\infty 2^{- (l+1)\frac{p-q}{q}}[\tilde u]_{W^{\frac{p-q-1}{q},q}(B_{2^l\rho})}^{q-1}  [\varphi]_{W^{\frac{p-q-1}{q},q}(\R)} \\
	& \aleq \sum_{l=1}^\infty 2^{- (l+1)\frac{p-q}{q}}[\tilde u]_{W^{\frac{p-q-1}{q},q}(B_{2^l\rho})}^{q-1}  ,
	\end{split}
	\end{equation}
	where the constant depends on $[\tilde u]_{W^{\frac{p-q-1}{q},q}(\R)} < \infty$ as well as  $[\varphi]_{W^{\frac{p-q-1}{q},q}(\R)} \leq 1$.
	By the fact that $\rho(x,y)\geq \tfrac r4$ for $x\in B_{2\rho}$ and $y\in \R/\Z \setminus B(x_0,r)$, we estimate the rest term by \eqref{reg:est:remainder1global} such that
	\begin{equation}\label{reg:eq:roughestimateR1tail}
	\begin{split}
		& |R^{1,(p,q)}_{B_{2\rho}\times (\R/\Z \setminus B(x_0,r))}(u,(\varphi u\wedge)_{ij})| \\
		& \aleq \int_{B_{2\rho}} \int_{\R/\Z \setminus B(x_0,r)} \left(\avint_{x\triangleright y} |u(z)-u(x)|^2 \, dz \right)^{q-2} \left(\avint_{x\triangleright y} |u(z_1)-u(x)| \, dz_1 \right)  \\
		& \cdot  \left( \avint_{x\triangleright y} |\varphi(z_2) - \varphi(x)|  \, dz_2 \right)  \frac{ dy\, dx}{\rho(x,y)^{p-q}} \\
		& \aleq (\tfrac r4)^{-(p-q+1)} \|u\|_{L^\infty}^{2q-3} \|\varphi\|_{L^1}.
	\end{split}
	\end{equation}
	Note that by  $\varphi\in W^{\frac{p-q-1}{q},q}(\R,\R)$ with $\supp\varphi \subset B_\rho$ together with the Sobolev inequalities \Cref{app:thm:classobinequ} and \Cref{app:thm:sobinequ} leads to 
	\begin{equation}\label{reg:eq:phiL1}
		\|\varphi\|_{L^1} \aleq \rho\,  [\varphi]_{W^{\frac{p-q-1}{q},q}(\R)} \aleq \rho.
	\end{equation}
	The estimates on $R^{1,(p,q)}_{B_{2\rho}\times (\R/\Z\setminus B_{2\rho})}$ also work for $R^{1,(p,q)}_{(\R/\Z\setminus B_{2\rho}) \times B_{2\rho}} $ by symmetry.
	In case of $R^{1,(p,q)}_{(\R/\Z\setminus B_{2\rho}) \times (\R/\Z \setminus B_{2\rho})} $, we have to consider 
	\begin{align*}
		(\R/\Z \setminus B_{2\rho})\times(\R/\Z \setminus B_{2\rho}) & = (\R/\Z \setminus B(x_0,r))\times(\R/\Z \setminus B(x_0,r)) \\
		& \quad  \cup (\R/\Z \setminus B(x_0,r))\times(B(x_0,r) \setminus B_{2\rho}) \\
		& \quad  \cup (B(x_0,r) \setminus B_{2\rho})\times(\R/\Z \setminus B(x_0,r))\\
		& \quad \cup  (B(x_0,r) \setminus B_{2\rho})\times  (B(x_0,r) \setminus B_{2\rho}).
	\end{align*}
	$R^{1,(p,q)}_D$ can be estimated for first three domains thereof like in \eqref{reg:eq:roughestimateR1tail}, due to $\supp \varphi \subset B_\rho$, which implies that either $\avint_{x\triangleright y} \varphi(z_2)- \varphi(x)\, dz_2 = 0$ or $\rho(x,y)\geq r$. For the fourth domain we proceed with help of the tail estimate \Cref{la:RminusBest} similar to \eqref{reg:eq:tailestforR1}. Therefore, the statement follows for the first remaining term in the case of $2\leq q<4$.
	
	For the range $q \geq 4$ in $R^{1,(p,q)}$ we notice that the same methods work as well, in particular, for $D\subset \R/\Z\times \R/\Z$,
	\begin{equation*}
		\begin{split}
			&|R^{1,(p,q)}_D(u,(\varphi u\wedge)_{ij})| \\
			& \aleq \|u\|_{L^\infty} \iint_{D} \left(\avint_{x\triangleright y} |u(z)-u(x)|^2 \, dz \right)^{2} \left(\avint_{x\triangleright y} |u(z_1)-u(x)| \, dz_1 \right)^{q-3}  \\
			& \quad \cdot  \left( \avint_{x\triangleright y} |\varphi(z_2) - \varphi(x)|  \, dz_2 \right)  \frac{ dy\, dx}{\rho(x,y)^{p-q}} .
		\end{split}
	\end{equation*}
	We only need to change the exponents in H\"older's inequality to  $1=  \frac 2q + \frac{q-3}{q} + \frac 1q$, since we deal with the factor $$\left(\avint_{x\triangleright y} |u(z)-u(x)|^2 \, dz \right)^2 \left| \avint_{x\triangleright y} u(z)-u(x) \, dz\right|^{q-4}$$ instead of $(\avint_{x\triangleright y} |u(z)-u(x)|^2 \, dz )^{q-2}$, and 
	slightly adapt the disjoint support estimate \Cref{la:RminusBest} as well as \eqref{reg:eq:roughestimateR1tail}, to achieve the desired result. 
	
	For the second and third remainder we observe that we have to deal for $j=1,2,3$ with the scalar products 
	\begin{align*}
	\avint_{x\triangleright y} (u(z)-u(x)) & \cdot ((\varphi u\wedge)_{ij}(z)-(\varphi u\wedge)_{ij}(x)) \, dz,\\
	\avint_{x\triangleright y} \avint_{x\triangleright y} (u(s)-u(t)) & \cdot ((\varphi u\wedge)_{ij}(s)-(\varphi u\wedge)_{ij}(t)) \, ds \, dt,  \text{ respectively},
	\end{align*}
	in contrast to 
	\begin{align*}
	\avint_{x\triangleright y} \avint_{x\triangleright y}  (u(z_1) - u(x)) \cdot((\varphi u\wedge)_{ij}(z_2) - (\varphi u\wedge)_{ij}(x)) \, dz_1 \, dz_2.
	\end{align*}
	However, the techniques presented for the first remainder, in particular adding zeros and using $(u\wedge)_{ij} \in T_u\S^2$ for $j=1,2,3$, work out similarly and lead for any $D\subset \R/\Z\setminus \R/\Z$ to
	\begin{equation*}
		\begin{split}
			&|R^{2,(p,q)}_D(u,(\varphi u\wedge)_{ij})| \\
			& \aleq \|u\|_{L^\infty} \iint_{D} \left(\avint_{x\triangleright y} |u(z)-u(x)|^2 \, dz \right) \left(\avint_{x\triangleright y} |u(z_1)-u(x)| \, dz_1 \right)^{q-1}  \\
			& \quad \cdot  \left( \avint_{x\triangleright y} |\varphi(z_2) - \varphi(x)|  \, dz_2 \right)  \frac{ dy\, dx}{\rho(x,y)^{p-q}}
		\end{split}
	\end{equation*}
	as well as
	\begin{equation*}
		\begin{split}
			&|R^{3,(p,q)}_D(u,(\varphi u\wedge)_{ij})| \\
			& \aleq \|u\|_{L^\infty} \iint_{D}  \left(\avint_{x\triangleright y} |u(z_1)-u(x)| \, dz_1 \right)^{q+1} \left( \avint_{x\triangleright y} |\varphi(z_2) - \varphi(x)|  \, dz_2 \right)  \frac{ dy\, dx}{\rho(x,y)^{p-q}} .
		\end{split}
	\end{equation*}
	Hence, by proceeding as for the first remainder, the statement follows for $k=2,3$ from adjusting the exponents used in H\"older's inequality and the tail estimate in \Cref{la:RminusBest}. 
	
	Let us turn the focus to the remainders $k=4,\ldots,7$ now.
	The remainder terms for $k=4$ and $k=7$ are easier to handle since 
	\begin{align}\label{reg:eq:Linfremain}
		\begin{split}
		| \avint_{x\triangleright y} \eta(z)u(z) \cdot ((\varphi u\wedge)_{ij})_{\RZ}\, dz| & \aleq \|\eta\|_{L^\infty} \|u\|^2_{L^\infty} \|\varphi\|_{L^1}, \\
		|\eta(x)u(x) \cdot ((\varphi u\wedge)_{ij})_{\RZ}| & \aleq \|\eta\|_{L^\infty} \|u\|^2_{L^\infty} \|\varphi\|_{L^1}, \\
		|\eta(y)u(y) \cdot ((\varphi u\wedge)_{ij})_{\RZ}| & \aleq \|\eta\|_{L^\infty} \|u\|^2_{L^\infty} \|\varphi\|_{L^1}.
		\end{split}
	\end{align}
	Note that $\|u\|_{L^\infty} \leq 1$, $\|\eta\|_{L^\infty}\aleq 1$, and by \eqref{reg:eq:phiL1} $\|\varphi\|_{L^1} \aleq \rho$.
	We then obtain that the remaining factors emerge to the energy $\E^{p,q}(u)$, in particular
	\begin{align*}
		|R^{k,(p,q)}_\eta (u,(\varphi u\wedge)_{ij})| \aleq  \rho\,  \E^{p,q}(u),
	\end{align*}
	since in case $k=4$ we can extend the integrand with an extra factor $(1-\tfrac 12 c)$ due to its boundedness, cf. \Cref{rm:reg:boundedfactorc}, and in case $k=7$ the necessary factors of the energy are already given. It remains to study the cases $k=5$ and $k=6$, whose factors in the integrand are not necessarily comparable with the energy. We begin with $k=5$ and observe by $\supp \eta \subset B(x_0,\tfrac r2)$ that 
	\[
		\begin{split}
			& R_{\eta}^{5,(p,q)}(u,(\varphi u\wedge)_{ij}) \\
			& = - q \int_{B(x_0,\frac r2)} \int_{\R/\Z}   (a - \tfrac 14 b^2)^{\frac{q-2}{2}}k(x,y)^{-\frac{p}{2}} (a-\tfrac 12 b) \, \left(\eta(x)u(x) \cdot ((\varphi u\wedge)_{ij})_{\RZ}\right) \frac{dy\, dx}{\rho(x,y)^{p-q}} .
		\end{split}
	\]
	By splitting the integration domain into 
	\[
		B(x_0,\tfrac r2) \times \R/\Z = (B(x_0,\tfrac r2) \times B(x_0,r)) \cup (B(x_0,\tfrac r2) \times (\R/\Z \setminus B(x_0,r))),
	\]
	we have for the first domain by the previous comments on the factors appearing in the integrand, inequalities \eqref{reg:eq:Linfremain} and \eqref{reg:eq:phiL1}, Jensen's inequality, identity \Cref{la:id2}, and Sobolev embedding \Cref{la:sob1}, that 
	\begin{align}\label{reg:eq:SobnormforR5}
		\begin{split}
		& |R_{\eta, B(x_0,\tfrac r2) \times B(x_0,r) }^{5,(p,q)}(u,(\varphi u\wedge)_{ij})| \\
		& \aleq \|\eta\|_{L^\infty} \|u\|^2_{L^\infty} \|\varphi\|_{L^1} \left( \int_{B(x_0,r)} \int_{B(x_0, r)}   \frac{\avint_{x\triangleright y} |u(z)-u(x)|^{2q} \, dz}{\rho(x,y)^{p-q}} dy\, dx\right)^{\frac 12}  \\
		& \aleq \rho [u]_{W^{\frac{p-q-1}{2q},2q}(B(x_0,r))}^q\\
		& \aleq \rho \,  r^{\frac{p-q-2}{2}} [u]_{W^{\frac{p-q-1}{q},q}(B(x_0,r))}^q, 
		\end{split}
	\end{align}
	and for the second domain we notice $\rho(x,y) \geq \tfrac r2$ for $x\in B(x_0,\tfrac r2)$ and $y\in \R/\Z \setminus B(x_0,r)$, which leads by \eqref{reg:eq:Linfremain}, estimates like \eqref{reg:eq:roughestimateR1tail}, and \eqref{reg:eq:phiL1} to 
	\begin{align}\label{reg:eq:roughestforR5}
		\begin{split}
		|R_{\eta, B(x_0,\tfrac r2) \times (\R/\Z \setminus B(x_0,r)) }^{5,(p,q)}(u,(\varphi u\wedge)_{ij})| &\aleq (\tfrac r2)^{-(p-q)} \|u\|_{L^\infty}^{q+2} \|\eta\|_{L^\infty} \|\varphi\|_{L^1} \\
		& \aleq  \rho \, r^{-(p-q)} .
		\end{split}
	\end{align}
	In case of $k=6$ we distinguish for the integration domain the cases 
	\begin{align*}
		\R/\Z\times\R/\Z & = B(x_0,r) \times B(x_0,r) \\
		& \quad \cup B(x_0,r) \times (\R/\Z \setminus B(x_0,r)) \\
		& \quad \cup (\R/\Z\setminus B(x_0,r)) \times B(x_0,r) \\
		& \quad \cup (\R/\Z\setminus B(x_0,r)) \times (\R/\Z\setminus B(x_0,r)).
	\end{align*}
	For the first integration domain we gain the same estimate as in \eqref{reg:eq:SobnormforR5} and for the other integration domains 
	we find by $\supp \eta \subset B(x_0,\tfrac r2)$ that either  $\rho(x,y)\geq \tfrac r2$ or 
	\[
		 \avint_{x\triangleright y} \eta(z)u(z) \cdot ((\varphi u\wedge)_{ij})_{\RZ}\, dz = 0.
	\]
	Therefore,  we deduce the same estimate like in \eqref{reg:eq:roughestforR5}.
	All remainder cases got estimated hereby.
\end{proof}

\subsection{Regularity theory for $\mathcal{E}^{q}$-critical points: Proof of \Cref{pr:decayest}}\label{subsec:RegularityTheory}

In this section we finally apply the grand machinery of showing  Hölder regularity for (essentially) fractional harmonic maps, which correspond to the first derivative of our critical knots of interest. We establish a proof for the scale-invariant tangent-point energies, i.e. for $\tp^{p,q}$ with $p=q+2$ and $q\geq 2$, along the lines of \cite{S15} and \cite{BRS19}, but face some major obstacles due to the local definition of critical points for scale-invariant tangent-point energies, cf. \Cref{def:wcp}, on the way. Our main goal is to show the decay estimate \Cref{pr:decayest}.

In order to attain the desired statement, we begin with estimating the Gagliardo semi-norm of $u$ by an operator $\Gamma_{\beta,B}u$ reminding of the Riesz potential, which is introduced in the following.

First we recall that the term containing the highest order in the Euler-Lagrange equation of $\E^q$, $q\geq 2$, for maps $u,\varphi:\R/\Z\rightarrow\R^3$ and some small interval $B\subset\R/\Z$ is given by
\begin{align*}
Q_{B\times B}(u,\varphi) & := Q_{B\times B}^{(q+2,q)}(u,\varphi) \\ 
& = q \int_{B}\int_{B}   \left|\avint_{x\triangleright y} u(z) -u(x)\, dz \right|^{q-2} k(x,y)^{-\frac{q+2}{2}} \\
& \hspace{2cm} \cdot \avint_{x\triangleright y} \avint_{x\triangleright y}  (u(z_1) - u(x)) \cdot(\varphi(z_2) - \varphi(x)) \, dz_1 \, dz_2  \frac{dy\, dx}{\rho(x,y)^{2}}.
\end{align*}
Note that for our functions of interest the factor $k(x,y)^{-\frac{q+2}{2}}$ is {strictly positive and} bounded, cf.~\Cref{rm:reg:boundedfactorc}.

As in \cite{S15} and \cite{BRS19}, we now define a vector-valued potential for some $0<\beta<1$ by 
\begin{align}\label{def:GammaQB}
	\begin{split}
	\Gamma_{\beta,B\times B} \, u(z):= & \int_{B} \int_{B}   \left|\avint_{x\triangleright y} u(z) -u(x)\, dz \right|^{q-2} \, k(x,y)^{-\frac{q+2}{2}} \\
	& \quad \cdot   \avint_{x\triangleright y} \avint_{x\triangleright y}  (u(z_1) - u(x)) \, (|z-z_2|^{\beta-1} - |z-x|^{\beta-1}) \, dz_1 \, dz_2  \frac{dy\, dx}{\rho(x,y)^{2}}.
	\end{split}
\end{align}
following the definition of the Riesz potential $\lapin{\beta}$ of order $\beta$, which is defined by
\[
\lapin{\beta} f (x) = \int_{\R}  |z-x|^{\beta -1} f(z) \, dz. 
\]
The inverse of the Riesz potential $\lapin{\beta}$ is called the fractional Laplacian of order $\beta$, which for $\beta \in (0,2)$ has the form
\[
\lapla{\frac \beta 2}  f (x)= c \int_{\R} \frac{f(y)-f(x)}{|x-y|^{1+\beta}} \, dy
\]
for some $c>0$, cf. \cite{Hitchhiker}.
Applied to our situation, we hence observe that 
\begin{align}\label{eq:QRiesz}
Q_{B\times B}(u, \varphi) = c \int_{\R} \Gamma_{\beta,B\times B} u(z)\cdot  \lapla{\frac{\beta}{2}} \varphi(z) \, dz.
\end{align}

Note that  we have 
\begin{equation}\label{eq:reg:RieszopGamma}
\lapin{\alpha} \Gamma_{\beta,B\times B} \,  u =  \Gamma_{\alpha + \beta,B \times B}  \, u 
\end{equation}
for any $\alpha,\beta>0$.

Our first interim result is the following.

\begin{proposition}(Left-hand side estimates)\label{pr:reg:lhsestimate}
	Let $q\geq 2$, $\tfrac 1q - \tfrac 1\p>0$ small, and $\gamma:\R /\Z \rightarrow \R^3$ be a homeomorphism with locally small tangent-point energy $\tp^{q+2,q}$ around  an interval $B(x_0,r)$ in $\R/\Z$, in the sense of definition \Cref{def:homeotp}. Moreover, denote the unit tangent field of $\gamma$ by  $u:\RZ\rightarrow \S^2$ such that $\int_{\RZ} u = 0$ and for any $y_0\in B(x_0,\tfrac r2)$, choose $\rho>0$ such that $B_{2^L\rho}:= B(y_0,2^L\rho) \subset B(x_0,r)$ for large $L\in\N$. 
	
	Then we have for any $\delta>0$
	\begin{align*}
	[u]^q_{W^{\frac{1}{q}, q}(B_\rho)} \aleq & \ [u]_{W^{\frac{1}{q}, q}(B_{2^L\rho})} \|\chi_{B_{2^K\rho}} \Gamma_{\frac 1\p,B_{2^L\rho}\times B_{2^L\rho}}u\|_{L^{\frac{\p}{\p-1} }} +\delta [u]^q_{W^{\frac{1}{q}, q}(B_{2^L\rho})} \\
	&+ C_\delta \left([u]^q_{W^{\frac{1}{q}, q}(B_{2^L\rho})} -[u]^q_{W^{\frac{1}{q}, q}(B_{\rho})}\right)
	\end{align*}
	for any $L,K\in\N$ large enough with $L\gg K$. The constant in this inequality only depends on $q$.
\end{proposition}

\begin{proof}
	Let $\eta \in C^\infty_c (B_{2\rho})$, $\eta \equiv 1$ on $B_\rho$, with $|\nabla^k \eta| \leq C(k) \rho^{-k}$. Set $(u)_{B_{2\rho}\setminus B_\rho} := \tfrac{1}{|B_{2\rho}\setminus B_\rho|} \int_{B_{2\rho}\setminus B_\rho} u$ and 
	\[
	\psi(x) := \eta(x)(u(x)-(u)_{B_{2\rho}\setminus B_\rho}).
	\]
	Then for any $x,y \in B_{\rho}$ we have 
	\begin{align*}
	|\avint_{x\triangleright y}  \psi(z) - \psi(x) \, dz |^2  =  |\avint_{x\triangleright y}  u(z) - u(x) \, dz |^2
	\end{align*}
	and therefore by \Cref{pr:Qequiv}  for any $L\geq 2$
	\begin{align*}
	& [u]_{W^{\frac{1}{q}, q}(B_\rho)}^q \\
	& \aleq \int_{B_{2^L\rho}} \int_{B_{2^L\rho}}   |\avint_{x\triangleright y} u(z) -u(x)\, dz |^{q-2} \,  k(x,y)^{-\frac{q+2}{2}} \,  	|\avint_{x\triangleright y}  \psi(z) - \psi(x) \, dz |^2  \frac{dy\, dx}{\rho(x,y)^{2}} .
	\end{align*}
	Now we decompose
	\begin{align*}
	\psi(z)-\psi(x) = ( &u(z)-u(x)) - (1-\eta(z))(u(z)-u(x)) \\
	& + (\eta(z)-\eta(x))(u(x)-(u)_{B_{2\rho} \setminus B_\rho}),
	\end{align*}
	which leads to 
	\[
	[u]_{W^{\frac{1}{q}, q} (B_\rho)}^q \aleq I - II + III,
	\]
	where
	\[
	\begin{split}
	I &:=  \iint_{{B_{2^L\rho}}^2}  |\avint_{x\triangleright y} u(z) -u(x)\, dz |^{q-2} \, k(x,y)^{-\frac{q+2}{2}} \\ 
	& \quad \cdot \avint_{x\triangleright y} \avint_{x\triangleright y} (u(z_1) - u(x)) \cdot (\psi(z_2) - \psi(x)) \, dz_1 \, dz_2  \frac{dx\, dy}{\rho(x,y)^{2}}, \\
	II &:= \iint_{{B_{2^L\rho}}^2} |\avint_{x\triangleright y} u(z) -u(x)\, dz |^{q-2} \, k(x,y)^{-\frac{q+2}{2}} \\ 
	& \quad \cdot \avint_{x\triangleright y} \avint_{x\triangleright y} ((1-\eta(z_1))(u(z_1) - u(x))) \cdot (\psi(z_2) - \psi(x)) \, dz_1 \, dz_2 \frac{dx\, dy}{\rho(x,y)^{2}}, \\
	III & :=  \iint_{{B_{2^L\rho}}^2}   |\avint_{x\triangleright y} u(z) -u(x)\, dz |^{q-2} \, k(x,y)^{-\frac{q+2}{2}} \\ 
	& \quad \cdot \avint_{x\triangleright y} \avint_{x\triangleright y}  ( (\eta(z_1)-\eta(x))(u(x)-(u)_{B_{2\rho}\setminus B_\rho})) \cdot (\psi(z_2) - \psi(x)) \, dz_1 \, dz_2 \frac{dx\, dy}{\rho(x,y)^{2}}.
	\end{split}
	\]
	
	We start to deal with the terms involving the cutoff function. Observe by the boundedness of the factor $k(x,y)^{-\frac{q+2}{2}}$ in \Cref{rm:reg:boundedfactorc}, H\"older's inequality for $1= \frac{q-2}{q} + \frac 1q + \frac 1q$, Jensen's inequality and identification Lemma \ref{la:id2} that
	\begin{align*}
	& |II|  \\
	& \aleq  \iint_{(B_{2^L\rho})^2}  \left(\avint_{x\triangleright y} |u(z)-u(x)| \, dz \right)^{q-2} \\
	& \quad \cdot \avint_{x\triangleright y}  (1-\eta(z_1)) \, |u(z_1) - u(x)| \, dz_1 \,\avint_{x\triangleright y} |\psi(z_2) - \psi(x)| \, dz_2 \frac{dx\, dy}{\rho(x,y)^{2}} \\
	& \aleq \left( \iint_{(B_{2^L\rho})^2}  \frac{ \avint_{x\triangleright y} |u(z) - u(x)|^{q} \, dz}{\rho(x,y)^{2}} \, dx\, dy \right)^{\frac{q-2}{q}} \\
	&\quad \cdot \left( \iint_{(B_{2^L\rho})^2}  \frac{ \avint_{x\triangleright y} (1-\eta(z_1))^q |u(z_1) - u(x)|^{q} \, dz_1}{\rho(x,y)^{2}} \, dx\, dy \right)^{\frac 1q}  \\
	& \quad \cdot \left( \iint_{(B_{2^L\rho})^2}  \frac{ \avint_{x\triangleright y} |\psi(z_2) - \psi(x)|^{q} \, dz_2}{\rho(x,y)^{2}} \, dx\, dy \right)^{\frac 1q}\\
	& \aleq [u]_{W^{\frac{1}{q}, q}(B_{2^L\rho})}^{q-2} [\psi]_{W^{\frac{1}{q}, q}(B_{2^L\rho})} \left(\iint_{(B_{2^L\rho})^2}  (1-\eta(z))^q |u(z)-u(x)|^q  \frac{dz\, dx}{\rho(z,x)^{2}} \right)^{\frac 1q}.
	\end{align*}
	Then $[\psi]_{W^{\frac{1}{q}, q}(B_{2^L\rho})} \aleq [u]_{W^{\frac{1}{q}, q}(B_{2^L\rho})}$  by \Cref{pr:psiuds}, the assumption $\eta \equiv 1$ on $B_\rho$,  and Young's inequality lead to
	\begin{align*}
	|II| & \aleq [u]_{W^{\frac{1}{q}, q}(B_{2^L\rho})}^{q-1} \left(\int_{B_{2^L\rho}} \int_{B_{2^L\rho}\setminus B_\rho}  |u(z)-u(x)|^q  \frac{dz\, dx}{\rho(z,x)^{2}} \right)^{\frac 1q} \\
	& \aleq \delta [u]^q_{W^{\frac{1}{q}, q}(B_{2^L\rho})} + C_\delta \left([u]^q_{W^{\frac{1}{q}, q}(B_{2^L\rho})} -[u]^q_{W^{\frac{1}{q}, q}(B_{\rho})}\right). 
	\end{align*}
	Regarding $III$, we estimate along the lines of $II$, first of all
	\begin{align*}
	|III| \aleq [u]_{W^{\frac{1}{q}, q}(B_{2^L\rho})}^{q-2} [\psi]_{W^{\frac{1}{q}, q}(B_{2^L\rho})} \left(\int_{B_{2^L\rho}} \int_{B_{2^L\rho}} |\eta(x)-\eta(z)|^q |u(x)-(u)_{B_{2\rho} \setminus B_\rho}|^q  \frac{dz\, dx}{\rho(z,x)^{2}} \right)^{\frac 1q}
	\end{align*}
	and next with help of \Cref{pr:psiuds1} for $L\geq 2$ as well as Young's inequality
	\begin{align*}
	|III| & \aleq [u]_{W^{\frac{1}{q}, q}(B_{2^L\rho})}^{q-1} \left(\int_{B_{2^L\rho}} \int_{B_{2^L\rho}} |\eta(x)-\eta(z)|^q |u(x)-(u)_{B_{2\rho}\setminus B_\rho}|^q  \frac{dz\, dx}{\rho(z,x)^{2}} \right)^{\frac 1q} \\
	& \aleq [u]_{W^{\frac{1}{q}, q}(B_{2^L\rho})}^{q-1} \left( [u]^q_{W^{\frac{1}{q}, q}(B_{2^L\rho})} -[u]^q_{W^{\frac{1}{q}, q}(B_{\rho})} \right)^{\frac 1q} \\
	& \aleq \delta [u]^q_{W^{\frac{1}{q}, q}(B_{2^L\rho})} + C_\delta \left([u]^q_{W^{\frac{1}{q}, q}(B_{2^L\rho})} -[u]^q_{W^{\frac{1}{q}, q}(B_{\rho})}\right). 
	\end{align*}
	Hence 
	\[
	|II| + |III| \aleq \delta [u]^q_{W^{\frac{1}{q}, q}(B_{2^L\rho})} + C_\delta \left([u]^q_{W^{\frac{1}{q}, q}(B_{2^L\rho})} -[u]^q_{W^{\frac{1}{q}, q}(B_{\rho})}\right).
	\]
	For the remaining term $I$ we see by definition of $Q_{B\times B}$ in \eqref{def:reg:QB} and \Cref{pr:psiuds} for $L\geq 1$
	\begin{align*}
	|I| & \aleq [\psi]_{W^{\frac 1q,q}(\R)} \sup_{\varphi \in C^\infty_c (B_{2\rho}, \R^3), [\varphi]_{W^{\frac 1q,q}(\R)} \leq 1} |Q_{B_{2^L \rho}\times B_{2^L\rho}} (u,\varphi)| \\
	& \aleq [u]_{W^{\frac 1q,q}(B_{2^L\rho})} \sup_{\varphi \in C^\infty_c (B_{2\rho}, \R^3), [\varphi]_{W^{\frac 1q,q}(\R)} \leq 1} |Q_{B_{2^L\rho}\times B_{2^L\rho}} (u,\varphi)|. 
	\end{align*}
	By using the identity \eqref{eq:QRiesz} for $\p>q$ with $\frac 1q - \frac 1\p >0$ small and introducing cutoff functions $\eta_{B_R}\in C^\infty_c (B_{2R})$ with $\eta_{B_R} \equiv 1$ on $B_R$ and $\|\nabla^k \eta_{B_R}\|_{L^\infty} \aleq R^{-k}$ for $R>0$, we get for any $K\geq 1$ and some $\varphi \in C^\infty_c (B_{2\rho},\R^3)$ such that $[\varphi]_{W^{\frac 1q,q}(\R)} \leq 1$ 
	\begin{align*}
	|Q_{B_{2^L \rho}\times B_{2^L \rho}} (u,\varphi)| & = c \left| \int_{\R} \Gamma_{\frac 1\p,B_{2^L \rho}\times B_{2^L \rho}} u(z) \cdot   \lapla{\frac{1}{2\p}} \varphi(z) \, dz \right| \\
	& \aleq \left| \int_{\R} \eta_{B_{2^{K-1}\rho }} (z) \Gamma_{\frac 1\p,B_{2^L \rho}\times B_{2^L \rho}} u(z) \cdot  \lapla{\frac{1}{2\p}} \varphi(z) \, dz \right| \\
	& \quad + \sum_{k=K}^\infty   \left| \int_{\R} (\eta_{B_{2^k\rho}} -\eta_{B_{2^{k-1}\rho}} ) (z) \Gamma_{\frac 1\p,B_{2^L \rho}\times B_{2^L \rho}} u(z) \cdot \lapla{\frac{1}{2\p}} \varphi(z) \, dz \right| 
	\end{align*}
	We estimate the first term by H\"older's inequality, Sobolev's inequality, cf. \Cref{app:thm:sobinequ}, and $[\varphi]_{W^{\frac 1q,q}(\R)} \leq 1$  as
	\begin{align*}
	&\left| \int_{\R} \eta_{B_{2^{K-1}\rho }} (z) \Gamma_{\frac 1\p,B_{2^L \rho}\times B_{2^L \rho}} u(z) \cdot  \lapla{\frac{1}{2\p}} \varphi(z) \, dz \right| \\
	& \aleq \| \eta_{B_{2^{K-1}\rho }} \Gamma_{\frac 1\p,B_{2^L \rho}\times B_{2^L \rho}} u  \|_{L^{\frac{\p}{\p-1}}} \| \lapla{\frac{1}{2\p}} \varphi \|_{L^{\p}} \\
	& \aleq \| \chi_{B_{2^K\rho }} \Gamma_{\frac 1\p,B_{2^L \rho}\times B_{2^L \rho}} u  \|_{L^{\frac{\p}{\p-1}}} \ [\varphi]_{W^{\frac 1q,q}(\R)} \\
	& \aleq  \| \chi_{B_{2^{K}\rho }} \Gamma_{\frac 1\p,B_{2^L \rho}\times B_{2^L \rho}} u  \|_{L^{\frac{\p}{\p-1}}}.
	\end{align*}
	Using integration by parts and the property $\lapin\alpha \Gamma_{\beta,B_{2^L \rho}\times B_{2^L \rho}} u =  \Gamma_{\alpha + \beta,B_{2^L \rho}\times B_{2^L \rho}} u $, cf. \eqref{eq:reg:RieszopGamma}, the second part can be rewritten as
	\begin{align*}
	& \sum_{k=K}^\infty   \left| \int_{\R} (\eta_{B_{2^k\rho}} -\eta_{B_{2^{k-1}\rho}} ) (z) \Gamma_{\frac 1\p,B_{2^L \rho}\times B_{2^L \rho}} u(z) \cdot  \lapla{\frac{1}{2\p}} \varphi(z) \, dz \right|  \\
	& = \sum_{k=K}^\infty   \left| \int_{\R}   \lapla{\frac 1q - \frac 1\p} \left( \lapin{\frac 2q - \frac2\p}  \Gamma_{\frac 1\p,B_{2^L \rho}\times B_{2^L \rho}} u \right) (z) \cdot \left( (\eta_{B_{2^k\rho}} -\eta_{B_{2^{k-1}\rho}} ) \lapla{\frac{1}{2\p}} \varphi \right) (z) \, dz \right|  \\
	& = \sum_{k=K}^\infty   \left| \int_{\R}  \Gamma_{\frac 2q - \frac 1\p,B_{2^L \rho}\times B_{2^L \rho}} u(z) \cdot \lapla{\frac 1q - \frac 1\p} \left( (\eta_{B_{2^k\rho}} -\eta_{B_{2^{k-1}\rho}} ) \lapla{\frac{1}{2\p}} \varphi \right) (z) \, dz \right| 
	\end{align*}
	and then estimated by H\"older's inequality, \Cref{pr:app:GammaEst} for $\frac 1q - \frac 1\p>0$ small enough, the localization argument \Cref{app:loc1}, and Sobolev's inequality \Cref{app:thm:sobinequ}, 
	\begin{align*}
	& \sum_{k=K}^\infty   \left| \int_{\R}  \Gamma_{\frac 2q - \frac 1\p,B_{2^L \rho}\times B_{2^L \rho}} u(z) \cdot \lapla{\frac 1q - \frac 1\p} \left( (\eta_{B_{2^k\rho}} -\eta_{B_{2^{k-1}\rho}} )  \lapla{\frac{1}{2\p}} \varphi \right) (z) \, dz \right| \\
	& \aleq \sum_{k=K}^\infty \|\Gamma_{\frac 2q - \frac 1\p,B_{2^L \rho}\times B_{2^L \rho}} u \|_{L^{{(1-\frac 2q + \frac 1\p)}^{-1}}} \,  \|\lapla{\frac 1q - \frac 1\p} ( (\eta_{B_{2^k\rho}} -\eta_{B_{2^{k-1}\rho}} ) \lapla{\frac{1}{2\p}} \varphi ) \|_{L^{(\frac 2q - \frac 1\p)^{-1}}} \\
	& \aleq \sum_{k=K}^\infty [u]^{q-1}_{W^{\frac 1q,q}(B_{2^L\rho}) } \, 2^{-\sigma k} \,  \|\lapla{\frac{1}{2\p}} \varphi \|_{L^{\p}} \\
	& \aleq \sum_{k=K}^\infty  2^{-\sigma k} \, [u]^{q-1}_{W^{\frac 1q,q}(B_{2^L\rho}) } ,
	\end{align*} 
	for some $\sigma>0$.
	The statement of the proposition follows by choosing $K$ large enough.
\end{proof}

In the next step we need to investigate estimates involving the operator $\Gamma_{\frac 1\p,B\times B} u$, which appears in the left-hand side estimates \Cref{pr:reg:lhsestimate}, to obtain the decay estimate \Cref{pr:decayest} subsequently.
We start by splitting the operator $\Gamma_{\frac 1\p,B\times B} u$ by projecting it into the linear space spanned by $u$ as well as into linear space orthogonal to $u$. More precisely, we observe by $|u|=1$ a.e. that
\begin{equation}\label{eq:reg:projofop}
\begin{split}
	& \| \chi_{B_{2^K\rho }} \Gamma_{\frac 1\p,B_{2^L \rho}\times B_{2^L \rho}} u  \|_{L^{\frac{\p}{\p-1}}} \\
	& \aleq \| \chi_{B_{2^K\rho }} u \, \cdot \, \Gamma_{\frac 1\p,B_{2^L \rho}\times B_{2^L \rho}} u  \|_{L^{\frac{\p}{\p-1}}} + \| \chi_{B_{2^K\rho }} u \, \wedge  \,  \Gamma_{\frac 1\p,B_{2^L \rho}\times B_{2^L \rho}} u  \|_{L^{\frac{\p}{\p-1}}}.
\end{split}
\end{equation}

Here we recall that $v\wedge$ for any $v\in \R^3$ is given by the $\R^{3\times 3}$-matrix 
\[
v\wedge = \begin{pmatrix}
0 & -v_3 & v_2 \\
v_3 & 0 & - v_1 \\
-v_2 & v_1 & 0
\end{pmatrix}.
\]

We then treat each part of the splitting separately. However, both estimates are based on effects of
integration by compensation using non-linear commutators as well as information from the Euler-Lagrange equations.

\begin{lemma}(Right-hand side estimates I)\label{la:reg:rhs1}
	Let $q\geq 2$, $\tfrac 1q - \tfrac 1\p>0$ small, and $\gamma:\R /\Z \rightarrow \R^3$ be a homeomorphism with small tangent-point energy $\tp^{q+2,q}$ around the interval $B(x_0,r)$ in $\R/\Z$, in the sense of definition \Cref{def:homeotp}. Denote the unit tangent field of $\gamma$ by  $u:\RZ\rightarrow \S^2$ such that $\int_{\RZ} u = 0$ and let $\tilde{u}$ be a $W^{\frac 1q,q}$-extension of $u|_{B(x_0,r)}$ from $B(x_0,r)$ to $\R$ as discussed in \Cref{reg:rm:extension}. 
	Moreover, for $y_0 \in B(x_0,\tfrac r2)$, choose $\rho>0$ such that $B_{2^{2L}\rho}:= B(y_0,2^{2L}\rho) \subset B(x_0,r)$ for large $L\in\N$.
	
	Then we have 
	\begin{align*}
	& \| \chi_{B_{2^K\rho }} u \, \cdot \, \Gamma_{\frac 1\p,B_{2^L \rho}\times B_{2^L \rho}} u  \|_{L^{\frac{\p}{\p-1}}} \\
	&  \aleq [u]^q_{W^{\frac 1q,q }(B_{2^{2L}\rho})} + \sum_{k=1}^\infty 2^{-\sigma (L+k)} [\tilde u]^q_{W^{\frac 1q,q }(B_{2^{2L+k}\rho})},
	\end{align*}
	for any $K,L\in\N$ large enough with $L\gg K$.
	The constant in this inequality only depends on $q$.
\end{lemma}
\begin{proof}
	First observe by $|u|=1$ a.e. in $B(x_0,r)$ that 
	\begin{align*}
	& u(z) \cdot \int_x^y \int_x^y ( u(z_1)-u(x)) (|z-z_2|^{\frac 1\p-1} - |z-x|^{\frac 1\p -1} ) \, dz_1 \, dz_2 \\
	& = -\frac 12 \int_x^y \int_x^y (u(z_1)-u(x)) \cdot (u(z_1)+u(x) -2 u(z)) (|z-z_2|^{\frac 1\p-1} - |z-x|^{\frac 1\p -1} ) \, dz_1 \, dz_2 
	\end{align*}
	for almost any $x,y,z \in B(x_0,r)$. 
	Therefore,  we have with help of the definition of $\Gamma$ in \eqref{def:GammaQB} and the boundedness of the appearing factor $\, k(x,y)^{-\frac{q+2}{2}}$ by \Cref{rm:reg:boundedfactorc} 
	\begin{align*}
	& |\chi_{B_{2^K\rho }} u (z) \, \cdot \, \Gamma_{\frac 1\p,B_{2^L \rho}\times B_{2^L \rho}} u (z) | \\
	& \aleq \chi_{B_{2^L\rho }} (z) \int_{B_{2^L\rho}} \int_{B_{2^L\rho}} \left(\avint_{x\triangleright y} |u(z_0) - u(x)| \, dz_0 \right)^{q-2} 
	 \\
	& \cdot  \avint_{x\triangleright y}  |u(z_1)-u(x)| |u(z_1)+ u(x) -2 u(z)| \, dz_1 \avint_{x\triangleright y} ||z-z_2|^{\frac 1\p-1} - |z-x|^{\frac 1\p -1}| \, dz_2 \frac{d y \, dx}{\rho(x,y)^2 }.  
	\end{align*}
	Observe that $u$ is evaluated only in $ B(x_0,r)$ here as both $B_{2^K\rho}\subset B_{2^L\rho} \subset B(x_0,r)$ ($L\gg K$). Since $u$ and $\tilde u$ coincide on $B(x_0,r)$ by construction, cf. \Cref{reg:rm:extension}, we can continue with $\tilde u$ from now on as it is globally $W^{\frac 1q,q}$-regular in comparison to $u$, which is only in $W^{\frac 1q,q}(B(x_0,r),\R^3)$. 
	
For the next step, we need the notation of the uncentered Hardy-Littlewood maximal function, which is given by
\[
\mathcal{M}f(x) = \sup_{B(x,r)\ni y }\frac{1}{|B(y,r)|} \int_{B(y,r)} |f(z)| \, dz. 
\]
By \Cref{pr:estmax} for small $\delta > 0$
	\[
		\left(\avint_{x\triangleright y} |\tilde u(z_0) - \tilde u(x)| \, dz_0 \right)^{q-2} \aleq  |x-y|^{(\frac 1\p - \delta) (q-2)}\left(\mathcal{M}\mathcal{M} \lapla{\frac{\frac 1\p - \delta}{2} } \tilde u(x) \right)^{q-2}.
	\]
	Now we decompose the integral $\avint_{x\triangleright y}  |\tilde u(z_1)-\tilde u(x)| |\tilde u(z_1)+\tilde u(x) -2 \tilde u(z)| \, dz_1 $ into four terms 
	\begin{align}
	& \avint_{x\triangleright y}  |\tilde u(z_1)-\tilde u(x)| |\tilde u(z_1)+\tilde u(x) -2 \tilde u(z)| \, dz_1 \nonumber \\
	& \leq \avint_{x\triangleright y}  |\tilde u(z_1)-\tilde u(y)|^2  \, dz_1  \label{eq:avu1}\\
	& \quad + \avint_{x\triangleright y}  |\tilde u(z_1)-\tilde u(y)| \, dz_1 \, |\tilde u(y)+\tilde u(x) -2 \tilde u(z)|  \label{eq:avu2} \\
	& \quad +  |\tilde u(y)-\tilde u(x)|\,  \avint_{x\triangleright y}  |\tilde u(z_1)- \tilde u(y)| \, dz_1 \label{eq:avu3} \\
	& \quad +  |\tilde u(y)-\tilde u(x)|\,  |\tilde u(y)+\tilde u(x) -2 \tilde u(z)|.\label{eq:avu4}
	\end{align}
	
	For the first term \eqref{eq:avu1}, we obtain by duality for some $\varphi\in C_c^\infty(\R)$ with $\|\varphi\|_{L^{\p}}\leq 1$ 
	\[
	\begin{split}
	&  \| \chi_{B_{2^K\rho }} u \, \cdot \, \Gamma_{\frac 1\p,B_{2^L \rho}\times B_{2^L \rho}} u  \|_{L^{\frac{\p}{\p-1}}} \\
	& \aleq \bigg\|  \int_{B_{2^L\rho}} \int_{B_{2^L\rho}} |x-y|^{(\frac 1\p - \delta) (q-2)}\left(\mathcal{M}\mathcal{M} \lapla{\frac{\frac 1\p - \delta}{2} }\tilde u(x) \right)^{q-2} \\
	& \cdot  \avint_{x\triangleright y}  |\tilde u(z_1)-\tilde u(y)|^2  \, dz_1  \avint_{x\triangleright y} ||\cdot -z_2|^{\frac 1\p-1} - |z-x|^{\frac 1\p -1}| \, dz_2 \frac{d y \, dx}{\rho(x,y)^2 } \bigg\|_{L^{\frac{\p}{\p-1}}(B_{2^K\rho})}\\
	& \aleq \int_{\R} \int_{B_{2^L\rho}} \int_{B_{2^L\rho}} \left( \mathcal{M}\mathcal{M} \lapla{\frac{\frac 1\p - \delta}{2} }  \tilde u(x) \right)^{q-2} |\varphi(z)| \\
	& \cdot  \avint_{x\triangleright y}  | \tilde u(z_1)-\tilde u(y)|^2  \, dz_1  \avint_{x\triangleright y} ||z-z_2|^{\frac 1\p-1} - |z-x|^{\frac 1\p -1}| \, dz_2  \, |x-y|^{(\frac 1\p - \delta) (q-2)-2}  \,d y \, dx \, dz.
	\end{split}
	\]
	Hence, for an admissible $\eps\in(0,1)$ and $(\frac 1\p -\delta)(q-1)+  \frac{1}{2q} -1 < \eps < \frac 1\p$, we can apply Lemma \ref{la:threetermforuquad} to the previous inequality and achieve
	\begin{align}\label{reg:est:rhs11stterm}
	\begin{split}
	& \bigg\|  \int_{B_{2^L\rho}} \int_{B_{2^L\rho}} |x-y|^{(\frac 1\p - \delta) (q-2)}\left(\mathcal{M}\mathcal{M} \lapla{\frac{\frac 1\p - \delta}{2} }\tilde u(x) \right)^{q-2} \\
	& \cdot  \avint_{x\triangleright y}  |\tilde u(z_1)-\tilde u(y)|^2  \, dz_1  \avint_{x\triangleright y} ||\cdot -z_2|^{\frac 1\p-1} - |z-x|^{\frac 1\p -1}| \, dz_2 \frac{d y \, dx}{\rho(x,y)^2 } \bigg\|_{L^{\frac{\p}{\p-1}}(B_{2^K\rho})}\\
	& \aleq \int_{\R}  \lapin{(\frac 1\p -\delta) (q-1)+ \frac{1}{2q} + \eps -1} \left(\chi_{B_{2^L\rho}}\mathcal{M}\mathcal{M} \lapla{\frac{\frac 1\p - \delta}{2} }  \tilde u(z) \right)^{q-2}\\
	& \qquad \cdot  \mathcal{M} \left(\mathcal{M}\lapla{\frac{\frac 1\p -\delta}{2}}  \tilde u\,  \mathcal{M}\lapla{\frac{1/(2q)}{2}}  \tilde u\right) (z) \ \lapin{\frac 1\p - \eps} |\varphi|(z) \, dz \\
	& \quad + \int_{\R} \left(\chi_{B_{2^L\rho}}\mathcal{M}\mathcal{M} \lapla{\frac{\frac 1\p - \delta}{2} }\tilde u(z) \right)^{q-2} \\
	& \qquad \cdot \lapin{(\frac 1\p -\delta) (q-1)+ \frac{1}{2q} + \eps -1}  \mathcal{M} \left(\mathcal{M}\lapla{\frac{\frac 1\p -\delta}{2}} \tilde u\,  \mathcal{M}\lapla{\frac{1/(2q)}{2}} \tilde u\right) (z) \ \lapin{\frac 1\p - \eps}|\varphi|(z) \, dz \\
	& \quad  +  \int_{\R} \lapin{(\frac 1\p -\delta)(q-1)+ \frac{1}{2q} + \eps -1} \left( \chi_{B_{2^L\rho}}\mathcal{M}\mathcal{M} \lapla{\frac{\frac 1\p - \delta}{2} } \tilde u(z) \right)^{q-2} \\
	& \qquad \cdot \mathcal{M} \mathcal{M} \lapla{\frac{\frac 1\p -\delta}{2}}  \tilde u (z) \,  \mathcal{M} \mathcal{M} \lapla{\frac{1/(2q)}{2}} \tilde u(z) \ \lapin{\frac 1\p - \eps} |\varphi|(z) \, dz\\
	& \quad +  \int_{\R}  \left(\chi_{B_{2^L\rho}}\mathcal{M}\mathcal{M} \lapla{\frac{\frac 1\p - \delta}{2} }  \tilde u(z) \right)^{q-2} \\
	& \qquad \cdot \lapin{(\frac 1\p -\delta)(q-1)+ \frac{1}{2q} + \eps -1} \left(\mathcal{M} \mathcal{M} \lapla{\frac{\frac 1\p -\delta}{2}}  \tilde u \, \mathcal{M} \mathcal{M} \lapla{\frac{1/(2q)}{2}} \tilde  u\right)(z) \ \lapin{\frac 1\p - \eps} |\varphi|(z) \, dz.
	\end{split}
	\end{align}
	The integrals appearing on the right-hand side do make sense as they can be traced back to $\tilde u\in W^{\frac 1q,q}(\R)$ by applying H\"older's inequality and the following estimates:
	First we have by Hardy-Littlewood maximal inequality and Sobolev's inequality  \Cref{app:thm:sobinequ}
	\[
	\begin{split}
	& \| (\chi_{B_{2^L\rho}}\mathcal{M}\mathcal{M} \lapla{\frac{\frac 1\p - \delta}{2} } \tilde u)^{q-2}\|_{L^{((\frac 1\p - \delta)(q-2))^{-1}}} \\
	& = \left(\int_{\R} | \chi_{B_{2^L\rho}}(z)\mathcal{M}\mathcal{M} \lapla{\frac{\frac 1\p - \delta}{2} } \tilde u(z)|^{(\frac 1\p-\delta)^{-1}} \, dz \right)^{(\frac 1\p-\delta)(q-2)} \\
	& \aleq \left(\int_{\R} | \lapla{\frac{\frac 1\p - \delta}{2} } \tilde u(z)|^{(\frac 1\p-\delta)^{-1}} \, dz \right)^{(\frac 1\p-\delta)(q-2)} \\
	& \aleq [\tilde u]^{q-2}_{W^{\frac 1q,q}(\R)} < \infty
	\end{split}
	\]
	and with additional help of \Cref{app:thm:classobinequ}
	\[
	\begin{split}
	&\|\lapin{(\frac 1\p -\delta)(q-1)+ \frac{1}{2q} + \eps -1} (\chi_{B_{2^L\rho}}\mathcal{M}\mathcal{M} \lapla{\frac{\frac 1\p - \delta}{2} }  \tilde u)^{q-2}\|_{L^{(1-\frac{1}{2q}- (\frac 1\p-\delta)-\eps)^{-1}}} \\
	& \aleq  \| (\chi_{B_{2^L\rho}}\mathcal{M}\mathcal{M} \lapla{\frac{\frac 1\p - \delta}{2} }  \tilde u)^{q-2}\|_{L^{((\frac 1\p - \delta)(q-2))^{-1}}} \\
	& \aleq  [\tilde u]^{q-2}_{W^{\frac 1q,q}(\R)} < \infty.
	\end{split}
	\]
	In a similar fashion with slightly adapted H\"older's inequality we observe
	\[
	\begin{split}
	& \left\|\mathcal{M} \left(\mathcal{M}\lapla{\frac{\frac 1\p -\delta}{2}}  \tilde u\,  \mathcal{M}\lapla{\frac{1/(2q)}{2}}  u\right)\right\|_{L^{\frac{1}{\frac{1}{2q}+(\frac 1\p-\delta)}}} \\
	&\aleq  \|\lapla{\frac{\frac 1\p - \delta}{2} }  \tilde u \|_{L^{\frac 1\p-\delta}} \ \|\lapla{\frac{1}{4q} }  \tilde u\|_{L^{2q}} \\ & \aleq [\tilde u]^{2}_{W^{\frac 1q,q}(\R)} < \infty,
	\end{split}
	\]
	and hence also
	\[
	\begin{split}
	\left\|  \mathcal{M} \mathcal{M} \lapla{\frac{\frac 1\p -\delta}{2}} \tilde u  \,  \mathcal{M} \mathcal{M} \lapla{\frac{1/(2q)}{2}}  \tilde u\right\|_{L^{\frac{1}{\frac{1}{2q}+(\frac 1\p-\delta)}}} &  \aleq [\tilde u]^{2}_{W^{\frac 1q,q}(\R)} ,\\
	\left\|  \lapin{(\frac 1\p -\delta) (q-1)+ \frac{1}{2q} + \eps -1}  \left(\mathcal{M} \left(\mathcal{M}\lapla{\frac{\frac 1\p -\delta}{2}} \tilde u\,  \mathcal{M}\lapla{\frac{1/(2q)}{2}} \tilde u\right)\right)\right\|_{L^{\frac{1}{1-(\frac 1\p-\delta)(q-2)-\eps}}} & \aleq [\tilde u]^{2}_{W^{\frac 1q,q}(\R)} ,\\
	\left\| \lapin{(\frac 1\p -\delta)(q-1)+ \frac{1}{2q} + \eps -1} \left(\mathcal{M} \mathcal{M} \lapla{\frac{\frac 1\p -\delta}{2}}  \tilde u \,  \mathcal{M} \mathcal{M} \lapla{\frac{1/(2q)}{2}}  \tilde u\right)\right\|_{L^{\frac{1}{1-(\frac 1\p-\delta)(q-2)-\eps}}} & \aleq  [\tilde u]^{2}_{W^{\frac 1q,q}(\R)},  
	\end{split}
	\]
	as well as
	\[
	\begin{split}
	\|\lapin{\frac 1\p - \eps} |\varphi| \|_{L^{\frac 1 \eps}} & \aleq \|\varphi\|_{L^{\p}}.
	\end{split}
	\]
	Therefore, we conclude by composing these observations and recalling $\|\varphi\|_{L^{\p}}\leq 1$, that
	\[
	\begin{split}
	& \bigg\|  \int_{B_{2^L\rho}} \int_{B_{2^L\rho}} |x-y|^{(\frac 1\p - \delta) (q-2)}\left(\chi_{B_{2^L\rho}}\mathcal{M}\mathcal{M} \lapla{\frac{\frac 1\p - \delta}{2} }\tilde u(x) \right)^{q-2} \\
	& \cdot  \avint_{x\triangleright y}  |\tilde u(z_1)-\tilde u(y)|^2  \, dz_1  \avint_{x\triangleright y} ||z-z_2|^{\frac 1\p-1} - |z-x|^{\frac 1\p -1}| \, dz_2 \frac{d y \, dx}{\rho(x,y)^2 } \bigg\|_{L^{\frac{\p}{\p-1}}(B_{2^K\rho})}\\
	& \aleq [\tilde u]^q_{W^{\frac 1q,q}(\R)} .
	\end{split}
	\]
	In view of the desired decay estimate \Cref{pr:decayest}, we need to take advantage of the local Sobolev regularity of $u$, i.e. $u\in W^{\frac 1q,q}(B(x_0,r))$. Recall that $u$ and $\tilde u$ coincide on $B(x_0,r)$ by construction, cf.~\Cref{reg:rm:extension}, and therefore also their Gagliardo seminorms on subsets of $B(x_0,r)$. For this reason, we localize the previous estimates with help of the factor $\chi_{B_{2^L\rho}}$ to gain for some $\sigma=\sigma(q)>0$ the upper bound
	\[
	[u]^q_{W^{\frac 1q,q }(B_{2^{2L}\rho})} + \sum_{k=1}^\infty 2^{-\sigma (L+k)} [\tilde u]^q_{W^{\frac 1q,q }(B_{2^{2L+k}\rho})} 
	\]
	instead. In particular, we applied \Cref{app:loc3} to \eqref{reg:est:rhs11stterm}, localized maximal inequality \Cref{app:loc4} and localized Sobolev inequality \Cref{app:locSobinequ} here. 
	
	We proceed with the remaining terms \eqref{eq:avu2}, \eqref{eq:avu3}, and \eqref{eq:avu4}, shortly 
	\[
	\begin{split}
	& U(x,y,z) :=   \eqref{eq:avu2} + \eqref{eq:avu3} + \eqref{eq:avu4} \\
	& =  \avint_{x\triangleright y}  |\tilde u(z_1)-\tilde u(y)| \, dz_1 \, |\tilde u(y)+\tilde u(x) -2 \tilde u(z)| \\
	& \quad + |\tilde u(y)-\tilde u(x)|\,  \avint_{x\triangleright y}  |\tilde u(z_1)- \tilde u(y)| \, dz_1  \\
	& \quad +  |\tilde u(y)-\tilde u(x)|\,  |\tilde u(y)+\tilde u(x) -2 \tilde u(z)|,
	\end{split}
	\]
	in the following similar manner. We begin with estimating the first factors $\avint_{x\triangleright y} |\tilde u(z_1)-\tilde u(y)|\, dz_1$, respectively, $|\tilde u(y)-\tilde u(x)|$, by Proposition \ref{pr:estmax} and Lebesgue's differentiation theorem by 
	\[
	|x-y|^{\frac 1\p-\delta} \left(\mathcal{M}\mathcal{M} \lapla{\frac{\frac 1\p - \delta}{2} }\tilde u(x) + \mathcal{M}\mathcal{M} \lapla{\frac{\frac 1\p - \delta}{2} }\tilde u(y) \right).
	\]
	Together with Lemma \ref{la:threetermforGandH} we gain the upper bound
	\[
	\begin{split}
	& \int_{B_{2^L\rho}} \int_{B_{2^L\rho}} |x-y|^{(\frac 1\p - \delta) (q-1)}\left(\mathcal{M}\mathcal{M} \lapla{\frac{\frac 1\p - \delta}{2} } \tilde u(x) + \mathcal{M}\mathcal{M} \lapla{\frac{\frac 1\p - \delta}{2} } \tilde u(y)  \right)^{q-1} \\
	& \cdot |x-y|^{\frac 1q+\eps} \brac{\mathcal{M}\mathcal{M} \lapla{\frac{1}{4q}} \tilde u(x) + \mathcal{M}\mathcal{M} \lapla{\frac{1}{4q}} \tilde u(y) + \mathcal{M}\mathcal{M} \lapla{\frac{1}{4q}} \tilde u(z)}\\
	& \cdot  k_{ \frac 1\p - \frac{1}{2q} - \eps,\frac{1}{\p}}(x,y,z) \frac{d y \, dx}{\rho(x,y)^2 }.
	\end{split}
	\]
	We arrive for some $\varphi\in C_c^\infty(\R)$ with $\|\varphi\|_{L^{\p}}\leq 1$  at
	\[
	\begin{split}
	& \bigg\| \int_{B_{2^L\rho}}\int_{B_{2^L\rho}} |x-y|^{(\frac 1\p - \delta) (q-1)}\left( \mathcal{M}\mathcal{M} \lapla{\frac{\frac 1\p - \delta}{2} }\tilde u(x)+ \mathcal{M}\mathcal{M} \lapla{\frac{\frac 1\p - \delta}{2} }\tilde u(y) \right)^{q-1} \\
	& \quad \cdot  U(x,y,\cdot) \avint_{x\triangleright y} ||\cdot-z_2|^{\frac 1\p-1} - |z-x|^{\frac 1\p -1}| \, dz_2 \frac{d y \, dx}{\rho(x,y)^2 } \bigg\|_{L^{\frac{\p}{\p-1}}(B_{2^K\rho})} \\
	& \aleq  \iiint_{\R^3} \left( (\chi_{B_{2^L\rho}}\mathcal{M}\mathcal{M} \lapla{\frac{\frac 1\p - \delta}{2} } \tilde u)^{q-1}(x)+ (\chi_{B_{2^L\rho}}\mathcal{M}\mathcal{M} \lapla{\frac{\frac 1\p - \delta}{2} } \tilde u)^{q-1}(y) \right) \\
	&  \quad \cdot \brac{\mathcal{M}\mathcal{M} \lapla{\frac{1}{4q}} \tilde u(x) + \mathcal{M}\mathcal{M} \lapla{\frac{1}{4q}} \tilde u(y) + \mathcal{M}\mathcal{M} \lapla{\frac{1}{4q}} \tilde u(z)}\, k_{ \frac 1\p - \frac{1}{2q} - \eps,\frac{1}{\p}}(x,y,z) \\
	& \quad \cdot   |\varphi(z)|  \  \frac{d x \, dy \, dz}{|x-y|^{2 - (\frac 1\p-\delta)(q-1) - \frac 1q - \eps} } \\
	& \aleq \|\lapla{\frac{\frac 1\p - \delta}{2} } \tilde u \|^{q-1}_{L^{\frac 1\p-\delta}} \ \|\lapla{\frac{1}{4q} }  \tilde u\|_{L^{2q}} \  \|\varphi\|_{L^{\p}} \\
	& \aleq [\tilde u]^q_{W^{\frac 1q,q}(\R)} ,
	\end{split}	
	\]	
	where we applied \Cref{app:pro:FGHKernel} and then estimated analogously to the case \eqref{eq:avu1}. In consideration of $u\in W^{\frac 1q,q}(B(x_0,r))$, we make use of the factor $\chi_{B_{2^L\rho}}$ again and localize the estimate as described in the case above. 
\end{proof}

It remains to treat the second term appearing on the right-hand side of the projection estimate \eqref{eq:reg:projofop}.

\begin{lemma}(Right-hand side estimates II)\label{la:reg:rhs2}
	Let $q\geq 2$, $\tfrac 1q - \tfrac 1\p>0$ small, and $\gamma:\R /\Z \rightarrow \R^3$ be a locally critical embedding of $\tp^{q+2,q}$ in the interval $B(x_0,r) \subset \R/\Z$, in the sense of \Cref{def:wcp}. Denote the unit tangent field of $\gamma$ by  $u:\RZ\rightarrow \S^2$ such that $\int_{\RZ} u = 0$ and let $\tilde{u}$ be a $W^{\frac 1q,q}$-extension of $u|_{B(x_0,r)}$ from $B(x_0,r)$ to $\R$ as discussed in \Cref{reg:rm:extension}. 
	Moreover, for $y_0 \in B(x_0,\tfrac r2)$, choose $\rho>0$ such that $B_{2^{2L}\rho}:= B(y_0,2^{2L}\rho) \subset B(x_0,r)$ for large $L\in\N$.

	Then we have
	\begin{align*}
		& \| \chi_{B_{2^K\rho }} u \, \wedge  \,  \Gamma_{\frac 1\p,B_{2^{L}\rho} \times B_{2^{L}\rho}} u  \|_{L^{\frac{\p}{\p-1}}}\\
		& \aleq [u]^q_{W^{\frac 1q,q }(B_{2^{2L}\rho})} + [u]^{2q-3}_{W^{\frac 1q,q }(B_{2^{2L}\rho})}  + \sum_{l=1}^\infty 2^{-\sigma (K+l)} [\tilde u]^{q-1}_{W^{\frac 1q,q }(B_{2^{2L+l}\rho})} \\
		& \quad + 2^{-\sigma K} [u]^{q-1}_{W^{\frac 1q,q }(B_{2^{2L}\rho})} +  \rho \left(\E^{q}(u) +  r^{-3} + [u]_{W^{\frac{1}{q},q}(B(x_0,r))}^q \right),
	\end{align*}
	for any large enough $K\in\N$ and $L\in\N$ with $L\gg K$. 
	The constant, besides depending on $q$, may also depend on global properties of $u$ such as $\|u\|_{L^\infty}$, $\|\tilde u\|_{L^\infty}$, $[u]_{W^{\frac{1}{q},q}(B(x_0,r))}$, and $[\tilde{u}]_{W^{\frac 1q,q}(\R)}$.
\end{lemma}
\begin{proof}
	We prove the statement along the lines of \cite[Lemma 3.8]{BRS19}. We first have by duality that 
	\[
	\| \chi_{B_{2^K\rho }} u \, \wedge  \,  \Gamma_{\frac 1\p,B_{2^{L}\rho}\times B_{2^{L}\rho}} u  \|_{L^{\frac{\p}{\p-1}}} \aleq \int_{\R} u(z) \, \wedge \, \Gamma_{\frac 1\p,B_{2^{L}\rho}\times B_{2^{L}\rho}} u (z) \cdot \psi(z) \, dz
	\]
	for a map $\psi\in C^\infty_c(B_{2^K\rho}, \R^3)$ with $\| \psi\|_{L^{\p}}\leq 1$. Hence, it is sufficient to show for a scalar $\psi\in C^\infty_c(B_{2^K\rho})$ that 
	\[
	\begin{split}
	&\left| \int_{\R} u(z) \, \wedge \, \Gamma_{\frac 1\p,B_{2^{L}\rho}\times B_{2^{L}\rho}} u (z) \cdot \psi(z) \, dz \right| \\
	& \aleq [u]^q_{W^{\frac 1q,q }(B_{2^{2L}\rho})} + [u]^{2q-3}_{W^{\frac 1q,q }(B_{2^{2L}\rho})}  + \sum_{l=1}^\infty 2^{-\sigma (K+l)} [\tilde u]^{q-1}_{W^{\frac 1q,q }(B_{2^{2L+l}\rho})} \\
	& \quad + 2^{-\sigma K} [u]^{q-1}_{W^{\frac 1q,q }(B_{2^{2L}\rho})} +  \rho \left(\E^{q}(u) +  r^{-3} + [u]_{W^{\frac{1}{q},q}(B(x_0,r))}^q \right).
	\end{split}
	\]
	Note that $u$ in 
	\[
		\left| \int_{\R} u(z) \, \wedge \, \Gamma_{\frac 1\p,B_{2^{L}\rho}\times B_{2^{L}\rho}} u (z) \cdot \psi(z) \, dz \right|
	\]
	is only considered on $B(x_0,r)$ due to $\supp\psi \subset B_{2^K\rho}$, the local definition of $\Gamma$ for $B_{2^L\rho}$ in \eqref{def:GammaQB}, and $B_{2^K\rho}\subset B_{2^L\rho} \subset B(x_0,r)$ ($L\gg K$). Therefore, using that $u$ coincides with $\tilde u$ on $B(x_0,r)$ by construction, cf.~\Cref{reg:rm:extension}, we can continue with $\tilde u$ from now on as it is globally $W^{\frac 1q,q}$-regular in contrast to $u$, which is only in $W^{\frac 1q,q}(B(x_0,r),\R^3)$. However, if later on we obtain $W^{\frac 1q,q}$-seminorms of $\tilde u$ restricted to subsets of $B(x_0,r)$, we may switch back to the original function $u$. 

	With the help of usual cutoff functions, i.e. $\eta_{B_R}\in C^\infty_c (B_{2R})$ with $\eta_{B_R} \equiv 1$ on $B_R$ and $\|\nabla^k \eta_{B_R}\|_{L^\infty} \aleq R^{-k}$ for $R>0$, we decompose the left-hand side into 
	\[
	\int_{\R} \tilde u(z) \, \wedge \, \Gamma_{\frac 1\p,B_{2^{L}\rho}\times B_{2^{L}\rho}} \tilde u (z) \, \psi(z) \, dz = I + \sum_{l=1}^\infty II_l,
	\]
	where 
	\[
	\begin{split}
	I &:= \int_{\R} \lapla{\frac{1}{2\p}} (\eta_{B_{2^{2K}\rho}} \lapin{\frac{1}{\p}} \psi)(z) \tilde u(z) \, \wedge \, \Gamma_{\frac 1\p,B_{2^{L}\rho}\times B_{2^{L}\rho}} \tilde u (z)  \, dz, \\
	II_l &:= \int_{\R} \lapla{(\frac 1q - \frac{1}{\p})} \left( \lapla{\frac{1}{2\p}} (\eta_{{B_{2^{2K+l+1}\rho}} \setminus B_{2^{2K+l}\rho}} \lapin{\frac{1}{\p}} \psi) \tilde u \right)(z)  \, \wedge \, \lapin{(\frac 2q - \frac{2}{\p})} \Gamma_{\frac 1\p,B_{2^{L}\rho}\times B_{2^{L}\rho}} \tilde u (z)  \, dz.
	\end{split}
	\]
	
	We first focus on the terms $II_l$. 
	For that, we define  
	\[
		\varphi (z)  :=\lapla{(\frac 1q - \frac{1}{\p})}  \left( \lapla{\frac{1}{2\p}} (\eta_{{B_{2^{2K+l+1}\rho}} \setminus B_{2^{2K+l}\rho}} \lapin{\frac{1}{\p}} \psi) \tilde u \right)(z),
	\]
	and get by equalities \eqref{eq:QRiesz} as well as \eqref{eq:reg:RieszopGamma}, and the boundedness of the factor $ k(x,y)^{-\frac{q+2}{2}}$ by \Cref{rm:reg:boundedfactorc},
	that 
	\begin{align} \label{reg:eq:splitPhi1}
	\begin{split}
	& \left|  \int_{\R}  \varphi(z) \wedge \lapin{(\frac 2q - \frac{2}{\p})} \Gamma_{\frac 1\p,B_{2^{L}\rho}\times B_{2^{L}\rho}} \tilde u (z)  \, dz \right| \\
	& = \left|  \int_{\R} \varphi(z) \wedge \Gamma_{\frac 2q - \frac 1\p,B_{2^{L}\rho}\times B_{2^{L}\rho}} \tilde u (z)  \, dz \right| \\
	& =  c \, \bigg| \iint_{B_{2^{L}\rho}^2} |\avint_{x\triangleright y} \tilde u(z_0)-\tilde u(x) \, dz_0 |^{q-2} \, k(x,y)^{-\frac{q+2}{2}} \\ 
	& \quad \left( \avint_{x\triangleright y} \avint_{x\triangleright y} ((\tilde u\, \wedge)_{ij} (z_1)- (\tilde u\, \wedge)_{ij} (x)) \cdot \left(\lapin{(\frac 2q - \frac1\p)} \varphi (z_2)- \lapin{(\frac 2q - \frac1\p)} \varphi (x)\right) \, dz_1 \, dz_2  \right)_{i=1}^3 \frac{dx\, dy}{\rho(x,y)^2} \bigg| \\
	& \aleq  \iint_{B_{2^{L}\rho}^2}\left( \avint_{x\triangleright y} |\tilde u(z_0)-\tilde u(x)| \, dz_0 \right)^{q-2} \avint_{x\triangleright y} |\tilde u(z_1)-\tilde u(x)| \, dz_1 \\
	& \quad \cdot \avint_{x\triangleright y} |\lapin{(\frac 2q - \frac1\p)} \varphi(z_2) - \lapin{(\frac 2q - \frac1\p)} \varphi (x)| \, dz_2 \frac{dx\, dy}{\rho(x,y)^2} \\
	& \aleq [\tilde u]^{q-1}_{W^{\frac 1q,q} (B_{2^{L}\rho})} [\lapin{(\frac 2q - \frac1\p)}\varphi]_{W^{\frac 1q,q} (B_{2^{L}\rho})} \\
	& \aleq [u]^{q-1}_{W^{\frac 1q,q} (B_{2^{L}\rho})} \| \varphi\|_{L^{\frac{1}{\frac 2q - \frac 1\p}}},
	\end{split}
	\end{align}
	where we applied the identity \Cref{la:id2} and Sobolev inequality \Cref{app:thm:sobinequ} at the end. It remains to estimate $\| \varphi\|_{L^{\frac{1}{\frac 2q - \frac 1\p}}}$. For this purpose we introduce the three term commutator for any $\alpha > 0$ by 
	\[
	\label{eq:threetermcommdef}
	H_\alpha(f,g) = \lapla{\frac \alpha 2} (fg) - f\lapla{\frac \alpha 2} g - g \lapla{\frac \alpha 2} f.
	\]
	Taking advantage of the three term commutator estimate \Cref{app:est:threetermcomm} and, in addition, of the uniform boundedness of $\tilde u$, cf.~\Cref{reg:rm:extension}, H\"older's inequality, Sobolev inequalities \Cref{app:thm:classobinequ} and \Cref{app:thm:sobinequ} for $\frac 1q -\frac 1\p >0$ very small, and \Cref{app:loc1}, we estimate
	\[
	\begin{split}
	& \|\varphi_1\|_{L^{\frac{1}{\frac 2q - \frac 1\p}}} \\
	& \aleq \|\tilde u\|_{L^\infty}\|\lapla{\frac 12 (\frac 1\p + 2(\frac{1}{q}-\frac 1\p))} (\eta_{{B_{2^{2K+l+1}\rho}} \setminus B_{2^{2K+l}\rho}} \lapin{\frac{1}{\p}} \psi) \|_{L^{\frac{1}{\frac 2q - \frac 1\p}}}\\
	& \quad + \| \lapla{\frac{1}{q}-\frac 1\p} \tilde u \|_{L^{\frac{1}{\frac 2q - \frac2\p}}} \| \lapla{\frac{1}{2\p}} (\eta_{{B_{2^{2K+l+1}\rho}} \setminus B_{2^{2K+l}\rho}} \lapin{\frac{1}{\p}} \psi) \|_{L^{\p}} \\
	& \quad+  \|H_{\frac 2q - \frac2\p}(\tilde u, \lapla{\frac{1}{2\p} } (\eta_{{B_{2^{2K+l+1}\rho}} \setminus B_{2^{2K+l}\rho}} \lapin{\frac{1}{\p}} \psi)) \|_{L^{\frac{1}{\frac 2q - \frac 1\p}}}\\
	& \aleq  \left( \|\tilde  u \|_{L^\infty} + \| \lapla{\frac{1}{2\p}} \tilde u \|_{L^{\p}} \right) \, 2^{-(K+l+1)\frac{\p-1}{\p}}\|\psi \|_{L^{\p}}\\
	& \quad+  \| \lapla{\frac{1}{q}-\frac 1\p} \tilde u \|_{L^{\frac{1}{\frac 2q - \frac2\p}}} \|\lapla{\frac 12 (\frac 1\p + 2(\frac{1}{q}-\frac 1\p))} (\eta_{{B_{2^{2K+l+1}\rho}} \setminus B_{2^{2K+l}\rho}} \lapin{\frac{1}{\p}} \psi) \|_{L^{\frac{1}{\frac 2q - \frac 1\p}}} \\
	& \aleq 2^{-(K+l)\sigma} \left( \|\tilde u\|_{L^\infty} + [\tilde u ]_{W^{\frac 1q,q }(\R)} \right)
	\end{split}
	\]
	for some small $\sigma = \sigma (q)>0$.

	Now we switch to the term $I$. We again start by introducing
	\[
	\varphi(z):=\eta_{B_{2^{2K}\rho}} \lapin{\frac{1}{\p}} \psi (z),
	\]
	and observe by \Cref{app:loc1} together with our assumption $\| \psi\|_{L^{\p}}\leq 1$ that
	\begin{equation}\label{reg:est:phibounded}
	\| \lapla{\frac{1}{2\p}} \varphi \|_{L^{\p}} \aleq 1.
	\end{equation}
	Next we split $I$ into three terms with respect to the three term commutator \eqref{eq:threetermcommdef} as 
	\[
	\begin{split}
	I_1 & := \int_{\R} \lapla{\frac{1}{2\p}} (\varphi \tilde u)(z) \, \wedge \, \Gamma_{\frac 1\p,B_{2^{L}\rho }\times B_{2^{L}\rho}} \tilde u (z)  \, dz, \\
	I_2 & := - \int_{\R} \varphi(z) \lapla{\frac{1}{2\p}} \tilde u(z) \, \wedge \, \Gamma_{\frac 1\p,B_{2^{L}\rho}\times B_{2^{L}\rho}} \tilde u (z)  \, dz, \\
	I_3 & := - \int_{\R} H_{\frac 1\p} (\varphi,\tilde u) (z) \, \wedge \, \Gamma_{\frac 1\p,B_{2^{L}\rho}\times B_{2^{L}\rho}} \tilde u (z)  \, dz. 
	\end{split}
	\]
	For the term $I_1$, we have up to a constant depending on $q$ by \eqref{eq:QRiesz}
		\[
	\begin{split}
	|I_1|  & =  c \left|\left(Q^{(q)}_{B_{2^{L}\rho}\times B_{2^{L}\rho}} (u,(\varphi \tilde u \, \wedge)_{ij}^T) \right)_{i=1}^3\right|\\
	& =c  \bigg| \iint_{B_{2^{L}\rho}^2} (\avint_{x\triangleright y} \tilde u(z_0)-\tilde u(x) \, dz_0 )^{q-2}  \, k(x,y)^{-\frac{q+2}{2}}  \\
	& \quad \left( \avint_{x\triangleright y} \avint_{x\triangleright y} (\tilde u (z_1)- \tilde u (x)) \cdot \left( (\varphi \tilde u \, \wedge)_{ij}^T (z_2)-  (\varphi \tilde u \, \wedge)_{ij}^T (x)\right) \, dz_1 \, dz_2  \right)_{i=1}^3 \frac{dy\, dx}{\rho(x,y)^2} \bigg|.
	\end{split}
	\]
	Note that $\tilde u$ is considered only on $B_{2^L\rho}\subset B(x_0,r)$ here anymore. Hence, we can exchange $\tilde u$ back to $u$, as they coincide on $B(x_0,r)$ by construction, cf.~\Cref{reg:rm:extension}.
	Then we simplify the expression by the skew-symmetry of $\varphi u\wedge$ and split $I_1$ by triangle inequality into 
	\[
	\begin{split}
		& \sum_{m=1}^4 \bigg| \iint_{D_m} (\avint_{x\triangleright y} u(z_0)-u(x) \, dz_0 )^{q-2} \, k(x,y)^{-\frac{q+2}{2}}   \\
		& \left( \avint_{x\triangleright y} \avint_{x\triangleright y} (u (z_1)- u (x)) \cdot \left( (\varphi u \, \wedge)_{ij} (z_2)-  (\varphi u \, \wedge)_{ij} (x)\right) \, dz_1 \, dz_2  \right)_{j=1}^3 \frac{dy\, dx}{\rho(x,y)^2} \bigg| \\
		& \aleq  \sum_{m=1}^4 \sum_{j=1}^3 |Q^{(q)}_{D_m}(u,(\varphi u \, \wedge)_{ij})|,
	\end{split}
	\]
	where 
	\[
	\begin{split}
		D_1 & = \R/\Z \times \R/\Z,\\
		D_2 & = (\R/\Z \setminus B_{2^{L}\rho}) \times B_{2^{L}\rho} , \\
		D_3 & = B_{2^{L}\rho} \times (\R/\Z \setminus B_{2^{L}\rho}),\\
		D_4 & = (\R/\Z \setminus B_{2^{L}\rho}) \times (\R/\Z \setminus B_{2^{L}\rho}).
	\end{split}
	\]
	For the term with integration domain $D_1$, we employ the assumption of locally critical embeddedness of $\gamma$ with respect to $\tp^{q+2,q}$ in $B(x_0,r)$. Hence \Cref{reg:th:critmap} yields that $u$ is a critical map of $\E^{q}_{\eta}$ in $B(x_0,r)$ for some suitable smooth function $\eta  \in C_c^\infty(B(x_0,\tfrac r2),[0,\infty))$. As therefore $u$ fulfills the Euler-Lagrange equations \Cref{la:EulerLagrange} for  any $(\varphi u \wedge)_{ij} \in T_u\S^2$, $j=1,2,3$, we observe
	\[
	\begin{split}
	|Q^{(q)} (u, (\varphi u \wedge)_{ij}) | \aleq \sum_{k=1}^3 |R^{k,(q)}(u, (\varphi u \wedge)_{ij}) | + \sum_{k=4}^7 |R^{k,(q)}_\eta(u, (\varphi u \wedge)_{ij}) |,
	\end{split}
	\]
	where we set $R^{k,(q)} = R^{k,(q+2,q)}$ for $q\geq2$. Therefore, \Cref{pr:reg:estremainderwedge} leads to
	\begin{align*}
		& |Q^{(q)} (u, (\varphi u \wedge)_{ij}) |  \\
		& \aleq  [u]_{W^{\frac{1}{q},q}(B_{2^{3K}\rho})}^{2q-3}  + [u]_{W^{\frac{1}{q},q}(B_{2^{3K}\rho})}^{q+1}\\ 
		& + \sum_{l=1}^\infty 2^{-\sigma(3K+l)}[\tilde{u}]_{W^{\frac{1}{q},q}(B_{2^{3K + l}\rho})}^{q-1}  +  \rho \left(\E^{q}(u) +  r^{-3} + [u]_{W^{\frac{1}{q},q}(B(x_0,r))}^q\right)
	\end{align*}
	for $0 < \sigma := \tfrac 2q \leq  1$. 
	We estimate the terms with domains $D_m$ for $m=2,3,4$ up to positive constant by methods similar to \eqref{reg:est:remainder1global} by
	\[
	\begin{split}
	\|u\|_{L^\infty}\iint_{D_m}  \left( \avint_{x\triangleright y} |u(z_0)-u(x)| \, dz_0 \right)^{q-1}  \left( \avint_{x\triangleright y} |\varphi(z_1) - \varphi (x)| \, dz_1 \right) \frac{dx\, dy}{\rho(x,y)^2}. 
	\end{split}
	\]
	Due to the fact that $u$ is not known to be in the fractional Sobolev space $W^{\frac 1q,q}$ outside of $B(x_0,r)$, we need to distinguish for the domain $D_2$ the cases $(\R/\Z\setminus B(x_0,r))\times B_{2^{L}\rho}$ and $(B(x_0,r)\setminus B_{2^{L}\rho})\times B_{2^{L}\rho}$. The first case we estimate along the lines of \eqref{reg:eq:roughestimateR1tail} since $\rho(x,y)\geq \tfrac r4$ and to the second case we can apply due to $\supp \varphi \subset B_{2^{2K}\rho}$ an adapted version of \Cref{la:RminusBest}, which gives in total with $\tilde u \in W^{\frac 1q,q}(\R)$ the upper bound up to a positive constant
	\[
		\rho \, r^{-3} + \sum_{l=1}^\infty 2^{-\sigma(L+l)}[\tilde u]^{q-1}_{W^{\frac 1q,q}(B_{2^{L+2K+l}\rho})}  .
	\]
	We treat the term with domain $D_3$ similarly by symmetry. For the integration domain $D_4$ we can deduce the same bound in the following manner. First we also distinguish several subdomains according to the only locally known fractional Sobolev regularity of $u$ in $B(x_0,r)$. In particular, we have to study the cases 
	\begin{align*}
		(\R/\Z \setminus B_{2^L\rho})\times(\R/\Z \setminus B_{2^L\rho}) & = (\R/\Z \setminus B(x_0,r))\times(\R/\Z \setminus B(x_0,r)) \\
		& \quad  \cup (\R/\Z \setminus B(x_0,r))\times(B(x_0,r) \setminus B_{2^L\rho}) \\
		& \quad  \cup (B(x_0,r) \setminus B_{2^L\rho})\times(\R/\Z \setminus B(x_0,r))\\
		& \quad \cup  (B(x_0,r) \setminus B_{2^L\rho})\times  (B(x_0,r) \setminus B_{2^L\rho}),
	\end{align*}
	where we observe in the first three cases that the double integral either equals $0$ due to $\supp \varphi \subset B_{2^{2K}\rho}$ or can be estimated with help of $\rho(x,y)\geq r$ similarly to \eqref{reg:eq:roughestimateR1tail}. For the last appearing case we use an adapted version of \Cref{la:RminusBest}.
	
	For the term $I_2$, we first observe by the definition of $\Gamma_{\beta, B\times B}$ in \eqref{def:GammaQB} 
	\[
	\begin{split}
	I_2 & \aeq - \int_{B_{2^{L}\rho}} \int_{B_{2^{L}\rho}} | \avint_{x\triangleright y} \tilde u(z_0)-\tilde u(x) \, dz_0 |^{q-2} \, k(x,y)^{-\frac{q+2}{2}} \\
	& \avint_{x\triangleright y} (\tilde u(z_1)-\tilde u(x))^T \, dz_1 \ \avint_{x\triangleright y} (\lapin{\frac 1\p} (\varphi\lapla{\frac{1}{2\p}}\tilde u\wedge)(z_2) - \lapin{\frac 1\p} (\varphi\lapla{\frac{1}{2\p}}\tilde u\wedge)(x)) \, dz_2 \frac{dy\, dx}{\rho(x,y)^2}.
	\end{split}
	\]
	Now since $\tilde u\wedge$ is orthogonal to $\tilde u$, we rewrite 
	\[
	\begin{split}
	& \avint_{x\triangleright y} (\tilde u(z_1)-\tilde u(x))^T \, dz_1 \ \avint_{x\triangleright y} (\lapin{\frac 1\p} (\varphi\lapla{\frac{1}{2\p}}\tilde u\wedge)(z_2) - \lapin{\frac 1\p} (\varphi\lapla{\frac{1}{2\p}}\tilde u\wedge)(x)) \, dz_2 \\
	& = \avint_{x\triangleright y}  (\tilde u(z_1)-\tilde u(x))^T \, dz_1 \  \avint_{x\triangleright y} \tilde{\Phi}(z_2,x)  \, dz_2,
	\end{split}
	\]
	where 
	\[
	\tilde{\Phi}(z_2,x) := \lapin{\frac 1\p} (\varphi\lapla{\frac{1}{2\p}}\tilde u\wedge)(z_2) - \lapin{\frac 1\p} (\varphi\lapla{\frac{1}{2\p}}\tilde u\wedge)(x) - \tfrac 12 (\tilde u\wedge(z_2) - \tilde u\wedge (x)) (\varphi(x)+ \varphi(y)).
	\]
	Therefore, we have by \Cref{rm:reg:boundedfactorc}
	\begin{equation}\label{eq:reg:estI2}
	\begin{split}
	|I_2| & \aleq \int_{B_{2^{L}\rho}} \int_{B_{2^{L}\rho}} \left( \avint_{x\triangleright y} |\tilde u(z_0)-\tilde u(x)| \, dz_0 \right)^{q-1}  \avint_{x\triangleright y} |\Phi(z_2,x)| \, dz_2 \frac{dy\, dx}{\rho(x,y)^2},
	\end{split}
	\end{equation}
	where
	\[
	\Phi(z_2,x):=  \lapin{\frac 1\p} (\varphi\lapla{\frac{1}{2\p}}\tilde u)(z_2) - \lapin{\frac 1\p} (\varphi\lapla{\frac{1}{2\p}}\tilde u)(x) - \tfrac 12 (\tilde u(z_2) - \tilde u(x)) (\varphi(x)+ \varphi(y)).
	\]
	Now we are in the position to apply an adapted version of \cite[Lemma 6.6]{S15}. For that reason, we begin with defining $U:= \lapla{\frac{1}{2\p}}\tilde u$, and find
	\[
	\begin{split}
	\Phi(z_2,x)	& =  \lapin{\frac 1\p} (\varphi \, U)(z_2) - \lapin{\frac 1\p} (\varphi \,  U)(x) - \tfrac 12 (\tilde u(z_2) - \tilde u(x)) (\varphi(x)+ \varphi(y)) \\
	& \aleq  \int_{\R} (|z_2-z|^{\frac 1\p -1}- |x-z|^{\frac 1\p -1}) \,  U(z) \, \varphi(z) \, dz \\
	& -\tfrac 12 \int_{\R} (|z_2-z|^{\frac 1\p -1}- |x-z|^{\frac 1\p -1}) \,  U(z) \, (\varphi(x)+\varphi(y)) \, dz \\
	& \aleq -\tfrac 12 \int_{\R} (|z_2-z|^{\frac 1\p -1}- |x-z|^{\frac 1\p -1}) \,  U(z) \, (\varphi(x)+\varphi(y)-2 \varphi(z)) \, dz. 
	\end{split}
	\]
	Together with \Cref{la:threetermforGandH} we obtain for $\frac 1\p < \frac 1q$ large enough and $\eps < \frac 1\p - \frac{1}{2q} < 1$ small enough
	\[
	\begin{split}
	& |\avint_{x\triangleright y} \Phi(z_2,x) \, dz_2 | \\
	& \aleq  \int_{\R} \left( \avint_{x\triangleright y}  | |z_2-z|^{\frac 1\p -1}- |x-z|^{\frac 1\p -1}| \, dz_2 \right)  |U(z)| \, |\varphi(x)+\varphi(y)-2 \varphi(z)| \, dz \\
	& \aleq  \int_{\R}   |x-y|^{\frac 1q+\eps} \, (\mathcal{M}\mathcal{M}\lapla{\frac{1}{4q}}\varphi(x) + \mathcal{M}\mathcal{M}\lapla{\frac{1}{4q}}\varphi(y) + \mathcal{M}\mathcal{M}\lapla{\frac{1}{4q}}\varphi(z) ) \\
	& \cdot k_{\frac 1\p - \frac{1}{2q}-\eps, \frac 1\p}(x,y,z) |U(z)| \, dz.
	\end{split}
	\]
	Furthermore, we have by \Cref{pr:estmax} for small $\delta >0$
	\[
	\begin{split}
	\left(\avint_{x\triangleright y} |\tilde u(z_0)-\tilde u(x)|\, dz_0 \right)^{q-1} \aleq |x-y|^{(\frac 1\p - \delta)(q-1)} \left(\mathcal{M}\mathcal{M} \lapla{\frac{\frac 1\p-\delta}{2}} \tilde u(x)\right)^{q-1}.
	\end{split}
	\] 
	We conclude, 
	\[
	\begin{split}
	|I_2| & \aleq \int_{\R} \int_{\R } \int_{\R }   |x-y|^{ \frac 1q+\eps + (\frac 1\p - \delta)(q-1) - 2} \chi_{B_{2^{L}\rho}}(x) \left( \mathcal{M}\mathcal{M} \lapla{\frac{\frac 1\p-\delta}{2}} \tilde u(x)\right)^{q-1}  \\
	& \cdot (\mathcal{M}\mathcal{M}\lapla{\frac{1}{4q}}\varphi(x) + \mathcal{M}\mathcal{M}\lapla{\frac{1}{4q}}\varphi(y) + \mathcal{M}\mathcal{M}\lapla{\frac{1}{4q}}\varphi(z) ) \\
	& \cdot k_{\frac 1\p - \frac{1}{2q}-\eps, \frac 1\p}(x,y,z) \,  |U(z)| \, dy \, dz \, dx .
	\end{split}
	\]
	Arguing along the lines of the proof of \Cref{la:reg:rhs1}, we therefore gain with assumption \eqref{reg:est:phibounded}
	\[
	\begin{split}
	|I_2| & \aleq  \|\lapla{\frac{\frac 1\p - \delta}{2} } \tilde u \|^{q-1}_{L^{\frac 1\p-\delta}} \ \|\lapla{\frac{1}{4q} }\varphi\|_{L^{2q}} \  \|U\|_{L^{\p}}\\
	& \aleq [\tilde u]^{q-1}_{W^{\frac 1q,q}(\R)} \|\lapla{\frac{1}{2\p}}\varphi \|_{L^{\p}} \,  [\tilde u]_{W^{\frac 1q,q}(\R)} \\
	& \aleq  [\tilde u]^{q}_{W^{\frac 1q,q}(\R)},
	\end{split}
	\]
	where we applied Sobolev's inequality in the last step. To take $u\in W^{\frac 1q,q}(B(x_0,r))$ into consideration, we localize this estimate by introducing the factor $\chi_{B_{2^{L}\rho}}$ in the inequality \eqref{eq:reg:estI2} and using \Cref{app:loc3}, localized maximal inequality \Cref{app:loc4}, and localized Sobolev inequality \Cref{app:locSobinequ}.
	
	With help of integration by parts, equality \eqref{eq:reg:RieszopGamma}, and H\"older's inequality, the last term $I_3$ can be bounded by
	\begin{align*}
		I_3 & = - \int_{\R} H_{\frac 1\p} (\varphi,\tilde u) (z) \, \wedge \, \Gamma_{\frac 1\p,B_{2^{L}\rho}\times B_{2^{L}\rho}} \tilde u (z)  \, dz \\
		& \aleq \int_{\R} |\lapla{\frac 1q - \frac 1\p} H_{\frac 1\p} (\varphi,\tilde u) (z) | \, | \lapin{\frac 2q - \frac 2\p} \Gamma_{\frac 1\p,B_{2^{L}\rho}\times B_{2^{L}\rho}} \tilde u (z) | \, dz \\
		& \aleq \|\lapla{\frac 1q -\frac 1\p} H_{\frac 1\p} (\varphi,\tilde u) \|_{L^{(\frac 1q - \frac 1 \p)^{-1}}} \, \|\Gamma_{\frac 2q- \frac 1\p,B_{2^{L}\rho}\times B_{2^{L}\rho}} \tilde u\|_{L^{(1-\frac 2q + \frac 1\p)^{-1}}}.
	\end{align*}
	Then applying the three-term-commutator estimate \Cref{app:est:threetermcomm}, \Cref{pr:app:GammaEst}, Sobolev inequality \Cref{app:thm:sobinequ} and assumption \eqref{reg:est:phibounded}, we obtain
	\begin{align*}
		I_3 & \aleq \|\lapla{\frac{1}{2\p}} \varphi \|_{L^\p}  \|\lapla{\frac{1}{2\p}} \tilde u \|_{L^\p} [\tilde u]_{W^{\frac 1q,q} (B_{2^{L}\rho })}^{q-1} \\
		& \aleq [\tilde u]_{W^{\frac 1q,q} (\R)} [\tilde u]_{W^{\frac 1q,q} (B_{2^{L}\rho })}^{q-1}.
	\end{align*}
	Also this estimate can get localized, in particular by the localized Sobolev inequality \Cref{app:locSobinequ} and the localized version of the three-term-commutator estimate \Cref{app:est:threetermcomm}.
	
	In total, using again that $u$ and $\tilde u$ coincide on $B(x_0,r)$ by construction, cf.~\Cref{reg:rm:extension}, we obtain an estimate of the form 
	\begin{align*}
		& \| \chi_{B_{2^K\rho }} u \, \wedge  \,  \Gamma_{\frac 1\p,B_{2^{L}\rho} \times B_{2^{L}\rho}} u  \|_{L^{\frac{\p}{\p-1}}}\\
		  &  \aleq  2^{-\sigma K} [u]^{q-1}_{W^{\frac 1q,q }(B_{2^{2L}\rho})} + [u]_{W^{\frac{1}{q},q}(B_{2^{3K}\rho})}^{2q-3}  + [u]_{W^{\frac{1}{q},q}(B_{2^{3K}\rho})}^{q+1}\\ 
		& + \sum_{l=1}^\infty 2^{-\sigma(3K+l)}[\tilde{u}]_{W^{\frac{1}{q},q}(B_{2^{3K + l}\rho})}^{q-1}  +  \rho \left(\E^{q}(u) +  r^{-3} + [u]_{W^{\frac{1}{q},q}(B(x_0,r))}^q\right) \\
		& + \rho \, r^{-3} + \sum_{l=1}^\infty 2^{-\sigma(L+l)}[\tilde u]^{q-1}_{W^{\frac 1q,q}(B_{2^{L+2K+l}\rho})}\\
		& + [u]^q_{W^{\frac 1q,q }(B_{2^{2L}\rho})} + \sum_{l=1}^\infty 2^{-\sigma (L+l)} [\tilde u]^q_{W^{\frac 1q,q }(B_{2^{2L+l}\rho})} \\
		& + [u]^{q-1}_{W^{\frac 1q,q }(B_{2^{L}\rho})} \left([u]_{W^{\frac 1q,q }(B_{2^{3K}\rho})}  + \sum_{l=1}^\infty 2^{-\sigma (L+k)} [\tilde u]_{W^{\frac 1q,q }(B_{2^{3K+l}\rho})} \right),
	\end{align*}
	which can be simplified by choosing $L\geq 3K$ and factoring out the constant $[u]_{W^{\frac 1q,q}(B(x_0,r))}$, wherever it makes sense, to 
	\begin{align*}
		& \| \chi_{B_{2^K\rho }} u \, \wedge  \,  \Gamma_{\frac 1\p,B_{2^{L}\rho} \times B_{2^{L}\rho}} u  \|_{L^{\frac{\p}{\p-1}}}\\
		& \aleq [u]^q_{W^{\frac 1q,q }(B_{2^{2L}\rho})} + [u]^{2q-3}_{W^{\frac 1q,q }(B_{2^{2L}\rho})}  + \sum_{l=1}^\infty 2^{-\sigma (K+l)} [\tilde u]^{q-1}_{W^{\frac 1q,q }(B_{2^{2L+l}\rho})} \\
		& \quad + 2^{-\sigma K} [u]^{q-1}_{W^{\frac 1q,q }(B_{2^{2L}\rho})} +  \rho \left(\E^{q}(u) +  r^{-3} + [u]_{W^{\frac{1}{q},q}(B(x_0,r))}^q \right).
	\end{align*}
\end{proof}

It only remains to prove the decay estimate based on the elaborated left-hand side and right-hand side estimates.

\begin{proof}[Proof of \Cref{pr:decayest}]
	This proof is in the spirit of \cite[Proposition~3.9]{BRS19}. 
	
	First let $K_0$ be a large number, which will be specified later, and we set for $K\geq K_0$ $L:= 10K$ and $N:= 20K$. Moreover, let $\eps,\delta>0$ be small numbers, that will be chosen in the following, and assume
	\begin{equation}\label{reg:est:inismall}
		[\tilde u]_{W^{\frac 1q,q}(B_{2^N\rho})} <\eps.
	\end{equation}
	We then combine the left-hand side and right-hand side estimates to obtain a recursive estimate. First recall from the left-hand side estimate \Cref{pr:reg:lhsestimate} that there exists a large constant $C_\delta >0$ such that
	\begin{align*}
		[u]^q_{W^{\frac{1}{q}, q}(B_\rho)} \aleq & \ [u]_{W^{\frac{1}{q}, q}(B_{2^L\rho})} \|\chi_{B_{2^K\rho}} \Gamma_{\frac 1\p,B_{2^L\rho}\times B_{2^L\rho}}u\|_{L^{\frac{\p}{\p-1} }} +\delta [u]^q_{W^{\frac{1}{q}, q}(B_{2^L\rho})} \\
		&+ C_\delta \left([u]^q_{W^{\frac{1}{q}, q}(B_{2^L\rho})} -[u]^q_{W^{\frac{1}{q}, q}(B_{\rho})}\right)
	\end{align*}
	for $K$ large and $\tfrac 1q - \tfrac 1\p>0$ small enough.
	In the next step we split the operator $\Gamma_{\frac 1\p,B_{2^K\rho}\times B_{2^L \rho}}$ by \eqref{eq:reg:projofop} into
	\begin{equation*}
	\begin{split}
		& \| \chi_{B_{2^K\rho }} \Gamma_{\frac 1\p,B_{2^L \rho}\times B_{2^L \rho}} u  \|_{L^{\frac{\p}{\p-1}}} \\
		& \aleq \| \chi_{B_{2^K\rho }} u \, \cdot \, \Gamma_{\frac 1\p,B_{2^L \rho}\times B_{2^L \rho}} u  \|_{L^{\frac{\p}{\p-1}}} + \| \chi_{B_{2^K\rho }} u \, \wedge  \,  \Gamma_{\frac 1\p,B_{2^L \rho}\times B_{2^L \rho}} u  \|_{L^{\frac{\p}{\p-1}}}.
	\end{split}
	\end{equation*}
	The first term on the right-hand side may be estimated by \Cref{la:reg:rhs1}
	\begin{align*}
		\| \chi_{B_{2^K\rho }} u \, \cdot \, \Gamma_{\frac 1\p,B_{2^L \rho}\times B_{2^L \rho}} u  \|_{L^{\frac{\p}{\p-1}}} \aleq [u]^q_{W^{\frac 1q,q }(B_{2^{2L}\rho})} + \sum_{l=1}^\infty 2^{-\sigma (L+l)} [\tilde u]^q_{W^{\frac 1q,q }(B_{2^{2L+l}\rho})},
	\end{align*}
	whereas the second term on the right-hand side is by \Cref{la:reg:rhs2} bounded by
	\begin{align*}
		& \| \chi_{B_{2^K\rho }} u \, \wedge  \,  \Gamma_{\frac 1\p,B_{2^{L}\rho} \times B_{2^{L}\rho}} u  \|_{L^{\frac{\p}{\p-1}}}\\
		& \aleq [u]^q_{W^{\frac 1q,q }(B_{2^{2L}\rho})} + [u]^{2q-3}_{W^{\frac 1q,q }(B_{2^{2L}\rho})}  + \sum_{l=1}^\infty 2^{-\sigma (K+l)} [\tilde u]^{q-1}_{W^{\frac 1q,q }(B_{2^{2L+l}\rho})} \\
		& \quad + 2^{-\sigma K} [u]^{q-1}_{W^{\frac 1q,q }(B_{2^{2L}\rho})} +  \rho \left(\E^{q}(u) +  r^{-3} + [u]_{W^{\frac{1}{q},q}(B(x_0,r))}^q \right).
	\end{align*}
	 We conclude with setting $\theta:= \tfrac{\sigma}{20}>0$  that there exists a large constant $C= C(q,r, \E^q(u), \|u\|_{L^\infty}, \|\tilde u\|_{L^\infty}, [u]_{W^{\frac{1}{q}, q}(B(x_0,r))}, [\tilde u]_{W^{\frac{1}{q}, q}(\R)})>0$ such that 
	\begin{align*}
		& [u]^q_{W^{\frac{1}{q}, q}(B_\rho)} \\
		& \leq 
		C [u]^q_{W^{\frac{1}{q}, q}(B_{2^N\rho})} \bigg( [u]_{W^{\frac 1q,q }(B_{2^{N}\rho})}
		  + \delta +2^{-\theta N} \bigg)\\
		& \quad + C C_\delta \bigg([u]^q_{W^{\frac{1}{q}, q}(B_{2^N\rho})} -[u]^q_{W^{\frac{1}{q}, q}(B_{\rho})}\bigg) 
		+ C \rho\\
		& \quad + C \sum_{l=1}^\infty 2^{-\theta (N+l)} [\tilde u]^{q}_{W^{\frac 1q,q }(B_{2^{N+l}\rho})}.
	\end{align*}
	Now we employ the hole-filling technique: We add $C C_\delta  [u]^q_{W^{\frac{1}{q}, q}(B_{2^N\rho})}$ to both sides of the inequality and divide by $2C C_\delta + 1$ subsequently so that we obtain under consideration of the initial bound \eqref{reg:est:inismall} on $[u]^q_{W^{\frac{1}{q}, q}(B_{2^N\rho})} $
	\begin{align*}
		[u]^q_{W^{\frac{1}{q}, q}(B_\rho)} & \leq [u]^q_{W^{\frac{1}{q}, q}(B_{2^N\rho})} \frac{\eps  + \delta +2^{-\theta N} + 2C C_\delta}{2CC_\delta +1} + \rho \\
		& \quad + \sum_{l=1}^\infty 2^{-\theta (N+l)} [\tilde u]^{q}_{W^{\frac 1q,q }(B_{2^{N+l}\rho})}.
	\end{align*}
	If we then choose $\delta$ and $\eps$ small enough, whereas $K_0$ big enough, such that
	\[
	\eps + \delta + 2^{-\theta K_0} \leq \tfrac 12,
	\]
	we gain by defining
	\[
		0< \tau := \frac{\frac 12 +2 C C_\delta}{2 CC_\delta +1} <1
	\]
	the desired estimate for any $N\geq N_0$ with  $N_0=20K_0$
	\begin{align*}
		& [u]^q_{W^{\frac{1}{q}, q}(B_\rho)}  \leq \tau [u]^q_{W^{\frac{1}{q}, q}(B_{2^N\rho})}   + \sum_{l=1}^\infty 2^{-\theta (N+l)} [\tilde u]^{q}_{W^{\frac 1q,q }(B_{2^{N+l}\rho})} + \rho.
	\end{align*}
\end{proof}


\appendix

\section{Gagliardo-Sobolev Space}
Recall that the seminorm for the fractional Sobolev space $W^{s,p}(B)$ for $s \in (0,1)$, $p \in (1,\infty)$, and $B\subset \R$ is given as 
\[
 [f]_{W^{s,p}(B)} = \brac{\int_{B}\int_{B} \frac{|f(x)-f(y)|^p}{|x-y|^{1+sp}}\, dx\, dy}^{\frac{1}{p}}.
\]
In this section we gather a few useful facts that we need throughout the paper. Most likely all of them are known at least to experts and we do not claim any originality here, but we could not find them in the literature.

We begin with two identifications for the fractional Sobolev space.
\begin{lemma}[Identification 1]\label{la:id1}
Let $s \in (0,1)$, $p \in (1,\infty)$. Then for any ball $B \subset \R$ or $B = \R$ and any $f \in C_c^\infty(\R)$,
 \[
 [f']_{W^{s,p}(B)}^p := \int_{B} \int_{B} \frac{|f'(x)-f'(y)|^p}{|x-y|^{n+sp}}\, dx\, dy \aeq \int_{B}\int_{B} \frac{\abs{\frac{f(y)-f(x)-f'(x)(y-x)}{|x-y|}}^p}{|x-y|^{n+sp}}\, dx\, dy 
\]
The constant depends on $s$ and $p$, but not on the set $B$ or the function $f$.
\end{lemma}
\begin{proof}
\underline{The $\aleq$-estimate}.

We have
\[
	|f'(x)-f'(y)|^p \aleq |f'(x)- \tfrac{1}{y-x}\int_x^y f'(z)\, dz|^p + |f'(y)- \tfrac{1}{x-y}\int_y^x f'(z)\, dz|^p.
\]
By the fundamental theorem of calculus,
\[
\begin{split}
	&|f'(x)- \tfrac{1}{y-x}\int_x^y f'(z)\, dz|^p + |f'(y)- \tfrac{1}{x-y}\int_y^x f'(z)\, dz|^p \\
	=&
	2 \frac{|f(y)-f(x)-f'(x)(y-x)|^p}{|y-x|^p}.
	\end{split}
\]
Thus, we have the $\aleq$-inequality,
\[
 \int_{B} \int_{B} \frac{|f'(x)-f'(y)|^p}{|x-y|^{n+sp}}\, dx\, dy \aleq \int_{B}\int_{B} \frac{\abs{\frac{f(y)-f(x)-f'(x)(y-x)}{|x-y|}}^p}{|x-y|^{n+sp}}\, dx\, dy.
\]
\underline{The $\ageq$-estimate}.

For the opposite inequality, by the fundamental theorem of calculus and Jensen's inequality,
\[
\begin{split}
	\abs{\frac{f(y)-f(x)-f'(x)(y-x)}{|x-y|}}^p & = \abs{\frac{(y-x) (\tfrac{1}{y-x}\int_x^y f'(z) -f'(x) \, dz)}{|x-y|}}^p \\
	& \leq \frac{1}{|y-x|} \int_{(x,y)} |f'(z)- f'(x)|^p\, dz.
\end{split}
\]
We integrate both sides over $B$ in $x$ and $y$,

\[
\begin{split}
	\int_{B}\int_{B} \frac{\abs{\frac{f(y)-f(x)-f'(x)(y-x)}{|x-y|}}^p}{|x-y|^{1+sp}}\, dy\, dx  
	\leq& \int_{B}\int_{B}\int_{(x,y)} \frac{|f'(z)- f'(x)|^p}{|x-y|^{2+sp}}\, dz\, dy\, dx \\
	\leq&\int_{B}\int_{B}\int_{x >z >y} \frac{|f'(z)- f'(x)|^p}{|x-y|^{2+sp}}\, dz\, dy\, dx \\
 &+\int_{B}\int_{B}\int_{y > z > x} \frac{|f'(z)- f'(x)|^p}{|x-y|^{2+sp}}\, dz\, dy\, dx\\
 \end{split}
\]
Observe that by Fubini for any $x \in B$, 
\begin{equation}\label{eq:id1:est13344}
\begin{split}
 &\int_{B}\int_{y > z > x} \frac{|f'(z)- f'(x)|^p}{|x-y|^{2+sp}}\, dz\, dy \\
 \leq&\int_{B \cap \{z > x\}}|f'(z)- f'(x)|^p \int_{z}^\infty \frac{1}{|x-y|^{2+sp}}\, dy\, dz \\
 =&\int_{B \cap \{z > x\}}|f'(z)- f'(x)|^p \frac{1}{1+sp}\frac{1}{|x-z|^{1+sp}}\, dz \\
 \leq&\frac{1}{1+sp}\int_{B} \frac{|f'(z)- f'(x)|^p}{|x-z|^{1+sp}}\, dz \\
 \end{split}
\end{equation}
Thus,
\[
 \int_{B}\int_{B}\int_{y >z >x} \frac{|f'(z)- f'(x)|^p}{|x-y|^{2+sp}}\, dz\, dy\, dx \aleq [f']_{W^{s,q}(B)}^p,
\]
and likewise 
\[
 \int_{B}\int_{B}\int_{x > z > y} \frac{|f'(z)- f'(x)|^p}{|x-y|^{2+sp}}\, dz\, dy\, dx \aleq [f']_{W^{s,q}(B)}^p. 
\]
We can conclude.

\end{proof}

\begin{lemma}[Identification 2]\label{la:id2}
Let $s \in (0,1)$, $p \in (1,\infty)$. For any $g \in C_c^\infty(\R)$ and any $B \subset \R$ a ball or $B=\R$ we have
 \[
 [g]_{W^{s,p}(B)}^p \aeq \int_{B}\int_{B} \frac{\mvint_{(x,y)} |g(x)-g(z)|^p\, dz}{|x-y|^{1+sp}}\, dx\, dy.
\]
The constant depends on $s$ and $p$ but not on the set $B$ or the function $g$.
\end{lemma}
\begin{proof}
\underline{The $\aleq$-estimate}

We have by Jensen's inequality,
\[
\begin{split}
	|g(x)-g(y)|^p \aleq& \abs{g(x)- \mvint_{(x,y)} g(z)\, dz}^p + \abs{g(y)- \mvint_{(x,y)} g(z)\, dz}^p\\
	 \leq&\mvint_{(x,y)} \abs{g(x)- g(z)}^p\, dz+ \mvint_{(x,y)} \abs{g(y)- g(z)}^p\, dz.
	\end{split}
\]
Thus,
\[
 [g]_{W^{s,p}(B)}^p \aleq \int_{B}\int_{B} \frac{\mvint_{(x,y)} |g(x)-g(z)|^p\, dz}{|x-y|^{1+sp}}\, dx\, dy.
\]

\underline{The $\ageq$-estimate}

The opposite direction is a consequence of Fubini's theorem and the fact that $B$ is convex. Indeed,
\[
\begin{split}
	&\int_{B}\int_{B} \frac{\mvint_{(x,y)} |g(x)-g(z)|^p\, dz}{|x-y|^{1+sp}}\, dy\, dx \\
	\leq&\int_{B}\int_{B}\int_{x >z >y} \frac{|g(x)- g(z)|^p}{|x-y|^{2+sp}}\, dz\, dy\, dx \\
 &+\int_{B}\int_{B}\int_{y > z > x} \frac{|g(x)- g(z)|^p}{|x-y|^{2+sp}}\, dz\, dy\, dx.\end{split}
	\]
Now we argue as in \eqref{eq:id1:est13344} to obtain the claim.
\end{proof}

\begin{lemma}[Sobolev embedding]\label{la:sob1}
 Let $B \subset \R$ a ball. For $s,t \in (0,1)$, $t < s$, $p, q \in (1,\infty)$ with \[s-\frac{1}{p} \geq t-\frac{1}{q},\]
we have
\begin{equation}\label{eq:sob1:1}
 [f]_{W^{t,q}(B)} \leq C(s,t,p,q)\ \diam(B)^{s-t-\frac{1}{p}+\frac{1}{q}}\ [f]_{W^{s,p}(B)}.
\end{equation}
If $B = \R$ and $s-\frac{1}{p} = t-\frac{1}{q}$, then
\begin{equation}\label{eq:sob1:2}
 [f]_{W^{t,q}(B)} \leq C(s,t,p,q)\ [f]_{W^{s,p}(B)}.
\end{equation}
The constant $C(s,t,p,q)$ does not depend on $f$ and $B$. 
\end{lemma}
\begin{proof}
We first treat \underline{\eqref{eq:sob1:2}, for $B=\R$, $s-\frac{1}{p} = t-\frac{1}{q}$}. 

In $\R$ we can use the abstract Sobolev embedding theorem for Triebel-spaces,
\[
 [g]_{W^{t,q}(\R)} \aleq [g]_{W^{s,p}(\R)}.
\]
Indeed, by \cite[Proposition, p.14]{RS96}, for $s \in (0,1)$,
\[
 [g]_{W^{t,q}(\R)} \aeq [g]_{\dot{F}^t_{q,q}(\R)}.
\]
By \cite[2.2.3 p.31]{RS96} we have
\[
 [g]_{F^{t}_{qq}(\R)} \aleq [g]_{F^{s}_{pp}(\R)}.
\]
So \eqref{eq:sob1:2} is established.

Next we treat \underline{\eqref{eq:sob1:1}, for a ball $B$}. 

Observe that for any $x_0 \in \R$ and $r > 0$ we have
\[
 [f(x_0+r\cdot)]_{W^{t,q}(B(0,1))} = r^{t-\frac{1}{q}}[f]_{W^{t,q}(B(x_0,r))}
\]
and 
\[
 [f(r\cdot)]_{W^{s,p}(B(0,1))} = r^{s-\frac{1}{p}}[f]_{W^{t,q}(B(x_0,1))}.
\]
So \eqref{eq:sob1:1} follows by scaling and translation from the case $B= B(0,1)$.

Moreover, we can assume that $(f)_{B} := \mvint_{B} f = 0$. Indeed, once we have shown \eqref{eq:sob1:1} under the assumption $(f)_{B} =0$ we can apply it to $f-(f)_{B}$ to get the full result.

So from now on we assume $(f)_{B} = 0$ and $B = B(0,1)$.

Set 
\[
 g(x) := \begin{cases}
          f(x) \quad &|x| \leq 1\\
          f(x/|x|^2) \quad & |x| > 1.
         \end{cases}
\]
We claim that 
\begin{equation}\label{eq:sob1:extensiongoal}
 [g]_{W^{t,q}(\R)} \aeq [f]_{W^{t,q}(B)}, \quad [g]_{W^{s,p}(\R)} \aeq [f]_{W^{s,p}(B)}.
\end{equation}
Indeed, we have 
\[
 [g]_{W^{t,q}(\R)} \geq [g]_{W^{t,q}(B)} = [f]_{W^{t,q}(B)}
\]
which establishes the $\ageq$-case for \eqref{eq:sob1:extensiongoal}.
For the $\aleq$-case of \eqref{eq:sob1:extensiongoal} observe that 
\[
 [g]_{W^{t,q}(\R)}^q = [g]_{W^{t,q}(B(0,1))}^q + [g]_{W^{t,q}(B(0,1)^c)}^q + 2 \int_{B(0,1)^c} \int_{B(0,1)} \frac{|g(x)-g(y)|^q}{|x-y|^{1+sq}}\, dx\, dy.
\]
First we observe by the substitution $\tilde{x} := x/|x|^2$,
\[
 \begin{split}
  [g]_{W^{t,q}(B(0,1)^c)}^q =&\int_{B(0,1)}\int_{B(0,1)} \frac{|f(\tilde{x})-f(\tilde{y})|^q}{|\tilde{x}/|\tilde{x}|^2-\tilde{y}/|\tilde{y}|^2|^{1+tq}}\, |\tilde{x}|\, |\tilde{y}|\, d\tilde{x}\, d\tilde{y}.
 \end{split}
\]
For $\tilde{x},\tilde{y} \in B(0,1)$ we have $|\tilde{x}-\tilde{y}| \leq |\tilde{x}/|\tilde{x}|^2-\tilde{y}/|\tilde{y}|^2|$ and thus
\[
 \begin{split}
  [g]_{W^{t,q}(B(0,1)^c)}^q \leq &\int_{B(0,1)}\int_{B(0,1)} \frac{|f(\tilde{x})-f(\tilde{y})|^q}{|\tilde{x}-\tilde{y}|^{1+tq}}\, d\tilde{x}\, d\tilde{y} = [f]_{W^{t,q}(B(0,1))}^q.
 \end{split}
\]
Similarly,
\[
 \begin{split}
  &\int_{B(0,1)^c} \int_{B(0,1)} \frac{|g(x)-g(y)|^q}{|x-y|^{1+sq}}\, dx\, dy\\
  \leq&\int_{B(0,1)}\int_{B(0,1)} \frac{|f(x)-f(\tilde{y})|^q}{|x-\tilde{y}/|\tilde{y}|^2|^{1+tq}}\, |\tilde{y}|\, dx\, d\tilde{y}.
 \end{split}
\]
For $x,\tilde{y} \in B(0,1)$ we have $|x-\tilde{y}| \aleq |x-\tilde{y}/|\tilde{y}|^2|$ and thus
\[
 \begin{split}
  &\int_{B(0,1)^c} \int_{B(0,1)} \frac{|g(x)-g(y)|^q}{|x-y|^{1+sq}}\, dx\, dy\\
  \aleq&[f]_{W^{t,q}(B(0,1))}^q.
 \end{split}
\]
This establishes \eqref{eq:sob1:extensiongoal}.

From \eqref{eq:sob1:extensiongoal} and \eqref{eq:sob1:2} we obtain \eqref{eq:sob1:1} in the case $s-\frac{1}{p} = t-\frac{1}{q}$.

For \eqref{eq:sob1:1} in the case $s-\frac{1}{p} > t-\frac{1}{q}$ we define 
\[
 h(x) := \eta(x) g(x),
\]
where $\eta \in C_c^\infty(B(0,2),[0,1])$, $\eta \equiv 1$ in $B(0,1)$ is the typical cutoff function. We apply the inhomogeneus Sobolev-inequality, \cite[Proposition, p.14]{RS96}, \cite[2.2.3 p.31]{RS96} to $h$, namely
\[
 [h]_{W^{t,q}(\R)} \aleq  [h]_{W^{s,p}(\R)} + \|h\|_{L^{p}(\R)} \leq  [h]_{W^{s,p}(\R)} + \|g\|_{L^{p}(B(0,2))}.
\]
Now its not too difficult to obtain
\[
 [h]_{W^{s,p}(\R)} \leq \|g\|_{L^{p}(B(0,2))} + [g]_{W^{s,p}(\R)}.
\]
Moreover,
\[
 \|g\|_{L^{p}(B(0,2))} \leq \|f\|_{L^{p}(B(0,1))} = \|f-(f)_{B(0,1)}\|_{L^{p}(B(0,1))}\aleq [f]_{W^{s,p}(B(0,1))},
\]
where in the last step we used the fact that $(f)_{B(0,1))} =0$ and Jensen's inequality. This establishes \eqref{eq:sob1:1}.
\end{proof}

\begin{theorem}[Classical Sobolev Inequality] \label{app:thm:classobinequ} \cite[Theorem 1.5]{S15}
	Let $n\in\N$, $s\geq t\geq 0$, $p\in (1,\tfrac{n}{s-t})$ and define $p_{s,t}^* = \tfrac{np}{n-(s-t)p}$. Then we have for any $f\in C^\infty_0(\R^n)$ 
	\[
	\| (-\lap)^{\frac t2} f \|_{L^{p^*_{s,t}}(\R^n)} \aleq 	\| (-\lap)^{\frac s2} f \|_{L^{p}(\R^n)},
	\]
	or in other words
	\[
		\|\lapin{s-t}f \|_{L^{p^*_{s,t}}(\R^n)} \aleq \|f\|_{L^{p}(\R^n)}.
	\]
\end{theorem}

\begin{theorem}[Sobolev Inequality] \label{app:thm:sobinequ} \cite[Theorem 1.6]{S15}
	Let $n\in\N$, $s>t\geq 0$, $p\in (1,\tfrac{n}{s-t})$ and define $p_{s,t}^* = \tfrac{np}{n-(s-t)p}$. Then we have for any $f\in C^\infty_0(\R^n)$ that
	\[
	\| \lapla{\frac t2} f \|_{L^{p^*_{s,t}}(\R^n)} \aleq \left(\int_{\R^n} \int_{\R^n} \frac{|f(x)-f(y)|^p}{|x-y|^{n+sp}} \, dz \, dy \right)^{\frac 1p},
	\]
	or in other words, let $s+\delta < n$ and $p\leq \frac n \delta$, then for $p_{s,\delta}^* = \tfrac{np}{n-\delta p}$
	\[
		\left(\int_{\R^n} \int_{\R^n} \frac{|\lapin{\delta}f(x)-\lapin{\delta}f(y)|^p}{|x-y|^{n+sp^*_{s,\delta}}} \, dz \, dy \right)^{\frac{1}{p^*_{s,\delta}}} \aleq \| f \|_{L^{p}(\R^n)}.
	\]
\end{theorem}

\begin{proposition} \label{pr:psiuds1}\cite[Proposition D.2]{S15}
	Let $s\in(0,1)$, $q\in (1,\infty)$, and $\eta\in C^\infty_c(B_{2\rho})$ with $\eta\equiv 1$ on $B_\rho$. Then for any $L\in \N$, $L>1$,
	\[
		\int_{B_{2^L\rho}} \int_{B_{2^L\rho}} \frac{|\eta (x)-\eta(y)|^q |u(y)-(u)_{B_{2\rho}\setminus B_\rho}|}{|x-y|^{1+sp}} \, dx \, dy \aleq [u]^q_{W^{s,p}(B_{2^L\rho})} - [u]^q_{W^{s,p}(B_{\rho})}.
	\]
\end{proposition}

\begin{proposition} \label{pr:psiuds}\cite[Proposition D.3]{S15}
	Let $s\in(0,1)$, $q\in (1,\infty)$, $\eta\in C^\infty_c(B_{2\rho})$ with $\eta\equiv 1$ on $B_\rho$, and
	\[
	\psi(x):= \eta(x)(u(x)-(u)_{B_{2 \rho}\setminus B_{\rho} }),
	\]
	then for any $L\in \N$, $L>1$,
	\[
		[\psi]_{W^{s,p}(\R)} \aleq [u]_{W^{s,p}(B_{2^L \rho})}.
	\]
\end{proposition}

We also find use of an adapted version of \cite[Proposition D.4]{S15}: 

\begin{proposition} \label{pr:app:GammaEst}
	For $\delta >0$ small enough, we have
	\begin{align*}
	 \|\Gamma_{\frac 1q + \delta,B_{\rho}} u \|_{L^{1/{(1-\frac 1q - \delta)}}} \aleq  [u]^{q-1}_{W^{\frac 1q,q}(B_\rho) } . 
	\end{align*} 
\end{proposition}
\begin{proof}
	For some $\varphi\in C_c^\infty(\R)$, $\|\varphi\|_{L^{1/(\frac 1q + \delta)}}\leq 1$, we have by  \Cref{rm:reg:boundedfactorc}, H\"older's inequality, the identification for fractional Sobolev spaces \Cref{la:id2}, and the Sobolev inequality \Cref{app:thm:classobinequ}
	\[
	\begin{split}
		&\|\Gamma_{\frac 1q + \delta,B_{\rho}} u \|_{L^{{1/(1-\frac 1q - \delta)}}} \\
		& \aleq \int_{\R} \Gamma_{\frac 1q + \delta,B_{\rho}} u (z) \, \varphi(z) \, dz \\
		& \aleq \int_{{B_\rho}} \int_{B_\rho} \avint_{x\triangleright y} |u(z_1) - u(x)|^{q-1} \, dz_1 \avint_{x\triangleright y} |\lapin{\frac 1q - \delta} \varphi (z_2) - \lapin{\frac 1q -\delta} \varphi(x) | \, dz_2 \frac{ dy \, dx}{\rho(x,y)^2}\\
		& \aleq  \left(\int_{{B_\rho}} \int_{B_\rho}  \frac{\avint_{x\triangleright y}  |u(z_1) - u(x)|^{q} \, dz_1}{\rho(x,y)^2}\, dy \, dx \right)^{\frac{q-1}{q}} \, \left( \int_{\R} \int_{ \R}  \frac{\avint_{x\triangleright y} |\lapin{\frac 1q - \delta} \varphi (z_2) - \lapin{\frac 1q -\delta} \varphi(x)|^q \, dz_2}{|x-y|^2}\, dx \, dy\right)^{\frac 1q}\\
		& \aleq  [u]^{q-1}_{W^{\frac 1q,q}(B_\rho) } \|\varphi\|_{L^{1/(\frac 1q + \delta)}}.
	\end{split}
	\]
\end{proof}

\section{Localization Arguments}

For the convenience of the reader, we recall some results related to localization. For a good overview of these statements we refer to \cite{S15}, but they can be found throughout the literature.

\begin{proposition}\cite[Proposition B.2]{S15} \label{app:loc1}
	Let $p>1$, $t\in(0,1)$, and $\delta \geq 0$ small. Then for any $\varphi\in C^\infty_0 (B_{2^K})$ and $L>2$ we have
	\[
	\begin{split}
	\|\lapla{\frac \delta 2} ((\eta_{B_{2^{K+L}}}- \eta_{B_{2^{K+L-1}}}) \lapla{\frac t2} \varphi  ) \|_{L^{\frac{np}{n+\delta p}}} & \aleq 2^{-L (\frac{n}{\frac{p}{p-1}}+t)} \|\lapla{\frac t2} \varphi \|_{L^{p}} \\
	\|\lapla{\frac{t+\delta}{2}} ((\eta_{B_{2^{K+L}}}- \eta_{B_{2^{K+L-1}}}) \,  \lapin{t} \varphi  ) \|_{L^{\frac{np}{n+\delta p}}} & \aleq 2^{-L \frac{n}{\frac{p}{p-1}}} \| \varphi \|_{L^{p}}
	\end{split}
	\]
\end{proposition}

\begin{proposition}\cite[Proposition B.3]{S15} \label{app:loc2}
	Let $s\in(0,n)$ and $p\in(1,\tfrac ns)$. Then we have for some $\sigma>0$ and any $L\in\N$ 
	\[
	\begin{split}
	\|\lapin{s} f\|_{L^{\frac{np}{n-sp}}(B_\rho)} \aleq \|f\|_{L^{p}(B_{2^L\rho})} + \sum_{l=1}^\infty 2^{-\sigma(L+l)}  \|f\|_{L^{p}(B_{2^{L+l}\rho})}.
	\end{split}
	\]
\end{proposition}

\begin{proposition}\cite[Proposition B.4]{S15} \label{app:loc3}
	Let $s_1, s_2, s_3\in[0,n)$ and $p_1,p_2,p_3 \in (1,\infty)$ such that 
	\[
	p_i^* := \frac{np_i}{n-s_ip_i} \in (1,\infty).
	\]
	If moreover 
	\[
	\sum_{i=1}^3 \tfrac{1}{p_i} - \tfrac{s_i}{n} = 1,
	\]
	then we have the following pseudo-local behavior for any $L\in\N$ and some $\sigma>0$:
	\[
	\begin{split}
	& \int_{\R^n} \lapin{s_1}(\chi_{B_\rho} f_1) \ \lapin{s_2} f_2 \  \lapin{s_3} f_3 \\
	& \aleq \|f_1\|_{L^{p_1}(B_{2^L\rho})} \|f_2\|_{L^{p_2}(B_{2^L\rho})} \|f_3\|_{L^{p_3}(B_{2^L\rho})} \\
	& \quad + \sum_{l=1}^\infty 2^{-(L+l)\sigma} \|f_1\|_{L^{p_1}(B_{2^{L+l}\rho})} \|f_2\|_{L^{p_2}(B_{2^{L+l}\rho})} \|f_3\|_{L^{p_3}(B_{2^{L+l}\rho})}.
	\end{split}
	\]
\end{proposition}

\begin{proposition} \label{app:loc4}
	Let $p>1$. Then we have for some $\sigma>0$ and any $L\in\N$ 
	\[
	\begin{split}
	\|\mathcal{M} f\|_{L^{p}(B_\rho)} \aleq \|f\|_{L^{p}(B_{2^L\rho})} + \sum_{l=1}^\infty 2^{-\sigma(L+l)}  \|f\|_{L^{p}(B_{2^{L+l}\rho})}.
	\end{split}
	\]
\end{proposition}
\begin{proof}
	We first split by Fatou's lemma and Minkowski inequality
	\[
	\begin{split}
		\|\M(f)\|_{L^p(B_\rho)} \leq \|\M(\chi_{B_{2^L\rho}}f)\|_{L^p(B_\rho)}  + \sum_{l=1}^\infty \|\M(\chi_{B_{2^{L+l}\rho}\setminus B_{2^{L+l-1}\rho}} f)\|_{L^p(B_\rho)} .
	\end{split}
	\]
	For the first term, we get by Hardy-Littlewood maximal inequality
	\[
	\begin{split}
		\|\M(\chi_{B_{2^L\rho}}f)\|_{L^p(B_\rho)} & \leq \|\M(\chi_{B_{2^L\rho}}f)\|_{L^p(\R^n)} \\
		& \leq \|\chi_{B_{2^L\rho}}f\|_{L^p(\R^n)} =  \|f\|_{L^p(B_{2^L\rho})}.
	\end{split}
	\]
	For the remaining terms, we observe for any $x\in B_\rho$ by the definition of the Hardy-Littlewood maximal function and H\"older's inequality
	\[
	\begin{split}
		\M(\chi_{B_{2^L\rho}}f)(x) & = \sup_{r>0}\frac{1}{|B(x,r)|} \int_{B(x,r)}\chi_{B_{2^{L+l}\rho}\setminus B_{2^{L+l-1}\rho}}(y) |f(y)| \, dy \\
		& \aleq (2^{L+l}\rho)^n \int_{B_{2^{L+l}\rho}} \chi_{B_{2^{L+l}\rho}\setminus B_{2^{L+l-1}\rho}} (y) |f(y)| \, dy \\
		& \aleq (2^{L+l}\rho)^n (2^{L+l}\rho)^{n-\frac np} \|f\|_{L^{p}(B_{2^{L+l}\rho})}.
	\end{split}
	\]
	Therefore, we obtain
	\[
	\begin{split}
		\|\M(\chi_{B_{2^{L+l}\rho}\setminus B_{2^{L+l-1}\rho}} f)\|_{L^p(B_\rho)} & = \left( \int_{B_\rho} |\M(\chi_{B_{2^{L+l}\rho}\setminus B_{2^{L+l-1}\rho}} f)(x)|^p \, dx\right)^{\frac 1p} \\
		& \aleq (2^{L+l}\rho)^n (2^{L+l}\rho)^{n-\frac np} \rho^{\frac np} \|f\|_{L^{p}(B_{2^{L+l}\rho})} \\
		& \approx 2^{-\frac np (L+l)}\|f\|_{L^{p}(B_{2^{L+l}\rho})} .
	\end{split}
	\]
\end{proof}

We also need a localized version of the Sobolev inequality \Cref{app:thm:sobinequ}:
\begin{lemma}\cite[Lemma C.1]{S15}\label{app:locSobinequ}
	Given $0 < t < s < 1$ and define $p_s=\tfrac ns$, $p_t=\tfrac nt$. Then for any $L\in\Z$ and $K\in \N$, we have 
	\[
	\|\chi_{B_{2^L}} \lapla{\frac t2} f\|_{L^{p_t}} \aleq [f]_{W^{s,p_s}(B_{2^{L+K}})} + \sum_{k=1}^\infty 2^{-\sigma(K+k)}  [f]_{W^{s,p_s}(B_{2^{L+K+k}})}.
	\]
\end{lemma}

\subsection{Three-Term-Commutator Estimates}

For $\alpha > 0$ the three term commutator is defined by
\[
H_\alpha(f,g) = \lapla{\frac \alpha 2} (fg) - f\lapla{\frac \alpha 2} g - g \lapla{\frac \alpha 2} f.
\]
This operator measures the deviation from the Leibniz rule for $\laps{\alpha}$. For fractional harmonic maps it was discovered in \cite{DLR11a,DLR11b} how $H_\alpha(\cdot,\cdot)$ takes the role of the div-curl term, and in particular it was shown that in $\R^1$, $(-\lap)^{\frac{1}{4}} H_{1/2}(f,g)$ belongs to the Hardy-space if $(-\lap)^{\frac{1}{4}} f, (-\lap)^{\frac{1}{4}} g \in L^2(\R)$, see also \cite{LS20}. There have been multiple extensions since, the following estimate and its localized version on the three term commutator will be helpful.

\begin{theorem}\cite[Theorem~A.1]{S15} \label{app:est:threetermcomm}
	For any $\eps>0$ small and $p\in(1,\infty)$, we have
	\[
		\|\lapla{\frac \eps 2} H_\alpha (f,g)\|_{L^p} \aleq \|\lapla{\frac \alpha 2} f\|_{L^{p_1}} \|\lapla{\frac \alpha 2} g\|_{L^{p_2}},
	\]
	where $p_1,p_2\in(1,\tfrac n\alpha]$ such that
	\[
		\tfrac 1p = \tfrac{1}{p_1} + \tfrac{1} {p_2} - \tfrac{\alpha-\eps}{n}.
	\]
	If $\supp f \subset B_{2^K}$, then we have for any $L\in\N$
	\[
	\begin{split}
		\|\lapla{\frac \eps 2} H_\alpha (f,g)\|_{L^p} \aleq \|\lapla{\frac \alpha 2} f\|_{L^{p_1}} 
		&\Big (\|\lapla{\frac \alpha 2} g\|_{L^{p_2}(B_{2^{K+L}})} \\
		\quad + \sum_{k=1}^\infty 2^{-\sigma(L+k)} \|\lapla{\frac \alpha 2} g\|_{L^{p_2}(B_{2^{K+L+k}})} \Big ).
		\end{split}
	\]
\end{theorem}

\section{Tail estimates}

\begin{lemma}\label{la:RminusBest}
	Let $\varphi \in C^\infty_c(B_r)$, $p\geq q+2$, $q\geq2$, and $L,k\in\N$. Then we can estimate 
	\[
	\begin{split}
	&  \int_{D_1}  \int_{D_2} \left(\avint_{x\triangleright y} |u(z)-u(x)|^2 \, dz \right)^{q-2} \left(\avint_{x\triangleright y} |u(z_1)-u(x)| \, dz_1 \right)  \\
	& \cdot  \left( \avint_{x\triangleright y} |\varphi(z_2) - \varphi(x)|  \, dz_2 \right)  \frac{ dy\, dx}{\rho(x,y)^{p-q}} \\
	& \aleq 2^{-(L+k) \frac{p-q}{q}}\, [u]_{W^{\frac{p-q-1}{q},q}(B_{2^{L+k}r})}^{2q-3} \, [\varphi]_{W^{\frac{p-q-1}{q},q}(B_{r})},
	\end{split}
	\]
	where 
	\[
	\begin{split}
		\textbf{Case 1:} & \quad D_1\times D_2 = B_{2^Lr} \times (B_{2^{L+k}r}\setminus B_{2^{L+k-1}r}),\\
		\textbf{Case 2:} & \quad D_1\times D_2 = (B_{2^{L+k}r}\setminus B_{2^{L+k-1}r}) \times B_{2^Lr}, \\
		\textbf{Case 3:} & \quad D_1\times D_2 = (B_{2^{L+k}r}\setminus B_{2^{L+k-1}r}) \times (B_{2^{L+k}r}\setminus B_{2^{L+k-1}r}).
	\end{split}
	\]
\end{lemma}
\begin{proof}
	Since $\avint_{x\triangleright y} |\varphi(z_2) - \varphi (x)| \, dz_2 = 0$ for either $x,y < -r$ or $x,y >r$ due to the support of $\varphi$ in $B_r$, we only need to consider the cases 
		\[
		\begin{split}
			\textbf{Case 1:} & \quad -2^Lr<x<r, \quad 2^{L+k-1}r < y < 2^{L+k}r, \\
			\textbf{Case 2:} & \quad 2^{L+k-1}r < x < 2^{L+k}r, \quad -2^Lr<y<r, \textnormal{ and } \\
			\textbf{Case 3:} & \quad -2^{L+k}r < x < -2^{L+k-1}r, \quad 2^{L+k-1}r < y < 2^{L+k}r.
		\end{split}
		\]
	The cases 
	\[
	\begin{split}
		\textbf{Case 1:} & \quad -r<x<2^Lr, \quad -2^{L+k}r < y < -2^{L+k-1}r, \\
		\textbf{Case 2:} & \quad -2^{L+k}r < x < -2^{L+k-1}r, \quad -r<y<2^Lr, \textnormal{ and } \\
		\textbf{Case 3:} & \quad  2^{L+k-1}r < x < 2^{L+k}r, \quad -2^{L+k}r < y < -2^{L+k-1}r,
	\end{split}
	\]
	follow analogously. We first examine \textbf{Case 1}. By using H\"older's inequality, Jensen's inequality and the Sobolev embedding \Cref{la:sob1}, in consideration of constants depending on the domain, we get
	\[
	\begin{split}
	&  \int_{B_{2^Lr}}  \int_{B_{2^{L+k}r}\setminus B_{2^{L+k-1}r}} \rho(x,y)^{-(p-q)} \left( \rho(x,y)^{-1} \int_x^y |u(z)-u(x)|^2 \, dz \right)^{q-2} \\
	& \cdot \left( \rho(x,y)^{-1} \int_x^y |u(z_1)-u(x)| \, dz_1 \right)    \left( \rho(x,y)^{-1} \int_x^y |\varphi(z_2) - \varphi(x)|  \, dz_2 \right)   dy\, dx \\
	&\aleq (2^{L+k}r)^{-(p-q)-q} \, (2^{L+k}r) \int_{B_{2^{L+k}r}} \left( \int_{B_{2^{L+k}r}}|u(z)-u(x)|^2 \, dz \right)^{q-2} \\
	& \cdot\left( \int_{B_{2^{L+k}r}} |u(z_1)-u(x)| \, dz_1 \right)    \left(  \int_{B_{2^{L+k}r}} |\varphi(z_2) - \varphi(x)|  \, dz_2 \right)  \, dx \\
	&\aleq (2^{L+k}r)^{-(p-q)-q+1} (2^{L+k}r)^{(q-1)\frac{q-2}{q}} \left(\iint_{(B_{2^{L+k}r})^2}  |u(z)-u(x)|^{2q} \, dz \, dx\right)^{\frac{q-2}{q}}  \\
	& \cdot(2^{L+k}r)^{(q-1)\frac 1q}  \left(\iint_{(B_{2^{L+k}r})^2} |u(z_1) - u (x)|^{q} \, dz_1 \, dx \right)^{\frac 1q} \\ 
	& \cdot(2^{L+k}r)^{(q-1)\frac 1q}  \left(\iint_{(B_{2^{L+k}r})^2} |\varphi(z_2) - \varphi (x)|^{q} \, dz_2 \, dx \right)^{\frac 1q} \\
	& \aleq  (2^{L+k}r)^{-(p-q)} (2^{L+k}r)^{(p-q) \frac{q-2}{q}}\, [u]^{2q-4}_{W^{\frac{p-q-1}{2q},2q}(B_{2^{L+k}r})}  (2^{L+k}r)^{(p-q) \frac{1}{q}}\, [u]_{W^{\frac{p-q-1}{q},q}(B_{2^{L+k}r})} \\
	& \cdot  r^{(p-q) \frac{1}{q}}\, [\varphi]_{W^{\frac{p-q-1}{q},q}(B_{r})} \\
	& \approx 2^{-(L+k)\frac{p-q}{q}} [u]_{W^{\frac{p-q-1}{q},q}(B_{2^{L+k}r})}^{2q-3} \, [\varphi]_{W^{\frac{p-q-1}{q},q}(B_{r})}.
	\end{split}
	\]
	\textbf{Case 2} differs from \textbf{Case 1} by symmetry only in the integration on $y$, where we used
	\[
		 \int_{B_{2^Lr}}  dy\aleq 2^{L+k}r.
	\]
	In \textbf{Case 3} the estimates also hold since the distance of $x$ and $y$ is even greater than in the previous cases. 
\end{proof}

\section{Mean Value Arguments}

In the following we introduce mean value arguments and compensation effects, which turn out to be crucial for elaborating the right-hand side estimates, subsequently. 

The first statement is a typical mean value argument, cf. \cite[Lemma 3.3.]{MSY20}.
\begin{lemma}\label{la:typicalmv}
	Let $\alpha \in \R$ and $a,b \in \R$ with $|a-b| \aleq \min\{|a|,|b|\}$. Then for any $\eps \in [0,1]$,
	\[
	||a|^\alpha-|b|^\alpha|\aleq |a-b|^\eps\,\min\{ |a|^{\alpha-\eps}, |b|^{\alpha-\eps}\}
	\]
\end{lemma}

In our situation we will have to deal with the expression
\[
\frac{1}{|x-y|} \int_{(x,y)} ||z-z_2|^{\alpha-1} - |z-x|^{\alpha -1}| \, dz_2.
\]
The following lemma tells us, that it behaves very similarly to 
\[
||z-y|^{\alpha-1} - |z-x|^{\alpha -1}|.
\]

\begin{lemma}\label{la:mvmvf}
	Let $x,y,z \in \R$ three distinct points be inside a geodesic ball $B \subset \R$ and $\alpha \in (0,1)$.
	Set 
	\[
	F(x,y,z) := \frac{1}{|x-y|} \int_{(x,y)} ||z-z_2|^{\alpha-1} - |z-x|^{\alpha -1}| \, dz_2.
	\]
	\begin{itemize}
		\item If 
		\begin{equation}\label{eq:mvmvf}
		|x-y| \aleq \min\{|x-z|,|y-z|\}
		\end{equation}
		then for any $\eps \in [0,1]$,
		\begin{equation}\label{eq:mvmvf:claim}
		F(x,y,z) \aleq |x-y|^\eps\, \min\{|x-z|^{\alpha-\eps-1}, |y-z|^{\alpha-\eps-1}\}.
		\end{equation}
		\item If 
		\begin{equation}\label{eq:mvmvf2}
		|x-z| \aleq \min\{|x-y|,|y-z|\}
		\end{equation}
		then for any $\eps \in [0,1]$,
		\begin{equation}\label{eq:mvmvf:claim2}
		F(x,y,z) \aleq |x-z|^{\alpha - 1} \aleq |x-y|^\eps |x-z|^{\alpha -\eps -1}.
		\end{equation}
		\item If 
		\begin{equation}\label{eq:mvmvf3}
		|y-z| \aleq \min\{|x-y|,|x-z|\}
		\end{equation}
		then for any $\eps \in [0,1]$,
		\begin{equation}\label{eq:mvmvf:claim3}
		F(x,y,z) \aleq |y-z|^{\alpha - 1}  \aleq |x-y|^\eps |y-z|^{\alpha -\eps -1}
		\end{equation}
	\end{itemize}
\end{lemma}
\begin{proof}[Proof of Lemma~\ref{la:mvmvf}, \eqref{eq:mvmvf:claim}]
	Observe that \eqref{eq:mvmvf} implies
	\[
	|x-z| \aeq |y-z|.
	\]

	\underline{Case 1: $|z-x| \leq 10 |x-y|$.}
	
	In this case $|x-y| \aeq |z-x|$, and for any $z_2 \in (x,y)$ we have $|z-z_2| \aleq |x-y|\aleq |z-x|$. We then simply integrate
	\[
	\begin{split}
	F(x,y,z) &  \aleq |x-y|^{-1} \int_{|z-z_2| \aleq |x-y|} |z-z_2|^{\alpha-1} dz_2 \\
	& \aeq |x-y|^{\alpha-1} \\
	& \aeq |x-y|^{\eps} \min\{|x-z|^{\alpha-\eps-1}, |y-z|^{\alpha-\eps-1}\}.
	\end{split}
	\]
	
	\underline{Case 2: $|z-x| > 10 |x-y|$.}
	In this case, for any $z_2 \in (x,y)$ we have $|z-z_2| \ageq |x-y|$ and $|z_2-x| \leq |x-y| \aleq |z_2 -z|$.
	
	In particular, for any $z_2 \in (x,y)$
	\[
	|z_2-x| \aleq \min\{|z_2-z|,|z-x|\}.
	\]
	By the typical mean value theorem argument, \Cref{la:typicalmv},
	\[
	||z-z_2|^{\alpha-1} - |z-x|^{\alpha -1}| \aleq |z_2-x|\, |z-x|^{\alpha-2} \aleq |x-y|\, |z-x|^{\alpha-2}.
	\]
	Integrating this, we obtain
	\[
	\begin{split}
	F(x,y,z) & \aleq |x-y| |z-x|^{\alpha-2} \\
	& \aleq |x-y|^\eps |z-x|^{\alpha-\eps-1} \\
	& \aeq |x-y|^\eps\, \min\{|x-z|^{\alpha-\eps-1}, |y-z|^{\alpha-\eps-1}\}.
	\end{split}
	\]
	This settles \eqref{eq:mvmvf:claim}.
\end{proof}

\begin{proof}[Proof of Lemma~\ref{la:mvmvf}, \eqref{eq:mvmvf:claim2}]
	Observe that \eqref{eq:mvmvf2} implies that $|x-y| \aeq |y-z|$.
	
	Observe that for $\alpha > 0$ the function $z_2 \mapsto |z-z_2|^{\alpha-1}$ is integrable with antiderivative $\aeq |z-z_2|^{\alpha}$, we have 
	\[
	\begin{split}
	&|x-y|^{-1} \int_{(x,y)} |z-z_2|^{\alpha-1}\\
	\aleq& |x-y|^{-1} \brac{|z-x|^\alpha + |z-y|^\alpha}\\
	=&  |x-y|^{-1} |z-x|^\alpha + |x-y|^{-1}|z-y|^\alpha\\
	\aleq& |z-x|^{\alpha-1} + |x-y|^{\alpha-1}\\
	\aleq& |z-x|^{\alpha-1} + |z-x|^{\alpha-1}\\
	\end{split}
	\]
	In the last step we used that $\alpha \in (0,1)$ (i.e.\  $\alpha - 1<0$)
	
	Also observe that 
	\[
	|x-y|^{-1} \int_{(x,y)} |z-x|^{\alpha-1} dz_2 = |z-x|^{\alpha-1}.
	\]
	This settles \eqref{eq:mvmvf:claim2}.
	
\end{proof}

\begin{proof}[Proof of Lemma~\ref{la:mvmvf}, \eqref{eq:mvmvf:claim3}]
	In this case observe that \eqref{eq:mvmvf3} implies $|x-y| \aeq|z-x|$ so ($\alpha < 1$)
	\[
	\frac{1}{|x-y|} \int_{(x,y)} |z-x|^{\alpha -1} \aeq |x-y|^{\alpha-1} \aleq |y-z|^{\alpha -1}
	\]
	For the remainder we argue as for \eqref{eq:mvmvf:claim2} and have
	\[
	\begin{split}
	&|x-y|^{-1} \int_{(x,y)} |z-z_2|^{\alpha-1}\\
	\aleq& |x-y|^{-1} \brac{|z-x|^\alpha + |z-y|^\alpha}\\
	=&  |x-y|^{-1} |z-x|^\alpha + |x-y|^{-1}|z-y|^\alpha\\
	\aeq &  |z-x|^{\alpha-1} + |z-y|^{\alpha-1}\\
	\aleq &  |z-y|^{\alpha-1}.\\
	\end{split}
	\]
	This settles \eqref{eq:mvmvf:claim3}.
\end{proof}

For the upcoming statement, we need the notation of the uncentered Hardy-Littlewood maximal function, which is given by
\[
\mathcal{M}f(x) = \sup_{B(x,r)\ni y }\frac{1}{|B(y,r)|} \int_{B(y,r)} |f(z)| \, dz
\]

Let us recall the following proposition first. 

\begin{proposition}\cite[Proposition 6.6.]{S18neumann} \label{pr:estmax}
	For any $\alpha \in [0,1]$,
	\[
	|u(x)-u(y)| \aleq |x-y|^\alpha \brac{\mathcal{M} \lapla{\frac \alpha 2} u(x) + \mathcal{M} \lapla{\frac \alpha 2} u(y) }.
	\]
	This implies
	\[
	\avint_{x\triangleright y}  |u(z_1)-u(x)| \, dz_1 \aleq |x-y|^\alpha \mathcal{M}\mathcal{M} \lapla{\frac \alpha 2} u(x),
	\]
	and
	\[
	\avint_{x\triangleright y}  |u(z_1)-u(y)| \, dz_1 \aleq |x-y|^\alpha \mathcal{M}\mathcal{M} \lapla{\frac \alpha 2} u(y).
	\]
\end{proposition}

We develop an adapted version of \cite[Proposition 6.3]{S15} now.

\begin{lemma}\label{la:threetermforGandH}
	Let 
	\[
	G(x,y,z) :=  |u(y)+u(x) -2 u(z)| \avint_{x\triangleright y} ||z-z_2|^{\frac 1\p-1} - |z-x|^{\frac 1\p -1}| \, dz_2
	\]
	and 
	\[
	H(x,y,z) := \avint_{x\triangleright y}  |u(z_1)-u(x)| \, dz_1 \avint_{x\triangleright y} ||z-z_2|^{\frac 1\p-1} - |z-x|^{\frac 1\p -1}| \, dz_2.
	\]
	Then for any $\alpha < \frac{1}{\p}$ and $\eps\in(0,1-\alpha)$ such that $\eps < \frac 1\p - \frac \alpha 2$, $G(x,y,z)$ and $H(x,y,z)$ are, up to a constant, bounded from above by 
	\[
	|x-y|^{\alpha+\eps} \brac{\mathcal{M}\mathcal{M} \lapla{\frac \alpha 4} u(x) + \mathcal{M}\mathcal{M} \lapla{\frac \alpha 4} u(y) + \mathcal{M}\mathcal{M} \lapla{\frac \alpha 4} u(z)}\, k_{ \frac 1\p - \frac \alpha 2 - \eps,\frac{1}{\p}}(x,y,z),
	\]
	where $k_{s,\gamma}$ has the form 
	\begin{align}
	k_{s,\gamma} (x,y,z) & :=  \min \{|x-z|^{s-1}, |y-z|^{s-1} \} \label{app:eq:k1}\\
	& \quad + \left(\frac{|y-z|}{|x-y|} \right)^{\gamma-s} |y-z|^{s-n} \chi_{\{|y-z| \aleq \min\{|x-y|,|x-z|\} \}} \label{app:eq:k2} \\
	& \quad + \left(\frac{|x-z|}{|x-y|} \right)^{\gamma-s} |x-z|^{s-n} \chi_{\{|x-z| \aleq \min\{|x-y|,|y-z|\} \}}.\label{app:eq:k3}
	\end{align}
\end{lemma}
\begin{proof}
	In the case of \underline{$|x-y|\aleq \min \{|x-z|,|y-z| \}$}, we estimate by Proposition \ref{pr:estmax}
	\[
	\begin{split}
	& |u(y)+u(x) -2 u(z)| \\
	& \leq |u(x)-u(y)| + 2 |u(y)-u(z)| \\
	& \aleq |x-y|^{\frac \alpha 2}\brac{\mathcal{M} \lapla{\frac \alpha 4} u(x) + \mathcal{M} \lapla{\frac \alpha 4} u(y) } + |y-z|^{\frac \alpha 2}\brac{\mathcal{M} \lapla{\frac \alpha 4} u(y) + \mathcal{M} \lapla{\frac \alpha 4} u(z) },
	\end{split}	
	\]
	respectively,
	\[
	\begin{split}
	& \avint_{x\triangleright y}  |u(z_1)-u(x)| \, dz_1 \aleq |x-y|^{\frac \alpha 2} \mathcal{M}\mathcal{M} \lapla{\frac \alpha 4} u(x).
	\end{split}	
	\]
	
	Therefore, we get by \eqref{eq:mvmvf:claim} for $\gamma_1=\frac \alpha 2 + \eps, \gamma_2=\alpha + \eps \in [0,1]$
	\[
	\begin{split}
	& G(x,y,z) \\
	& \aleq |x-y|^{\frac \alpha 2}\brac{\mathcal{M} \lapla{\frac \alpha 4} u(x) + \mathcal{M} \lapla{\frac \alpha 4} u(y) }  \avint_{x\triangleright y} ||z-z_2|^{\frac 1\p-1} - |z-x|^{\frac 1\p -1}| \, dz_2\\
	& \quad + |y-z|^{\frac \alpha 2}\brac{\mathcal{M} \lapla{\frac \alpha 4} u(y) + \mathcal{M} \lapla{\frac \alpha 4} u(z) } \avint_{x\triangleright y} ||z-z_2|^{\frac 1\p-1} - |z-x|^{\frac 1\p -1}| \, dz_2 \\
	& \aleq \brac{\mathcal{M} \lapla{\frac \alpha 4} u(x) + \mathcal{M} \lapla{\frac \alpha 4} u(y) }  |x-y|^{\frac \alpha 2 + \gamma_1} |y-z|^{\frac 1\p - \gamma_1 - 1}\\
	& \quad + \brac{\mathcal{M} \lapla{\frac \alpha 4} u(y) + \mathcal{M} \lapla{\frac \alpha 4} u(z) }  |y-z|^{\frac \alpha 2 + \frac 1\p - \gamma_2 -1} |x-y|^{\gamma_2},
	\end{split}
	\]
	respectively,
	\[
	\begin{split}
	H(x,y,z) & \aleq |x-y|^{\frac \alpha 2} \mathcal{M} \mathcal{M} \lapla{\frac \alpha 4} u(x)  \avint_{x\triangleright y} ||z-z_2|^{\frac 1\p-1} - |z-x|^{\frac 1\p -1}| \, dz_2 \\
	& \aleq \mathcal{M} \mathcal{M} \lapla{\frac \alpha 4} u(x) |x-y|^{\alpha + \eps} |y-z|^{\frac 1\p - \frac \alpha 2 - \eps-1}. 
	\end{split}
	\]
	In the case of \underline{$|y-z|\aleq \min \{|x-z|,|x-y| \}$}, we start with the same estimates as in the previous case, but apply \eqref{eq:mvmvf:claim3} for $\gamma_1=\frac \alpha 2 + \eps, \gamma_2=\alpha + \eps \in [0,1]$ as
	\[
	\begin{split}
	& G(x,y,z) \\
	& \aleq |x-y|^{\frac \alpha 2}\brac{\mathcal{M} \lapla{\frac \alpha 4} u(x) + \mathcal{M} \lapla{\frac \alpha 4} u(y) }  \avint_{x\triangleright y} ||z-z_2|^{\frac 1\p-1} - |z-x|^{\frac 1\p -1}| \, dz_2\\
	& \quad + |y-z|^{\frac \alpha 2}\brac{\mathcal{M} \lapla{\frac \alpha 4} u(y) + \mathcal{M} \lapla{\frac \alpha 4} u(z) } \avint_{x\triangleright y} ||z-z_2|^{\frac 1\p-1} - |z-x|^{\frac 1\p -1}| \, dz_2 \\
	& \aleq \brac{\mathcal{M} \lapla{\frac \alpha 4} u(x) + \mathcal{M} \lapla{\frac \alpha 4} u(y) }  |x-y|^{\frac \alpha 2} |y-z|^{\frac 1\p - 1}\\
	& \quad + \brac{\mathcal{M} \lapla{\frac \alpha 4} u(y) + \mathcal{M} \lapla{\frac \alpha 4} u(z) }  |y-z|^{\frac \alpha 2} |y-z|^{\frac 1\p - 1}\\
	& \aleq \brac{\mathcal{M} \lapla{\frac \alpha 4} u(x) + \mathcal{M} \lapla{\frac \alpha 4} u(y) }  |x-y|^{\frac \alpha 2 + \gamma_1} |y-z|^{\frac 1\p - \gamma_1 - 1} \left(\frac{|y-z|}{|x-y|}\right)^{\gamma_1}\\
	& \quad + \brac{\mathcal{M} \lapla{\frac \alpha 4} u(y) + \mathcal{M} \lapla{\frac \alpha 4} u(z) } |x-y|^{\gamma_2} |y-z|^{\frac \alpha 2 + \frac 1\p - \gamma_2 - 1} \left(\frac{|y-z|}{|x-y|}\right)^{\gamma_1+(\gamma_2-\gamma_1)},
	\end{split}
	\]
	respectively,
	\[
	\begin{split}
	H(x,y,z) & \aleq |x-y|^{\frac \alpha 2} \mathcal{M} \mathcal{M} \lapla{\frac \alpha 4} u(x)  \avint_{x\triangleright y} ||z-z_2|^{\frac 1\p-1} - |z-x|^{\frac 1\p -1}| \, dz_2 \\
	& \aleq \mathcal{M} \mathcal{M} \lapla{\frac \alpha 4} u(x) |x-y|^{\frac \alpha 2} |y-z|^{\frac 1\p -1} \\
	& = \mathcal{M} \mathcal{M} \lapla{\frac \alpha 4} u(x) |x-y|^{\alpha + \eps} |y-z|^{\frac 1\p -\frac \alpha 2 - \eps -1}\left(\frac{|y-z|}{|x-y|}\right)^{\frac \alpha 2 + \eps} . 
	\end{split}	
	\]
	For the last case of \underline{$|x-z|\aleq \min \{|y-z|,|x-y| \}$}, we gain by Proposition \ref{pr:estmax} that
	\[
	\begin{split}
	& |u(y)+u(x) -2 u(z)| \\
	& \leq |u(x)-u(y)| + 2 |u(x)-u(z)| \\
	& \aleq |x-y|^{\frac \alpha 2}\brac{\mathcal{M} \lapla{\frac \alpha 4} u(x) + \mathcal{M} \lapla{\frac \alpha 4} u(y) } + |x-z|^{\frac \alpha 2}\brac{\mathcal{M} \lapla{\frac \alpha 4} u(x) + \mathcal{M} \lapla{\frac \alpha 4} u(z) },
	\end{split}	
	\]
	and hence, by \eqref{eq:mvmvf:claim2} for $\gamma_1=\frac \alpha 2 + \eps, \gamma_2=\alpha + \eps \in [0,1]$
	\[
	\begin{split}
	& G(x,y,z) \\
	& \aleq |x-y|^{\frac \alpha 2}\brac{\mathcal{M} \lapla{\frac \alpha 4} u(x) + \mathcal{M} \lapla{\frac \alpha 4} u(y) }  \avint_{x\triangleright y} ||z-z_2|^{\frac 1\p-1} - |z-x|^{\frac 1\p -1}| \, dz_2\\
	& \quad + |x-z|^{\frac \alpha 2}\brac{\mathcal{M} \lapla{\frac \alpha 4} u(x) + \mathcal{M} \lapla{\frac \alpha 4} u(z) } \avint_{x\triangleright y} ||z-z_2|^{\frac 1\p-1} - |z-x|^{\frac 1\p -1}| \, dz_2 \\
	& \aleq \brac{\mathcal{M} \lapla{\frac \alpha 4} u(x) + \mathcal{M} \lapla{\frac \alpha 4} u(y) }  |x-y|^{\frac \alpha 2} |x-z|^{\frac 1\p - 1}\\
	& \quad + \brac{\mathcal{M} \lapla{\frac \alpha 4} u(x) + \mathcal{M} \lapla{\frac \alpha 4} u(z) }  |x-z|^{\frac \alpha 2} |x-z|^{\frac 1\p - 1}\\
	& \aleq \brac{\mathcal{M} \lapla{\frac \alpha 4} u(x) + \mathcal{M} \lapla{\frac \alpha 4} u(y) }  |x-y|^{\frac \alpha 2 + \gamma_1} |x-z|^{\frac 1\p - \gamma_1 - 1} \left(\frac{|x-z|}{|x-y|}\right)^{\gamma_1}\\
	& \quad + \brac{\mathcal{M} \lapla{\frac \alpha 4} u(x) + \mathcal{M} \lapla{\frac \alpha 4} u(z) } |x-y|^{\gamma_2} |x-z|^{\frac \alpha 2 + \frac 1\p - \gamma_2 - 1} \left(\frac{|x-z|}{|x-y|}\right)^{\gamma_1+(\gamma_2-\gamma_1)},
	\end{split}
	\]
	respectively,
	\[
	\begin{split}
	H(x,y,z) & \aleq |x-y|^{\frac \alpha 2} \mathcal{M} \mathcal{M} \lapla{\frac \alpha 4} u(x)  \avint_{x\triangleright y} ||z-z_2|^{\frac 1\p-1} - |z-x|^{\frac 1\p -1}| \, dz_2 \\
	& \aleq \mathcal{M} \mathcal{M} \lapla{\frac \alpha 4} u(x) |x-y|^{\frac \alpha 2} |x-z|^{\frac 1\p -1} \\
	& = \mathcal{M} \mathcal{M} \lapla{\frac \alpha 4} u(x) |x-y|^{\alpha + \eps} |x-z|^{\frac 1\p -\frac \alpha 2 - \eps -1}\left(\frac{|x-z|}{|x-y|}\right)^{\frac \alpha 2 + \eps} . 
	\end{split}	
	\]
	Eventually, note that by Lebesgue differentiation theorem, for a.e.  $x\in\R$,
	\[
	\mathcal{M} \lapla{\frac \alpha 4} u(x) \aleq \mathcal{M} \mathcal{M} \lapla{\frac \alpha 4} u(x).
	\]
\end{proof}

Furthermore, we also need a result inspired by \cite[Proposition 6.4]{S15}.

\begin{proposition} \label{app:pro:FGHKernel}
	Let $F,G,H:\R^n \rightarrow \R_+$, $\alpha\in(0,n)$, $s,\beta \in (0,1)$, $s+\alpha < \beta$, and consider
	\[
		I:= \int_{\R^n}\int_{\R^n}\int_{\R^n} (F(x)+F(y)) (G(x)+G(y)+G(z)) |x-y|^{\alpha-n} H(z) k_{s,\beta}(x,y,z) \, dx\, dy \, dz,
	\]
	where $k_{s,\beta}(x,y,z)$ is of the form \eqref{app:eq:k1}, \eqref{app:eq:k2}, or \eqref{app:eq:k3}. Then 
	\[
		I\leq \int_{\R^n} G \ H\  \lapin{s+\alpha} F +  \int_{\R^n} F\ G \ \lapin{s+\alpha} H +  \int_{\R^n} F\ \lapin{\alpha} G\  \lapin{s}H +  \int_{\R^n} G\ \lapin{\alpha}F \ \lapin{s}H.
	\]
\end{proposition}

\begin{lemma}\label{la:threetermforuquad}
	Let $G,H:\R\rightarrow \R_+$, $\alpha,\beta \in (0,1)$ such that $\beta < \alpha < \frac 1\p$, and 
	\[
	\begin{split}
	I & := \int_{\R} \int_{\R} \int_{\R} G(x) \avint_{x\triangleright y} |u(z_1)-u(y)|^2 \, dz_1 H(z)  \avint_{x\triangleright y} ||z-z_2|^{\frac 1\p-1} - |z-x|^{\frac 1\p -1}| \, dz_2  \\
	&  \quad \cdot |x-y|^{-2+\alpha(q-2)}  \, dx \, dy \, dz.
	\end{split}
	\]
	Then 
	\[
	\begin{split}
	I 
	& \aleq  \int_{\R}  \lapin{\alpha (q-1)+ \beta + \eps -1}G \,  \mathcal{M} \left(\mathcal{M}\lapla{\frac \alpha 2} u \mathcal{M}\lapla{\frac \beta 2} u\right) \lapin{\frac 1\p - \eps} H \\
	& \quad + \int_{\R} G \, \lapin{\alpha (q-1)+ \beta + \eps -1}  \mathcal{M} \left(\mathcal{M}\lapla{\frac \alpha 2} u \mathcal{M}\lapla{\frac \beta 2} u\right) \lapin{\frac 1\p - \eps} H  \\
	& \quad  +  \int_{\R} \lapin{\alpha(q-1)+ \beta + \eps -1} G  \ \mathcal{M} \mathcal{M} \lapla{\frac \alpha 2} u \,  \mathcal{M} \mathcal{M} \lapla{\frac \beta 2} u \ \lapin{\frac 1\p - \eps} H\\
	& \quad +  \int_{\R}  G  \  \lapin{\alpha(q-1)+ \beta + \eps -1} \left(\mathcal{M} \mathcal{M} \lapla{\frac \alpha 2} u \,  \mathcal{M} \mathcal{M} \lapla{\frac \beta 2} u\right) \ \lapin{\frac 1\p - \eps} H 
	\end{split}
	\]
	for any admissible $\eps\in(0,1)$, $\alpha(q-1)+  \beta -1 < \eps < \frac 1\p$. 
\end{lemma}
\begin{proof}
	We begin with treating the term $\avint_{x\triangleright y} |u(z_1)-u(y)|^2 \, dz_1$ by Proposition \ref{pr:estmax} as 
	\[
	\begin{split}
	& \avint_{x\triangleright y} |u(z_1)-u(y)|^2 \, dz_1 \\
	& \aleq \avint_{x\triangleright y} |z_1-y|^{\alpha} \left(\mathcal{M} \lapla{\frac \alpha 2} u(z_1) + \mathcal{M} \lapla{\frac \alpha 2} u(y) \right) |z_1-y|^{\beta} \left(\mathcal{M} \lapla{\frac \beta 2} u(z_1) + \mathcal{M} \lapla{\frac \beta 2} u(y) \right) \, dz_1 \\
	& \leq |x-y|^{\alpha + \beta} \bigg( \avint_{x\triangleright y} \mathcal{M} \lapla{\frac \alpha 2} u(z_1)  \mathcal{M} \lapla{\frac \beta 2} u(z_1) \, dz_1 +\avint_{x\triangleright y} \mathcal{M} \lapla{\frac \alpha 2} u(z_1) \, dz_1   \mathcal{M} \lapla{\frac \beta 2} u(y) \\
	& \quad +    \mathcal{M} \lapla{\frac \alpha 2} u(y) \avint_{x\triangleright y} \mathcal{M} \lapla{\frac \beta 2} u(z_1) \, dz_1 +  \mathcal{M} \lapla{\frac \alpha 2} u(y)  \mathcal{M} \lapla{\frac \beta 2} u(y) \bigg) \\
	& \aleq |x-y|^{\alpha + \beta} \bigg( \mathcal{M} \left( \mathcal{M} \lapla{\frac \alpha 2} u \mathcal{M} \lapla{\frac \beta 2} u \right) (y) +  \mathcal{M} \mathcal{M} \lapla{\frac \alpha 2} u(y)  \mathcal{M} \mathcal{M} \lapla{\frac \beta 2} u(y) \bigg).
	\end{split}
	\]
	We first consider the cases \underline{$|x-y| \aleq \min\{|x-z|,|y-z|\}$} and \underline{$|y-z| \aleq \min\{|x-z|,|x-y|\}$} and observe by \eqref{eq:mvmvf:claim} and \eqref{eq:mvmvf:claim3} that 
	\[
	\avint_{x\triangleright y} ||z-z_2|^{\frac 1\p-1} - |z-x|^{\frac 1\p -1}| \, dz_2 \aleq |x-y|^\eps |y-z|^{\frac 1\p - \eps -1}
	\]
	for an admissible $\eps>0$. Thereby, we get
	\[
	\begin{split}
	& \iiint_{\R^3} G(x)\ \mathcal{M} \left( \mathcal{M} \lapla{\frac \alpha 2} u \mathcal{M} \lapla{\frac \beta 2} u \right) (y) \  H(z)  |x-y|^{-2+\alpha(q-1) + \beta + \eps}  |y-z|^{\frac 1\p - \eps -1}  \, dx \, dy \, dz \\
	&  + \iiint_{\R^3} G(x) \ \mathcal{M} \mathcal{M} \lapla{\frac \alpha 2} u(y) \,  \mathcal{M} \mathcal{M} \lapla{\frac \beta 2} u(y) \  H(z)  |x-y|^{-2+\alpha(q-1) + \beta + \eps}  |y-z|^{\frac 1\p - \eps -1}  \, dx \, dy \, dz  \\
	& \approx \int_{\R} \lapin{\alpha(q-1)+ \beta + \eps -1} G (y) \ \mathcal{M} \left( \mathcal{M} \lapla{\frac \alpha 2} u \mathcal{M} \lapla{\frac \beta 2} u \right) (y) \ \lapin{\frac 1\p - \eps} H(y) \, dy  \\
	& +  \int_{\R} \lapin{\alpha(q-1)+ \beta + \eps -1} G (y) \ \mathcal{M} \mathcal{M} \lapla{\frac \alpha 2} u(y) \,  \mathcal{M} \mathcal{M} \lapla{\frac \beta 2} u(y) \ \lapin{\frac 1\p - \eps} H(y) \, dy.\\
	\end{split}
	\]
	For the case \underline{$|x-z| \aleq \min\{|y-z|,|x-y|\}$}, we have by \eqref{eq:mvmvf:claim2}
	\[
	\avint_{x\triangleright y} ||z-z_2|^{\frac 1\p-1} - |z-x|^{\frac 1\p -1}| \, dz_2 \aleq |x-y|^\eps |x-z|^{\frac 1\p - \eps -1}
	\]
	for an admissible $\eps>0$, which implies 
	\[
	\begin{split}
	& \iiint_{\R^3} G(x) \ \mathcal{M} \left( \mathcal{M} \lapla{\frac \alpha 2} u \mathcal{M} \lapla{\frac \beta 2} u \right) (y) \  H(z)  |x-y|^{-2+\alpha(q-1) + \beta + \eps}  |x-z|^{\frac 1\p - \eps -1}  \, dx \, dy \, dz \\
	&  + \iiint_{\R^3} G(x) \ \mathcal{M} \mathcal{M} \lapla{\frac \alpha 2} u(y) \,  \mathcal{M} \mathcal{M} \lapla{\frac \beta 2} u(y) \  H(z)  |x-y|^{-2+\alpha(q-1) + \beta + \eps}  |x-z|^{\frac 1\p - \eps -1}  \, dx \, dy \, dz  \\
	& \approx \int_{\R} G (x) \  \lapin{\alpha(q-1)+ \beta + \eps -1} \mathcal{M} \left( \mathcal{M} \lapla{\frac \alpha 2} u \mathcal{M} \lapla{\frac \beta 2} u \right) (x) \ \lapin{\frac 1\p - \eps} H(x) \, dy  \\
	& +  \int_{\R}  G (x) \ \lapin{\alpha(q-1)+ \beta + \eps -1} \left( \mathcal{M} \mathcal{M} \lapla{\frac \alpha 2} u(x) \,  \mathcal{M} \mathcal{M} \lapla{\frac \beta 2} u (x)\right)  \ \lapin{\frac 1\p - \eps} H(x) \, dy.
	\end{split}
	\]
\end{proof}

\bibliographystyle{abbrv}%
\bibliography{bib}%

\end{document}